\documentclass[psamsfonts]{amsart}
\usepackage[utf8]{inputenc}
\usepackage{amsfonts}
\usepackage{hyperref}
\hypersetup{
	pdftitle={Endoscopy for metaplectic affine Hecke categories},
	pdfsubject={Mathematics},
	pdfauthor={Gurbir Dhillon, Yau Wing Li, Zhiwei Yun, Xinwen Zhu},
	pdfkeywords={Hecke algebra, arXiv, Preprint}
}
\usepackage{amsmath}
\usepackage{xcolor}
\usepackage{amsthm}
\usepackage{pdflscape}
\usepackage{mathrsfs}
\usepackage{bbold}
\usepackage{stmaryrd}
\usepackage[skip=5pt, indent=10pt]{parskip}
\usepackage{dsfont}
\usepackage{tikz}
\usepackage{amssymb}

\usepackage[all]{xy}

\usepackage{color}
\usepackage[shortlabels]{enumitem}
\usepackage{hyperref}
\usepackage{pgf,tikz}
\usepackage{mathrsfs}
\usetikzlibrary{arrows}
\usepackage{tikz-cd}
\usepackage{graphicx}
\usepackage{filecontents}

\usepackage[mathscr]{euscript}
\newtheorem{thm}{Theorem}[section]
\newtheorem{cor}[thm]{Corollary}
\newtheorem{prop}[thm]{Proposition}
\newtheorem{lem}[thm]{Lemma}
\newtheorem{lemma}[thm]{Lemma}
\newtheorem{conj}[thm]{Conjecture}

\theoremstyle{definition}
\newtheorem{defn}[thm]{Definition}

\newtheorem{exmp}[thm]{Example}

\newtheorem{con}[thm]{Construction}

\theoremstyle{remark}
\newtheorem{rem}[thm]{Remark}

\makeatletter
\let\c@equation\c@thm
\makeatother
\numberwithin{equation}{section}

\newcommand{\R}{\mathbb{R}} 
\newcommand{\Z}{\mathbb{Z}}   
\newcommand{\p}{\mathbb{P}}   
\newcommand{\Q}{\mathbb{Q}}

\newcommand{\V}{\mathbb{V}}

\newcommand{\sF}{\mathcal{F}}
\newcommand{\sG}{\mathcal{G}}
\newcommand{\sH}{\mathcal{H}}

\newcommand{\sM}{\mathscr{M}}

\newcommand{\G}{\mathbb{G}}

\newcommand{\A}{\mathbb{A}}
\newcommand{\X}{\mathbb{X}}

\newcommand{\pha}{\phantom{ }}

\newcommand{\bsl}{\backslash}

\newcommand{\extw}{\widetilde{W}}

\newcommand{\gm}{\mathbb{G}_m}
\newcommand{\ga}{\mathbb{G}_a}

\newcommand{\cmc}{{}_{\chi}\cM_{\chi'}}
\newcommand{\cwc}{{}_\chi\tilW_{\chi'}}
\newcommand{\ctc}{{}_{\chi}\cT_{\chi'}}
\newcommand{\cvc}{{}_{\chi}\mathbb{V}_{\chi'}}
\newcommand{\xvx}{\mathbb{V}_{\Xi}}
\newcommand{\xtx}{\cT_{\Xi}}

\newcommand{\cht}{\check{\theta}}

\newcommand{\dcm}{\delta^{\chi\mon}}
\newcommand{\fm}{\mathfrak{m}}
\newcommand{\I}{\on{IndCoh}}

\newcommand{\Sph}{\on{Sph}}
\newcommand{\Sphc}{\Sph_{G, \chi_c}}
\newcommand{\Sphh}{\Sph_{H}}

\newcommand{\gmc}{\gm^{\cen}}

\newcommand{\ra}{\rightarrow}

\newcommand{\sur}{\twoheadrightarrow}
\newcommand{\inj}{\hookrightarrow}

\newcommand{\xra}{\xrightarrow}

\newcommand{\on}{\operatorname}

\newcommand{\tLG}{\widetilde{LG}}

\newcommand{\tI}{\widetilde{I}}

\newcommand{\Tv}{T^\vee}

\DeclareMathOperator{\Aut}{Aut}

\DeclareMathOperator{\id}{id}
\DeclareMathOperator{\spec}{Spec}

\DeclareMathOperator{\ad}{ad}
\DeclareMathOperator{\mult}{mult}

\DeclareMathOperator{\Hom}{Hom}
\DeclareMathOperator{\End}{End}

\DeclareMathOperator{\Ad}{Ad}
\DeclareMathOperator{\aff}{aff}
\DeclareMathOperator{\Gal}{Gal}
\DeclareMathOperator{\Spec}{Spec}

\DeclareMathOperator{\Sym}{Sym}

\DeclareMathOperator{\Ch}{Ch}
\DeclareMathOperator{\CH}{CH}

\DeclareMathOperator{\Av}{Av}
\DeclareMathOperator{\Ind}{Ind}
\DeclareMathOperator{\Res}{Res}

\DeclareMathOperator{\Coh}{Coh}

\DeclareMathOperator*{\uralim}{\underrightarrow\lim}

\DeclareMathOperator{\pt}{pt}

\DeclareMathOperator{\RHom}{\textbf{R} \kern -2pt \Hom}
\DeclareMathOperator{\Rhom}{\textbf{R}\kern -2pt \hom}
\DeclareMathOperator{\RuHom}{\textbf{R} \underline{\kern -0pt \Hom}}

\DeclareMathOperator{\cen}{cen}

\DeclareMathOperator{\Gr}{Gr}

\DeclareMathOperator{\red}{red}

\DeclareMathOperator{\colim}{colim}
\DeclareMathOperator{\std}{\Delta}
\DeclareMathOperator{\costd}{\nabla}

\newcommand\cC{\mathcal{C}}
\newcommand\cD{\mathcal{D}}
\newcommand\cE{\mathcal{E}}
\newcommand\cF{\mathcal{F}}
\newcommand\cG{\mathcal{G}}
\newcommand\cH{\mathcal{H}}

\newcommand\cK{\mathcal{K}}
\newcommand\cL{\mathcal{L}}
\newcommand\cM{\mathscr{M}}

\newcommand\cO{\mathcal{O}}

\newcommand\cS{\mathscr{S}}
\newcommand\cT{\mathscr{T}}

\newcommand\cV{\mathcal{V}}

\newcommand{\B}{\mathbb{B}}
\renewcommand{\H}{\mathbb{H}}

\newcommand{\CC}{\mathbb{C}}
\newcommand{\ZZ}{\mathbb{Z}}
\newcommand{\QQ}{\mathbb{Q}}
\newcommand{\RR}{\mathbb{R}}
\newcommand{\FF}{\mathbb{F}}

\newcommand{\LL}{\mathbb{L}}
\newcommand{\cc}{\mathbb{c}}
\newcommand{\bgg}{\mathbb{g}}

\newcommand{\frf}{\mathfrak{f}}

\newcommand{\Fun}{\operatorname{Fun}}

\renewcommand{\H}{\mathscr{H}}
\renewcommand{\sH}{\mathscr{H}}

\newcommand{\stMst}{{^{\st}\hspace{-.2mm}}\H^{\st}}
\newcommand{\ssMss}{{^{\st}\hspace{-.2mm}}\sM^{\st}}
\newcommand{\stM}{{^{\st \hspace{-.2mm}}}\H}

\newcommand{\sds}{{}^{\st}\hspace{-.5mm}\delta^{\st}}
\newcommand{\usds}{{}^{\st}\hspace{-.2mm}\delta^{\st}}
\newcommand{\Mst}{\H^{\st}}

\newcommand{\ktc}{\on{K}^b \xtx}
\newcommand{\st}{\textup{s}}

\newcommand{\halpha}{\check{\alpha}}

\newcommand{\bG}{{\mathbb G}}

\newcommand{\coker}{\on{coker}}

\newcommand{\bs}{\backslash}
\newcommand{\isom}{\xar{\sim}}
\newcommand{\Oblv}{\mathrm{Oblv}}
\newcommand{\mon}{\text{\textendash}mon}

\newcommand{\Lie}{\on{Lie}}

\def\bF{\mathbf{F}}
\def\bG{\mathbf{G}}

\newcommand\fra{\mathfrak{a}}

\newcommand\frc{\mathfrak{c}}

\newcommand\fre{\mathfrak{e}}
\newcommand\frg{\mathfrak{g}}
\newcommand\frh{\mathfrak{h}}
\newcommand\fri{\mathfrak{i}}

\newcommand\frl{\mathfrak{l}}
\newcommand\frakm{\mathfrak{m}}

\newcommand\frs{\mathfrak{s}}
\newcommand\frt{\mathfrak{t}}
\newcommand\fru{\mathfrak{u}}

\newcommand\frz{\mathfrak{z}}

\newcommand\tilI{\widetilde{I}}

\newcommand\tilW{\widetilde{W}}

\newcommand\GL{\textup{GL}}

\newcommand\SL{\textup{SL}}

\newcommand\SO{\textup{SO}}

\newcommand\Spin{\textup{Spin}}

\newcommand\PSp{\textup{PSp}}

\newcommand{\Gm}{\mathbb{G}_m}
\newcommand\Ga{\mathbb{G}_a}

\newcommand{\der}{\textup{der}}

\newcommand\xch{\mathbb{X}^*}
\newcommand\xcoch{\mathbb{X}_*}

\numberwithin{equation}{section}

\newcommand{\ft}{\mathfrak{t}}

\newcommand{\fh}{\mathfrak{h}}

\newcommand{\Vect}{\on{Vect}}

\newcommand{\Dmod}{\on{D-mod}}
\renewcommand{\mod}{\text{\textendash}\operatorname{mod}}

\newcommand{\xar}[1]{\xrightarrow{#1}}

\renewcommand{\mon}{\on{-mon}}

\renewcommand{\and}{\quad \on{and} \quad}

\newcommand{\chG}{\check{G}}

\newcommand{\chI}{\check{I}}

\newcommand{\ckappa}{\check{\kappa}}
\renewcommand{\ft}{\mathfrak{t}}

\newcommand{\cLambda}{\check{\Lambda}}
\newcommand{\clambda}{\check{\lambda}}
\newcommand{\cmu}{\check{\mu}}
\newcommand{\cPhi}{\check{\Phi}}
\newcommand{\chGG}{\check{\GG}}
\newcommand{\chBB}{\check{\BB}}
\newcommand{\chAA}{\check{\AA}}
\newcommand{\chPP}{\check{\PP}}
\newcommand{\chgg}{\check{\gg}}
\newcommand{\chbb}{\check{\bb}}
\newcommand{\chaa}{\check{\aa}}
\newcommand{\chHH}{H^\vee}
\newcommand{\chhh}{\frh^\vee}

\newcommand{\incl}{\hookrightarrow}

\newcommand{\surj}{\twoheadrightarrow}

\newcommand{\Ql}{\Q_{\ell}}
\newcommand{\Qlbar}{\overline{\Q}_\ell}

\renewcommand{\j}[1]{\langle{#1}\rangle}
\newcommand{\wt}[1]{\widetilde{#1}}

\newcommand\quash[1]{}

\newcommand{\ov}{\overline}

\newcommand{\tl}[1]{[\![#1]\!]}
\newcommand{\lr}[1]{(\!(#1)\!)}
\newcommand\sss{\subsubsection}
\newcommand\xr{\xrightarrow}
\newcommand\op{\oplus}
\newcommand\ot{\otimes}
\newcommand\bt{\boxtimes}

\renewcommand\c\circ

\renewcommand\a\alpha
\renewcommand\b\beta
\renewcommand\d\delta
\newcommand{\e}{\epsilon}
\renewcommand{\th}{\theta}

\newcommand{\io}{\iota}

\renewcommand\r\rho
\newcommand{\s}{\sigma}
\newcommand{\Sig}{\Sigma}

\newcommand{\z}{\zeta}

\renewcommand{\l}{\lambda}
\renewcommand{\L}{\Lambda}
\newcommand{\om}{\omega}
\newcommand{\Om}{\Omega}

\DeclareMathOperator{\tors}{tors}

\DeclareMathOperator{\AS}{AS}
\DeclareMathOperator{\Bun}{Bun}

\newcommand\chk{\textup{char}(k)}

\newcommand\GG{\mathbb{G}}
\newcommand\TT{\mathbb{T}}
\newcommand\WW{\mathbb{W}}
\renewcommand\AA{\mathbb{A}}
\newcommand\PP{\mathbb{P}}

\newcommand\BB{\mathbb{B}}
\newcommand\II{\mathbb{I}}

\renewcommand\SS{\mathbb{S}}

\renewcommand\gg{\mathbb{g}}
\renewcommand\aa{\mathbb{a}}
\newcommand\bb{\mathbb{b}}

\newcommand\opp{\on{op}}
\newcommand\Wa{W_{\aff}}
\newcommand\Grot{\G_{m}^{\textup{rot}}}
\newcommand\Nm{\on{Nm}}

\newcommand\all{\textup{all}}

\renewcommand{\ge}{\geqslant}
\renewcommand{\le}{\leqslant}
\renewcommand{\geq}{\geqslant}
\renewcommand{\leq}{\leqslant}
\newcommand\Aff{\textup{Aff}}

\newcommand{\Corr}{\textup{Corr}}

\renewcommand{\Im}{\textup{Im}}

\newcommand\lmod{\mathrm{-mod}}

\newcommand\Out{\textup{Out}}
\newcommand{\Perf}{\textup{Perf}}

\newcommand{\Pic}{\textup{Pic}}
\newcommand\pr{\textup{pr}}

\newcommand\Proj{\textup{Proj}}

\newcommand{\Modu}{\operatorname{Mod}}

\newcommand{\lincat}{\mathrm{Lincat}}

\newcommand{\calg}{\mathrm{CAlg}}

\newcommand{\AlgSp}{\mathrm{AlgSp}}
\newcommand{\indalg}{\mathrm{IndAlgSp}}
\newcommand{\prestk}{\mathrm{PreStk}}

\newcommand{\mono}{\mathrm{mon}}
\newcommand{\qcqs}{\mathrm{qcqs}}

\newcommand{\verti}{\mathrm{vert}}

\newcommand{\fint}{\mathrm{fp}}
\newcommand{\indfp}{\mathrm{indfp}}

\newcommand{\indcoh}{\mathrm{IndCoh}}

\newcommand{\Lpol}{L_{\mathrm{pol}}}

\title{Endoscopy for metaplectic affine Hecke categories}

\author{Gurbir Dhillon, Yau Wing Li, Zhiwei Yun, and Xinwen Zhu}

\begin{document}
	\begin{abstract}
		For a possibly twisted loop group $LG$, and any character sheaf of its Iwahori subgroup, we identify the associated affine Hecke category with a combinatorial category of Soergel bimodules. In fact, we prove such results for affine Hecke categories arising from central extensions of the loop group $LG$. Our results work for mod $\ell$ or integral $\ell$-adic coefficients.

  As applications, we obtain endoscopic equivalences between affine Hecke categories, including the derived Satake equivalence for metaplectic groups, and a series of conjectures by Gaitsgory in quantum geometric Langlands.

	\end{abstract}
	
	\date{\today}
	\maketitle
	\tableofcontents
	
\section{Introduction}

Let $G$ be a quasi-split group over a local function field $K=k\lr{t}$. Affine Hecke categories play a basic role in the categorical representation theory of $G$ entirely similar to the role of affine Hecke algebras in the representation theory of $p$-adic groups, and indeed these two stories are related via categorical trace.

In this paper we establish some basic structural results about monodromic affine Hecke categories, in particular their relation to ones with trivial monodromy, and develop some first applications; further consequences will be described in a sequel. In particular, the present paper is a contribution to the depth zero part of the local geometric and local quantum Langlands correspondences.

In the next subsection, we provide some contextual and motivating discussion for our results, which will be described precisely in Sections \ref{ss:statements} and \ref{ss:apps} below. The reader may wish to skip directly there. 

\subsection{Some context and motivation}

\sss{Affine Hecke algebras} The Iwahori--Matsumoto affine Hecke algebra \cite{IM} controls the block of unramified principal series for the $p$-adic group, i.e., the Bernstein block containing the trivial representation. For general depth zero blocks, one has analogous Hecke algebras, whose structure was determined by Morris \cite{M}. For positive depth blocks, thanks to works of Howe, Adler, Bushnell, Kutzko, Moy, Prasad, Reeder, Roche, Yu, Kim, ..., culminating in the recent paper of Adler, Fintzen, Mishra and Ohara \cite{AFMO}, it is known that suitably defined Hecke algebras are always isomorphic to a depth zero Hecke algebra for a twisted Levi subgroup; therefore the study of general blocks is reduced to the depth zero case \cite{AFMO}, assuming that the residue characteristic $p$ does not divide the order of the absolute Weyl group of $G$. 

A new subtlety when passing from the usual affine Hecke algebra to a general depth zero Hecke algebra is that the latter may involve a two cocycle on the group $\Om$ controlling the non-Coxeter part of the Hecke algebra.

Now consider the more general setting of representations of {\em metaplectic covers} of $p$-adic groups. Here, one has a good understanding of the unramified principal series block, suggested by Kazhdan and proven by Savin \cite{Sa}. For general blocks many basic questions remain open.  As a first indication - whether one encounters nontrivial two cocycles for principal series, even in depth zero, was as far as we can tell an open question.

\sss{Affine Hecke categories} Let us now turn to the geometric theory. The affine Hecke category that geometrizes the Iwahori-Matsumoto affine Hecke algebra consists of sheaves on the loop group $LG$ equivariant with respect to the Iwahori subgroup on both sides. Here, the Kazhdan-Lusztig basis for the affine Hecke algebra correspond to the intersection complexes of Iwahori double cosets \cite{kazhdan1980schubert}. Soergel \cite{soergel1990kategorie} introduces a certain bimodules over the polynomial ring that allows one to recover the affine Hecke category via Coxeter combinatorics. A celebrated theorem of Bezrukavnikov \cite{B} gives a description of the affine Hecke category in terms of coherent sheaves on the Steinberg variety of the Langlands dual group, categorifying the isomorphism on the level of Grothendieck groups by Kazhdan and Lusztig \cite{KL}.

The next step is to consider geometric versions of the affine Hecke category attached to the depth zero principal series blocks of $p$-adics. They consist of sheaves on the loop group $LG$ that are equivariant on both sides with respect to the Iwahori subgroup $I$ together with a character sheaf $\chi$ (rank one local system) on its torus quotient $\AA$. The metaplectic analog of such a category replaces the loop group by its Kac-Moody central extension. We call all such monoidal categories {\em monodromic affine Hecke categories}, and they are the focus of the present paper. 

Based on the work of Morris, we should expect that every monodromic affine Hecke category to be the semi-direct product of the affine Hecke category associated to a smaller {\em endoscopic} group with trivial monodromy, along with an action of a semisimple monoidal category $\cM_\Om$ whose simple objects are indexed by a discrete group $\Omega$. The most optimistic hope would be that $\cM_\Om$ is monoidally equivalent to $\Om$-graded vector spaces. However, a priori the monoidal structure on $\cM_\Om$ may involve twisting by a $3$-cocycle of $\Om$. This is related to the possible $2$-cocycles that show up in the depth zero affine Hecke algebras.

One of our main results states that the above expectation for the twisted affine Hecke category is indeed true. In particular, the $3$-cocycle on $\Om$ is trivializable. One may compare this to a result of Roche \cite{Ro} that shows that for principal series of $p$-adic groups, the $2$-cocycle is indeed trivializable.

When we turn to geometrizations of general blocks, it is expected that every positive depth block should again be equivalent to a depth zero block for a twisted Levi subgroup of $G$, provided the characteristic of $p$ is not too small (see J.Xia's thesis \cite{Xia} for a special case, and D.Yang \cite{Yang} for a more general treatment). Moreover, as we will mention in Section \ref{sss:trace}, a miracle of taking categorical trace implies that general depth zero blocks, and not just the principal series ones, can also be recovered from the monodromic affine Hecke categories.  Therefore we believe that the geometric study of all blocks of a $p$-adic group reduces to the study of the monodromic affine Hecke categories.

\sss{From Hecke categories to representations}\label{sss:trace}

Fix a finite field $\mathbb{F}_q$ along with an algebraic closure $\bF$. For a reductive group $H$ defined over $\mathbb{F}_q$, 
by the works \cite{BFO, BN, LuTrunc} in various sheaf-theoretic contexts, the cocenter of the affine Hecke category $D(B_{\bF} \bs H_{\bF} / B_{\bF})$ recovers the category of {unipotent character sheaves}.  If instead one takes the Frobenius-twisted cocenter, Lusztig \cite{LuUnip} shows that one recovers all unipotent representations of $H(\mathbb{F}_q)$, and not just principal series representations.

Recall that the set of irreducible representations of $H(\mathbb{F}_q)$ breaks into Lusztig series, indexed essentially by semisimple conjugacy classes in the dual group $\check{H}(\FF_q)$, with the unipotent representations corresponding to the identity element. Lusztig showed that, at least when $H$ has connected center, the representations in the Lusztig series corresponding to the conjugacy class of $\s\in \check{H}(\FF_q)$ were in canonical bijection with unipotent representations of the centralizer $C_{\check{H}}(\s)(\FF_q)$. In the last statement, one can replace the centralizer of $\s$ by its Langlands dual, i.e., the endoscopic group $H_\s$, since the set of unipotent representations for a group and its dual are in canonical bijection.

It was shown by Lusztig--Y. \cite{LY} that the above result can be deduced from a categorical statement at the level of the Hecke categories. Namely, if we denote  by $\chi_\sigma$ the character sheaf on the abstract Cartan $T$ of $H$ corresponding to $\sigma$, as well as its pullback to a Borel $B$ of $H$, they obtained a monoidal equivalence
$$D((B_\bF, \chi_\sigma) \bs H_{\bF} / (B_{\bF}, \chi_\sigma)) \simeq D(B_{\sigma,\bF} \bs H_{\sigma,\bF} / B_{\sigma, \bF}). $$
Here $B_\s$ is a Borel subgroup of $H_\s$. By passing to Frobenius-twisted cocenters, this recovers the previous bijection between $\s$-Lusztig series of $H$ and unipotent representations of $H_\s$.

The trace construction has been carried out for $p$-adic groups  in recent work of Z. \cite{Zh}. It is shown in \emph{loc. cit.} that, among other things, the category of depth zero representations of $p$-adic groups can be obtained from the Frobenius-twisted cocenter of the twisted affine Hecke category (a similar construction for unipotent representations was previously obtained by Hemo--Z.). 

Lusztig \cite{Lupadic} gives a classification of unipotent representations of a quasi-split $p$-adic group in terms of Langlands dual data. 
For general depth zero representations of $G(\FF_q\lr{t})$, it is natural to expect a canonical bijection between the representations in a fixed Lusztig series and unipotent representations of a certain {\em endoscopic loop group}. For example, when $G$ is split, its Lusztig series corresponding to the semisimple conjugacy class $\s$ of $\check{G}(\FF_q)$ should be in bijection with the unipotent representations of $G_\s(\mathbb{F}_q(\!(t)\!))$ for the usual endoscopic group $G_\s$ attached to $(G,\s)$. More generally, for metaplectic covers,  one should replace the Langlands dual group $\check{H}$ with the metaplectic dual group \cite{FL, We}.

A major motivation for the present work is that its results contain many, but not all, of the steps necessary to establish such a correspondence, and its geometric refinement as \cite{LY}, for the depth zero representations for quasi-split groups and their metaplectic covers. Further results in this vein will appear in the sequel to this work.

\subsection{Statement of main results}
\label{ss:statements}

\sss{} Let $k$ be an algebraically closed field of characteristic zero or an algebraic closure of a finite field, and write $K := k(\!(t)\!).$ 

\sss{} Let $G$ be a connected, reductive, and quasi-split group over $K$. The set of $K$-points of $G$ are naturally the $k$-points of a group ind-scheme over $k$, the loop group $LG$. We write $I \subset LG$ for an Iwahori subgroup.  We consider a Kac--Moody central extension 
$$1 \rightarrow \mathbb{G}_m \rightarrow \widetilde{LG} \rightarrow LG \rightarrow 1,$$
cf. Section \ref{ss: cent ext} for the precise type of central extensions we consider. Write $\widetilde{I}\subset \wt{LG}$ for the preimage of $I$.

\sss{}\label{sss:intro sheaves} 
In this introduction, when discussing categories of sheaves on a stack $X$ over $k$, we will mean any of the following stable $\infty$-categories linear over a coefficient field $E$:

\begin{enumerate}[(i)]

\item For any base field $k$, consider constructible complexes on the \'etale site of $X$ with coefficient field $E$ an algebraic extension of either $\mathbb{F}_\ell$ or $\Ql$. Here $\ell$ is a prime number different from $\on{char}(k)$.

\item When $\on{char}(k) = 0$,  consider  complexes of D-modules, so that the coefficient field $E=k$.

\end{enumerate}

In the main body of the paper, we allow the coefficients $E$ to be a DVR in the context (i), see \S\ref
{SS: sheaf theory}.

\sss{} For a torus $H$ over $k$, a {\em character sheaf} $\chi$ on $H$ is a tame rank one local system on $H$ with a trivialization at the identity.  If the torus $H$ acts on a stack $X$, we may form the category $D((H,\chi)\bs X)$ of (strictly) $(H,\chi)$-equivariant sheaves on $X$. The category of $(H,\chi)$-monodromic sheaves on $X$ is the full subcategory of $D(X)$ generated under colimits by the essential image of the forgetful functor $D((H,\chi)\bs X)\to D(X)$. (In the main body of the article, we will use a different but equivalent definitions of these categories.)

\sss{} Let $I^+$ be the pro-unipotent radical of $\wt I$, and let $\AA=I/I^+$, $\wt\AA=\wt I/I^+$ be the quotient tori. Let $\chi$ be  a character sheaf on $\wt \AA$. Our main object of interest is the monoidal category of bi-$(\widetilde{I}, \chi)$-monodromic sheaves on $\widetilde{LG}$, which we denote by $_{\chi}\cM_{\chi}$. Concretely, $_{\chi}\cM_{\chi}$ is the full subcategory of sheaves on
\begin{equation}
    Z=I^+\bs \wt{LG}/I^+
\end{equation}
generated under colimits by $(\wt\AA,\chi)$-equivariant objects under both the left and right translations by $\wt\AA$.

Our basic structural theorem on ${}_\chi\cM_\chi$ asserts that it may be reconstructed as a monoidal category from certain combinatorial data naturally attached to it. 

\sss{}  To describe these combinatorial data, we first recall the description of ${}_\chi\cM_\chi$ when $\chi$ is trivial given in \cite{BY}.

Consider the extended affine Weyl group $\widetilde{W}$ associated to $\widetilde{LG}$ that acts on the cocharacter lattice $\xcoch(\wt\AA)$ of $\wt\AA=\wt I/I^+$. It has a normal subgroup $W_{\aff} \subset \wt{W}$ generated by affine simple reflections $S_{\aff}$. We recall that $(W_{\aff}, S_{\aff})$ is a Coxeter system. Consider the quadruple $$(\wt{W}, W_{\aff}, S_{\aff}, \xcoch(\wt\AA)).$$
  
Results from \cite{BY} show that when $\chi$ is trivial, one can recover ${}_\chi\cM_\chi$ (for $\chi$ trivial) as a monoidal category from the above quadruple.

\sss{} For a nontrivial character sheaf $\chi$, we will similarly attach a quadruple to ${}_\chi\cM_\chi$
\begin{equation} \label{e:mondata}(\wt{W}_\chi, \wt{W}_\chi^\circ, S_\chi, \frf_\chi).
\end{equation}

The action of $\widetilde{W}$ on $\wt\AA$ permutes the isomorphism classes of character sheaves on $\wt \AA$, and we denote by $\widetilde{W}_\chi$ the stabilizer of $\chi$.  Consider further the subgroup $$\widetilde{W}^\circ_\chi \subset \widetilde{W}_\chi$$ generated by the reflections $s_{\halpha}$ in $W_{\aff}$ corresponding to those affine coroots $\halpha$ such that the pullback of $\chi$ along $\halpha: \Gm\to \wt\AA$ is trivial. This is a normal subgroup equipped with a set of reflections $S_\chi$ making it a Coxeter system  $(\wt{W}^\circ_\chi, S_\chi)$. 

The last item $\frf_\chi$ in the quadruple \eqref{e:mondata} is a more flexible replacement for $\xcoch(\wt\AA)$, namely it is the formal group of the dual torus of $\wt\AA$ over $E$. It carries an action of $\wt{W}_\chi$ given by the restriction of the action of $\wt{W}$.

The formal scheme $\frf_\chi$ is unchanged if $\wt\AA$ undergoes an isogeny of degree prime to $\textup{char}(E)$. 
This flexibility is what ultimately will allow us to obtain applications such as the metaplectic derived Satake equivalence (Theorem \ref{t:dersat}).

\sss{} 
Given the quadruple $\eqref{e:mondata}$, we associate to it a monoidal category ${}_\chi\cS_\chi$ as follows.

First, let $\on{SBim}^\circ_\chi\subset \indcoh(\frf_\chi^2)$ be essentially the category of usual Soergel bimodules: it is monoidally generated by the Bott--Samelson sheaves $\beta(s)$ ($s \in S_\chi$) supported on the union of the diagonal and the graph of $s$ on $\frf_\chi$.

From $\on{SBim}^\circ_\chi$, we then construct a presentable stable monoidal $\infty$-category ${}_{\chi}\cS^\circ_\chi$; this is a certain combinatorially constructed localization of the ind-completion of the bounded homotopy category of $\on{SBim}^\circ_\chi$.

In general, we have
\begin{equation}\label{Wa rtimes Om}
    \wt{W}_\chi \simeq \Omega_\chi \ltimes \wt{W}_\chi^\circ
\end{equation}
where $\Om_\chi$ is the subgroup of $\wt W_\chi$ that normalizes the subset $S_\chi$. Let $\Vect_{\Om_\chi}$ denote the monoidal category of $\Om_\chi$-graded $E$-vector spaces. The action of $\Om_\chi$ on $(W^\circ_\chi, S_\chi)$ induces an action of $\Vect_{\Om_\chi}$ on ${}_\chi \cS^\circ_\chi$. We form the semi-direct product monoidal category
$${}_\chi \cS_\chi := \Vect_{\Omega_\chi} \ltimes {}_\chi \cS^\circ_\chi.$$

Writing ${}_\chi \cS_\chi $ as a direct sum $\oplus_{\om \in \Om_\chi}{}_\chi \cS^\circ_\chi$, we define $\on{SBim}_\chi\subset{}_\chi \cS_\chi $ to be the full subcategory equal to $\oplus_{\om \in \Om_\chi}\on{SBim}^\circ_\chi$ under the prior decomposition. We call objects in $\on{SBim}_\chi$ {\em Soergel bimodules}.

Our basic structure theorem then reads as follows. 

\begin{thm} \label{t:maintheo} \label{t:mainthm} Assume that the characteristic of $E$ is not two. There is an equivalence of monoidal $\infty$-categories
\begin{equation}\label{main eq}
    {}_{\chi}\cM_\chi \simeq {}_{\chi}\cS_{\chi}
\end{equation}
that sends cofree tilting sheaves to Soergel bimodules. 

In particular, the monodromic affine Hecke category ${}_\chi \cM_\chi$ may be reconstructed from the combinatorial datum $(\wt{W}_\chi, \wt{W}_\chi^\circ, S_\chi, \frf_{\chi})$. 
\end{thm}

\begin{rem}

\begin{enumerate}
    \item Cofree tilting sheaves are counterparts  of the free-monodromic tilting sheaves introduced in \cite{BY},  which are more convenient to work with in our categorical setting (presentable stable categories). Cofree tilting sheaves first appeared in the work \cite{CD} and also played an important role in \cite{Zh}.
    \item The equivalence \eqref{main eq} matches natural highest weight structures on both sides, and in particular exchanges standard sheaves in ${}_\chi\cM_\chi$ and Rouquier complexes in ${}_\chi\cS_\chi$ indexed by the same element in $\tilW_\chi$. 
    \item In \S\ref{s:mon vs strict} we will explain how to pass back and forth between strictly equivariant and monodromic categories via categorical constructions. The above result has a counterpart for the strictly equivariant affine Hecke category as well.  
    \item We believe that our results also hold in the Betti setting (i.e., $k=\CC$ with an arbitrary coefficient field $E$), and most of the argument goes through without change except those in Section \ref{s:Soergel}, where one can use "Kirillov model" to replace the use of the Artin-Schreier sheaf.
    \item When the characteristic of $E$ is two, the definitions of Bott-Samelson sheaves and ${}_{\chi}\cS_{\chi}$ must be modified for the theorem to hold. See Remark~\ref{rem:true BS sheaf} for details.
\end{enumerate}

\end{rem}

\sss{}  An immediate consequence of Theorem \ref{t:maintheo} for the internal structure of ${}_\chi \cM_\chi$ is a monoidal equivalence
\begin{equation}\label{intro all blocks}
    {}_\chi \cM_\chi \simeq \Vect_{\Omega_\chi} \ltimes {}_\chi \cM^\circ_\chi,
\end{equation}
where ${}_\chi \cM^\circ_\chi$ is the block (direct summand) of ${}_\chi \cM_\chi$ containing the monoidal unit. 

We emphasize that the equivalence  \eqref{intro all blocks} provides the simplest possible answer for the basic structure theory of monodromic affine Hecke categories. Namely, it shows all potential three cocycle twists which are {\em a priori} present are always trivializable. This appears to be a nontrivial assertion. For instance, it is not immediately implied by any of the standard conjectures in this area that we are aware of, e.g. spectral descriptions of the Hecke categories. 

\sss{} Another consequence of Theorem \ref{t:maintheo} is that the neutral block ${}_\chi \cM^\circ_\chi$ is always equivalent to the affine Hecke category of another group with trivial character sheaf, where the group in general is the product of a reductive group and the loop group of a semisimple, simply connected group. One may even hope to define an ind-group $\cH$ whose affine Hecke category (with the trivial character sheaf) is equivalent to the whole ${}_\chi\cM_\chi$. Such a group deserves to be called an {\em endoscopic loop group} for $(\wt{LG}, \chi)$.

Although it is unclear how to construct $\cH$ from the pair $(\wt{LG}, \chi)$ in general, we propose a special case here and will leave the general case for future investigation. When $G$ is split and $\chi$ is trivial along the central $\mathbb{G}_m$, $\chi$ defines an element $s$ in the dual torus $\mathbb{A}^\vee$ in the Langlands dual group $G^\vee$. Let $H$ be the usual endoscopic group of $(G,s)$: it is the Langlands dual to the centralizer $Z_{G^\vee}(s)$ (there is a way to define $H$ even when $Z_{G^\vee}(s)$ is disconnected). In this case, one may take $\cH$ to be the loop group of $H$.

\sss{Related work}
For the finite monodromic Hecke category attached to a reductive group $H$ and a character sheaf $\chi$ on its abstract Cartan $T_H$, the equivariant version of Theorem \ref{t:mainthm} with $\Qlbar$-coefficients is proved in \cite{LY}. The argument there relies on a different kind of Soergel functor that is an analog of taking equivariant cohomology. In the monodromic setting and with modular coefficients, the structure of the neutral block is described in terms of Soergel bimodules in the PhD thesis of V.Gouttard \cite{Go}. 

For loop groups $LG$ (without central extension), the structure of the neutral block $\cM_\chi^\circ$ (or rather its equivariant version) is described in terms of Soergel bimodules in the PhD thesis \cite{Li} of one of us. The method used in \cite{Li} is an affine analog of that of \cite{LY}. In the recent work of Eberhardt--Eteve \cite{EE}, among other interesting results, they give a Soergel-type description for the (K-theoretic) monodromic Hecke categories attached to any {\em connected} Kac--Moody group. Their result however does not cover all blocks of the Hecke category for a loop group. In the upcoming work of Sandvik, he proves an endoscopic equivalence for the neutral block of the equivariant Hecke categories for Kac--Moody groups with modular coefficients.

\sss{} Let us briefly comment on our method of proving Theorems \ref{t:maintheo}.  
By its definition, up to renormalization issues, the combinatorial category ${}_\chi \cS_\chi$ is essentially a full subcategory of the endofunctors of $\indcoh(\frf_\chi)$. The formal scheme $\frf_\chi$ is designed so that $\indcoh(\frf_\chi)$ is equivalent to $D(\wt{\mathbb{A}}/(\wt{\mathbb{A}}, \chi\mon))$, the category of $\chi$-monodromic sheaves on $\wt\AA$. The main content of Theorem \ref{t:mainthm} is therefore the existence of a suitable action of ${}_\chi \cM_\chi$ on $D(\wt{\mathbb{A}}/(\wt{\mathbb{A}},\chi\mon))$. 

In \cite{LY}, for the monodromic Hecke categories of a reductive group $H$, Lusztig--Y. used the natural action of $D((B, \chi) \bs H / (B, \chi))$ on the category of nondegenerate Whittaker sheaves 
$$D((N^-, \phi) \bs H / (B, \chi))$$
to prove an analogue of \eqref{intro all blocks}. 

While same construction works in the monodromic setting, its direct analogue is insufficient for loop groups which are disconnected in general: the naively defined affine nondegenerate Whittaker model now splits as a direct sum of copies of $D(\wt{\mathbb{A}} / (\wt{\mathbb{A}},\chi\mon))$, indexed by the component group of the loop group $LG$.

Finding a remedy for the above issue is therefore one of the central tasks of the present work. With some over-simplification, our method goes as follows. One has an identification of the affine nondegenerate Whittaker model with a category of sheaves on the moduli stack of $G$-bundles on $\mathbb{P}^1$ with level structure at two marked points $0$ and $\infty$, where our original Iwahori together with $\chi$ sits at $0$ and the generic additive character $\psi$ is imposed at $\infty$. We then impose a further equivariance for the action of a carefully chosen subgroup $\Sig$ at $\infty$, commuting with $\phi$ in a suitable sense,  
and show this pares the affine nondegenerate Whittaker model down to having the desired size, i.e., one copy of $D(\wt{\mathbb{A}} / (\wt{\mathbb{A}},\chi\mon))$. The actual story is slightly more complicated in that $\Sig$ may be slightly larger than what we exactly need, resulting in a mild choice in the construction of the equivalence \eqref{main eq} if $k$ is not the algebraic closure of a finite field.

\subsection{Applications}\label{ss:apps}
Theorem \ref{t:maintheo} is well suited to establishing equivalences between monodromic affine Hecke categories associated to different loop groups and characters, as long as their combinatorial quadruples are isomorphic. Such equivalences are anticipated broadly throughout the local geometric Langlands correspondence.

In this paper, we use this to settle some of the basic remaining conjectures in depth zero for the quantum local Langlands correspondence, due to Gaitsgory \cite{Ga}. 

\sss{} We work in the de Rham setting as in  \S\ref{sss:intro sheaves}(ii), i.e., with D-modules over an algebraically closed field $k$ of characteristic zero. 

Let $G$ be a split connected reductive group over $K=k\lr{t}$, i.e., it is the base change of a connected reductive group $\GG$ over $k$. Let $\kappa$ be a nondegenerate {level}, i.e., an Ad-invariant nondegenerate bilinear form on the Lie algebra of $\GG$. This gives rise to a monoidal category of $\kappa$-twisted D-modules on $LG$, which we denote by
$$\Dmod_\kappa(I \bs LG / I).$$

Let $\chG$ (resp. $\chGG$) denote the connected split reductive group over $K$ (resp. $k$) that is Langlands dual to $G$ (resp. $\chGG$).
\footnote{We use two notations for the Langlands dual groups: the $\check{}$ notation denotes a group defined over the same field as the automorphic side, such as $K$ or $k$, while the ${}^\vee$ notation denotes a group defined over the coefficient ring $E$.} One has a nondegenerate level $\ckappa$ for $\chG$, obtained by requiring that the forms restrict to dual nondegenerate forms on the Cartan subalgebras of $\GG$ and $\chGG$. In particular,  we may form the level $-\check{\kappa}$, and the associated Hecke category 
$$\Dmod_{-\ckappa}(\chI \bs L\chG / \chI).$$
The following was conjectured by Gaitsgory \cite[Conjecture 3.10]{Ga}. 

\begin{thm} There is a $t$-exact monoidal equivalence
    $$\Dmod_\kappa(I \bs LG / I) \simeq \Dmod_{-\ckappa}(\chI \bs L\chG / \chI).$$
\end{thm}
If we write $L^+G$ for the arc group of $\GG\times_{\Spec k}\Spec \cO_K$, and similarly for $L^+ \chG$, the following spherical analogue was also conjectured by Gaitsgory \cite{Ga}. 
\begin{thm} There is a $t$-exact monoidal equivalence
$$\Dmod_\kappa(L^+G \bs LG / L^+G) \simeq \Dmod_{-\ckappa}(L^+\chG \bs L \chG / L^+ \chG).$$
\end{thm}

In fact, we more generally prove an equivalence, conjectured by Gaitsgory, matching the Hecke 2-categories, where one runs over all pairs of standard parahorics contained in the arc groups of $G$ and $\chG$. We also obtain similar statements when we further introduce a character twist on the Iwahori itself. 

Let us explicitly highlight a particularly useful instance of the above, namely the case of the metaplectic double cover of the symplectic group.  
\begin{exmp}  Consider a symplectic group $G = \on{Sp}_{2n}$, and the unique level giving the short coroots squared length {\em one}; denote this level by $\frac{1}{2}$. In this case, the dual level for $\chG = \on{SO}_{2n+1}$ is half integral, but integral for the Spin group, so the above equivalences simplify to 
\begin{align*}\Dmod_{\frac{1}{2}}(I \bs L\on{Sp}_{2n} / I) &\simeq \Dmod(I \bs L\on{SO}_{2n+1}/I)\\ \Dmod_{\frac{1}{2}}(L^+\on{Sp}_{2n} \bs L\on{Sp}_{2n} / L^+ \on{Sp}_{2n}) &\simeq \Dmod(L^+\SO_{2n+1} \bs L \SO_{2n+1} / L^+ \SO_{2n+1}).\end{align*}
The latter equivalence is used in the construction of Coulomb branches in non-cotangent types \cite{coloumb}.
\end{exmp}

\sss{} We are now back to the general sheaf-theoretic setting of Section \ref{sss:intro sheaves}.

Next we will state some analogues of the preceding in the case when the level $\kappa$ is rational. These are geometric analogues of the function-theoretic local Shimura correspondence of Savin \cite{Sa}. 

Let $G$ be a split connected reductive group over $K$, $\wt{LG}$ the corresponding centrally extended loop group, and $\chi_c$ a character sheaf of finite order on the central $\mathbb{G}_m$. Let us denote the category of $\chi_c$-equivariant sheaves on $\wt{LG}$ by $D_{\chi_c}(LG)$, and, using that the central extension canonically splits over $I$, we may form the $\chi_c$-twisted affine Hecke category 
$D_{\chi_c}(I \bs LG / I)$ of bi-$I$-equivariant objects of $D_{\chi_c}(LG)$. 

On the other hand, from the data of $G$ and $\chi_c$ we may canonically associate a split connected reductive group $H^\vee$ over $E$, the {\em metaplectic dual group} (see Sections \ref{ss:meta dual}-\ref{ss:meta dual2} for details). Let $H$ be a connected split reductive group over $K$ that is Langlands dual to $H^\vee$. \footnote{In some sources, the notation $H$ and $H^\vee$ are switched.} 
We have the following. 

\begin{thm}\label{t:dersat} There are $t$-exact monoidal equivalences
\begin{eqnarray}
    D_{\chi_c}(I \bs LG / I) \simeq D(I_H \bs LH / I_H),\\
\label{e:meta Sat}    D_{\chi_c}(L^+G \bs LG / L^+G) \simeq D(L^+H \bs L H / L^+ H).
\end{eqnarray}
\end{thm}

In particular, passing to the hearts of the $t$-structures in \eqref{e:meta Sat}, which are preserved under convolution, one recovers, at the level of monoidal categories, the abelian Satake equivalence of Finkelberg--Lysenko \cite{FL}. 

Combine \eqref{e:meta Sat} with the derived Satake equivalence of \cite[Theorem 5]{BF} (see also the cocomplete version \cite[Cor. 12.5.5]{AG}), we arrive at a spectral description of the full metaplectic Satake category in terms of the metaplectic dual group $H^\vee$.

\begin{cor}[Metaplectic derived Satake equivalence]
    Suppose the coefficient field $E$ has characteristic zero. There is an equivalence of monoidal categories
    \begin{equation}
        D_{\chi_c}(L^+G \bs LG / L^+G) \simeq \indcoh_{\on{nilp}}((\pt/H^\vee)\times_{\frh^{\vee}/H^\vee}(\pt/H^\vee)).
    \end{equation}
\end{cor}

\sss{Future applications}
In the sequel to this paper, we will give a spectral description of the monodromic Hecke categories $\cM_\chi$ when no central extension is involved, confirming a conjecture of Bezrukavnikov \cite[Conjecture 58]{B}, i.e., the depth zero local geometric Langlands correspondence with restricted variation.

On the other hand, the equivalence \eqref{intro all blocks} expressing $\cM_\chi$ as a semidirect product without any cocycle twistings has consequences in the usual representation theory of metaplectic groups: it will allow us to show the triviality of cocycles in the Hecke algebras for metaplectic groups in certain cases, and to give a transparent treatment of the metaplectic Casselman--Shalika formula.

\subsection{Acknowledgments} It is a pleasure to thank Tomoyuki Arakawa, Roman Bezrukavnikov, Dougal Davis, Misha Finkelberg, Joakim F\ae rgeman,  Charles Fu, Dennis Gaitsgory, Mark Goresky, Jacob Lurie, Sergey Lysenko, Sam Raskin, Simon Riche, Jeremy Taylor, Kari Vilonen, and David Yang for useful conversations and correspondence. 

G.D. was supported by NSF Postdoctoral Fellowship under grant No. 2103387. 

Y.W.L. was supported by the University of Melbourne,
IAS School of Mathematics, the National Science Foundation under Grant No. DMS-1926686, and the ARC
grant FL200100141.

Z.Y. was supported by the Simons Foundation.

X.Z. was supported by NSF grant under DMS-2200940.

\section{Monodromic categories}\label{S: monodromic and equivariant categories}
\subsection{Category theory}\label{SS: category theory}
In this article, we will work within the $\infty$-categorical setting (in a mild way). The main advantage to work in this framework is that it allows us to pass between monodromic and equivariant settings, and between Iwahori and parahoric levels easily. 
Our main result on the Soergel bimodule type descriptions of the monodromic affine Hecke categories, from which one deduces the identification of the monodromic affine Hecke category of $G$ with the (unipotent monodromic) affine Hecke category of an endoscopic group,  can be phrased easily within the framework of the usual triangulated categories.
 
We will let $\lincat$ denote the category of presentable stable $\infty$-categories, with functors being continuous (a.k.a. colimit preserving) functors. Recall that $\lincat$ is equipped with Lurie's tensor product, making it a symmetric monoidal category. We recall that limits and colimits exist in $\lincat$, and the tensor products commute colimits separately in each variable. On the other hand, the forgetful functor from $\lincat$ to the (very big) category of all $\infty$-categories preserves limits (but not colimits).

For a commutative ring $E$, let $\Modu_E$ denote the stable category of $E$-modules. This is a commutative algebra object in $\lincat$.
Let $\lincat_E$ denote the the category of $\Modu_E$-modules in $\lincat$, equipped with the induced tensor product. Objects in $\lincat_E$ will be called as $E$-linear categories.
For $\cC,\cD\in \lincat_E$, unless specified, we write $\cC\otimes\cD$ for the tensor product in $\lincat_E$.
We recall that limits and colimits exist in $\lincat_E$, the tensor products commute colimits separately in each variable, and the forgetful functor $\lincat_E\to\lincat$ preserves both limits and colimits.

On the other hand, in this article, we do not need derived algebraic geometry. All the involved schemes and stacks are classical. 

\subsection{Sheaf theory}\label{SS: sheaf theory}
\subsubsection{}
Let $k$ be an algebraically closed field.
Let $\AlgSp^{\fint}_k$ be the category of algebraic spaces of finite presentation over $k$. In this article, we will consider the following association to $X\in \AlgSp^{\fint}_k$ an $E$-linear category $D(X)$ of sheaves on it:
\begin{enumerate}
\item For $k$ of arbitrary characteristic, let $\ell$ be a prime number different from $\textup{char}(k)$. Let $E$ be either a DVR finite over  $\mathbb{Z}_\ell$, or an algebraic field extension over $\FF_\ell$ or $\mathbb{Q}_\ell$. Let $D(X)$ be the $E$-linear category of $\ell$-adic sheaves on $X$;

\item If $k$ is a field of characteristic zero, we let $E=k$ and let $D(X)$ be the ind-completion of the category of holonomic algebraic D-modules on $X$.

\end{enumerate}

We will refer these cases as \'etale and de Rham setting. 
We note that in all the above cases, $E$ is a regular Noetherian local ring complete with respect to its maximal ideal $\fm$.  We denote the residue field by $$\fre := E / \fm.$$
In particular, if $E$ is a field, we have $E = \fre$. In addition, for ease of notation, we will refer to all the cases above when $E$ is not a field by simply saying that $E$ is {\em integral}.

Note that $D(\Spec k)=\Modu_E$. Also recall that if $f: X\to Y$ is a morphism in $\AlgSp^{\fint}_k$, we have pairs of adjoint functors 
\[
f^*:D(Y)\rightleftharpoons D(X):f_*,\quad f_!:D(X)\rightleftharpoons D(Y): f^!.
\]
Note that in the de Rham setting, there is also a ``bigger" category consisting of all algebraic D-modules on $X$. We do not use this version in this article.

\sss{} For our purpose, we need to extend the domain of sheaf theory $D$ in order to deal with some (mildly) infinite-dimensional geometric objects. In addition, to rigorously define the monoidal structure on the (monodromic) affine Hecke categories in the $\infty$-categorical setting, we need to make the assignment $X\leadsto D(X)$ in a highly functorial way.

Here is the standard process to make such extension (e.g. see \cite[Sect. 8]{Zh} for a detailed account): Let $\cC$ be a(n $\infty$-)category with finite products. Associated to $\cC$, there is a symmetric monoidal ($\infty$-)category $\Corr(\cC)$ of correspondences, whose objects the same as those of $\cC$ and with the morphisms from $X$ to $Y$ given by correspondences between $X$ and $Y$, i.e. a diagram of the form 
\begin{equation}\label{eq: correspondence}
   \xymatrix{
Z\ar_{f}[d] \ar^{g}[r]& X\\
   Y &
   }
\end{equation} 
To save space, we sometimes also write the above morphism as $X\xleftarrow{g} Z\xrightarrow{f} Y$.
Compositions of morphisms are given by compositions of correspondences.
The tensor product of $X$ and $Y$ in $\Corr(\cC)$ is given by the product $X\times Y$ in $\cC$. If $\verti$ is a class of morphisms in $\cC$ that is stable under compositions and base change, one can define a subcategory $\Corr(\cC)_{\verti;\all}\subset\Corr(\cC)$ with $f$ in \eqref{eq: correspondence} belonging to $\verti$. 

For later discussions of affine Hecke categories,
we recall the following general paradigm. Let $\cC$ be as above. Let $f:X\to Y$ be a morphism in $\cC$. Then $X\times_YX$ admits a canonical associative algebra structure as an object in $\Corr(\cC)$. Informally, the multiplication is given by the correspondence
\[
X\times_YX\times X\times_YX\xleftarrow{\id\times\Delta_X\times\id}X\times_YX\times_YX\xrightarrow{\id\times f\times \id}X\times_YY\times_YX=X\times_YX,
\]
and the unit is given by
\[
\Spec k\xleftarrow{p_X} X\xrightarrow{\Delta_{X/Y}}X\times_YX.
\]
If $f$ and $\Delta_{X/Y}$ belong to $\verti$, then $X\times_YX$ is an algebra object in $\Corr(\cC)_{\verti;\all}$. If $W\to Y$ is another morphism in $\cC$, then $X\times_YW$ is a left $(X\times_YX)$-module in $\Corr(\cC)_{\verti;\all}$.

\sss{} We apply the above paradigm to the case $\cC=\AlgSp^{\fint}_k$. We may now encode $X\leadsto D(X)$ as a lax symmetric monoidal functor
\begin{equation}\label{eq:sheaf theory finite}
D: \Corr(\AlgSp_k^{\fint})\to \lincat_E,
\end{equation}
which sends $X$ to $D(X)$ and a morphism as in \eqref{eq: correspondence} to $f_*\circ g^!:D(X)\to D(Y)$. We refer to \cite[\S 10.2]{Zh} for a careful construction of such sheaf theory in the first case (i.e. the $\ell$-adic setting), which literally applies to the second case.

Let $\AlgSp_k^{\qcqs}$ be the category of quasi-compact and quasi-separated (qcqs) algebraic spaces over $k$, and let $\fint$ be the class of finitely presented (fp) morphisms. 
Recall that every such algebraic space $X$ can be written as an inverse limit 
\begin{equation}\label{eq: noetherian approx}
    X=\lim_i X_i
\end{equation} with each $X_i$ of finite presentation over $k$, and with transitioning maps being affine morphisms, and every fp morphism $f:X\to Y$ arises as the base change of a morphism $(f_i: X_i\to Y_i)\in \AlgSp^{\fint}_k$. 
We thus can extend \eqref{eq:sheaf theory finite} as a functor
\begin{equation}\label{eq:sheaf theory qcqs}
    D: \Corr(\AlgSp_k^{\qcqs})_{\fint;\all}\to \lincat_E
\end{equation}
via the (operadic) left Kan extension along the fully faithful embedding $\Corr(
\AlgSp_k^{\fint})\subset \Corr(\AlgSp_k^{\qcqs})_{\fint;\all}$.

Explicitly, for every $X\in \AlgSp_k^{\qcqs}$, written as in \eqref{eq: noetherian approx}, we have
\[
D(X)=\colim_i D(X_i)
\]
with transitioning functors being $!$-pullbacks. 

\begin{rem}
    We warn the readers that for general $X$, $D(X)$ may not be what one naively would expect. For example, if $X=\Spec L$ is the spectrum of an algebraically closed field $L$ transcendental over $k$ (say $k$ is of characteristic $p>0$), and if $E=\Ql$, then $D(X)\neq \Modu_E$. But we shall not worry about this.
\end{rem}

Let $\prestk_k$ denote the category of prestacks over $k$, i.e. the category of (accessible) functors from the category $\calg_k$ of commutative $k$-algebras to the category of groupoids.

Let $\indalg_k\subset\prestk_k$ be the full subcategory of (strict) ind-algebraic spaces over $k$, by which we mean prestacks that can be written as $\colim X_i$, with each $X_i\in\AlgSp_k^{\qcqs}$ and with all transitioning maps $\iota_{i,j}:X_i\to X_j$ being finitely presented closed embeddings. If each $X_i$ in the above presentation can be chosen to be a qcqs scheme over $k$, then $X$ is called an ind-scheme.

We say a morphism $f:X\to Y$ of prestacks to be indfp if for every $Y'\to Y$ with $Y'\in\AlgSp_k^\qcqs$, the base change $X':=Y'\times_YX\to Y'$ can be presented as $X'=\colim X_i\to Y'$ with each $X_i\to Y'$ a fp morphism in $\AlgSp_k^{\qcqs}$ and with all transitioning maps $\iota_{i,j}: X_i\to X_j$ being fp closed embeddings.

We will let
\begin{equation}\label{eq:sheaf theory prestk}
    D: \Corr(\prestk_k)_{\indfp;\all}\to \lincat_E
\end{equation}
be the right Kan extension of \eqref{eq:sheaf theory qcqs} along the inclusion $\Corr(\AlgSp_k^{\qcqs})_{\fint;\all}\subset \Corr(\prestk_k)_{\indfp;\all}$.  Note that as the inclusion is not a full embedding, a priori the value of $D(X)$ on $X\in\AlgSp_k^{\qcqs}$ may change. But it is a non-trivial theorem that in fact, the restriction of \eqref{eq:sheaf theory prestk} to $\Corr(\AlgSp_k^{\qcqs})_{\fint;\all}$ is still \eqref{eq:sheaf theory qcqs}. (E.g. see \cite[Sect. 10.4.5]{Zh} for a detailed discussion.) On the other hand, for a general prestack $X$, the value of $D$ at $X$ is 
\[
D(X)=\lim_{S\in (\prestk_k)_{/X}\times_{\prestk_k}\AlgSp_k^{\qcqs}} D(S).
\]
Informally, this means that the limit taken over all qcqs algebraic spaces $S$ over $X$ with the transitioning functors being $!$-pullbacks. 

For those prestacks $X$ encountered in practice, the limit over the large index category $(\prestk_k)_{/X}\times_{\prestk_k}\AlgSp_k^{\qcqs}$ can often be reduced to the limit over a much smaller index category.

\begin{exmp}
Let $U$ be a qcqs scheme over $k$ with an action by an affine group scheme $H$ over $k$. Let $H^\bullet\times U$ denote the simplicial scheme arising from the action map.
Let $X=H\bs U$ be the \'etale quotient. 
Then
\[
D(X)=\lim D(H^\bullet\times U),
\]
with transitioning functors given by $!$-pullbacks. This follows from the \'etale descent.

One can show that if $H$ is pro-unipotent, i.e. $H=\lim_i H_i$ with each $H_i$ a unipotent algebraic group and transitioning maps $H_j\to H_i$ surjective with kernel unipotent, and if $H$ acts trivially on $U$, then $D(X)\cong D(U)$.
\end{exmp}

\begin{exmp}
Let $X$ be an ind-scheme over $k$. Let $X \simeq \colim X_i$ be a presentation of $X$ as a colimit of qcqs schemes where the all transition maps $\iota_{i,j}:X_i \ra X_j$ being fp closed immersions. Then 
\[
D(X)=\lim D(X_i)\cong \uralim_i D(X_i),
\]
with transitioning functors in the limit presentation given by $!$-pullbacks, and in the colimit presentation given by $*$-pushforwards. 
\end{exmp}

\sss{Monoidal structure}\label{sss:sheaf conv} Now let $f: X\to Y$ be an $\indfp$ morphism of prestacks, and assume that the relative diagonal $\Delta_{X/Y}: X\to X\times_YX$ also belongs to $\indfp$. Then $D(X\times_YX)$ acquires a canonical $E$-linear monoidal structure, and for every morphism $W\to Y$, $D(X\times_YW)$ is a left $D(X\times_YX)$-module.

Our major applications will be the construction (in the $\infty$-categorical setting) of the monoidal structure on the (monodromic) affine Hecke categories, and to define their module categories. Here are some simpler examples.

\begin{exmp}
    If $f: X\to X$ is the identity map, then $X\times_XX=X$ and the monoidal structure is just the underlying monoidal structure of the symmetric monoidal structure of $D(X)$. In addition, for $W\to Y=X$ the action of $D(X)$ on $D(W)$ is induced by the symmetric monoidal functor $D(X)\to D(W)$ by $!$-pullback.
\end{exmp}

\begin{exmp}\label{ex: monoidal structure of sheaf on group}
    Let $H$ be an algebraic group over $k$. Then $\pt=\spec k \to \mathbb{B} H$ is fp and the corresponding monoidal structure of $H=\pt\times_{\mathbb{B} H}\pt$ is the usual one on $D(H)$ induced by $*$-pushforward along the multiplication of $H$, usually called the $*$-convolution. Now suppose $X$ is equipped with an action of $H$. Then there is a tautological map $H\bs X\to \mathbb{B} H$, and we have $X=\pt\times_{\mathbb{B} H}H\bs X$. The corresponding action of $D(H)$ on $D(X)$ is the usual one induced by $*$-pushforward long the action map.
\end{exmp}

\subsection{Character sheaves}\label{sstorus}

\subsubsection{} As before, let $k$ be an algebraically closed field, and let $E$ be the coefficient ring of our sheaf theory $D$. We recall the notion of character sheaves.

\begin{defn}
    Let $H$ be an algebraic group over $k$ and let $\mult: H\times H\to H$ denote the multiplication. Let $E'$ be a finite $E$-algebra. By a {\em character $E'$-local system} or an {\em $E'$-linear character sheaf} of $H$, we mean a rank one $E'$-local system $\chi$ on $H$, equipped with an isomorphism which is associative 
\[\mult^*\chi \isom \chi \boxtimes_{E'} \chi.\]
\end{defn}

Note that such an isomorphism necessarily induces a rigidification of $\chi$ at the unit of $H$. We let $\Ch(H ; E')$ denote the groupoid of $E'$-character sheaves on $H$. We will omit $E'$ from the notation when it is clear. It naturally forms a Picard groupoid (and so its isomorphism classes form an abelian group).

Let $H_1, H_2$ be two algebraic groups. Then it is easy to check that the exterior tensor product induces an equivalence of groupoids 
\begin{equation}\label{eq:character sheaf tensor product}
\Ch(H_1; E')\times \Ch(H_2; E')\stackrel{\cong}{\to} \Ch(H_1\times H_2; E'),\quad (\chi_1,\chi_2)\mapsto \chi_1\boxtimes_{E'}\chi_2.
\end{equation}
Indeed, if $\chi$ is a character sheaf on $H_1\times H_2$. Let $\chi_1=\chi|_{H_1\times \{e\}}$ and $\chi_2=\chi|_{\{e\}\times H_2}$. Then using the isomorphism $H_1\times H_2\cong (H_1\times\{e\})\times(\{e\}\times H_2)\xrightarrow{\mult} H_1\times H_2$ and the character property of $\chi$, we see that $\chi\cong \chi_1\boxtimes_{E'}\chi_2$.

\sss{The case of a torus}
Let $H=T$ be an algebraic torus over $k$. In the \'etale setting, let $\pi_1^{t}(T)$ be the tame \'etale fundamental group of $T$ based at $1$. It is well-known that the groupoid $\Ch(T; E')$ is discrete, and 
there is a natural isomorphism
\[
\Ch(T;E')\cong \Hom_{cts}(\pi_1^{t}(T), {E'}^\times), 
\]
sending $\chi$ to its stalk at the unit of $T$, equipped with the action of $\pi_1^{t}(T)$. Here $\Hom_{cts}$ denotes the set of continuous homomorphisms.

It turns out that in both the \'etale and de Rham settings, $\Ch(T ; E')$ is the set of $E'$-points of an ind-scheme $\mathrm{LS}_{\Tv,K}^{t,\Box}$ classifying tame ``$\Tv$-local systems" on $\Spec K$ (where $K=k\lr{t}$), which will be discussed in Sect. \ref{sss:framed Loc}.

\sss{The case of $\Ga$}\label{SSS: char sheaf on Ga}
To construct the affine Whittaker model in Section \ref{ss:AffineWhittaker}, we need a particular rank one character sheaf on $\Ga$.

When $\chk=p$ and we are in the \'etale context, we suppose that $E$ contains a $p$-th root of unity and fix a nontrivial character $\psi: \FF_p\to E^{\times}$. Let $\AS_{\psi}$ be the corresponding Artin--Schreier on $\Ga$.  When $\chk=0$ and we are in the $D$-module context, $\AS_\psi$ should be replaced by the exponential $D$-module on $\Ga$. For notational simplicity, we still denote the exponential $D$-module by $\AS_\psi$.

\sss{Case of integral coefficients}\label{sss: int coef}
In this article, we allow the coefficient ring $E$ of our sheaf theory (in the \'etale setting) to be integral. Such generality introduces a well-known complication: There are different $E$-character sheaves with the same reduction as $\fre$-character sheaves. We now make some preliminary discussions in order to address this issue later on. Readers primarily interested in the case where $E$ is an algebraically closed field can skip the following discussion.

We assume that $H$ is connected. 
Let $\chi\in \Ch(H;\fre_\chi)$ be a character sheaf on $H$, where $\fre_\chi$ is a finite extension of $\fre$. We assume that $\fre_\chi$ is the field of the definition of $\chi$. We refer to such a character sheaf as {\em residual}. We let $E_\chi$ denote the unique unramified extension of $E$ with residue field $\fre_\chi$. Then $\chi$ lifts canonically to an $E_\chi$-linear character sheaf $[\chi]$ via the Teichm\"uller lifting $\fre_\chi^\times\to E_\chi^\times$.

Consider the collection 
of all $E'$-linear character sheaves $\wt{\chi}$, for all $E'$ local rings finite over $E$, whose reductions modulo the maximal ideal $\frakm'$ of $E'$ are $\chi$, and denote this by 
\begin{equation}\label{eq: lifting of char sheaves}
\mathscr{L}_\chi := \{ \wt{\chi}: \wt{\chi} \underset{E'}\otimes (E'/\frakm')=\chi  \}.
\end{equation}
If we let $u$ denote the trivial character sheaf on $H$, then there is a natural bijection 
$$ \mathscr{L}_u\simeq \mathscr{L}_\chi,\quad \wt{u}\mapsto \wt{u}\otimes_{E}[\chi].
$$

\subsection{Monodromic sheaves} \label{ss:monshvs}
We review the theory of monodromic sheaves in this subsection. For a more detailed discussion, we refer to \cite[\S 4.1]{Zh}.

\subsubsection{}Let $H$ be a connected algebraic group over $k$. Let $D_{\mono}(H)\subset D(H)$ denote the full $E$-linear subcategory generated by all $E'$-linear character sheaves with all finite $E$-algebra $E'$.
We call $D_{\mono}(H)$ the category of monodromic sheaves on $H$.

Note that if $H_1$ and $H_2$ are two connected algebraic groups, there exterior tensor product functor $D(H_1)\otimes D(H_2)\to D(H_1\times H_2)$ restricts to an equivalence
\begin{equation}\label{eq: monodromic exterior tensor product}
D_{\mono}(H_1)\otimes D_{\mono}(H_2)\cong D_{\mono}(H_1\times H_2).
\end{equation}
That is, the classical K\"unneth formula implies that the functor is fully faithful. On the other hand, \eqref{eq:character sheaf tensor product} implies that the functor is also essential surjective. 

Recall that $D(H)$ is naturally a  monoidal category, given by $*$-convolution (see Example \ref{ex: monoidal structure of sheaf on group}). We will write $\cF_1\star\cF_2:=\mult_*(\cF_1\boxtimes\cF_2)$ for the convolution product.

\begin{prop}\label{p: monoidal unit}
  The natural embedding $\iota^{\mono}:D_{\mono}(H)\subset D(H)$ admits a continuous right adjoint, which in addition admits a canonical monoidal structure.
\end{prop}
Note that this proposition in particular says that $D_{\mono}(H)$ has a natural monoidal structure.
\begin{proof}
    As both $D_{\mono}(H)$ and $D(H)$ are compactly generated and the inclusion $D_{\mono}(H)\subset D(H)$ preserves compact objects, the right adjoint
    \[
    \Av^{\mono}:=(\iota^{\mono})^R: D(H)\to D_{\mono}(H)
    \]
    is continuous. We show it has a natural monoidal structure. 

    First by \eqref{eq: monodromic exterior tensor product}, we see that the natural map 
\[
\Av^{\mono}(\cG_1)\boxtimes\Av^{\mono}(\cG_2)\to \Av^{\mono}(\cG_1\boxtimes\cG_2)
\] 
is an isomorphism,  for $\cG_i\in D(H_i)$ for $i=1,2$.  Now for $\cF\in D_{\mono}(H)$, we have 
\begin{multline*}
\Hom(\cF, \Av^{\mono}(\cG_1\star\cG_2))=\Hom(\cF,\cG_1\star\cG_2)=\Hom(\mult^*\cF, \cG_1\boxtimes\cG_2)\\
=\Hom(\mult^*\cF,\Av^{\mono}(\cG_1\boxtimes\cG_2))=\Hom(\cF,\Av^{\mono}(\cG_1)\star\Av^{\mono}(\cG_2)).
\end{multline*}
In particular, we see that $\cG_1\star\cG_2\cong \Av^{\mono}(\cG_1\star\cG_2)$ if $\cG_1,\cG_2\in D_{\mono}(H)$. Therefore, $D_{\mono}(H)\subset D(H)$ is closed under the convolution product of $D(H)$. 

Next, let $i_e: \{e\}\to H$ be the inclusion and let $\delta_e=(i_e)_*E$. This is the monoidal unit of $D(H)$. Now for $\cF\in D_{\mono}(H)$, 
we have 
\[
\Av^{\mono}(\delta_e)\star\cF=\Av^{\mono}(\delta_e)\star\Av^{\mono}(\cF)=\Av^{\mono}(\delta_e\star\cF)=\Av^{\mono}(\cF)=\cF.
\]

The above arguments show that at the homotopy level, $D_{\mono}(H)$ has a monoidal structure, with $\Av^{\mono}(\delta_e)$ a monoidal unit. In addition, $\Av^{\mono}$ is monoidal. By \cite[Lemma 7.19]{Zh}, this can be upgraded at $\infty$-categorical level. 
\end{proof}

\begin{rem}
    In fact, $D_{\mono}(H)$ is a two-sided ideal of $D(H)$ and $\Av^{\mono}$ is $D(H)$-bilinear. See \cite[Proposition 4.17]{Zh}.
\end{rem}

In the sequel, we let
\[
\d^{\mono}:=\Av^{\mono}(\d_e)
\]
denote the monoidal unit of $D_{\mono}(H)$, and call it the cofree monodromic local system on $H$.

Now let $f: H_1\to H_2$ be a homomorphism of connected algebraic group over $k$. If $\chi_2\in \Ch(H_2;E')$, then $\chi_1:=f^*\chi_2\in \Ch(H_1;E')$. Therefore, the pullback functor $f^*: D(H_2)\to D(H_1)$ restricts to a pullback functor 
\begin{equation}\label{eq: star-pullback-mono}
f^*: D_{\mono}(H_2)\to D_{\mono}(H_1),
\end{equation}
which admits a (continuous) right adjoint, given by
\begin{equation}\label{eq: star-pushforward-mono}
f_*^{\mono}= \Av^{\mono}\circ (f_*|_{D_{\mono}(H_1)}).
\end{equation}

\subsubsection{$\chi$-monodromic subcategories}\label{sss: chi-monodromic category}
Let $\chi\in \Ch(H;\fre_\chi)$ be a character sheaf as in Sect. \ref{sss: int coef}. We define 
\begin{equation} \label{eq: inccharshv2}
\iota^{\chi\mon}: D_{\chi\mon}(H) \subset D(H; E)
\end{equation}
to be the full subcategory of $D(H;E)$ generated by the character sheaves $\wt{\chi} \in \mathscr{L}_\chi$, where $\mathscr{L}_\chi$ is defined in \eqref{eq: lifting of char sheaves}. As the latter sheaves are all compact objects of $D(H; E)$, this inclusion again admits a right adjoint, denoted by $\Av^{\chi\mon}$. In addition, let 
$$
\dcm=\Av^{\chi\mon}(\delta_e).
$$
Then Proposition \ref{p: monoidal unit} holds with $D_{\mono}(H)$ replaced by $D_{\chi\mon}(H)$ and $\Av^{\mono}$ replaced by $\Av^{\chi\mon}$, with the same proof. The monoidal unit of $D_{\chi\mon}(H)$ is $\dcm$. We call $\dcm$ the cofree $\chi$-monodromic local system on $H$. Note that the pullback functor  \eqref{eq: star-pullback-mono} restricts to a functor from $D_{\chi_2\mon}(H_2)$ to $D_{\chi_1\mon}(H_1)$.

To simplify expositions, in the sequel we will use the notation to $(\chi\on{-})\mono$ to denote either $\chi$-monodromic or monodromic version.

\sss{Spectral side}\label{sss:framed Loc}
In the rest of this subsection, we assume that $H=T$ is a torus. In this case, we can study $D_{(\chi\on{-})\mono}(H)$ via (ind-)coherent sheaves on an ind-scheme $\mathrm{LS}_{\Tv,K}^{t,\Box}$ over $E$ parameterizing (framed) ``tame" $T^\vee$-local systems on $\Spec K$, which we now introduce.

\begin{itemize}
\item In the \'etale setting, we define $\mathrm{LS}_{\Tv,K}^{t,\Box}$ via its functor of points. Let $\pi_1^t(\Spec K)$ be the tame fundamental group of $\Spec(K)$. 
For a commutative $E$-algebra $R$, 
$$\mathrm{LS}_{\Tv,K}^{t,\Box}(R)=\Hom_{cts}(\pi_1^t(\Spec K), T^\vee(R))$$ consisting of those $\rho: \pi_1^t(\Spec K)\to T^\vee(R)$ such that for every one dimensional representation of $T^\vee$ on $V$, $V$ is a union of finite $E$-submodules $V_i$, stable under the action of $\pi_1^t(\Spec K)$ and the resulting representation of $\pi_1^t(\Spec K)$ on $V_i$ is continuous in the usual sense. When $E$ is a field of characteristic zero, this is the moduli space of rigidified \'etale $T^\vee$-local systems on $\Spec K$ with restricted variation as defined in \cite[Section 1.4]{AGKRRV.restricted.local.systems}. With general coefficients, this is defined in \cite{zhu2020coherent}.

\item In the de Rham setting (so $E=k$), $\mathrm{LS}_{\Tv,K}^{t,\Box}$ is the moduli space of rigidified de Rham $T^\vee$-local systems on $\Spec K$ with restricted variation as defined in \cite[Section 1.4]{AGKRRV.restricted.local.systems}. 
\end{itemize}

We give more concrete descriptions of $\mathrm{LS}_{\Tv,K}^{t,\Box}$.

\begin{itemize}
    \item In the \'etale setting we may assume $E=\mathbb{Z}_\ell$, as other cases arise as the base change along $\mathbb{Z}_\ell\to E$. Once we fix a topological generator of $\pi_1^t(\Spec K)$, then $\mathrm{LS}_{\Tv,K}^{t,\Box}$ is a sub-indscheme of $T^\vee_{\mathbb{Z}_\ell}$ which is the union of all closed subschemes $Z\subset T^\vee_{\mathbb{Z}_\ell}$ that are finite over $\mathbb{Z}_\ell$ and such that $Z\otimes\mathbb{F}_\ell\subset T^\vee_{\mathbb{F}_\ell}$ is of prime-to-$p$ order. 
    \item In the de Rham setting (so $E=k$), $\mathrm{LS}_{\Tv,K}^{t,\Box}$ is canonically isomorphic to the formal completion of $\mathfrak t^\vee/\xcoch(\Tv)$ along all closed points. Note that when $E=\mathbb{C}$, via the exponential map, we can identify $\mathfrak t^\vee/\xcoch(\Tv)$ with $\Tv$ and so identify $\mathrm{LS}_{\Tv,K}^{t,\Box}$ with the formal completion of $\Tv$ along all closed points.    
\end{itemize}

In both cases,
note that the underlying reduced subscheme of $\mathrm{LS}_{\Tv,K}^{t,\Box}$ is just the union of points, each of which is of the form $\chi=\Spec \fre_\chi$ for a finite field extension $\fre_\chi$ of $\fre$, corresponding to a residual character sheaf $\chi\in \Ch(T;\fre_\chi)$. 

We let $\frf_\chi\subset \mathrm{LS}_{\Tv,K}^{t,\Box}$ be the union of all closed subschemes $Z$ of $\mathrm{LS}^{t, \square}_{T^\vee,K}$ such that $\chi\in Z$ is the unique closed point. Then 
\begin{equation}\label{eq: geometric block decomposition}
    \on{LS}^{t, \square}_{\Tv,K} \simeq \underset{\chi} \sqcup \hspace{.5mm} \frf_\chi.
\end{equation}

When $E=\fre$ is a field, $\frf_\chi$ is just the formal completion of 
$\on{LS}^{t, \square}_{T^\vee}$ at $\chi$. However, when $E$ is integral, the ind-scheme $\frf_\chi$ is a very complicated geometric object. For this reason, we also let $\frz_\chi$ denote the formal completion of $\on{LS}^{t, \square}_{T^\vee}$ at $\chi$. 
We always have a closed embedding
\begin{equation}\label{eq: z subset f}
\frz_\chi\subset\frf_\chi,
\end{equation}
which is an isomorphism if $E=\fre$ is a field. When $\chi=u$ corresponds to the trivial $\Tv$-local system, we shall also write $\frz_\chi\subset \frf_\chi$ simply as $\frz\subset\frf$.

Note that 
\begin{equation*}\label{eq:Soergel ring R}
R_\chi:=\Fun(\frz_\chi)
\end{equation*}
is isomorphic to a power series ring over $E_\chi$ and $\frz_\chi=\mathrm{Spf} R_\chi$. Explicitly, if we choose a (topological) generator of $\pi_1^t(\GG_m)$ and identify $\chi$ as a closed point of $\Tv\otimes\fre$, and if we write $\Tv=\Spec E[x_1^{\pm 1},\ldots,x_n^{\pm 1}]$, then 
    \begin{equation}\label{eq: ring R}
    R_\chi\simeq E_\chi\tl{x_1-x_1([\chi]),x_2-x_2([\chi]),\ldots,x_n-x_n([\chi])}.
    \end{equation}
The description in de Rham case is similar.

Although $\frf_\chi$ is complicated (when $E$ is integral), its ring of functions is well-understood.
\begin{lem}\label{lem: function on frf}
    The embedding \eqref{eq: z subset f} induces an isomorphism $\Fun(\frf_\chi)\cong \Fun(\frz_\chi)$. 
\end{lem}

\begin{proof}
To see this, recall from Section \ref{sstorus} that $\frf_\chi$ is a filtered colimit under closed embeddings of schemes $Z_\alpha$, each finite over $E$, and in particular $\fm$-adically complete, where $\fm$ denotes the maximal ideal of $E$. In particular, we have 
\begin{equation} \label{e:fvsz}\Fun(\frf_\chi) \simeq \varprojlim_\alpha \on{Fun}(Z_\alpha) \simeq \varprojlim_{\alpha, n} \on{Fun}(Z_\alpha) / (\fm^n) \simeq \on{Fun}(\frz_\chi),\end{equation}as desired. 
\end{proof}

\begin{rem}
 We note that using the natural group structure on $\mathrm{LS}_{\Tv,K}^{t,\Box}$, we have
\begin{equation}
\frf_\chi=(\frf\otimes_{E}E_\chi)\cdot [\chi].
\end{equation}
\end{rem}

In the sequel if $K$ is clear, we will just write  $\mathrm{LS}_{\Tv}^{t,\Box}$ instead of $\mathrm{LS}_{\Tv,K}^{t,\Box}$.

\subsubsection{Coherent description of monodromic sheaves}

Let $\indcoh(\mathrm{LS}_{\Tv}^{t,\Box})$ be the ind-completion of the category of coherent sheaves on $\mathrm{LS}_{\Tv}^{t,\Box}$. The following result can be regarded as a version of Mellin transform of monodromic sheaves on tori, or a tame geometric local Langlands correspondence for tori.

\begin{prop}\label{prop: tame geometric Langlands for tori}
\begin{enumerate}
    \item\label{prop: tame geometric Langlands for tori-1} There is a unique $t$-exact equivalence of monoidal categories
\[
\CH: \indcoh(\mathrm{LS}_{\Tv}^{t,\Box})\simeq D_{\mono}(T)
\]   
sending the skyscraper sheaf at $\chi\in \mathrm{LS}_{\Tv}^{t,\Box}(E')$ to the character $E'$-local system on $T$ corresponding to $\chi$. Here the monoidal structure of $D_{\mono}(T)$ is given by the convolution (provided by Proposition \ref{p: monoidal unit}), and the monoidal structure of $\indcoh(\mathrm{LS}_{\Tv}^{t,\Box})$ is given by the $!$-tensor product.

For $\chi$ a residual character sheaf, the functor $\CH$ restricts to an equivalence 
$$\indcoh(\frf_\chi)\simeq D_{\chi\mon}(T).$$ 

\item\label{prop: tame geometric Langlands for tori-2} Let $f^\vee:\mathrm{LS}_{\Tv_2}^{t,\Box}\to \mathrm{LS}_{\Tv_1}^{t,\Box}$ be the homomorphism of ind-schemes induced by $f$. Then under the equivalence $\CH$, the adjoint functors (see \eqref{eq: star-pullback-mono} and \eqref{eq: star-pushforward-mono})
$$
f^*:D_{\mono}(T_2)\rightleftharpoons D_{\mono}(T_1): f_*^{\mono}
$$ 
correspond to adjoint functors between ind-coherent sheaves
$$
(f^\vee)_*: \indcoh(\mathrm{LS}_{\Tv_2}^{t,\Box})\rightleftharpoons \indcoh(\mathrm{LS}_{\Tv_1}^{t,\Box}):(f^\vee)^!.
$$
\end{enumerate} 
\end{prop}

\begin{proof} We will sketch the proof in first the de Rham setting, and then second in the   \'etale setting.

In the de Rham setting, recall that the presentation of the algebra of differential operators on $T$ as the semi-direct product $$U(\ft) \ltimes \on{Fun}(T) \simeq  \Fun(\ft^\vee) \rtimes k[\X_*(T^\vee)]$$gives rise to a tautological {\em Mellin transform} equivalence 
$$M^r: D(T) \simeq \on{IndCoh}(\ft^\vee / \X_*(T^\vee)),$$
cf. the discussion in Section 2.1 of \cite{BZN}. We normalize this equivalence by a cohomological shift, so that it exchanges the delta D-module at the identity in $D(T)$, placed in cohomological degree zero,  with the dualizing sheaf on $\ft^\vee/\X_*(T^\vee)$. With this, it is standard that $M^r$ underlies a monoidal equivalence, exchanging $*$-convolution of D-modules with $!$-tensor product of ind-coherent sheaves.  

By construction, $M^r$ exchanges skyscraper sheaves on the spectral side with character sheaves on $T$, whence restricts to an equivalence of full subcategories
$$D_{\on{mon}}(T) \simeq \indcoh(\on{LS}_{T^\vee}^{t, \square})$$of the required form. 
Moreover, property (2) follows from noting that, by construction, the Mellin transform $M^r$ exchanges $$f^![2(\dim T_2 - \dim T_1)]: D(T_2) \rightarrow D(T_1),$$which agrees with $f^*$ on $D_{\on{mon}}(T_2)$, with $$(f^\vee)_*: \indcoh(\ft_2^\vee / \X_*(T_2^\vee)) \rightarrow \indcoh(\ft_1^\vee / \X_*(T_1^\vee)).$$

   It remains to sketch the proof of the equivalence $\CH$ in the \'etale context.  We first construct an exact additive functor between abelian categories
    \[\CH^{\heartsuit}:\Coh(\mathrm{LS}_{\Tv}^{t,\Box})^{\heartsuit}\to D(T)^{\heartsuit}.
    \]
Namely, the ring of regular functions of $\mathrm{LS}_{\Tv}^{t,\Box}$
is (suitably) completed group algebra of $\pi_1^t(T)$. Therefore, $\Coh(\mathrm{LS}_{\Tv}^{t,\Box})^{\heartsuit}$ can be identified with the abelian category of continuous $\pi_1^t(T)$-modules with underlying $E$-modules finite over $E$, which give $E$-local systems on $T$. This gives the desired functor $\CH^{\heartsuit}$.

Now, for $\cF_1,\cF_2\in \Coh(\mathrm{LS}_{\Tv}^{t,\Box})^{\heartsuit}$, we have
\[
\Hom_{\Coh(\mathrm{LS}_{\Tv}^{t,\Box})}(\cF_1,\cF_2)=\Hom_{\pi_1^t(T)}(\cF_1,\cF_2)=\Hom_{D(T)}(\CH(\cF_1),\CH(\cF_2)),
\]
where the last isomorphism uses the $K(\pi,1)$ property of $T$ (even in the \'etale setting). It follows that $\CH^{\heartsuit}$ can be promoted to a fully faithful embedding $\Coh(\mathrm{LS}_{\Tv}^{t,\Box})\to D(T)$. Clearly, the image of the functor are generated (by taking finite colimits and direct summands) by character sheaves on $T$ (and therefore in particular are compact objects in $D(T)$). It follows that we have a full faithful embedding $\indcoh(\mathrm{LS}_{\Tv}^{t,\Box})\to D(T)$ with the essential image $D_{\mono}(T)$. 

Note that if $Z\subset \frf_\chi$ is a closed subscheme, then $\Fun(Z)$ is a local finite $E$-algebra with residue field $\fre_\chi$. Therefore, $\CH(\cO_Z)$ is a character sheaf $\wt{\chi}$ belonging to $\mathscr{L}_\chi$. Clearly, every element in $\mathscr{L}_\chi$ arises in this way. It follows that $\CH$ restricts to an equivalence $\indcoh(\frf_\chi)\simeq D_{\chi\mon}(T)$ as desired. This proves Part \eqref{prop: tame geometric Langlands for tori-1} except the monoidality of $\CH$.

It is clear from the construction that $f^*: \Ch(T_2)\to \Ch(T_1)$ corresponds to $(f^\vee)_*: \Coh(\mathrm{LS}_{\Tv_2}^{t,\Box})^{\heartsuit}\to \Coh(\mathrm{LS}_{\Tv_1}^{t,\Box})^{\heartsuit}$. Then Part \eqref{prop: tame geometric Langlands for tori-2} follows.

Finally, applying Part \eqref{prop: tame geometric Langlands for tori-2} to $\mult: T\times T\to T$, we see that $\CH$ is monoidal.
\end{proof}

Here are some direct consequences of Proposition \ref{prop: tame geometric Langlands for tori}. 
\begin{cor}   (Block decomposition) \label{r:block} 
    We have that
$$D_{\on{mon}}(T) = \underset{\chi \mbox{ residual}} \bigoplus \hspace{.5mm} D_{\chi\mon}(T).$$
\end{cor}
\begin{proof}
    This follows from Proposition \ref{prop: tame geometric Langlands for tori} and the decomposition \eqref{eq: geometric block decomposition}.
\end{proof}

We need the following result to construct Soergel bimodules.
\begin{cor}\label{cor: Obtaining Soergel}
    We have a canonical equivalence of monoidal categories
    \[
    \End_{\lincat_E} D_{\mono}(T)\cong \indcoh(\mathrm{LS}_{\Tv}^{t,\Box}\times \mathrm{LS}_{\Tv}^{t,\Box}),
    \]
    under which, the natural monoidal structure on $\End_{\lincat_E} D_{\mono}(T)$ corresponds to the following monoidal structure: 
    \[
    \indcoh(\mathrm{LS}_{\Tv}^{t,\Box}\times \mathrm{LS}_{\Tv}^{t,\Box})\otimes \indcoh(\mathrm{LS}_{\Tv}^{t,\Box}\times \mathrm{LS}_{\Tv}^{t,\Box})\to \indcoh(\mathrm{LS}_{\Tv}^{t,\Box}\times \mathrm{LS}_{\Tv}^{t,\Box})
    \]
    \[
    (\cF_1,\cF_2)\mapsto (p_{13})_*\Delta_{23}^!(\cF_1\boxtimes\cF_2).
    \]
\end{cor}

\sss{Cofree monodromic local systems} We establish a few properties of cofree monodromic local systems which will used later to study monodromic affine Hecke categories.

Let $\chi$ be a residual character sheaf of $T$, and let $\delta^{\chi\mon}$ be the cofree monodromic local system as in Section \ref{sss: chi-monodromic category}. 
Also recall the ind-schemes $\frz_\chi\subset\frf_\chi$ introduced in Section \ref{sss:framed Loc}.
\begin{lem} There are canonical isomorphisms of algebras
$$\End(\dcm) \simeq \Fun(\frf_\chi) \simeq \Fun(\frz_\chi)=R_{\chi}.$$
In particular, the leftmost algebra is concentrated in degree zero.  \label{l:enddcm}
\end{lem}

\begin{proof} 

Note that under the equivalence of Proposition \ref{prop: tame geometric Langlands for tori}, the monoidal units are identified. Therefore, the cofree monodromic local system $\dcm$ is exchanged with the dualizing sheaf on $\frf_\chi$. 

Note that for an ind-scheme $\fri=\colim_{\alpha}Z_\alpha$ which is a filtered colimit of schemes of finite type along closed embeddings, we have 
\[
\End(\omega_{\fri})=\Hom(\colim_{\alpha}\omega_{Z_\alpha},\omega_{\fri})=\lim_{\alpha}\End(\omega_{Z_\alpha})=\lim_{\alpha}\Fun(Z_\alpha)=\Fun(\fri).
\]
Now the statement follows from Lemma \ref{lem: function on frf}.
 \end{proof}

The following results will be used in Section \ref{s:tiltbasics} to study the cofree tilting sheaves.

\begin{cor} \label{l:cofreeprop}
\begin{enumerate}
\item\label{l:cofreeprop-1} The object $\delta^{\chi\mon}$ is indecomposable.

\item\label{l:cofreeprop-2} For $I$ a finite set, consider the corresponding direct sum $\oplus_I \dcm$. Then any object $S$ is which is a summand of $\oplus_I \dcm$ is again isomorphic to a finite direct sum of copies of $\dcm$. 

\item\label{l:cofreeprop-3} Any self-extension 
$$0 \rightarrow \dcm \rightarrow c \rightarrow \dcm \rightarrow 0$$
admits a splitting. 
\end{enumerate}
\end{cor}

\begin{proof} Assertions \eqref{l:cofreeprop-1} and \eqref{l:cofreeprop-2} follow from the locality of $\on{Fun}(\frf_\chi)$, and assertion \eqref{l:cofreeprop-3} is the vanishing of $\on{H}^1 \Hom(\dcm, \dcm)$. 
\end{proof}

Proposition \ref{prop: tame geometric Langlands for tori} has another consequence, which will be used in the proof of Proposition \ref{p:std2std}.

\begin{lem} \label{l:compsincofree} Within $D_{\chi\mon}(T)$, consider the full stable subcategory generated by the unit under taking cones and retracts 
$$\langle \dcm \rangle \hookrightarrow D_{\chi\mon}(T).$$
Then this contains all the compact objects, i.e., we have an inclusion
$$D_{\chi\mon}(T)^c \hookrightarrow \langle \dcm \rangle \hookrightarrow D_{\chi\mon}(T).$$ 
\end{lem}

\begin{proof} We must equivalently argue on the spectral side that, if we write $\omega_{\frf_\chi}$ for the dualizing sheaf on $\frf_\chi$, we have the inclusion 
$$\indcoh(\frf_\chi)^c =  \on{Coh}(\frf_\chi) \subset \langle \omega_{\frf_\chi} \rangle.$$

As a preliminary reduction, we claim that $ \on{Coh}(\frf_\chi)$ is generated as an idempotent complete stable category by the dualizing sheaves $\omega_Z$ for every closed subscheme $Z \subset \frf_\chi$. Indeed, we claim this holds for any ind-scheme $\mathfrak{i}$ over $\Spec E$ which is a filtered colimit of affine schemes of finite type along closed embeddings. Indeed, by applying Serre duality, it is enough to note that every object of $ \Coh(\fri)$ admits a finite filtration with associated graded given by shifts of the structure sheaves of its closed subschemes, but this is now evident.

Therefore, for any $\wt{\chi} \in \mathscr{L}_\chi$, with corresponding closed subscheme
$$i:  Z_{\wt{\chi}} \hookrightarrow \frf_\chi,$$we must show that $\omega_{Z_{\wt{\chi}}}$ lies in $\langle \omega_{\frf_{\chi}} \rangle$. To see this, note we have a tautological isomorphism
$$\omega_{Z_{\wt{\chi}}} \simeq i_* i^! (\omega_{\frf_\chi}).$$

 Let us choose a topological generator for the tame fundamental group of $K$, and thereby identify $\frf_\chi$ as a formal subscheme $$\frf_\chi \hookrightarrow T^\vee_E.$$ If we consider the  composite embedding $$\iota: Z_{\wt{\chi}} \hookrightarrow \frf_\chi \hookrightarrow T^\vee_E$$ note that the endofunctor $\iota_* \iota^!$ of $\I(T^\vee_E)$ is given by tensoring with a perfect complex, thanks to the regularity of $E$ and whence $T^\vee_E$. In particular, it follows that $i_* i^! (\omega_{\frf_\chi})$ is representable as a summand of a finite iterated extension of shifts of $\omega_{\frf_\chi}$, as desired. 
\end{proof}

We will use the above lemma repeatedly as follows. Let us write $\wt{D}_{\chi\mon}(T)$ for the ind-completion of $\langle \dcm \rangle$. Explicitly 
$$\wt{D}_{\chi\mon}(T)\simeq \Modu_{R_\chi}.$$
We then have a tautological adjunction 
$$i_!: D_{\chi\mon}(T) \rightleftharpoons  \wt{D}_{\chi\mon}(T) : i^!,$$
where $i_!$ is the fully faithful embedding obtained by ind-extending the inclusion $$D_{\chi\mon}(T)^c \hookrightarrow \langle \dcm \rangle,$$and its right adjoint is obtained via ind-extending the inclusion $\langle \dcm \rangle \hookrightarrow D_{\chi\mon}(T)$. 

\begin{cor} \label{c:compsincofree} For any category $\cC$ in $\lincat_E$, a functor  $\xi: D_{\chi\mon}(T) \rightarrow \cC$ is the same data as an object $c$ of $\cC$ and a map of $E$-algebras 
$$\End(\dcm) \rightarrow \End(c)$$
such that the induced map 
$\wt{D}_{\chi\mon}(T) \rightarrow \cC$ factors through the Verdier quotient $$\wt{D}_{\chi\mon}(T) \xrightarrow{i^!} D_{\chi\mon}(T) \rightarrow \cC.$$
\end{cor}

\sss{} Next we consider functors between monodromic sheaves induced by homomorphisms of tori. 

\begin{prop}\label{lem: functoriality depth zero geom Langlands for tori}
Let $f: T_1\to T_2$ be a \emph{surjective} homomorphism of tori. Then both of the usual $*$- and $!$-pushforwards of all sheaves restrict to functors between monodromic categories. In particular, $f_*^{\mono}=f_*|_{D_{\mono}(T_1)}$, where $f_*^{\mono}$ is from \eqref{eq: star-pushforward-mono}. In addition, there is a canonical isomorphism of functors
\begin{equation}\label{eq: pseudo-properness for monodromic pushforwards}
f_![\dim T_1-\dim T_2]\cong f_*: D_{\mono}(T_1)\to D_{\mono}(T_2).
\end{equation}
\end{prop}

\begin{proof}We need to show that $f_*$ and $f_!$ send monodromic sheaves to monodromic sheaves and these two functors differ by a shift. We may assume that $f$ is either an isogeny, or $f$ is a surjective homomorphism with $\ker f$ a torus. 

We deal with these two cases separately.
When $f$ is an isogeny, then $f_*=f_!$ and they send a character sheaf $\chi$, corresponding to a representation of $\pi_1^t(T_1)$, to the representation $\pi_1^t(T)$ by induction, and therefore belongs to $D_{\mono}(T_2)$.

In the case $\ker f$ is connected, we may write $T_1=\ker(f)\times T_2$ so that $f$ becomes the projection to the second factor. Since the restriction of a character sheaf to $T_2$ is still a character sheaf,
we see both $f_*$ and $f_!$ send monodromic sheaves to monodromic sheaves. In addition, to see these two functors differ by a shift, it is enough to notice that $\mathrm{LS}_{T_2^\vee}^{t,\Box}\to \mathrm{LS}_{T_1^\vee}^{t,\Box}$ is a closed embedding of formal schemes. In this case $(f^\vee)^!$ is also the left adjoint of $(f^\vee)_*$ up to shifts. The claim follows.
\end{proof}

\begin{rem}
    In fact, the proposition holds when $T_1$ and $T_2$ are replaced by general connected affine algebraic groups. See \cite[Proposition 4.15]{Zh}.
\end{rem}

\begin{prop}\label{p:ch sh triv}
Let $f: T_1\to T_2$ be an 
isogeny of tori. That is, a surjective homomorphism with finite kernel $\Gamma$. 
Let $\chi_1\in\Ch(T_1)$. Then after possibly a finite extension of the coefficient ring $E$, $\chi_1$ descends to a character sheaf $\chi_2$ on $\Ch(T_2)$. In particular, $\chi_1|_{\Gamma}$ is trivial as a character sheaf.

If the order of $\Gamma(k)$ is invertible in $\fre$, then the pullback functor induces an equivalence $D_{\chi_2\mon}(T_2)\cong D_{\chi_1\mon}(T_1)$.
\end{prop}

\begin{proof}
  The induced map ${T'}^\vee\to \Tv$ is also an isogeny of tori (over $E$), and therefore $\pi^\vee: \mathrm{LS}_{{T'}^\vee}^{t,\Box}\to \mathrm{LS}_{\Tv}^{t,\Box}$ is surjective. The first statement follows. If the order of $\Gamma(k)$ is invertible in $E$, $\pi^\vee$ is \'etale and therefore induces an isomorphism ind-schemes $\frf_{\chi_{1}} \cong \frf_{\chi_{2}}$. The second statement follows as well. 
\end{proof}

In our application, we only use the existence of a descent of the character sheaf. The full equivalence is not required.

\subsection{Monodromic categories} \label{ss: monodromic categories}
\sss{} Let $X$ be a prestack with an action of a connected algebraic group $H$. Then we have an action of $D(H)$ on $D(X)$ (see Example \ref{ex: monoidal structure of sheaf on group}), called $*$-convolution and denoted by $\star$ as before. On the other hand, via the monoidal functor $\Av^{\chi\mon}:D(H)\to D_{\chi\mon}(H)$, we may regard $D_{\chi\mon}(H)$ as a (left) $D(H)$-module. Similarly, we can regard $D_{\mono}(H)$ as a (left) $D(H)$-module by Proposition \ref{p: monoidal unit}.

We will let $\Modu_{D(H)}(\lincat_E)$ denote the ($2$-)category of left $D(H)$-modules in $\lincat_E$. Recall that all $D(H)$-linear functors between two $D(H)$-modules $M$ and $N$ form an $E$-linear category $\Hom_{D(H)}(M,N)$.

\begin{defn} \label{d:equivsheaves}
Let $X$ be a prestack with an action of a connected algebraic group $H$ (from the left).
    We let 
    \[
    D((H,\mono)\bs X):=\Hom_{D(H)}(D_{\mono}(H),D(X)),
    \]
    be the category of $H$-monodromic sheaves on $X$. For a residual character sheaf $\chi$, let
    \[
    D((H,\chi\mon)\bs X):=\Hom_{D(H)}(D_{\chi\mon}(H),D(X)),
    \]
    be the category of $(H,\chi)$-monodromic sheaves on $X$.
\end{defn}

Similarly, we define the above categories for right action: suppose $H$ acts on $X$ from the right, we define $D(X/ (H,\chi\mon))$ as $D((H,\chi^{\vee}\mon)\bs X)$, where we interpret $H$ acting on $X$ from the left by $hx := xh^{-1}$. 

As before, in the sequel we use $(\chi\on{-})\mono$ to denote either $\chi$-monodromic or all monodromic version.
Here are a few basic facts about monodromic categories of sheaves.

\begin{enumerate}
    \item Let $X=H$ equipped with the left $H$-action. Then we have an equivalence
    \begin{eqnarray*}
    \Hom_{D(H)}(D_{(\chi\on{-})\mono}(H),D(H))&\cong& D_{(\chi\on{-})\mono}(H),\\ F&\mapsto& F(\d^{(\chi\on{-})\mono})\\
    (-)\star\cF & \mapsfrom & \cF,
    \end{eqnarray*}
    of the right $D(H)$-modules. Here, as usual, the right $D(H)$-module structure on $\Hom_{D(H)}(-,D(H))$ comes from the right action of $D(H)$ on $D(H)$. That is to say, $D_{(\chi\on{-})\mono}(H)$ equipped with the right $D(H)$-module structure is a left dual of $D_{(\chi\on{-})\mono}(H)$ as a left $D(H)$-module. It follows that for general $X$, the category $D((H,(\chi\on{-})\mono)\bs X)$ of ($\chi$-)monodromic sheaves on $X$ can be identified with
    \[
    \Hom_{D(H)}(D_{(\chi\on{-})\mono}(H),D(X))\cong D_{(\chi\on{-})\mono}(H)\otimes_{D(H)}D(X).
    \]
    In particular, we have $D((H,(\chi\on{-})\mono)\bs H)\cong D_{(\chi\on{-})\mono}(H)$.
    
    \item As the adjoint pair of functors
    \[
    \iota^{(\chi\on{-})\mono}: D_{(\chi\on{-})\mono}(H) \rightleftharpoons D(H): \Av^{(\chi\on{-})\mono}
    \]
    realize $D_{(\chi\on{-})\mono}(H)$ as a colocalization of $D(H)$ as $D(H)$-modules, we see that we have a pair of adjoint functors 
    \[
    \iota_X^{\chi\mon}: D((H,\chi\mon)\bs X)\rightleftharpoons D(X): \Av_X^{\chi\mon}
    \]
    realizing $D((H,\chi\mon)\bs X)$ as a colocalization of $D(X)$. 

\end{enumerate}

\begin{rem}\label{rem: monodromy action}
    Since $D((H,\chi\mon)\bs X)$ is a $D_{\chi\mon}(H)$-module, the endomorphism algebra $\End(\dcm)$ acts on every object in $D((H,\chi\mon)\bs X)$ in a functorial way. We call such action the monodromic action.
\end{rem}

\sss{} We discuss some functoriality for monodromic categories.
\begin{lem}\label{lem: monodromic category, pullback}
    Let $f:X\to Y$ be an $H$-equivariant morphism of prestacks. Then $f^!: D(Y)\to D(X)$ restricts to $f^!: D((H,\chi\mon)\bs Y)\to D((H,\chi\mon)\bs X)$. 
\end{lem}
\begin{proof}
    This follows from the fact that the pullback functor $f^!$ is $D(H)$-linear (which in turn follows from the base change).
\end{proof}

\begin{lem}\label{lem: monodromic category, change group}
Let $f: H\to H'$ be an isogeny as in Proposition \ref{p:ch sh triv} and suppose that $X$ is an $H'$-prestack. Let $\chi'\in\Ch(H')$ and $\chi$ its pullback to $H$.
We have
\[
D((H',\chi'\mon)\bs X)\subset D((H,\chi\mon)\bs X)
\]
as full subcategories of $D(X)$.
\end{lem}

\begin{proof}
By Proposition \ref{lem: functoriality depth zero geom Langlands for tori}, $f_*$ sends $D_{\chi\mon}(H)$ to $D_{\mono}(H')$. It is compatible with the (right) action of $D(H')$ and $D(H)$. That is, for $\cF\in D(H)$ and $\cG\in D_{\chi\mon}(H)$, we have
\[
 f_*(\cG\star \cF)\cong  f_*(\cG)\star f_*(\cF).
\]
Thus, it induces a fully faithful embedding
$D_{\chi\mon}(H)\otimes_{D(H)}D(H')\subset D(H')$. Note that $D_{\chi'\mon}(H')\subset D_{\chi\mon}(H)\otimes_{D(H)}D(H')$. Therefore, we have
\begin{multline*}
D((H',\chi'\mon)\bs X)=D_{\chi'\mon}(H')\otimes_{D(H')}D(X)\subset\\
 D_{\chi\mon}(H)\otimes_{D(H)}D(H')\otimes_{D(H')}D(X)=D_{\chi\mon}(H)\otimes_{D(H)}D(X)=
D((H,\chi\mon)\bs X), 
\end{multline*}
compatible with full embeddings into $D(X)$.
\end{proof}

\begin{rem}
    Note however that the above inclusion is strict in general. Indeed, when $T\to T'$ is an isogeny of tori, then under the equivalence in Proposition \ref{prop: tame geometric Langlands for tori}, the category $D_{\chi\mon}(T)\otimes_{D(T)}D(T')$ can be identified with the category of ind-coherent sheaves on $\frf_\chi\times_{\mathrm{LS}_{\Tv}^{t, \Box}} \mathrm{LS}_{{T'}^\vee}^{t,\Box}$. This ind-scheme usually strictly contains $\frf_{\chi'}$.
\end{rem}

\begin{lemma}\label{l:desc mon}
Let $\pi: H\to H'$ be an isogeny as in Proposition \ref{p:ch sh triv}. Let $\chi'\in \Ch(H')$, and let $\chi\in\Ch(H)$ be its pullback to $H$.

Let $X$ be an algebraic stack with an action of $H$, and let $X'=H'\times^{H}X=\ker\pi\bs X$. Write $\pi_X:X\to X'$.
\begin{enumerate}
\item The pullback $(\pi_X)^!: D(X')\to D(X)$ restricts to a functor $D((H',\chi'\mon)\bs X')\to D((H,\chi\mon)\bs X)$.
\item The resulting functor is an equivalence if the degree of $\pi$ is invertible in $E$.
\end{enumerate}
\end{lemma}
\begin{proof}
We consider the action of $H$ on $X'$ as the intermediate object. Then (1) follows from Lemma \ref{lem: monodromic category, pullback} and Lemma \ref{lem: monodromic category, change group}.

Note that the right adjoint of $(\pi_X)^!=(\pi_X)^*$ is given by $\d^{\chi'\mon}\star (\pi_X)_*(-)$. 
We consider the Cartesian diagram
\[
\xymatrix{
H\times X \ar[rr]\ar[d] && X\ar[d]\\
H'\times X\ar[r] & H'\times X'\ar[r] & X'.
}
\]
Then base change gives $(\pi_X)^!(\d^{\chi'\mon}\star (\pi_X)_*(\cF))=\d^{\chi\mon}\star \cF=\cF$ for $\cF\in D_{\chi\mon}(X)$. Therefore, $\d^{\chi'\mon}\star (\pi_X)_*(-)$ is fully faithful. 

On the other hand, for $\cF\in D_{\chi'\mon}(X')$, we have the adjunction
\[
\cF\to \d^{\chi'\mon}\star(\pi_X)_*((\pi_X)^!\cF),
\]
whose $!$-pullback to along $\pi_X$ is an isomorphism by what we have just shown.
As $(\pi_X)^!$ is conservative (by descent), this implies that $\cF\to \d^{\chi'\mon}\star(\pi_X)_*((\pi_X)^!\cF)$ is also an isomorphism. (2) follows.
\end{proof}

\section{Loop groups}\label{ss:loop group}
In this section, we review and discuss some basic properties of loop groups over $k$, their central extensions and the affine Weyl group combinatorics.
A reference for twisted loop groups is \cite{PR}.

In addition to standard material, we introduce a notion of affine pinning for loop groups in \S\ref{ss:Cphi}, and prove a result about its stabilizer under the adjoint action of the loop group (see Proposition \ref{p:Mphi to Om}), which seems to be new. The discussion of affine pinnings will be used to construct the Soergel functor in \S\ref{s:Soergel}.

\subsection{Twisted loop groups}\label{ss:tw loop}

\sss{}
For a group scheme $G$ over $K=k\lr{t}$, let $LG$ be the associated loop group whose $R$-points are $G(R\lr{t})$, for any $k$-algebra $R$.

On the other hand, for a $k$-group $\GG$ with $\gg$ its Lie algebra,  we use the notation $\GG\lr{t}$ to denote the loop group associated to the base change $G=\GG\times_{\Spec k} \Spec K$, with $\gg\lr{t}$ its Lie algebra.

\sss{}
Let $G$ be a connected reductive group over $K=k\lr{t}$ that splits over a tamely ramified extension $K'/K$. Any such $K'$ is of the form $K'=k\lr{t^{1/e}}$ for some $e$ prime to $p$. We do not require $K'$ to be a minimal splitting field for $G$. 

Let $\GG$ be the split form of $G$, defined over $k$. Then $G$ is the descent of $\GG_{K'}$ with respect to the $\Gal(K'/K)\cong \mu_{e}(k)$-descent datum given by the diagonal action on $\GG$ and on $K'$. Fix a pinning of $\GG$, which in particular contains a maximal torus $\TT$, a pair of opposite Borel subgroups  $\BB$ and $\BB^{\opp}$ such that $\BB\cap \BB^{\opp}=\TT$. Since $G$ is quasi-split, we may assume that the action of $\mu_{e}$ on $\GG$ is via a pinned automorphism $\s: \mu_{e}\to \Aut^{\textup{pin}}(\GG)$.

Then we have
\begin{equation*}
LG=\GG\lr{t^{1/e}}^{\mu_{e}}
\end{equation*}
where $\z\in \mu_{e}$ acts on $\GG$ by $\s(\z)$ and on $t^{1/e}$ by multiplication by $\z$.  Let $\II\subset \GG\tl{t^{1/e}}$ be the standard Iwahori subgroup of $\GG\lr{t^{1/e}}$ corresponding to $\BB$. Then the neutral component $I$ of $LG\cap \II$ is an Iwahori subgroup of $LG$.

Note that $LG$ contains the reductive group $\GG^{\mu_{e}}$. Let $\GG_{0}=\GG^{\s(\mu_{e}),\c}$ be its neutral component. Let $\AA=\TT^{\s(\mu_{e}), \c}$. Then $\AA$ is a maximal torus of $\GG_{0}$. The $K$-torus $\AA_{K}$ is a maximal split torus of $G$.

\subsection{Affine roots}

The Lie algebra $\Lie G$ is a $K$-vector space. When we think of it as a $k$-vector space, it can be identified with the Lie algebra of $LG$, which we denote by $L\frg$. 

We have an action of the one-dimensional torus $\Grot$ on $L\frg$ by scaling $t^{1/e}$. 

Let $\Phi_{\aff}$ be the set of affine roots of $LG$: its elements are nontrivial characters $\a\in \xch(\AA\times \Grot)$ that appear in the action of $\AA\times \Grot$ on $L\frg$. For $\a\in \Phi_{\aff}$, denote by $(L\frg)_{\a}$ is the corresponding root space. An affine root $\a$ is real if $\a|_{\AA}$ is nontrivial, in which case $(L\frg)_{\a}$ is one-dimensional over $k$.

Positive affine roots $\Phi^{+}_{\aff}$ are those that appear in the adjoint action of $\AA\times \Grot$ on $\Lie I$. This determines a set of simple affine roots $S_{\aff}$. Let $\Phi^{-}_{\aff}=-\Phi^{+}_{\aff}$ be the set of negative affine roots. The Lie algebra $\Lie I$ is the $t$-adic completion of the span of $\fra=\Lie \AA$ and the positive affine root spaces $(L\frg)_{\a}$, $\a\in \Phi_{\aff}^{+}$.

\subsection{(Extended) Affine Weyl group}\label{ss:Wa}
Let $T\subset G$ be the centralizer of $\AA$, which is a maximal torus in $G$ (over $K$). Consider the normalizer $N_{LG}(\AA)$. It fits into an exact sequence
\begin{equation}\label{NA ses}
1\to LT\to N_{LG}(\AA)\to W\to 1
\end{equation}
where $W=W(G,T)$ can be identified with $\WW^{\s(\mu_{e})}$ (where $\WW$ is the Weyl group of $\GG$ with respect to $\TT$). The inclusion $\GG_{0}\incl \GG$ induces an isomorphism 
\[
W(\GG_{0}, \AA)\isom \WW^{\s(\mu_{e})}=W.
\]

Let
\begin{equation*}
\tilW=\pi_{0}(N_{LG}(\AA))
\end{equation*}
be the extended affine Weyl group of $LG$. Using the canonical isomorphism $\pi_{0}(LT)\cong \xcoch(\TT)_{\mu_{e}}$, we get an exact sequence
\begin{equation}\label{eq: extended affine weyl}
1\to \xcoch(\TT)_{\mu_{e}}\to \tilW\to W\to 1.
\end{equation}

The affine Weyl group $\Wa\subset \tilW$ is generated by reflections $s_{\a}$ defined by affine real roots $\a$. It is a Coxeter group with simple reflections $s_{\a}$ for $\a\in S_{\aff}$. Let $\ell: W_{\aff}\to \ZZ_{\ge0}$ be the length function of the Coxeter group $\Wa$.

The group $\tilW$ permutes $\Phi_{\aff}$. The stabilizer of the set of affine simple roots $S_{\aff}$ is denoted by $\Om$. We have
\begin{equation*}
\tilW=\Wa\rtimes \Om.
\end{equation*}
We extend the length function to the whole $\tilW$ by declaring
\begin{equation*}
\ell(w\om)=\ell(w), \quad \forall w\in \Wa, \om\in \Om.
\end{equation*}
In particular, $\Om$ is the set of length zero elements in $\tilW$.

The group $\Om\cong \tilW/W_{\aff}$ is also identified with $\pi_{0}(LG)$.

\subsection{Globalization of $G$}\label{ss:glob G}

\sss{The group scheme $\cG$}
We use the notation from \S \ref{ss:tw loop}. In particular, $G$ is a connected reductive group over $K=k\lr{t}$ that splits over a tamely ramified extension $K'/K$. For such $G$, we can associate it with a group scheme $\cG$ over $\p^1=\p^{1}_{k}$ together with an isomorphism $\cG|_{\Spec K}\cong G$, as in \cite[Section 1.2]{HNY}. Here $\Spec K$ is identified with the formal puncture disk of $0$ in $\p^{1}$, such that $t$ is the standard coordinate on $\p^{1}$. 

Let $[e]: \p^{1}_{k}\to \p^{1}_{k}$ be the $e$-th power map in terms of standard coordinates. Let $[e]_{*}(\GG\times\p^{1})$ be the Weil restriction of the constant group scheme $\GG\times\p^{1}$ along $[e]$. Then  $\cG$ is defined as the fiberwise neutral component of $[e]_*(\GG\times \p^1)^{\mu_{e}}$, where $\mu_{e}$ acts diagonally on $\GG$ (via the pinned action $\s$) and on $\p^{1}$ by rotation.

\sss{Polynomial loop group}
We let $L_{\mathrm{pol}}G=\Res_{\GG_m/k}(\GG\times\GG_m)^{\mu_e}$ be the polynomial loop group of $G$, so $L_{\mathrm{pol}}G(R)=\cG(R[t,t^{-1}])=\GG(R[t^{1/e},t^{-1/e}])^{\mu_e}$. Note that $L_{\mathrm{pol}}G$ can be non-reduced.

\sss{The group $M$}\label{sss:M}
We let $M:=N_{L_{\mathrm{pol}}G}(I,\AA)^{\red}$ be the reduced normalizer of $(I, \AA)$ in $L_{\mathrm{pol}}G$, which is an ind-group over $k$. This group will only be used in the construction of the Soergel functor.

\begin{prop}\label{M ses}
We have an exact sequence
\begin{equation}\label{eq: AMOm}
1\to \AA\to M\to \Omega\to 1.
\end{equation}

\end{prop}

\begin{proof}
For ease of reading, we break the proof into a sequence of steps. 

{\it Step 1.} We first show that the sequence \eqref{eq: AMOm} is short exact when $\GG$ is a torus.

In this case, we denote $\GG,\cG$ and $G$ by $\TT, \cT$ and $T$. We have  $Z\GG_{0}=\AA=\TT^{\mu_{e},\circ}$. The norm map relates the situation of the split torus $\TT\times \wt\Gm$ (where $\wt\Gm=\Spec k[t^{1/e}, t^{-1/e}]$, viewed as a $\mu_{e}$-cover of $\Gm$) and $\cT$:

\begin{equation*}
\xymatrix{\TT\ar[d]^{\Nm}\ar[r] & (R_{\wt\Gm/k}(\TT\times \wt\Gm))^{\red}\ar[d]^{\Nm}\ar[r]^-{\wt q} & \xcoch(\TT)\ar[d]^{\ov\Nm}\\
\AA\ar[r] & (R_{\Gm/k}\cT)^{\red}\ar[r]^-{q} & \pi_{0}(LT)
}
\end{equation*}
We have $R_{\wt\Gm/k}(\TT\times\wt\Gm)^{\red}\cong \TT\times \xcoch(\TT)$, where $(a,\l)\in \TT\times \xcoch(\TT)$ corresponds to the section $a\l: \wt\Gm\to \TT$. The map $\wt q$ is given by the natural projection. Therefore the first row is short exact. By \cite[Theorem 5.1, part A]{PR}, $\ov \Nm: \xcoch(\TT)\to \pi_{0}(LT)$ identifies the latter with the $\mu_{e}$-coinvariants of $\xcoch(\TT)$. In particular, $\ov\Nm$ is surjective, hence $q$ is surjective. 

It remains to show that $\ker(q)$, which obviously contains $\AA$, is exactly $\AA$. By definition, $\ker(q)$ is contained in the neutral component of $LT$, which is pro-algebraic with reductive quotient $\AA$. We also know that $\ker(q)$ is reductive because it is a finite type subgroup of the $\mu_{e}$-invariants of $(R_{\wt\Gm/k}(\TT\times \wt\Gm))^{\red}$. Therefore $\ker(q)$ has to be equal to $\AA$.

{\it Step 2.} We show that the natural map $M \to \Om$ is surjective for a general reductive group $\GG$.

Recall the exact sequence 
\begin{equation} \label{eq: NLGcomponent} 1 \to (LT)^{\circ} \to N_{LG}(\AA) \to \extw \to 1.
\end{equation}
We claim that there exists a set-theoretic section of the map $N_{LG}(\AA) \to \extw$ whose image lies in the polynomial loop group.  

Note that $(L^+\GG)^{\mu_e, \circ}$ is a special parahoric subgroup of $LG$, which provides a splitting of the exact sequence \eqref{eq: extended affine weyl}. Moreover we can identify $W$ with $\WW^{\mu_e}$. Under this identification, we lift $W$ to the polynomial loop group via the Tits section. The argument in Step 1 shows that $\xcoch(\TT)_{\mu_{e}}$ can also be lifted to the polynomial loop group. Hence, the map $N_{L_{\mathrm{pol}} G}(\AA)^{\red} \to \extw$ is surjective. 

By restricting to the length zero elements of $\extw$, we conclude that $M \to \Om$ is surjective. 

{\it Step 3.} We show that the kernel of $M \to \Om $ is exactly $\AA$. 

By the exact sequence \eqref{eq: NLGcomponent}, the kernel of $M \to \Om$ is contained in the intersection $(LT)^{\circ} \cap L_{\mathrm{pol}}G$.
Since this intersection coincides with $(LT)^{\circ} \cap L_{\mathrm{pol}}T$, it remains to verify that $(LT)^{\circ} \cap L_{\mathrm{pol}}T = \AA$, which was proved in Step 1. 
\end{proof}

\subsection{Central extension}\label{ss: cent ext}

We will need a certain central extension $\tLG$ of $LG$ by a one-dimensional torus $\gmc$
\begin{equation*}
	1 \ra \gmc \ra  \tLG \ra LG \ra 1.
\end{equation*}

\subsubsection{Basic central extension}
The construction of the central extension is by the well-known procedure of determinant line bundles, which we recall below. 

The starting point is the determinant line bundle $\cL_{\det}$ on the affine Grassmannian $\Gr_{\GL_n}$ for $G=\GL_n$: its fiber over an $R$-point $\L\in \Gr_{\GL_n}(R)$, viewed as a projective $R\tl{t}$-submodule $\L\subset R\lr{t}^n$, is the line $\det_R(\L_0/\L\cap \L_0)\ot \det_R(\L/\L\cap \L_0)^{-1}$. Here $\L_0=R\tl{t}^n\in \Gr_{\GL_n}(R)$ is the standard lattice, $\det_R(M)\in \Pic(R)$ means taking the top exterior power of a projective $R$-module $M$ of finite rank. We still denote the pullback of $\cL_{\det}$ to $LG$ by $\cL_{\det}$. It carries a canonical multiplicative structure: a canonical isomorphism $\mult^*\cL_{\det}\cong \cL_{\det}\bt\cL_{\det}$ over $LG\times LG$ (where $\mult$ is the multiplication map on $LG$) satisfying further associativity. The $\Gm$-torsor $\tLG\to LG$ obtained by removing the zero section from the total space of $\cL_{\det}$ thus gives a central extension of $LG$, called the {\em basic central extension} of $LG$ for $G=\GL_n$.

Now for a connected reductive group $G$ over $K$ as in \S\ref{ss:tw loop}, we choose an algebraic representation $V$ of $\GG$ (a finite-dimensional $K'$-vector space)
\begin{equation*}
    \rho_V: \GG \to \GL(V)
\end{equation*}
with finite kernel. Then $\rho_V$ induces a homomorphism of ind-groups
\begin{equation}
    LG \inj \GG\lr{t^{1/e}} \to \GL(V)\lr{t^{1/e}}.
\end{equation}
By pulling back the basic central extension of $\GL(V)\lr{t^{1/e}}$ constructed in the previous paragraph, we get a central extension $\tLG\to LG$, depending on the choice of $V$.

For notations having to do with the centrally extended loop group $\tLG$, we add a tilde to the corresponding notation for the loop group $LG$. For example, $\wt \AA$ and $\tI$ are the maximal torus and Iwahori of $\tLG$ respectively. 

\subsubsection{Commutator pairing} \label{sss:commutator}
We will compute the commutator pairing of this
central extension when restrict to the torus.  

Let $P(V)$ be the multi-set of weights of $V$. We define $\langle \cdot, \cdot \rangle = \langle \cdot, \cdot \rangle_V : \xcoch(\TT) \otimes \xcoch(\TT) \to \ZZ$, by declaring that 
$$
\langle \cmu , \clambda \rangle_V := \sum_{\om \in P(V)} \om(\cmu) \om(\clambda), \qquad \cmu, \clambda \in \xcoch(\TT).$$ Note that $\langle \cdot, \cdot \rangle_V$ is positive definite by our assumption on $V$. For $\cmu \in \xcoch(\TT)$, we define $\cmu^*$ to be the unique element in $\xch(\TT)$ such that $\cmu^*(-) = \langle \cmu, - \rangle$.

The central extension $\widetilde{L\TT}$ gives rise to a bilinear commutator pairing $$c: L\TT \times L\TT \ra \gmc,$$ which sends $(g_1, g_2)$ to the commutator $[\wt g_{1}, \wt g_{2}]=\wt g_{1}\wt g_{2}\wt g^{-1}_{1}\wt g^{-1}_{2}$, for any liftings $\wt g_i$ of $g_i$. From \cite[Section 9]{Li}, we know that the commutator pairing of 
\begin{equation}\label{eq: commutator pairing}
c(\cmu(t^{1/e}), \clambda(x)) = x^{\langle \cmu , \clambda \rangle_V}, \qquad x \in \gm.
\end{equation}

\sss{Action of $\tilW$ on $\wt\AA$}\label{SSS: action of tilW on wtAA}
Since $N_{\tLG}(\wt\AA)$ normalizes $\wt \AA$, we know that $\extw$ acts on $\xcoch(\wt \AA)$ and $\xch(\wt \AA)$. We now describe the action of $\xcoch(\TT)_{\mu_e} \subset \extw$ on $\xcoch(\wt \AA)$ and $\xch(\wt \AA)$ explicitly using the pairing $\langle \cdot, \cdot \rangle_V$.

Let $K_{c}$ (resp. $\Lambda_{c}$) be a generator of $\xcoch (\gmc)$ (resp. $\xch (\gmc)$) that are dual to each other. Note that the natural map $\AA \inj \TT$ induces an injection $\xcoch(\AA) \inj \xcoch(\TT)$ and a surjection $\xch(\TT) \surj \xch(\AA)$. Hence, it makes sense to define $\langle \sigma \cmu , \clambda \rangle$ and view $\cmu^*$ as an element in $\xch(\AA)$ in the following lemma. 

\begin{lem} \label{l:conj action}
    Let $\nu$ be the composition
    \begin{equation}
        \nu: \xcoch(\TT) \surj \xcoch(\TT)_{\mu_e}\subset \tilW
    \end{equation}
    where the first map is the natural projection. For $\cmu \in \xcoch(\TT)$ and $\clambda \in \xcoch(\AA)$, the element $\nu(\cmu)\in \tilW$ sends $\clambda$ to $\clambda + \sum_{\sigma \in \mu_e} \langle \sigma \cmu , \clambda \rangle K_{c}$. Dually, the element $\nu(\cmu)$ sends $\chi$ to $\chi - \chi(K_{c}) \sum_{\sigma \in \mu_{e}}\sigma\cmu^*$ for $\chi \in \xch(\wt \AA)$.
\end{lem}
\begin{proof}
    Let $\Nm: \TT(K')\to T(K)$ be the norm map. For  any $\cmu \in \xcoch(\TT)$, $\Nm(\cmu(t^{1/e}))\in T(K)=LT(k)$ is a representative of $\nu(\cmu)$. From \eqref{eq: commutator pairing}, we know that
    \begin{equation}\label{comm sigma}
        c((\sigma \cmu)(t^{1/e}), \clambda(x)) = x^{\langle \sigma \cmu , \clambda \rangle_V}, \qquad \sigma \in \mu_e, x \in \gm.
    \end{equation}
    Now $\Nm(\cmu(t^{1/e}))=h\prod_{\s\in \mu_e}(\sigma \cmu)(t^{1/e})$ for some $h\in \TT$. The bilinearity of the commutator pairing $c$ and \eqref{comm sigma} imply that 
    \begin{equation*}
        c(\Nm(\cmu(t^{1/e})), \clambda(x))=x^{\j{\sum_{\s\in\mu_e}\s\cmu, \clambda}_V}\in \gmc.
    \end{equation*}
    The left side above is $\nu(\cmu)(\clambda)-\clambda$. The first statement of the lemma now follows. 
\end{proof}

\begin{prop} \label{p: affine weyl faithful}
     The subgroup of $\extw$ that acts trivially on $\wt \AA$ coincides with the torsion part of $\xcoch(\TT)_{\mu_e}$. In other words, we have an exact sequence, $$1 \ra \xcoch(\TT)_{\mu_e, \tors} \ra \extw \ra \Aut(\wt \AA).$$ In particular, $W_{\aff}$ acts faithfully on $\wt \AA$. 
\end{prop}

\begin{proof}
    It is clear that $W$ acts on $\AA$ faithfully. In view of the exact sequence \eqref{eq: extended affine weyl}, if $w\in \tilW$ acts trivially on $\wt\AA$, we must have $w\in \xcoch(\TT)_{\mu_e}$. Thus  
    it suffices to show that the kernel of  $\xcoch(\TT)_{\mu_e}\to \Aut(\wt\AA)$ is precisely the torsion subgroup $\xcoch(\TT)_{\mu_e, \tors}$.

    The action of $\xcoch(\TT)_{\mu_e}$ preserves $\gmc$ and is trivial on $\gmc$ and $\AA=\wt\AA/\gmc$. Therefore the action of $\xcoch(\TT)_{\mu_e}$ on $\wt\AA$ factors through a homomorphism $\xcoch(\TT)_{\mu_e}\to \Hom(\AA, \gmc)$. Since the latter is torsion-free, $\xcoch(\TT)_{\mu_e,\tors}$ must act trivially on $\wt\AA$.

    Now suppose $\cmu\in \xcoch(\TT)$ and $\nu (\cmu)\in \tilW$ acts trivially on $\wt \AA$. Lemma \ref{l:conj action} implies that $$\sum_{\sigma \in \mu_e} \langle \sigma \cmu , \clambda \rangle_V =0 $$ for all $\clambda \in \xcoch(\AA)$. In particular, setting $\clambda =\sum_{\sigma \in \mu_e} \sigma \cmu$, we obtain $$\langle \sum_{\sigma \in \mu_e} \sigma \cmu , \sum_{\sigma \in \mu_e} \sigma \cmu \rangle_V =0. $$ Hence, $\sum_{\sigma \in \mu_e} \sigma \cmu = 0$ since $\j{\cdot, \cdot}_V$ is positive definite.  The averaging map $\xcoch(\TT)_{\mu_e}\to \xcoch(\TT)^{\mu_e}$ sending the class of $\cmu$ to $\sum_{\s\in \mu_e}\s\cmu$ is an isomorphism after tensoring with $\QQ$, therefore its kernel is $\xcoch(\TT)_{\mu_e,\tors}$, which implies the class of $\cmu$ in $\xcoch(\TT)_{\mu_e}$ is torsion. 
    
    The kernel of the surjection $\Wa\to W$ is torsion-free, therefore $\Wa\cap \xcoch(\TT)_{\mu_e, \tors}=\{1\}$, hence $\Wa$ acts faithfully on $\wt\AA$.
\end{proof}

A direct corollary of Proposition \ref{p: affine weyl faithful} is as follows. 

\begin{cor} \label{c: untwisted faithful}
    If $G$ is split over $K$, then $\extw$ acts faithfully on $\wt \AA$. 
\end{cor}

\subsection{Affine pinning}\label{ss:Cphi}
\begin{defn}
    An affine pinning of $LG$ consists of $(I,\AA,\phi)$, where $I$ is an Iwahori subgroup of $LG$, $\AA$ a maximal torus of $I$, and $\phi: I^+\to \GG_a$ is a group homomorphism such that its differential $d\phi$ restricts to an isomorphism $(L\frg)_\alpha\cong k$ for every simple affine root $\a\in S_{\aff}$.
\end{defn}

We fix a $\mu_e$-invariant pair $(\BB, \TT)$ of $\GG$ as in \S \ref{ss:tw loop}, along with a globalization of $\GG$ as in \S \ref{ss:glob G}. This provides a natural choice of $(I, \AA)$ and $L_{\mathrm{pol}}G$.

For $\a\in S_{\aff}$, we write 
\begin{equation*}
    0\neq (d\phi)_{\a}\in ((L\frg)_{\a})^*
\end{equation*}
for the isomorphism $(L\frg)_\alpha\cong k$.  We identify $d\phi$ with the sum $\sum_{\a\in S_{\aff}}(d\phi)_{\a}$ as an element in $(L\frg)^*$.

\sss{The group $M_\phi$}\label{sss:M phi}
Recall that $M=N_{L_{\mathrm{pol}}G}(I,\AA)^{\red}$. Let
\begin{equation*}
    M_\phi:=N_{L_{\mathrm{pol}}G}(I,\AA,\phi)^{\red}.
\end{equation*}
Then $M_\phi$ is the subgroup of $M$ consisting of those elements that fix $\phi$.

\begin{prop}
\label{p:Mphi to Om}
The obvious inclusion $Z(\GG)\cap \AA\incl M_{\phi}$ and the projection $M_\phi\to \pi_0(LG)\cong \Om$ give a short exact sequence 
\begin{equation}\label{Mphi ses}
1\to Z(\GG)\cap \AA\to M_{\phi}\to \Om\to 1
\end{equation}
compatible with the short exact sequence \eqref{eq: AMOm}.    
\end{prop}
\begin{proof}
First we check that $M_{\phi}\cap \AA=Z(\GG)\cap \AA$. The inclusion $Z(\GG)\cap \AA\subset M_{\phi}$ is clear. Conversely, if $a\in \AA$ commutes with $d\phi$, then it commutes with the nonzero linear form $(d\phi)_\a\in (L\frg)_{\a}^*$ for each affine simple root $\a\in S_{\aff}$. Therefore $\a(a)=1$ for all $\a\in S_{\aff}$, which implies $a\in Z(\GG)$.

It remains to show that each $\om\in \Om$ can be lifted to an element $\dot\om\in M(k)$ that commutes with $d\phi$.

Let $S_{\aff}=\coprod_{\xi\in \Xi} S_{\xi}$ be the decomposition of affine simple roots according to the  irreducible components of the affine root system of $LG$.   Let $\ov\a\in \xch(\AA)$ be the restriction $\a|_{\AA}$. For each $\xi\in \Xi$,  $\{\ov\a\}_{\a\in S_{\xi}}$ satisfy a unique linear relation $\sum_{\a\in S_{\xi}}n_{\a}\ov\a=0$ with positive integer coefficients $\{n_{\a}\}_{\a\in S_{\xi}}$, such that the gcd of all $n_\a$ is equal to $1$. We have an exact sequence
\begin{equation}\label{A surj}
\AA\to \prod_{\a\in S_{\xi}}\Gm\xr{\nu} \prod_{\xi\in \Xi}\Gm\to 1
\end{equation}
where the $\xi$-coordinate of $\nu((c_{\a})_{\a\in S_{\aff}})$ is given by $\prod_{\a\in S_\xi}c_{\a}^{n_{\a}}$. Now let $\dot\om\in M(k)$ be any lifting of $\om$, and consider its coadjoint action on the $k$-line
\begin{equation*}
\Ad(\dot \om)\curvearrowright \LL_\xi:=\bigotimes_{\a\in S_{\xi}}((L\frg)_{\a}^*)^{\ot n_{\a}}.
\end{equation*}
Here we use that the permutation of $\om$ on $S_{\xi}$ preserves the function $\a\mapsto n_{\a}$. If we change $\dot \om$ to $a\dot \om $ for some $a\in \AA(k)$, the resulting action $\Ad(a\dot \om)$ on the line $\LL_\xi$ is remains unchanged, as $\prod_{\a\in S_{\xi}}\ov\a(a)^{n_{\a}}=1$. Therefore we get a well-defined scalar $\l_{\xi}(\om)\in k^{\times}$ by which $\Ad(\dot \om)$ acts on $\LL_\xi$, which is independent of the lifting $\dot \om$. In other words, there is a canonical action of $\Om$ on the line $\LL_\xi$ given by a homomorphism
\begin{equation}\label{Om scalar}
\l_\xi: \Om\to k^{\times}.
\end{equation}

By Lemma \ref{l:Om scalar} below, this homomorphism $\l_\xi$ is always trivial. Let $\dot \om\in M(k)$ be any lifting of $\om$. Write $\Ad(\dot \om)(d\phi)_{\a}=c_{\a}(d\phi)_{\om(\a)}$ for $\a\in S_{\aff}$. Then $\prod_{\a\in S_\xi}c_{\a}^{n_{\a}}=\l_\xi(\om)=1$ for each $\xi\in \Xi$. By \eqref{A surj}, there exists $a\in \AA(k)$ such that $\a(a)=c_{\a}^{-1}$ for all $\a\in S_{\aff}$. Then $\Ad(\dot \om a)$ sends $\phi_{\a}$ to $\phi_{\om(\a)}$. We conclude that $\dot \om a\in M_{\phi}(k)$ has image $\om \in \Om$. This completes the proof of the lemma.
\end{proof}

\begin{lemma}\label{l:Om scalar}
Recall $p=\chk$ is coprime to $e$. Then the homomorphism $\l_\xi$ in \eqref{Om scalar} is always trivial for all $\xi\in \Xi$. 
\end{lemma}
\begin{proof}
Let $G\to G^{\ad}$ be the adjoint quotient, and we similarly have a homomorphism $\l^{\ad}_{\xi}: \Om_{G^{\ad}}\to k^{\times}$. From the construction of $\l_\xi$ it is clear that it factors through $\l^{\ad}_\xi$ via the natural map $\Om\to \Om_{G^{\ad}}$. Therefore it suffices to treat the case where $G$ is of adjoint type. We can then further reduce to the case where the affine root system of $LG$ is irreducible. We drop the subscript $\xi$ in the rest of the proof.

The loop group $LG$ has a model over $\ZZ[\mu_e][1/e]$, using which we get an integral version of $\l$, a homomorphism $\L: \Om\to  \ZZ[\mu_e][1/e]^\times$ such that when $k$ has characteristic $p>0$, $\l$ is obtained from $\L$ by reduction mod $p$ (and embedding $\FF_p^\times\incl k^\times$). It therefore suffices to show that $\L$ is trivial, or equivalently, that  $\L_\CC: \Om\xr{\L}\ZZ[\mu_e][1/e]^\times \subset \CC^\times$ is trivial. We may therefore assume that the base field $k=\CC$.

We recall the Kostant section of $\bgg$. Fix a $\Out(\bgg)$-invariant regular nilpotent element $\mathbf{f}$ as a sum of elements in $\bgg_{-\b}$ for simple roots $\b$. Extend it to a principal $\frs\frl_2$-triple $\{\mathbf{e},\mathbf{h},\mathbf{f}\}$ in $\bgg^{\Out(\bgg)}$. Then the Kostant section (of $\mathbf{f}$ inside $\bgg$) is defined as $\mathbf{f}+\bgg^{\mathbf{e}}$. The composed map $\mathbf{f}+\bgg^\mathbf{e}\to \bgg\to \cc=\bgg \sslash \GG$ is an isomorphism, equivariant with respect to the $\Out(\bgg)$-action. 
It follows that we have an isomorphism $\mathbf{f}+\frg^\mathbf{e}\cong\frc$ over $K$ and therefore an isomorphism
\[\kappa: \mathbf{f}+L\frg^\mathbf{e}\cong L\frc.\]

We use the Killing form to identify $(L\frg)^*_{\alpha}$ with $(L\frg)_{-\alpha}$. Let $V=\oplus_{\alpha\in S_{\aff}} (L\frg)_{-\alpha}$. We have an isomorphism $V\sslash \AA\isom \LL$ given by the formula $\nu(x_{-\a})=\ot (x_{-\a})^{\ot n_{\a}}$, where $x_{-\a}\in (L\frg)_{-\a}$ for $\a\in S_{\aff}$. On the other hand, $(\mathbf{f}+\bgg^\mathbf{e})\cap V$ consists of elements of the form $\mathbf{f}+(L\frg)_{-\a_0}$, where $\a_0\in S_{\aff}$ is the unique affine simple root that is not a root of $\bgg_0$. The composition
\begin{equation*}
    \b: (\mathbf{f}+L\frg^\mathbf{e})\cap V\subset V\to V\sslash \AA\cong \LL
\end{equation*}
sends $\mathbf{f}+x_0$ to $x_0^{n_{\a_0}}=x_0\in (L\frg)_{-\a_0}$. In particular, $\b$ is an isomorphism. We have the following commutative diagram
\[
\xymatrix{
(\mathbf{f}+L\frg^\mathbf{e})\cap V\ar[r]^-{\b}\ar[d]^{i}& V\sslash \AA\cong \LL\ar[d]^{\gamma}\\
\mathbf{f}+L\frg^\mathbf{e} \ar[r]^-{\kappa}& L\frc
}
\]

We have already seen that both $\b$ and  $\kappa$ are isomorphisms. Since $i$ is an embedding, so is $\gamma$. Since $\gamma$ is equivariant under $\Ad(\dot\om)$, which acts trivially on $L\frc$, therefore $\Ad(\dot\om)$ acts trivially on $\LL$. 
\end{proof}

We record one more property of $M_\phi$.
\begin{prop}
    The group $M_\phi$ is commutative.
\end{prop}
\begin{proof}
    Clearly \eqref{Mphi ses} is a central extension of $\Omega$ by $Z(\GG_0)\cap\AA$. It is enough to show that the induced commutator pairing $\Omega\times\Omega\to Z(\GG_0)\cap\AA$ is trivial. As in the proof of Lemma \ref{l:Om scalar}, this pairing is in fact defined over a suitable localization of $\ZZ[\mu_e]$ and there we can check the triviality for $k=\CC$. Then as before, we regard $d\phi$ as an element in $L\frg$. We claim that it is a regular element, regarded as an element in the Lie algebra of $G$. (Recall that $G$ is a reductive group over the local field $K=k\lr{t}$.) If this is the case, then $M_\phi$ is a subgroup of the loop group of the centralizer $C_G(d\phi)$ of $d\phi$ in $G$, and therefore it is commutative.
    
    To prove the claim, we may even base change to $K'$. Then $d\phi$ is an element in $\bgg\lr{t^{1/e}}$. In fact, it belongs to $\bgg\tl{t^{1/e}}$, is a regular nilpotent element modulo $t^{1/e}$, as it is a sum over simple roots of $\GG$ of non-zero root vectors. This implies that $d\phi$ is regular, as desired.  
\end{proof}

\begin{rem}
    In fact, one can show that when the characteristic $p$ of $k$ is zero or larger than the Coxeter number of each simple factor of $\GG$, then $d\phi$ is regular semisimple. Its centralizer in $G$ is a (non-split) maximal torus $A_\phi$. Then $I_\phi:=(LA_\phi\cap I)^\circ$ is the unique Iwahori subgroup of $LA_\phi$. If we let $I_\phi^+$ denote the pro-unipotent radical of $I_\phi$, then $M_\phi\cong LA_\phi/I_\phi^+$.
\end{rem}

\subsection{A subgroup of $\wt M_\phi$}
Fix a determinantal central extension $\tLG \to LG$ as in \S\ref{ss: cent ext}.  Let $\wt M_\phi\subset \wt M$ be the central extension obtained by pulling back $\tLG \to LG$ along the embedding $M_\phi\subset \Lpol G\to LG$. By Proposition \ref{p:Mphi to Om}, there is a surjection
\begin{equation}\label{wt Mphi to Om}
    \wt M_\phi\surj \Om.
\end{equation}
For the purpose of constructing the Soergel functor in \S\ref{s:Soergel}, it would be ideal if this section has a homomorphic section. While we cannot guarantee that, in this subsection we construct a multi-section to \eqref{wt Mphi to Om}, i.e., a subgroup $\Sig\subset \wt M_\phi$ that surjects onto $\Om$ with finite kernel well under control. The result here will only be used in the proof of Theorem \ref{thm: monodromic-action}.

\sss{}\label{sss:section xcoch S}
Recall the connected center $\SS$ of $\GG$. The lattice $\xcoch(\SS)^{\mu_e}$ is the cocharacter lattice of the split part of the connected center $(ZG)^\circ$.  For $\cmu\in \xcoch(\SS)^{\mu_e}$ we have $t^{\cmu}\in M_\phi$. This gives a homomorphism
\begin{equation}
    \tau: \xcoch(\SS)^{\mu_e}\to M_\phi(k).
\end{equation}
By our assumption on the central extension in \S\ref{ss: cent ext}, the commutator pairing of the central extension $\wt{L'\GG}$ restricted to $(t^{1/e})^{\xcoch(\SS)}$ is trivial. In particular, the commutator pairing of the central extension $\wt M_\phi$ of $M_\phi$ is trivial when restricted to the image of $\tau$. Since $\xcoch(\SS)^{\mu_e}$ is free abelian, there exists a homomorphism
\begin{equation}\label{wt tau}
    \wt\tau: \xcoch(\SS)^{\mu_e}\to \wt M_\phi(k)
\end{equation}
lifting $\tau$.

\begin{con}\label{cons: Sigma}
    Fix the choice of a lifting $\wt\tau$ of $\tau$ as in \eqref{wt tau}, we construct  canonically a subgroup $\Sig\subset \wt M_\phi(k)$ containing $\Im(\wt\tau)$ with the following properties:
    \begin{itemize}
        \item The natural projection
        \begin{equation*}
            \pi_{\Sig}: \Sig\incl \wt M_\phi(k) \to \Om
        \end{equation*}
        is surjective.
        \item The kernel  $\Sig_{1}:=\ker(\pi_{\Sig})$ is a finite abelian group.
    \end{itemize}
\end{con}

The rest of the subsection is devoted to the construction of $\Sigma$.

\sss{Step 1}\label{sss:ov tau fin coker} Composing $\tau$ with the projection to $\pi_0(M_\phi)$ we get a homomorphism
\begin{equation}
    \ov\tau: \xcoch(\SS)^{\mu_e}\to \pi_0(M_\phi).
\end{equation}
We claim that $\ov\tau$ is injective with finite cokernel. Indeed, if we further project to $\Om$:
\begin{equation*}
        \ov\tau_\Om: \xcoch(\SS)^{\mu_e}\xr{\ov\tau}\pi_0(M_\phi)\surj\Om=\xcoch(\TT)_{\mu_e}/\xcoch(\TT^{\on{sc}})_{\mu_e},
\end{equation*}
the map $\ov\tau_\Om$ becomes an isomorphism after tensoring with $\QQ$. Since $\xcoch(\SS)^{\mu_e}$ is free abelian, $\ov\tau$ is injective with finite cokernel. Denote by $N$ the exponent of $\coker(\ov\tau)$.

\begin{rem}
    With a more careful analysis, one can show that the exponent of  $\coker(\ov\tau)$ divides $e^2f_{\GG}$, where $f_{\GG}$ is the determinant of the Cartan matrix of $\GG$. However, we do not need this information for the construction.
\end{rem}

\sss{Step 2} Consider the abelian group
\begin{equation}
    M^\flat_\phi(k)= M_\phi (k)/ \Im(\tau).
\end{equation}
Rewrite \eqref{Mphi ses} as a short exact sequence of abelian groups
\begin{equation}\label{Mphi split}
1\to (Z\GG\cap \AA)^{\c}(k)\to M_{\phi}(k)\to \pi_{0}(M_{\phi})\to 1.
\end{equation} 
By construction we have an exact sequence
\begin{equation}
    1\to (Z\GG\cap \AA)^\circ(k)\to M_\phi^\flat(k)\to \coker(\ov\tau)\to 1.
\end{equation}
Let
\begin{equation}
    \Gamma^\flat=M_\phi^\flat(k)[N]
\end{equation}
be the $N$-torsion subgroup of $M_\phi^\flat(k)$.

Using the fact that $(Z\GG\cap \AA)^\circ(k)$ is divisible, the following lemma is easy to see.
\begin{lemma}\label{l:Gamma flat}
    The group $\Gamma^\flat$ is finite and surjects onto $\pi_0(M_\phi^\flat)=\coker(\ov\tau)$.
\end{lemma}
\begin{proof}
    Finiteness follows from that both $(Z\GG\cap \AA)^\circ[N]$ and $\coker(\ov\tau)$ are finite. Since $(Z\GG\cap \AA)^\circ$ is a torus, $(Z\GG\cap \AA)^\circ(k)$ is divisible. Therefore every element of $\coker(\ov\tau)$ lifts to an element in $M_\phi^\flat(k)$ of the same order. This implies that $\Gamma^\flat\to \coker(\ov\tau)$ is surjective.
\end{proof}

Let $\Gamma\subset M_\phi(k)$ be the preimage of $\Gamma^\flat$ under the projection $M_\phi(k)\to  M^\flat_\phi(k)$. 

\begin{lemma}
    The projection $\Gamma\to \pi_0(M_\phi)$ is surjective with finite kernel of exponent dividing $N$.
\end{lemma}
\begin{proof}
    We have a commutative diagram where both rows are exact
\begin{equation}
    \xymatrix{1\ar[r] &  \Im(\tau)\ar[r]\ar[d]^{\cong} & \Gamma\ar[r] \ar[d]^{\gamma} & \Gamma^\flat\ar[r] \ar[d]^{\gamma^\flat} & 1\\
    1\ar[r] &  \Im(\ov\tau)\ar[r] & \pi_0(M_\phi)\ar[r] & \coker(\ov\tau) \ar[r] & 1}
\end{equation}
By Lemma \ref{l:Gamma flat}, $\gamma^\flat$ is surjective, which implies $\gamma$ is also surjective. Moreover, since the left vertical map is an isomorphism, $\ker(\gamma)\cong \ker(\gamma^\flat)$, which is finite and $N$-torsion.
\end{proof}

\sss{Step 3} In this final step, we construct a multi-section of $\Gamma$ to $\wt M_\phi$ containing $\Im(\wt\tau)$. Let $\wt\Gamma$ be the preimage of $\Gamma$ in $\wt M_\phi$. Then $\wt\Gamma$ is a two-step nilpotent group that fits into a short exact sequence
\begin{equation*}
1 \ra k^\times \ra  \wt\Gamma \ra \Gamma \ra 1.
\end{equation*}
Attached to this is the (alternating) commutator pairing 
\begin{equation*}
c:\Gamma\times \Gamma\to k^\times.
\end{equation*}

\begin{lemma}\label{l:Theta normal}
\begin{enumerate}
    \item The subgroup $\Theta:=\mu_N(k)\cdot \Im(\wt\tau)\subset \wt\Gamma$ is normal. 
    \item The larger subgroup $\Theta' :=\mu_{N^2}(k)\cdot \Im(\wt\tau)\subset \wt\Gamma$ is normal and the quotient $\wt\Gamma/\Theta'$ is abelian.
\end{enumerate}
\end{lemma}
\begin{proof}
(1) We have mentioned in \S\ref{sss:section xcoch S} that $c$ is trivial when restricted to $\Im(\tau)\times \Im(\tau)$. Therefore $c|_{\Gamma\times \Im(\tau)}$ factors through $\Gamma/\Im(\tau)\times \Im(\tau)=\Gamma^\flat\times \Im(\tau)$. Since $\Gamma^\flat$ is $N$-torsion, $c|_{\Gamma\times \Im(\tau)}\subset \mu_N(k)$. This implies that, for any $\wt\gamma\in \wt\Gamma$ and any $\s\in \Im(\wt\tau)$, $\wt\gamma \s \wt\gamma^{-1}\in \mu_N(k)\cdot \s$. Hence $\mu_N(k)\cdot \Im(\wt\tau)$ is a normal subgroup of $\wt\Gamma$.

(2) The quotient $\wt\Gamma/\Theta$ is a $2$-step nilpotent group that fits into an exact sequence
\begin{equation}\label{wt Gamma mod Theta}
    1\to k^\times/\mu_N(k)\to \wt\Gamma/\Theta\to \Gamma^\flat\to 1,
\end{equation}
where $k^\times/\mu_N(k)\cong k^\times$ is central. Since $\Gamma^\flat$ is $N$-torsion, the commutator pairing for \eqref{wt Gamma mod Theta} has values in the $N$-torsion of $k^\times/\mu_N(k)$, which is $\mu_{N^2}(k)/\mu_N(k)$. Therefore $(\wt\Gamma/\Theta)/(\mu_{N^2}(k)/\mu_N(k))=\wt\Gamma/\Theta'$ is abelian. 
\end{proof}

Now let 
\begin{equation}
    \Sigma^\flat:=(\wt\Gamma/\Theta')[N]
\end{equation}
be the $N$-torsion subgroup of the abelian group $\wt\Gamma/\Theta'$. Let $\Sig\subset \wt\Gamma$ be the preimage of $\Sig^\flat$. We check that $\Sig$ satisfies the requirements of Construction \ref{cons: Sigma}.

\begin{lemma}
    The canonical map $\s: \Sigma\to \pi_0(M_\phi)$ is surjective, and $\ker(\s)$ is a finite abelian group with exponent dividing $N^4$.
\end{lemma} 
\begin{proof} Since $\wt\Gamma/\Theta'$ is an extension of $\Gamma^\flat$ by $k^\times/\mu_{N^2}(k)$, which is divisible, we have an exact sequence
\begin{equation}\label{exact seq Sigma flat}
    1\to \mu_{N^3}(k)/\mu_{N^2}(k)\to \Sigma^\flat\to \Gamma^\flat\to1.
\end{equation}
Let $\Theta''=\mu_{N^3}(k)\cdot\Im(\wt\tau)\subset \wt\Gamma$, \eqref{exact seq Sigma flat} implies an exact sequence
\begin{equation}
    1\to \Theta''\to \Sigma\to \Gamma^\flat\to1.
\end{equation}
It fits into a commutative diagram
\begin{equation}
    \xymatrix{1\ar[r] &  \Theta''\ar[r]\ar[d] & \Sigma\ar[r] \ar[d]^{\sigma} & \Gamma^\flat\ar[r] \ar[d]^{\gamma^\flat} & 1\\
    1\ar[r] &  \Im(\ov\tau)\ar[r] & \pi_0(M_\phi)\ar[r] & \coker(\ov\tau) \ar[r] & 1}
\end{equation}
Both the left and right vertical maps are surjective by construction, we conclude that $\sigma: \Sig\to \pi_0(M_\phi)$ is surjective. The above diagram also gives an exact sequence
\begin{equation}
    1\to \mu_{N^3}(k)\to \ker(\s)\to \ker(\gamma^\flat)\cong (Z\GG\cap\AA)^\circ[N](k)\to 1.
\end{equation}
In particular, $\ker(\s)\subset \wt{Z\GG\cap \AA}(k)$ is finite abelian and $N^4$-torsion.
\end{proof}

The above lemma implies a fortiori that the further projection $\pi_\Sig: \Sigma\xr{\s} \pi_0(M_\phi)\surj \Om$ is surjective. The kernel $\ker(\pi_\Sig)$ is a subgroup of $\wt{Z\GG \cap \AA}(k)$, hence is abelian. Moreover, $\ker(\pi_\Sig)$ is an extension of $\ker(\pi_0(M_\phi)\to \Om)=\pi_0(Z\GG\cap \AA)$ by $\ker(\s)$, hence is finite. This finishes the construction of $\Sigma$.

\subsection{Integral affine coroots} \label{s:comboblock} For a real affine root $\a$ of $LG$, we can associate to a homomorphism $$\phi_\a: SL_2 \rightarrow \tLG,$$ integrating a choice of corresponding $\mathfrak{sl}_2$-triple. The restriction of $\phi_\a$ to the diagonal matrices, which is independent of the choice, yields a homomorphism $$\halpha : \gm \ra \wt\AA,$$ which is the (real) coroot corresponding to $\a$. We denote the obtained set of (real) affine coroots, and its subset of positive affine coroots, by
$$\cPhi_{\aff}^{+} \subset \cPhi_{\aff} \subset \mathbb{X}_*(\wt{\mathbb{A}}).$$ 

Since $N_{\tLG}(\wt\AA)$ acts on $\wt\AA$ by conjugation, $\tilW$ acts on the group of character sheaves on $\wt\AA$. For $\chi, \chi' \in \Ch(\wt\AA)$ two residual character sheaves (see Section \ref{sss: int coef} for this notion), let 
\begin{equation*}
	_{\chi} \tilW_{\chi'} := \{w \in \tilW : w\chi'=\chi \}.  
\end{equation*}
For ease of notation, we will write $\tilW_\chi$ instead of $_\chi \tilW_\chi$.

Within $\tilW_\chi$, we have a subgroup which plays a basic role in the present paper, constructed as follows. Associated to $\chi$ is the subset of {\em integral} coroots 
\begin{equation*}
	\cPhi_\chi = \{ \halpha \in \cPhi_{\aff}: (\halpha)^*(\chi) \text{ is trivial}\},  
\end{equation*}
and similarly its subset of positive integral coroots
$$\cPhi^{+}_\chi := \cPhi^{+}_{\aff} \cap \cPhi_\chi.$$

By the passing to the group generated by the corresponding reflections, we obtain a normal subgroup of $\tilW_\chi$, namely
\begin{equation*}
    \tilW_\chi^\circ = \langle s_\a : \a \in \cPhi_\chi \rangle \triangleleft \tilW_\chi. 
\end{equation*}
The results in \cite{Bo} show that $\tilW_\chi^\circ$ is naturally a Coxeter group $(\tilW_\chi^\circ, S_\chi)$, with the simple reflections $S_\chi$ corresponding to the indecomposable positive affine coroots in $\cPhi_\chi$, i.e., the elements of $\cPhi_\chi^{+}$ not expressible as a sum of two other elements. 

Finally, we set up some terminology for the arising group of cosets, and its variants for varying characters. Namely, for $\chi' \in \tilW \chi$, we consider the set 
\begin{equation*}
	_{\chi} \Omega_{\chi'} = \tilW_{\chi}^\circ \bsl {}_{\chi} \tilW_{\chi'}  = {}_{\chi} \tilW_{\chi'} / \tilW_{\chi'}^\circ.
\end{equation*}
When $\chi = \chi'$, for ease of notation we will write $\Omega_\chi$ instead of ${}_\chi \Omega_\chi$. Note that $\Omega_\chi$ is a quotient group of $\tilW_\chi$, and more generally ${}_{\chi}\Omega_{\chi'}$ is naturally a bitorsor over $\Omega_{\chi}$ and $\Omega_{\chi'}$. For each element $\beta$ in ${}_{\chi} \Om_{\chi'}$, we denote its corresponding coset by ${}_{\chi} \tilW_{\chi'}^{\beta}$.

We observe that there is a natural section $\sigma$ of the tautological projection $${}_{\chi} \tilW_{\chi'} \twoheadrightarrow {}_{\chi} \Omega_{\chi'}.$$Namely, each element $\beta$ in $_{\chi} \Omega_{\chi'}$ contains a unique element $w^{\beta}$ of minimal length, and we set $\sigma(\beta) = w^\beta$. Explicitly, $w^\beta$ is the unique element of the coset satisfying
\begin{equation} \label{e:charminimal}w^\beta( \cPhi^{+}_{\chi'}) \xrightarrow{\sim} \cPhi^{ +}_{\chi},\end{equation}
which shows minimal length elements are moreover closed under multiplication. Namely, we have the following lemma. 

\begin{lem}\label{l:minmult} Given cosets $\beta \in {}_\chi \Omega_{\chi'}$ and $\gamma \in {}_{\chi'}\Omega_{\chi''}$, define $\beta \star \gamma$ to be the subset of $\tilW$ given by the image of 
$$\beta \times \gamma \hookrightarrow \tilW \times \tilW \xrightarrow{\on{mult}} \tilW.$$
Then $\beta \star \gamma$ is a coset in ${}_\chi \Omega_{\chi''}$, and we have an equality of minimal length representatives
$$w^\beta w^\gamma = w^{\beta \star \gamma}.$$
\end{lem}

\begin{proof} To see the first assertion, if we pick coset representatives $y \in \beta$ and $z \in \gamma$, we note that
\begin{align*}\beta \star \gamma =  (y\tilW_{\chi'}^\circ) (\tilW_{\chi'}^\circ z) & = y \tilW_{\chi'}^\circ z = \tilW_{\chi}^\circ yz. \end{align*}
The second assertion is then immediate from \eqref{e:charminimal}.
\end{proof}

In particular, when $\chi =\chi'$, the section $\sigma$ provides a group homomorphism $\Omega_\chi \rightarrow \tilW_\chi$, whence 
$$\tilW_\chi \simeq \Omega_\chi \ltimes \tilW_\chi^\circ,$$
and moreover this is an extended Coxeter presentation, i.e., the conjugation action of $\Omega_\chi$ on $\tilW^\circ_\chi$ is by automorphisms as a Coxeter system. 

For a coset $\beta \in {}_\chi \Omega_{\chi'}$, we define a partial order $\leq_{\beta}$ and length function $\ell_{\beta}$ on ${}_{\chi} \tilW_{\chi'}^{\beta}$ by transport of structure via the isomorphism $$\tilW_\chi^\circ \simeq {}_{\chi} \tilW_{\chi'}^{\beta}, \quad \quad w \mapsto w w^{\beta}.$$ In other words, we have $$w_1w^{\beta} \leq_{\beta} w_2w^{\beta} \Leftrightarrow w_1 \leq_{\chi} w_2, \text{ and}$$
$$\ell_{\beta}(w_1w^{\beta})= \ell_{\chi}(w_1), $$ where $\ell_\chi: \tilW_\chi^\circ \rightarrow \mathbb{Z}_{\geqslant 0}$ is the usual length function attached to the Coxeter group $\tilW_\chi^\circ$ and $\leq_{\chi}$ denotes the standard partial order of $\tilW_\chi^\circ$. 

It is straightforward to check that these agree with the similarly defined length functions and Bruhat orders using instead 
$${}_{\chi} \tilW_{\chi'}^{\beta} \simeq \tilW_{\chi'}^\circ, \quad \quad w^{\beta} x \mapsfrom x.$$
Indeed, for the Bruhat order, this follows from the fact that conjugation by $w^{\beta}$ is an isomorphism of Coxeter systems $\tilW_\chi^\circ \simeq \tilW_{\chi'}^\circ,$ which in turn follows from Equation \eqref{e:charminimal}. For the length, one notes that for $z \in {}_{\chi} \tilW_{\chi'}^{\beta}$, $$\ell_{\beta}(z) = \# \{ \halpha \in \cPhi_{\chi'}^{+}: z(\halpha) \in -\cPhi_{\chi}^{+}\},$$  
which makes the claim manifest. Let us summarize the results above. 

\begin{prop}\label{p:conjmin}
    For a coset $\beta \in {}_\chi \Omega_{\chi'}$, left (resp. right) multiplication by the minimal element $w^{\beta}$ preserves the Bruhat order and the length function. In particular, conjugation by a minimal element is an isomorphism between Coxeter systems $\tilW_\chi^\circ$ and $\tilW_{\chi'}^\circ$.
\end{prop}

The following proposition is crucial to the induction arguments for results in the next sections. 

\begin{prop} \label{p:conjsimple}
     For a simple reflection $r$ in $\tilW_{\chi}^\circ$, there exists a minimal element $w^{\beta}$ such that $w^{\beta} r w^{\beta, -1}$ is a simple reflection in both $\tilW_{w^{\beta}\chi}^\circ$ and $\tilW$.
\end{prop}

\begin{proof}
Recall that, given a positive affine coroot $\halpha$, we may speak of its {\em height} $\on{ht}(\halpha)$. Namely, if we write $\halpha$ uniquely as a sum of simple coroots 
$$\halpha = \underset{i} \Sigma \hspace{.5mm} n_i \cdot \halpha_i, \quad \quad n_i \in \mathbb{Z}_{\geqslant 0},$$
where $\halpha_i \in S_{\aff}$ run over the simple positive affine coroots, we have 
$$\on{ht}(\halpha) := \underset{i} \Sigma \hspace{.5mm} n_i.$$

For an integral simple reflection $r$ as in the proposition, let us denote by $\halpha_r$ the corresponding positive real affine coroot, and proceed by induction on $\on{ht}(\halpha_r)$.
    
So, suppose that $\on{ht}(\halpha_r) > 1$. In this case, there exists a simple reflection $t$ in $\tilW$ such that $t(\halpha_r)$ is of strictly lower height, i.e., $$\halpha_r = t(\halpha_r) + n \cdot \halpha_t, \quad \quad n > 0.$$
In this case, it follows that $t \notin {}_{\chi}\tilW_\chi^\circ,$ otherwise $r$ would not be simple. If we consider the conjugation map 
$$\tilW \simeq \tilW, \quad \quad w \mapsto twt,$$
Proposition \ref{p:conjmin} implies that this restricts to an isomorphism of Coxeter systems 
$${}_\chi\tilW_\chi^\circ \simeq {}_{t\chi}\tilW_{t\chi}^\circ.$$
In particular, we deduce that $trt$ is a simple reflection in ${}_{t\chi}\tilW_{t\chi}^\circ$. Since $\ell(trt)$ is smaller than $\ell(r)$, the induction step is complete, as the product of minimal elements remains minimal, which is precisely the content of Lemma \ref{l:minmult}.
\end{proof}

\sss{A particular example} \label{sss: example} 
Since the group $\Om_{\chi}$ plays an important role to describe all blocks of the monodromic Hecke categories as mentioned in Theorem \ref{t:mainthm}, let us give a particularly interesting example to illustrate the group $\Om_{\chi}$ can be complicated. This example shows that the condition on the connectedness of the center in Theorem \ref{thm: twisted metaplectic} is indeed necessary for extending the equivalence to all blocks. More computations about $\Om_{\chi}$ and $\tilW_{\chi}$ can be found in \S \ref{s: Quantum Langlands}.

We will construct a pair $(\tLG, \chi)$ satisfying the following properties:
\begin{enumerate}
    \item $\tilW^\circ_\chi$ is trivial;
    \item $\Om_\chi$ is not abelian;
    \item $\Om_\chi$ has no finite order elements.
\end{enumerate} 
As a consequence, $\Om_\chi = \tilW_\chi$ cannot be isomorphic to a semi-direct product of a finite group and a lattice. 

Let $G$ be $\PSp_{6}$. We adopt the convention that the roots of $G$ are $\pm L_i \pm L_j, \pm 2L_i$ for $1 \leq i \neq j \leq 3$, and the coroots are $\pm e_i \pm e_j, \pm e_i$, where $\{e_i\}$ and $\{L_i\}$ are dual bases.  Then $\xcoch(\TT)$ is generated by the coroot lattice $\check{Q}$ together with the vector $\cmu :=(e_1+e_2+e_3)/2$, while $\xch(\TT)$ coincides with the root lattice $Q$. 

Via simple roots $L_1-L_2, L_2-L_3$ and $2L_3$, we have an isomorphism
\begin{equation*}
    \TT\isom \Gm^3.
\end{equation*}
Let $\th_i\in \Ch(\Gm)$ and let $\th=\th_1\bt \th_2\bt \th_3\in \Ch(\TT)$ under the above isomorphism. 

We choose any central extension $\tLG$ of $LG$ as described in Section \ref{ss: cent ext}, and denote the associated commutator pairing on $\xcoch(\TT)$ (see \S\ref{sss:commutator}) by $\langle \cdot, \cdot \rangle$. Write the rational number $\frac{2}{\j{e_i,e_i}}$ (which is independent of the choice of $i\in \{1,2,3\}$) in lowest term
\begin{equation*}
    \frac{2}{\j{e_i,e_i}}=\frac{a}{b}\in \QQ.
\end{equation*}
Let $\chi_{\cen}\in \Ch(\Gm^{\cen})$ be a character sheaf of order exactly $b$. Consider the character sheaf
\begin{equation*}
    \chi=\chi_{\cen}\bt\th\in \Ch(\wt\TT).
\end{equation*}

\begin{prop}
    Suppose $\th_3$ has order 4, and $\th_1,\th_2$ each has odd order as elements in $\Ch(\Gm)$. Then $\wt W_\chi$ fits into an exact sequence
    \begin{equation}\label{PSp6 W chi exact sequence}
        1 \ra \check{Q} \ra \extw_{\chi} \ra \Z /2 \ra 1
    \end{equation}
    which does not split. Moreover, $\wt W_\chi$ satisfies the three conditions stated in the beginning of Section \ref{sss: example}.
\end{prop}
\begin{proof}
    Let $\om \in \WW$ be the  element that fixes $L_1$ and $L_2$ while sending $L_3$ to $-L_3$.
    By Lemma \ref{l:conj action} and direct computation, it follows that $\extw_{\chi}$ is given by $$\extw_{\chi} = \{t^{\clambda}; \clambda\in \check{Q}\} \cup \{t^{\cmu +\clambda}\om;\clambda\in \check{Q}  \}.$$ 
    From this we get the exact sequence \eqref{PSp6 W chi exact sequence}.

Note that $(t^{\cmu+\clambda}\om)^2=t^{\cmu +\clambda + \om (\cmu +\clambda)}$, and $\cmu +\clambda + \om (\cmu +\clambda)$ is of the form $(2m+1)e_1+(2n+1)e_2$ for some $m,n \in \ZZ$. Therefore, $t^{\cmu +\clambda}\om$ is not of finite order. As a result, $\tilW_\chi$ contains no elements of finite order. This implies that \eqref{PSp6 W chi exact sequence} does not split. It also implies that $\tilW_\chi$ contains no reflections, and hence $\tilW^\circ_\chi$ is trivial.

To see that $\tilW_\chi$ is not abelian, we observe that $\om$ sends $e_3$ to $-e_3$. Hence, $t^{e_3}$ does not commute with $t^{\cmu + \clambda}$ for any $\clambda \in \check{Q}$. 
\end{proof}

A similar phenomenon happens for $\PSp_{4n+2}$ for $n \geq 1$.

\section{Monodromic Hecke category}

Now we introduce the main player of this article.

\label{s:heckecat}
\subsection{Monodromic Hecke category} 

\sss{General Hecke categories}\label{sss:gen Hk}
Let $\bG$ be an affine smooth integral model of $G$ over $\cO_K$, with arc group $L^+\bG$, i.e., $L^+\bG=\bG(R\tl{t})$ for  $k$-algebras $R$. Consider the indfp morphism
\[
X=\mathbb{B} L^+\bG\to Y=\mathbb{B} LG.
\]
Note that
\[
L^+\bG\bs LG/L^+\bG\cong X\times_Y X.
\]
On the other hand, the relative diagonal $X\to X\times_YX$ is a (finitely-presented) closed embedding. Therefore, we may apply the 
paradigm in Section \ref{sss:sheaf conv} to obtain a monoidal category
\[
D(L^+\bG\bs LG/L^+\bG).
\]
The monoidal structure on this category will be called the convolution product, and will be denoted as $\star$.

\begin{rem}
    The morphism $\mathbb{B} L^+\bG\to \mathbb{B} LG$ is representable by ind-algebraic spaces. But we do not know whether it is representable by ind-schemes. This is the reason at the very beginning we work with algebraic spaces of finite presentation over $k$ as the domain of our sheaf theories. 
\end{rem}

If in addition the embedding $L^+\bG\incl LG$ is equipped with a lifting to $\tLG$ (for example, when $L^+\bG$ is pro-unipotent, such a lifting exists and is unique), we may apply the above procedure to the indfp morphism
\[
X=\mathbb{B} L^+\bG\to \wt Y=\mathbb{B} \tLG
\]
to obtain the monoidal category
\[
D(L^+\bG\bs \tLG/L^+\bG).
\]

\sss{Affine flag variety}
Take $L^+\bG=I$ the Iwahori subgroup, the double cosets $I\bs LG/I$ are indexed by $\tilW$. For $w\in \tilW$, let $LG_{w}$ denote the $(I \times I)$-orbit on $LG$ containing the coset of $w$ in $N_{LG}(\AA)$. The closure relation of the $LG_w$'s is described by the Bruhat order of $\tilW$. We denote by $LG_{\le w}$ the closure of $LG_{w}$ in $LG$. 

Let $I^+$ be the pro-unipotent radical of $I$, which has a unique lifting to $\tLG$. For $w \in \tilW$, we let
\begin{equation*} 
Z = I^{+}\bs \tLG /I^{+}; \qquad Z_{\le w} = I^{+}\bs \tLG_{\le w}/I^{+}; \qquad Z_ w=I^{+}\bs \tLG_{w}/I^{+}. 
\end{equation*}
We let 
$$j_{(\le) w}: Z_{(\le) w}\hookrightarrow Z$$
denote the associated (locally) closed embeddings.

Applying the discussion in \S\ref{sss:gen Hk} to $L^+\bG=I^+$, we obtain a monoidal category
\[
D(I^+ \bs \tLG/I^+).
\]

We will make use of the following perverse $t$-structure on $D(I^+ \bs \tLG/ I^+)$. First as $\tLG/I^+$ is a strict ind-scheme of ind-finite type, there is a standard perverse $t$-structure on $D(\tLG/I^+)$. 
Let us denote the coconnective part by $D(\tLG/I^+)^{p,\ge 0}$. Then
 $D(I^+ \bs \tLG/ I^+)^{p,\ge 0}$ consist of those $\cF$ whose $!$-pullback to $\tLG/ I^+$ belonging to $D(\tLG/I^+)^{p,\ge\dim \wt\AA}$.

\sss{Monodromic affine Hecke categories}
We have the action of $\wt\AA\times \wt\AA$ on $Z$ from left and right, and therefore we have full subcategory
\[
\cM:=D((\wt \AA,\mono)\bs Z /(\wt \AA,\mono))\subset D(I^+ \bs \tLG/I^+),
\]
which is closed under monoidal product $\star$.
In addition, $\cM$ admits a unit, and therefore itself is a monoidal category. Note, however, that the above embedding is non-unital. 

For $w\in\widetilde{W}$, we let $\cM(w)$ and $\cM(\leq \!w)$ denote $$D((\wt \AA,\mono)\bs Z_w/(\wt \AA,\mono)) \quad \text{and} \quad D((\wt \AA,\mono)\bs Z_{\leq w}/(\wt \AA,\mono)),$$
respectively. 
For $?=*$ or $!$, we have fully faithful embedding
\begin{equation}\label{eq-w-semiorthogonal-piece}
(j_{w})_?: \cM(w)\to \cM,\quad (j_{\le w})_?: \cM(\le w)\to \cM.
\end{equation}

For $\chi,\chi'\in \Ch(\wt\AA)$ two residual character sheaves as in Corollary \ref{r:block}, we similarly let
\[
_{\chi}\cM_{\chi'}:=D((\wt \AA, \chi \mon)\bs Z /(\wt \AA, \chi' \mon))\subset \cM.
\]
Observe that $_{\chi}\cM_{\chi}$ is closed under the convolution product and admits a unit, and therefore is a monoidal category.

If $\Xi$ denotes a $\tilW$-orbit on $\Ch(\wt\AA)$, define
\begin{eqnarray}
\cM_\Xi:= \underset{\chi, \chi' \in \Xi} \bigoplus \hspace{1.5mm} {}_{\chi}\cM_{\chi'},
\end{eqnarray}
and note this is again a full subcategory of $\cM$ by Corollary \ref{r:block}. It is easy to see that $\sF \star \sG \in \pha _{\chi} \cM_{\chi''}$  for $\sF \in \pha _{\chi} \cM_{\chi'}$, $\sG \in \pha _{\chi'} \cM_{\chi''}$,  In particular, it induces a monoidal structure on the full subcategory $_\chi \cM_\chi$, which is also denoted by $\star$. 

\begin{rem}\label{rem:block decomp of monodromic hecke}
On the other hand, one easily check that if $\cF\in \pha _{\chi} \cM_{\chi'}, \cG\in \pha _{\chi''} \cM_{\chi'''}$, and if $\chi'\neq \chi''$ (as closed points in $\mathrm{LS}_{\wt{\AA}^\vee}^{t,\Box}$), then $\cF \star \cG=0$.  
\end{rem}

We note that a choice of a lifting $\dot{w}$ to $\tLG$ defines a morphism $$\iota_{\dot{w}}: \wt\AA\to Z_w,\quad t\mapsto t\dot{w},$$which admits a left inverse $\iota_{\dot{w}}^{-1}$. It induces an equivalence 
$(\iota_{\dot{w}^{-1}})^!:  D((\wt\AA,\mono)\bs\wt\AA)\cong \cM(w)$. To put certain objects introduced later in the heart of the perverse $t$-structure, it is convenient to consider a cohomological shift of $(\iota_{\dot{w}^{-1}})^!$.
Below, by a mild abuse of notation, we will write
\begin{equation} \label{e:onestratumunivmon}
\iota_{\dot{w}}:=(\iota_{\dot{w}^{-1}})^!\langle -\ell(w) \rangle:  D((\wt\AA,\mono)\bs\wt\AA)\cong \cM(w).
\end{equation}  
In particular,
$\cM(e)$ is canonically identified with $D((\wt\AA,\mono)\bs\wt\AA)$, via using the identity element of $\wt{LG}$; this naturally underlies a monoidal equivalence. For $?=*$ or $!$, composing with the functors $(j_w)_?$ from \eqref{eq-w-semiorthogonal-piece}, we obtain fully faithful embeddings
\begin{equation}\label{eq-extension-of-monodromic}
(j_{\dot{w}})_?:D((\wt{\mathbb{A}}, \on{mon}) \bs \wt{\mathbb{A}}) \overset{\iota_{\dot{w}}} \simeq \cM(w) \xrightarrow{(j_w)_?} \cM.
\end{equation}

Similarly for $\chi,\chi'\in \Ch(\wt\AA)$, we have ${}_{\chi}\cM(w)_{\chi'}$ and ${}_{\chi}\cM(\le \! w)_{\chi'}$. In this case, fixing a lifting $\dot{w}$ as before induces an equivalence
\begin{equation}\label{eq: stratum}
    D((\wt\AA,\chi\mon)\bs\wt\AA)\cong {}_{\chi}\cM(w)_{w^{-1}\chi},
\end{equation}
and $(j_{\dot{w}})_?$ restricts to a fully faithful embedding $ D((\wt\AA,\chi\mon)\bs\wt\AA)\to {}_{\chi}\cM_{w^{-1}\chi}$.

\subsubsection{Compact generators} \label{ss:compgens}Let us explicitly describe a tautological set of compact generators for $\cM.$

Recalling that $D((\wt{\mathbb{A}}, \on{mon}) \bs \wt{\mathbb{A}})$ is by definition compactly generated by the $E'$-linear character sheaves $\wt{\chi}$ on $\wt{\mathbb{A}}$, where $E'$ runs over all finite $E$-algebras, for any such character sheaf, we obtain a compact object of $\cM$, namely
$$j(w)_{!,\wt{\chi}} := (j_{\dot{w}})_!( \wt{\chi}).$$
Note in particular that the resulting object, up to non-canonical isomorphism, is independent of the choice of coset representative $\dot{w}$, justifying our suppression of it from our notation.

The following then is an immediate consequence of the definitions. 

\begin{lemma} The category $\cM$ is compactly generated by the sheaves $$j(w)_{!, \wt{\chi}}, $$for all $w \in \tilW,$ and all $E'$-linear character sheaves $\wt{\chi}$ on $\wt{\AA}$.

Moreover, the perverse $t$-structure on $\cM$ may be characterized by an object $\xi$ lies in $\cM^{p, \geqslant 0}$ if and only if 
$$\Hom(j(w)_{!, \wt{\chi}}, \xi) \in \Modu_E^{\geqslant 0},$$for all $w \in \tilW,$ and all $E'$-linear character sheaves $\wt{\chi}$ on $\wt{\AA}$.

\end{lemma}

By projecting onto ${}_{\chi}\cM_{\chi'}$, we now obtain a similar convenient set of compact generators.  Recall that to a residual character sheaf $\chi$, we attached a set of character sheaves $\mathscr{L}_\chi$, cf. \eqref{eq: lifting of char sheaves}.

\begin{cor}\label{c:compgensfixedmon} For any fixed $\chi, \chi'$, the category ${}_\chi\cM_{\chi'}$ is compactly generated by the objects $$j(w)_{!, \wt{\chi}}, \quad \text{for } w \in {}_\chi\tilW_{\chi'}, \wt{\chi} \in \mathscr{L}_\chi.$$ 
Moreover, the usual perverse $t$-structure on $\cmc$ may be characterized by an object $\xi$ lies in $\cmc^{p, \geqslant 0}$ if and only if 
$$\Hom(j(w)_{!, \wt{\chi}}, \xi) \in \Modu_E^{\geqslant 0}, \quad \quad w \in \cwc, \wt{\chi} \in \mathscr{L}_\chi.$$

\end{cor}

\begin{rem} Note that as each stratum contains only finitely many strata in its closure, we could equally well have worked with the $*$-extensions of character sheaves, instead of $!$-extensions, to obtain a set of compact generators. \end{rem}

\subsubsection{Standard and costandard objects}

For setting up Soergel theory, we will need additionally the following distinguished set of non-compact objects. In the following definition, $w, \iota,$ and $j$ are as in the preceding subsection. 
 
\begin{defn}
	 We define the {\em cofree monodromic standard sheaf} and {\em cofree monodromic costandard sheaf} by  
     $$\std(w) := j_{\dot{w},!} (\d^{\on{mon}})  \quad \text{and} \quad \costd(w) := j_{\dot{w},*} (\d^{\on{mon}}),$$
     respectively. 
\end{defn}

Similarly, if we fix a character sheaf $\chi$, by applying the projector $\iota_\chi^!$ on the left, or equivalently the projector $\iota_{w^{-1}\chi}^!$ on the right, to the previous objects we obtain the associated cofree (co)standard objects 
$$\std(w)_{\chi} := j_{\dot{w},!} (\dcm)  \quad \text{and} \quad \costd(w)_{\chi} := j_{\dot{w},*} (\dcm) ,$$
which belong to the category ${}_\chi \cM_{w^{-1}\chi}.$
As above, here we shall only consider these objects up to isomorphism; we will explain how to construct canonical elements within each isomorphism class in the subsequent sections.

\subsection{First properties of (co)standard objects} We now collect some basic assertions regarding (co)standard objects and their convolutions. These are monodromic extensions of some well-known facts for the (co)standard objects in the usual affine Hecke category
When there is no central extension, results in this subsection have appeared in \cite[Proposition 4.47]{Zh} with detailed proofs. So we will sometimes be sketchy in the sequel.

Recall the action of $\tilW$ on $\wt\AA$ as discussed in \S \ref{SSS: action of tilW on wtAA}. For $w\in \tilW$, we simply denote the induced automorphism of $\wt\AA$ as $w: \wt\AA\to\wt\AA$. If $\cL\in D_{\mono}(\wt\AA)$, we will write $w(\cL)$ instead of $w_*(\cL)$ for simplicity.

\subsubsection{Some rank one calculations}
Let $s$ be a simple reflection in $\tilW$. Let $\alpha_s$ (resp. $\halpha_s$) be the corresponding simple affine root (resp. coroot). Let $U_{\alpha_s}$ be the corresponding root subgroup and let $\phi_s: \GG_a\simeq U_{\alpha_s}$ denote the isomorphism determined by the affine pinning.
Let $\dot{s}$ be a lifting to $s$ to $\tLG$ so that $\dot{s}^{2}=\halpha_{s}(-1)\in \wt\AA$.

Our goal is to study certain object of the form $\cF\star \nabla(s)$, where $\cF$ is a sheaf on $\tLG/I^+$ obtained from the $*$- or $!$-pushforward from certain sheaf $\cF_0\in \tLG_s/I^+$. Clearly, $\cF\star\costd(s)$ is supported on $\tLG_{\leq s}/I^+$. We shall study its ($!$-)restriction to $\tLG_s/I^+$ and to $\wt{I}/I^+$ respectively.

For this purpose, it is convenient to use  the coordinate 
$$
\GG_a\times \wt\AA\simeq \tLG_{s}/I^+, \quad (x,t)\mapsto \phi_s(x)t\dot{s}.
$$ 
Then the map $ \tLG_{s}/I^+\to I^+\bs  \tLG_{s}/I^+\xrightarrow{\iota_{\dot{s}}^{-1}} \wt\AA$ corresponds to the second projection $\GG_a\times \wt\AA\to \wt\AA$. We let $j:\mathring{\tLG}_{s}/I^+\subset \tLG_{s}/I^+$ be the open embedding corresponding to $\gm\times\wt\AA\subset\GG_a\times\wt\AA$. 

The sheaf $\cF_0\in D(\tLG_s/I^+)$ we will be interested in is of the form $\cL[1]\boxtimes \cF_1\in D(\GG_a\times\wt\AA)$, where $\cL$ is a character sheaf on $\GG_a$, and $\cF_1\in D(\wt\AA/(\wt\AA\mon))$.

We first study $(\cF\star j_{\dot{s},*}(\delta^{\mono}))|_{\tLG_s/I^+}$, which must be of the form $$\cL[1]\boxtimes \cK_{?,\cL,\cF_1}\in D(\GG_a\times \wt\AA),$$ for some $\cK_{?,\cL,\cF_1}$ which we shall compute. Here $?=*$ or $!$, depending on whether $\cF$ is the $*$- or $!$-extension of $\cF_0$.

Consider the following diagram with all the commutative squares being Cartesian
\[
\xymatrix{
\wt\AA\times \mathring{\tLG}_{s}/I^+  \ar@{^{(}->}[d]_{j'}\ar[rr] && \tLG_{s}\times^{I^+}\tLG_{s}/I^+\ar@{^{(}->}[d]_{j'}\ar[r] & \tLG_{s}/I^+ \times I^+\bs \tLG_{s}/I^+ \ar@{^{(}->}[d]\\
 \wt\AA\times \tLG_{s}/I^+ \ar_{\pr_{1}}[d]\ar^-{(t,g)\mapsto}_-{(t\dot{s}g, g^{-1})}[rr] &&  \tLG_{\le s}\times^{I^+}\tLG_{s}/I^+\ar[r]\ar^{\mult}[d] &  \tLG_{\le s}/I^+ \times I^+\bs \tLG_{ s}/I^+\\
 \wt\AA \ar^{t\mapsto t\dot{s}}[rr] &&    \tLG_{\le s} /I^+&
}\]

Using the above coordinate of $\tLG_{s}/I^+$ and the map $\iota_{\dot{s}}^{-1}: I^+\bs \tLG_s/I^+\to \wt\AA$, and by an explicit $\SL_2$-calculation, we see that the composed horizontal map in the first row of the above diagram is given by
$$f:  \wt\AA\times \GG_m\times \wt\AA\to \GG_a\times\wt\AA\times \wt\AA,\quad (t,y, t')\mapsto (-\alpha_s(t)y^{-1}, t\halpha_s(-y^{-1})t' ,s(t')^{-1}\halpha_s(-1)).$$
Here $\alpha_s$ (resp. $\halpha_s$) is the affine simple root (resp. coroot) associated to $s$. The vertical maps of the left column of the diagram is given by $ \wt\AA\times \GG_m\times \wt\AA\xrightarrow{j'} \wt\AA \times \GG_a\times \wt\AA\xrightarrow{\pr_{1}} \wt\AA$. 

By base change, we have $$
\cK_{?,\cL,\cF_1}\simeq (\pr_{1})_* (j'_?(f^!(\cL\boxtimes \cF_1\boxtimes \delta^{\mono})))[1](1).
$$
We can simplify the r.h.s. a little bit. Consider the following isomorphism of tori
\[
h: \wt\AA\times\GG_m\to \GG_m\times \wt\AA, \quad (t,y)\mapsto (\alpha_s(t)y^{-1}, s(t)\halpha_s(y)),
\]
Using the Cartesian diagram
$$
\xymatrix{
\wt\AA\times \GG_m\times \wt\AA\ar^-f[rrrr]\ar_{\pr_{12}}[d]&&&& \GG_a\times\wt\AA\times \wt\AA\ar^{(x,t,t')\mapsto (x, s(t)t')}[d]\\
\wt\AA\times \GG_m\ar^-{h}[r] & \GG_m\times\wt\AA\ar@{^{(}->}[r] & \GG_a\times\wt\AA\ar^-{(x,t)\mapsto (-x,t)}[rr] &&\GG_a\times\wt\AA,
}
$$
the base change, and the fact $\delta^{\mono}\star\delta^{\mono}=\delta^{\mono}$, we see that
$$
pr_{12,*}f^!(\cL\boxtimes \cF_1\boxtimes \delta^{\mono}) \simeq h^!((\cL^{-1}|_{\GG_m})\boxtimes s(\cF_1)), 
$$ hence, \begin{equation} \label{eq: K?LF}
    \cK_{?,\cL,\cF_1}\simeq (\pr_{1})_*(j_?(h^!((\cL^{-1}|_{\GG_m})\boxtimes s(\cF_1))))[1](1).
\end{equation}

Now we specialize to some cases needed in the sequel. 
\begin{enumerate}
\item We suppose $?=*$. Then we may combine $(\pr_1)_* j_*$ as $(\pr_1)_*$, where $\pr_1: \wt\AA\times\GG_m\to\wt\AA$ is the first projection. Then we can rewrite \eqref{eq: K?LF} as
\begin{eqnarray*}
\cK_{*,\cL,\cF_1} & \simeq & \Av^{\mono}((\pr_1)_*(h^!((\cL^{-1}|_{\GG_m})\boxtimes s(\cF_1))[1](1))\\
      & \simeq & (\pr_1)_*(h^!( \Av^{\mono}(\cL^{-1}|_{\GG_m})\boxtimes s(\cF_1))[1](1).
\end{eqnarray*}
The first equivalence is due to the fact that $\cK_{*,\cL,\cF_1}$ is monodromic, hence $\cK_{*,\cL,\cF_1}=\Av^{\mono}(\cK_{*,\cL,\cF_1})$. The second equivalence follows from the fact that both $(\pr_1)_*$ and $h^!$ commute with $\Av^{\mono}$. Using Proposition \ref{prop: tame geometric Langlands for tori}, we can compute $\cK_{*,\cL,\cF_1}$ explicitly via the maps in spectral side
\[
\mathrm{LS}_{\wt\AA^\vee}^{t,\Box}\stackrel{\pr_1^\vee}{\longrightarrow} \mathrm{LS}_{\wt\AA^\vee\times\GG_m}^{t,\Box}\stackrel{h^\vee}{\longleftarrow} \mathrm{LS}_{\GG_m\times \wt\AA^\vee}^{t,\Box}.
\]

Namely, if $\cL=E$ is the constant sheaf on $\GG_a$ and $\cF_1=\delta^{\mono}$, then $\Av^{\mono}(\cL^{-1}|_{\GG_m})\boxtimes\delta^{\mono}$ corresponds to the dualizing sheaf of $\mathrm{LS}_{\{1\}\times \wt\AA^\vee}^{t,\Box}\subset \mathrm{LS}_{\GG_m\times \wt\AA^\vee}^{t,\Box}$. If we let 
\begin{equation}\label{eq: kernel of halphasvee}
\ker(\halpha_s)^\vee:=\mathrm{LS}_{\wt\AA^\vee}^{t,\Box}\times_{\mathrm{LS}_{\GG_m}^{t,\Box}}\{1\}=\mathrm{LS}_{\wt\AA^\vee}^{t,\Box}\times_{\pr_1^\vee,\mathrm{LS}_{\wt\AA^\vee\times\GG_m}^{t,\Box},h^\vee}\mathrm{LS}_{\{1\}\times \wt\AA^\vee}^{t,\Box},
\end{equation}
it follows from Proposition \ref{prop: tame geometric Langlands for tori} \eqref{prop: tame geometric Langlands for tori-2} that
\begin{equation}\label{eq: shrek-stalk-conv-costad-s-strata}
    \cK_{*,E,\delta^{\mono}}\simeq \CH(\omega_{\ker(\halpha_s)^\vee})[1](1).
\end{equation}

Next assume that $\cL=\AS_\psi$ is an Artin-Schreier sheaf. We need the following lemma. Let $\chi$ be a closed point of $\mathrm{LS}_{\wt\AA^\vee}^{t,\Box}$ of residue field $\fre_\chi$ and let $E_\chi$ be the unramified extension of $E$ with residue field $\fre_\chi$. Recall that $\chi$ can be regarded as an $\fre_\chi$-linear character sheaf, which admits a canonical lifting as an $E_\chi$-linear character sheaf $[\chi]$ (via the Teichm\"uller lifting $[\cdot]:\fre_\chi^\times\to E_\chi^\times$). 

\begin{lem}
    We have 
    \[
    \Av^{\mono}(\AS_\psi|_{\GG_m}[1])\simeq\bigoplus_{\chi} \dcm\otimes_{E_\chi} \cG(\chi^{-1},\psi),
    \]
    where $\chi$ runs through all residue character sheaves on $\gm$, and $\cG(\chi,\psi):=H^1(\GG_m, [\chi]\otimes \AS_\psi)$ is the rank one $E_\chi$-module.
\end{lem}
\begin{proof}
    This follows directly from the fact that 
    $$\Hom(\chi, \Av^{\mono}(\AS_\psi|_{\GG_m}[1]))=\Hom(\chi, \AS_\psi|_{\GG_m}[1])=H^1(\GG_m, \chi^{-1}\otimes\AS_\psi)$$ 
    is a one-dimensional $\fre_\chi$-module.
\end{proof}

Using the lemma, it follows that we have a natural isomorphism
\begin{equation}\label{eq: shrek stalk of average of AS}
    \cK_{*,\AS_\psi,\cF_1}\simeq \bigoplus_\chi(\dcm\star s(\cF_1))\otimes \cG(\halpha_s^{\vee}(\chi)^{-1}, \psi^{-1}).
\end{equation}

\item We suppose $?= {}!$. Here we will only consider the case when $\cL$ is the constant sheaf on $\GG_a$. In this case, 
$h ^!(\cL\boxtimes \delta^{\mono})$ is a monodromic sheaf on $\wt\AA\times\GG_m$.
Recall the following well-known result.
\begin{lem}
Let $\chi$ be a character sheaf on $\GG_m$ and $j: \gm \ra \ga$ be the open embedding. Then $\on{C}^*(\GG_a, j_!\chi)=0$. 
\end{lem} 
\begin{proof}
We prove this in the \'etale setting. 
We may assume that $\chi$ is of finite coefficient and therefore of finite order. In particular, $\chi$ is a direct summand of $(\GG_m\xrightarrow{[n]}\GG_m)_*E$ for some positive integer $n$. Then $\on{C}^*(\GG_a, j_!\chi)$ is a direct summand of $\on{C}^*(\GG_a, j_!E)=0$.
\end{proof}
We thus see that 
\begin{equation}\label{eq: support on s-strata}
\cK_{!,E,\cF_1}=0
\end{equation}
by \eqref{eq: K?LF}.
\end{enumerate}

Next, we study $(\cF\star j_{\dot{s},*}(\delta^{\mono}))|_{\wt{I}/I^+}$, which is in fact simpler. As before, the restriction is via the $!$-pullback. We have an analogous diagram
\[
\xymatrix{
\wt\AA\times \tLG_{s}/I^+  \ar@{=}[d]\ar[rr] && \tLG_{s}\times^{I^+}\tLG_{s}/I^+\ar@{^{(}->}[d]_{j'}\ar[r] & \tLG_{s}/I^+ \times I^+\bs \tLG_{s}/I^+ \ar@{^{(}->}[d]\\
 \wt\AA\times \tLG_{s}/I^+ \ar_{\pr_{1}}[d]\ar^-{(t,g)\mapsto(tg, g^{-1})}[rr] &&  \tLG_{\le s}\times^{I^+}\tLG_{s}/I^+\ar[r]\ar^{\mult}[d] &  \tLG_{\le s}/I^+ \times I^+\bs \tLG_{ s}/I^+\\
 \wt\AA \ar^{t\mapsto t}[rr] &&    \tLG_{\le s} /I^+&
}\]
Then using the coordinate as before and base change, it is easy to see that for $\cF_0=\cL[1]\boxtimes\delta^{\mono}$, the $!$-restriction of $\cF\star j_{\dot{s},*}(\delta^{\mono})$ to $\wt{I}/I^+=\wt\AA$ is given by 
\begin{equation}\label{eq: support on unit strate}
(\cF\star j_{\dot{s},*}(\delta^{\mono}))|_{\wt{I}/I^+}\simeq \delta^{\mono}\otimes C^\bullet(\GG_a,\cL).
\end{equation}

\subsubsection{Some properties of (co)standard objects}
The first assertion we need is the behavior of convolutions of (co)standard objects in the case when the lengths add. 

\begin{lem}\label{multi}
	For $w_1,w_2 \in \extw$ such that $\ell(w_1)+\ell(w_2)=\ell(w_1w_2)$, and for $\cL_1,\cL_2\in D((\wt\AA,\mono)\bs \wt\AA)$, there are natural isomorphisms 
    \[
     j_{\dot{w}_1,!}(\cL_1)\star j_{\dot{w}_2,!}(\cL_2)\simeq j_{\dot{w}_1\dot{w}_2,!}(\cL_1\star w_1(\cL_2)),
     \] 
     \[
     j_{\dot{w}_1,*}(\cL_1)\star j_{\dot{w}_2,*}(\cL_2)\simeq j_{\dot{w}_1\dot{w}_2,*}(\cL_1\star w_1(\cL_2)).
    \]
    In particular, we have
    \begin{equation} \label{e:allmons}\std(w_1) \star \std(w_2) \simeq \std(w_1w_2), \quad \quad \costd(w_1) \star \costd(w_2) \simeq \costd(w_1w_2), \end{equation}
    and for any fixed monodromy $\chi$
    \begin{equation} \label{e:onemon}\std(w_1)_{\chi} \star \std(w_2)_{w_1^{-1}\chi} \simeq \std(w_1w_2)_\chi, \quad \quad \costd(w_1)_{\chi} \star \costd(w_2)_{w_1^{-1} \chi} \simeq \costd(w_1w_2)_\chi. \end{equation} 
\end{lem}

\begin{proof} 
Consider the composition $$\tLG \times^{I^+} \tLG / I^+ \xra{q} \tLG \times^{\tI} \tLG / I^+ \xra{\mult} \tLG / I^+.$$
Note that the second map is ind-proper. By our assumption on the lengths, we know that $j_{\dot{w}_1,?}(\cL_1)\star j_{\dot{w}_2,?}(\cL_2)$ supports on the closure of the $w_1w_2$-stratum by a standard argument, as well as the multiplication map $$\tLG_{w_1} \times^{\tI} \tLG_{w_2}/I^{+} \ra \tLG_{w_1w_2}/I^{+}$$
is an isomorphism. As a result, we have the following diagram with both squares Cartesian
$$\xymatrix{ \wt \AA\dot{w}_1 \times \wt \AA\dot{w}_2 \ar[rr]^-{\on{mult}}_-{(a_1\dot{w}_1,a_2\dot{w}_2)\mapsto a_1w_1(a_2)\dot{w}_1\dot{w}_2} \ar[d]^\iota && \wt \AA\dot{w}_1\dot{w}_2 \ar[d]^{\iota} \\ \tLG_{w_1} \times^{I^+} \tLG_{w_2}/I^{+} \ar[rr]^{q}\ar^{j_{w_1,w_2}}[d] && \tLG_{w_1} \times^{\tI} \tLG_{w_2}/I^{+}\ar^{j_{w_1,w_2}}[d]\\
\tLG \times^{I^+} \tLG/I^{+} \ar[rr]^{q} && \tLG \times^{\tI} \tLG/I^{+}.}$$ 
Notice that $\on{mult}_{*}(\cL_1\boxtimes\cL_2)=\cL_1\star ((w_{1})_*\cL_2)$. In addition, notice that $q_*$ commutes with both $(j_{w_1,w_2})_*$ and $(j_{w_1,w_2})_!$ when restricted to the monodromic sheaves (for the $!$-pushforward case, one uses Proposition \ref{lem: functoriality depth zero geom Langlands for tori}). The lemma then follows.
\end{proof}

The next assertion we need is the standard fact that (co)standard objects are inverse to one another under convolution.

\begin{lem} \label{inverse}
	For $w \in \extw$ and $\cL_1,\cL_2\in D((\wt\AA,\mono)\bs \wt\AA)$, there is a natural isomorphism 
    \[
    j_{\dot{w},!}(\cL_1)\star j_{\dot{w}^{-1},*}(\cL_2)\cong j_{\dot{w},*}(\cL_1)\star j_{\dot{w}^{-1},!}(\cL_2)\cong j_{e,!}(\cL_1\star w(\cL_2)).
    \]
    In particular, we have
    \begin{equation} \label{e:kb1} \std(w) \star \costd(w^{-1}) \simeq \costd(w)\star \std(w^{-1})\simeq \Delta(e),\end{equation}
	and for any fixed monodromy $\chi$, there is an isomorphism \begin{equation} \label{e:kb2} \std(w)_{\chi} \star \costd(w^{-1})_{w^{-1}\chi} \simeq \costd(w)_{\chi} \star \std(w^{-1})_{w^{-1}\chi}\simeq  \Delta(e)_\chi. \end{equation}
\end{lem}

\begin{proof} 
By Lemma \ref{multi} we may reduce to the case where $w$ is length zero or a simple reflection, and $\cL_1=\cL_2=\delta^{\mono}$. That is, we need to prove \eqref{e:kb1}.
The case of length zero is clear. It therefore remains to verify, for a simple reflection $s$, that one has an isomorphism
$$\std(s) \star \costd(s) \simeq \std(e).$$

By \eqref{eq: support on s-strata}, we see that the convolution $\std(s) \star \costd(s)$ is supported on the neutral double coset $I^+\bs\tI/I^+$. By \eqref{eq: support on unit strate}, we know that $\std(s) \star \costd(s)|_{Z_e} \simeq \delta^{\mono}$. The desired isomorphism follows.
\end{proof} 

\begin{lem}\label{lem: conv of costd by costd}
    Let $s$ be a simple reflection in $\extw$. There is the following cofiber sequence
    \[
j_{e,*}(\delta^{\mono})\to j_{\dot{s},*}(\delta^{\mono})\star j_{\dot{s},*}(\delta^{\mono})\to j_{\dot{s},*}(\CH(\omega_{\ker\halpha_s^\vee}))[1]\to
    \]
    and similarly
        \[
j_{\dot{s},!}(\CH(\omega_{\ker\halpha_s^\vee}))\to  j_{\dot{s},!}(\delta^{\mono})\star j_{\dot{s},!}(\delta^{\mono})\to 
j_{e,!}(\delta^{\mono})\to. 
    \]
    Here $\ker\halpha_s^\vee$ is as in \eqref{eq: kernel of halphasvee}.
\end{lem}
\begin{proof}
    The first cofiber sequence follows from \eqref{eq: shrek-stalk-conv-costad-s-strata} and \eqref{eq: support on unit strate}. The second one follows from the first one by convolving it with $j_{\dot{s},!}(\delta^{\mon})\star j_{\dot{s},!}(\delta^{\mon})$ and Lemma \ref{inverse}.
\end{proof}

The following lemma concerns the cleanness of objects associated to simple reflections.

\begin{lem}\label{stdeqcostd}
	Let $s$ be a simple reflection in $\extw$, and $j_{s}: Z_s \rightarrow Z$ as before.     
    Then there is a natural cofiber sequence
    $$
   j_{e,!}(\CH(\omega_{\ker\halpha_s^\vee}))\to j_{\dot{s},!}(\delta^{\mono})\to j_{\dot{s},*}(\delta^{\mono})\to . $$
    In particular, 
     $$\std(s)_\chi \simeq \costd(s)_\chi$$
     if and only if $s \notin \extw^\circ_{\chi}$.  
\end{lem}

\begin{proof} 
The cofiber sequence follows from Lemma \ref{inverse} by convolving the first cofiber sequence in Lemma \ref{lem: conv of costd by costd} with $j_{\dot{s},!}(\delta^{\mono})$. 

Recall that $s\in \extw^\circ_{\chi}$ if and only if the pullback $\halpha_s^*(\chi)$ is the trivial character sheaf on $\GG_m$. In the spectral side, this means exactly when $\chi\in \ker\halpha_s^\vee$. Thus if $s \notin \extw^\circ_{\chi}$, then $\dcm\star \CH(\omega_{\ker\halpha_s^\vee})=0$ and therefore $\std(s)_\chi \simeq \costd(s)_\chi$. Conversely, if $\std(s)_\chi \simeq \costd(s)_\chi$ then $\CH(\omega_{\ker\halpha_s^\vee})\star \dcm=0$. This implies that $\chi\not\in \ker\halpha_s^\vee$, or equivalently $s \notin \extw^\circ_{\chi}$.
\end{proof}

\subsection{Block decomposition}
\label{s:block}
We now apply the results of the previous section to obtain the block decompositions of the categories ${}_{\chi}\cM_{\chi'}$. These will play a basic role in the Soergel theoretic description of the categories we will obtain in the remainder of the paper.

We first set up some relevant notations. Let $\beta$ be a block in ${}_{\chi} \Omega_{\chi'}$, i.e., an $\tilW_{\chi}^{\circ} \times \tilW_{\chi'}^{\circ}$ orbit in ${}_{\chi}\tilW_{\chi'}$, cf. Section \ref{s:comboblock}.

\begin{defn} For $\beta$ as above, we define $$i_\beta:  {}_{\chi} \cM _{\chi'}^{\beta} \hookrightarrow {}_{\chi}\cM_{\chi'}$$ to be the full subcategory generated under colimits and shifts by the objects  $$j(w)_{!, \wt{\chi}}, \quad \quad w \in \beta, \wt{\chi} \in \mathscr{L}_\chi.$$ 
\end{defn}

When $\chi = \chi'$ and $\beta$ is the block containing the identity element, we write ${}_{\chi} \cM _{\chi}^{\circ}$ instead of ${}_{\chi} \cM _{\chi}^{\beta}$.

\begin{rem} It follows from Lemma \ref{l:compsincofree} that  ${}_{\chi} \cM _{\chi'}^{\beta}$ is also generated under colimits and shifts by the non-compact objects $$\Delta(w)_\chi, \quad w \in \beta.$$
\end{rem}

The statement of the block decomposition is then given as follows.

\begin{prop}\label{blockdec} The direct sum of the inclusions $i_\beta$, for $\beta \in {}_{\chi}\Omega_{\chi'}$, is an equivalence, i.e., 
\begin{equation} \label{e:blockdec}\underset{\beta} \oplus \hspace{.5mm} i_\beta: \underset{\beta} \oplus \hspace{.5mm} {}_{\chi} \cM_{\chi'}^\beta \xrightarrow{\sim} {}_{\chi}\cM_{\chi'}. \end{equation}
\end{prop}

To prove the proposition, we will need one more lemma.

\begin{lemma} Given a simple reflection $s$, the convolution functor
$$ - \star \std(s): {}_\chi\cM_{\chi'} \rightarrow {}_\chi\cM_{s\chi'}$$
sends blocks to blocks. That is, given a block ${}_\chi\cM_{\chi'}^\beta$, its essential image under $- \star \std(s)$ is a block of ${}_{\chi}\cM_{s\chi'}$. \label{l:b2b}
\end{lemma}

\begin{rem} 
An immediate consequence of Lemma \ref{l:b2b} and Lemma \ref{multi} is that convolution with any standard sheaf $$\std(w), \quad \quad w \in \tilW,$$ sends blocks to blocks, and similarly for their inverse costandard sheaves, cf. Lemma \ref{inverse}. 
\end{rem}

\begin{proof} First suppose that $s \notin \tilW_{\chi'}^\circ$. In this case, using that $\tilW_{s\chi'}^\circ = s\tilW^\circ_{\chi'} s$, note that the map 
$${}_\chi\tilW_{\chi'} \xrightarrow{\sim} {}_\chi\tilW_{s\chi'}, \quad \quad w \mapsto ws $$
tautologically sends a block $\beta = w\tilW_{\chi'}^\circ$ to 
$$w\tilW_{\chi'}^\circ s = ws \tilW_{s\chi'}^\circ =: \beta s.$$

Passing to the categories, it is therefore enough to show for any $\wt{\chi} \in \mathscr{L}_\chi$ and $w \in {}_\chi\tilW_{\chi'}$ that one has an isomorphism
$$j(w)_{!, \wt{\chi}} \star \Delta(s) \simeq j(ws)_{!, \wt{\chi}}.$$
If $w < ws$, this is a case of Lemma \ref{multi}. If $w > ws$, we must equivalently show 
$$j(w)_{!, \wt{\chi}} \simeq j(ws)_{!, \wt{\chi}} \star \costd(s).$$
However, by Lemma \ref{stdeqcostd}, we have $$j(ws)_{!, \wt{\chi}} \star \costd(s) \simeq j(ws)_{!, \wt{\chi}} \star \std(s),$$
whence we are again done by Lemma \ref{multi}. 

It remains to address the case where $s \in \tilW_{\chi'}^\circ$. In this case, the map 
$${}_\chi\tilW_{\chi'} \rightarrow {}_\chi\tilW_{\chi'}, \quad \quad w \mapsto ws$$
tautologically sends any block $\beta$ bijectively onto itself. Passing to categories, we will therefore show that $- \star \std(s)$ restricts to an autoequivalence
$$- \star \std(s): {}_{\chi}\cM_{\chi'}^\beta \simeq {}_{\chi}\cM_{\chi'}^\beta.$$
More precisely, if we fix $x \in \beta$, and without loss of generality assume that $x < xs$, we will show that for any $\wt{\chi} \in \mathscr{L}_\chi$ that $- \star \Delta(s)$ restricts to an autoequivalence of the subcategory generated under shifts and colimits by the objects
$$ j(x)_{!, \wt{\chi}} \quad \text{and} \quad j(xs)_{!, \wt{\chi}}.$$
Indeed, by Lemma \ref{multi}, we have that \begin{equation}\label{e:puzz1}j(x)_{!, \wt{\chi}} \star \std(s) \simeq j(xs)_{!, \wt{\chi}}.\end{equation} In addition, consider the tautological triangle
$$\std(s) \rightarrow \costd(s) \rightarrow \xi \xrightarrow{+1}$$
as in Lemma \ref{stdeqcostd}. We therefore obtain a distinguished triangle 
\begin{equation}  \label{e:puzz2} j(xs)_{!, \wt{\chi}} \star \std(s) \rightarrow j(xs)_{!, \wt{\chi}} \star \costd(s) \rightarrow j(xs)_{!, \wt{\chi}} \star \xi \xrightarrow{+1}.
\end{equation}
As by Lemma \ref{multi} the middle term is $j(xs)_{!, \wt{\chi}} \star \costd(s) \simeq j(x)_{!, \wt{\chi}},$ and $j(xs)_{!, \wt{\chi}} \star \xi$ lies in the full subcategory generated by $j(xs)_{!, \wt{\chi}}$, the claim follows by combining \eqref{e:puzz1} and \eqref{e:puzz2}. 
\end{proof}

Finally, we are ready to prove the block decomposition. 
\begin{proof}[Proof of Proposition \ref{blockdec}.] As the essential image of the functor \eqref{e:blockdec} contains a set of compact generators, cf. Corollary \ref{c:compgensfixedmon}, it is enough to see fully faithfulness. I.e., we must show that given distinct blocks $\beta, \gamma$, elements $y \in \beta, w \in \gamma$, and elements $ \overline{\chi}, \wt{\chi} \in \mathscr{L}_\chi$, we have the vanishing 
$$\Hom( j(y)_{!, \overline{\chi}},  j(w)_{!, \wt{\chi}}) \simeq 0.$$ 
As in \cite[Proposition 4.11]{LY}, we prove this, for all possible twists on the right $\chi'$ simultaneously, via an induction on the length of $w$. 

If $w$ is of length zero, then the claim is clear, as for support reasons the above Hom space is nonvanishing if and only if $y = w$. For the inductive step, as $w$ is of positive length, we may write $w = xs$, where $x < w$ and $s$ is a simple reflection. By Lemma \ref{multi} and Lemma \ref{inverse}, we obtain that
\begin{align*} \Hom( j(y)_{!, \overline{\chi}},  j(w)_{!, \wt{\chi}}) &\simeq \Hom(j(y)_{!, \overline{\chi}}, j(x)_{!, \wt{\chi}} \star \Delta(s)) \\ &\simeq \Hom(j(y)_{!, \overline{\chi}} \star \nabla(s), j(x)_{!, \wt{\chi}}).
\end{align*}
We are therefore done by our inductive hypothesis and Lemma \ref{l:b2b}. \end{proof}

\section{Cofree tilting sheaves}

In this section, we collect some basic properties of cofree tilting sheaves in monodromic Hecke categories, including the classification of indecomposable objects, and the closure of cofree tilting sheaves under convolution. Then we will describe in Theorem \ref{t:renorm} a renormalization procedure that reconstructs the Hecke category from the ordinary addtive monoidal category of cofree tilting sheaves. Many results in this section are variants or generalizations of similar results in the unipotent monodromy situation.

\subsection{Definition of cofree tilting sheaves} \label{s:tiltbasics} 

\sss{} Let us begin with the relevant (standard) definition.

\begin{defn}[Cofree tilting sheaves] Let $\tau$ be an object of  $\cmc$.
\begin{enumerate}
\item We say $\tau$ {\em admits a  $\std$-flag} if it has a finite step filtration with associated graded isomorphic to cofree standard sheaves. 
\item We say $\tau$ {\em admits a $\costd$-flag} if it has a finite step filtration with associated graded isomorphic to cofree costandard sheaves. 
\item We say $\tau$ is a {\em cofree tilting sheaf} if it admits both a $\std$-flag and a $\costd$-flag.
\end{enumerate}
\end{defn}

Let us note a first useful consequence of the definition. 

\begin{lem} \label{l:duals} If an object $\tau \in \cmc$ admits a $\std$-flag or $\costd$-flag, then it is dualizable with respect to convolution, and its dual $\tau^\vee \in {}_{\chi'}\cM_{\chi}$ admits a $\costd$-flag, or $\std$-flag, respectively.

In particular, any cofree tilting sheaf is dualizable, and its dual is again cofree tilting. 
\end{lem}

\begin{proof} This follows from the definitions and Lemma \ref{inverse}.
\end{proof}

\sss{} We next establish some basic properties of cofree tilting sheaves.  

\begin{lem}\label{c:proptilt} The following hold. 
\begin{enumerate}
    \item\label{c:proptilt-1} An object $\cmc$ admits a $\Delta$-flag (resp. $\costd$-flag) if and only if its $*$-restriction (resp. $!$-restriction) to each stratum is isomorphic to a finite direct sum of copies of the cofree monodromic local system, concentrated in perverse degree zero, and is zero for all but finitely many strata.
    
    \item\label{c:proptilt-2} A direct summand of a cofree tilting sheaf is again cofree tilting sheaf. 

     \item\label{c:proptilt-3} Given two objects $\tau_1, \tau_2$ of $\cmc$, if $\tau_1$ admits a $\std$-flag, and $\tau_2$ admits a $\costd$-flag, then $\Hom(\tau_1, \tau_2)$ is concentrated in degree zero, and admits a filtration with associated graded isomorphic to $\Hom(\Delta(w)_\chi,\nabla(w)_\chi)$. In particular, it is a finite free $\Fun(\frf_\chi)$-module.     

\end{enumerate}
\end{lem}

\begin{proof} The assertion \eqref{c:proptilt-1} follow from standard Cousin filtration arguments along with Corollary \ref{l:cofreeprop} \eqref{l:cofreeprop-3}, which ensures a filtration by cofree monodromic sheaves on a given stratum may be split.

Assertion \eqref{c:proptilt-2} follows from \eqref{c:proptilt-1}, as by Corollary \ref{l:cofreeprop} \eqref{l:cofreeprop-2} their conditions are inherited by passing to direct summands. 

Finally, the Cousin filtration gives a finite filtration of $\Hom(\tau_1,\tau_2)$ with associated graded being $\Hom((j_y)^*\tau_1, (j_w)^!\tau_2)$. Then
assertion \eqref{c:proptilt-3} follows straightforwardly by d\'evissage from Lemma \ref{l:enddcm}, or more carefully from its consequence that
\begin{equation}\Hom( \std(y)_\chi, \costd(w)_\chi ) \cong \begin{cases} \Fun(\frf_\chi), & y = w, \\ 0, & y \neq w. \end{cases} \label{e:homstdcostd}\end{equation}
\end{proof}

\begin{rem} Let us explicitly mention that for cofree tilting $\tau$, the space $\Hom(\Delta(w)_\chi,\tau)$ is a finite free $\Fun(\frf_\chi)$-module. Therefore the problem of computing tilting characters, i.e., the multiplicities in the standard flags of indecomposable objects, are equivalent for $E$ and $\fre$, as indeed one should expect from the `deformation philosophy'. \end{rem}

\begin{defn} Let us denote the full subcategory of $\cmc$ consisting of tilting objects by $$\ctc  \hookrightarrow \cmc.$$
Similarly, we have the full subcategories ${}_\chi\cT_{\chi'}^{\beta} \inj {}_\chi\cM_{\chi'}^{\beta}  $ and ${}_\chi\cT_{\chi}^{\circ} \inj {}_\chi\cM_{\chi}^{\circ}$.
\end{defn}

We have the following basic property of $\ctc$.
\begin{prop}
The category $\ctc$ forms a discrete, Krull-Schmidt subcategory of $\cmc$.
Every object in $\cmc$ is a finite direct sum of indecomposable objects.
\end{prop}
We recall that a full subcategory of an $\infty$-category is called discrete if all homomorphism spaces in this subcategory are discrete.

We also recall that an (ordinary) additive category is called Krull-Schmidt if it is idemponent complete, every object is a finite direct sum of indecomposable objects, and the endomorphism ring of every indecomposable object is a local ring.
\begin{proof}
     That $\ctc$ is discrete follows from Lemma \ref{c:proptilt} \eqref{c:proptilt-3}. By Lemma \ref{c:proptilt} \eqref{c:proptilt-2}, it is also idempotent complete.

    Now let $\tau\in \ctc$. Then it is supported on finitely many strata and its restriction to each stratum (either via $*$- or $!$-pullback) is isomorphic to a finite direct sum of cofree monodromic local systems. Therefore, by Corollary \ref{l:cofreeprop}, $\tau$ is a finite direct sum of indecomposable objects.
    
    Now suppose $\tau$ is indecomposable. Then $\End(\tau)$ does not contain idemponent other than the identity. But
    by Lemma \ref{c:proptilt} \eqref{c:proptilt-3}, $\End(\tau)$ as a module is finite free over $\Fun(\frf_\chi)$, which is a complete local ring. Therefore, $\End(\tau)$ is local.
\end{proof}

\subsection{Classification by support} In this subsection, we will prove that indecomposable cofree tilting sheaves are classified by their support.

\sss{} We start with the construction many cofree tilting sheaves.

\begin{lem} \label{l: tilt simple reflection}
Let $s$ be a simple reflection in $\extw$. 
There exists an indecomposable cofree tilting sheaf $\tau(s)_{\chi}$ supported on $Z_{\le s}$ whose restriction to $Z_s$ is nonzero.
\end{lem}

\begin{proof}
When $s \notin \extw_{\chi}^{\circ}$, Lemma \ref{stdeqcostd} implies that $\std(s)_{\chi} \simeq \costd(s)_{\chi}$ is clean. Hence, we can take $\tau(s)_{\chi}$ to be the (co)standard sheaf. We may therefore assume that $s \in \extw_{\chi}^{\circ}$.

Recall the closed subscheme $\ker\halpha_s^\vee\subset\mathrm{LS}^{t,\Box}_{\wt\AA^\vee}$. Applying the functor $\CH$ to the canonical map $\omega_{\ker\halpha_s^\vee}\to \omega_{\mathrm{LS}^{t,\Box}_{\wt\AA^\vee}}$, and convolving with $\dcm$, we obtain a short exact sequence of monodromic sheaves on $\wt\AA$,
    \[
    0\to \CH(\omega_{\ker\halpha_s^\vee\cap\frf_\chi})\to \dcm\to \dcm\to 0.
    \]
    This gives a cofiber sequence $$j_{e,!}(\CH(\omega_{\ker\halpha_s^\vee\cap\frf_\chi}))\to \Delta(e)_\chi\to\Delta(e)_\chi.$$
Push-out this sequence by the cofiber sequence from Lemma \ref{stdeqcostd} gives
\[
\xymatrix{
j_{e,!}(\CH(\omega_{\ker\halpha_s^\vee\cap\frf_\chi}))\ar[r]\ar[d]& \Delta(e)_\chi\ar[r]\ar[d]&\Delta(e)_\chi\ar@{=}[d]\\
\Delta(s)_\chi\ar[r]\ar[d]& \tau(s)_\chi\ar[r]\ar[d]& \Delta(e)_\chi\\
\nabla(s)_\chi\ar@{=}[r] & \nabla(s)_\chi & 
}
\]
We thus construct a cofree tilting sheaf $\tau(s)_\chi$ with the desired support condition. Moreover, since $\std(s)_\chi$ is not clean, cf. Lemma \ref{stdeqcostd}, it follows that $\tau(s)_\chi$ must be indecomposable.
\end{proof}

We can now construct further cofree tilting sheaves via convolution with $\tau(s)_{\chi}$.

\begin{lem} \label{l: conv tilting s}
    Let $\tau \in {}_{\chi'} \cT_{\chi}$ be a cofree tilting sheaf and $\tau(s)_{\chi}$ as in Lemma \ref{l: tilt simple reflection}. Then $\tau \star \tau(s)_{\chi}$ is tilting.
\end{lem}

\begin{proof}
    We claim more generally that the functor
$$- \star \tau(s)_{\chi}: {}_{\chi'}\cM_{\chi} \rightarrow {}_{\chi'}\cM_{s\chi}$$
sends (co)standardly filtered objects to (co)standardly filtered objects. Let us show that $\Delta(w)_{\chi'} \star \tau(s)_{\chi}$ is filtered by copies of $\Delta(w)_{\chi'}$ and $\Delta(ws)_{\chi'}$, for any $w \in {}_{\chi'}\tilW_{\chi}$; a similar argument will apply for costandard objects. Indeed, if $w < ws$, this follows from using that $\tau(s)_{\chi}$ is filtered by copies of $\Delta(e)_{\chi}$ and $\Delta(s)_{\chi}$ and Lemma \ref{multi}. If $w > ws$, this follows from using that $\tau(s)_{\chi}$ is filtered by copies of $\costd(e)_{\chi}$ and $\costd(s)_{\chi}$ and Lemma \ref{multi}. 
\end{proof}

\sss{}  
We can now prove that, as usual, indecomposable cofree tilting sheaves are classified by their support.

\begin{prop} \label{p:classtilts}
For each $w \in \cwc$, there exists a unique up to non-unique isomorphism indecomposable cofree tilting sheaf $\tau(w)_{\chi}$ supported on $Z_{\le w}$ whose restriction to $Z_w$ is isomorphic to $\iota_{\dot{w}}(\dcm)$, where we recall $\iota_{\dot{w}}$ is as in \eqref{e:onestratumunivmon}.

Moreover, any indecomposable object of $\ctc$ is isomorphic to exactly one $\tau(w)_\chi$.
\end{prop}

\begin{proof}
{\em Step 1.}    We first prove the existence of $\tau(w)_{\chi}$ and simultaneously show that $\std(w)_{\chi}$ appears exactly once in the $\std$-flag of $\tau(w)_{\chi}$, meaning that the restriction of $\tau(w)_{\chi}$ to $Z_w$ is isomorphic to $\iota_{\dot{w}}(\dcm)$. The proof proceeds by induction on $\ell(w)$, allowing $\chi$ and $\chi'$ to vary.

When $\ell(w)=0$, we know that $\std(w)_{\chi} \simeq \costd(w)_{\chi}$ is clean by support reason. Hence, $\tau(w)_{\chi}$ exists and is isomorphic to the (co)standard sheaf. 

When $w$ is a simple reflection in $\tilW$, we have shown the existence of $\tau(w)_{\chi}$ and the multiplicity of $\std(w)_{\chi}$ in $\tau(w)_{\chi}$ is one in Lemma \ref{l: tilt simple reflection}.

Let $s$ be a simple reflection in $\extw$ such that $w > ws$. Assuming the tilting sheaf $\tau(ws)_{\chi}$ has been constructed, we now establish the existence of $\tau(w)_{\chi}$. By Lemma \ref{l: conv tilting s}, the convolution $\tau(ws)_{\chi} \star \tau(s)_{sw^{-1}\chi}$ is tilting. By the support reason and the inductive hypothesis, the multiplicity of $\std(w)_{\chi}$ in $\tau(ws)_{\chi} \star \tau(s)_{sw^{-1}\chi}$ is one. Hence, there is a unique indecomposable summand $\tau$ of $\tau(ws)_{\chi} \star \tau(s)_{sw^{-1}\chi}$ that is nonzero on $Z_w$. This summand is our desired $\tau(w)_{\chi}$, completing the proof of its existence.

{\em Step 2.} Let $\tau_1,\tau_2$ be two cofree tilting sheaves supported on $Z_{\le w}$. We notice that in associated graded of the filtration of $\Hom(\tau_1,\tau_2)$ as in (the proof of) Lemma \ref{c:proptilt} \eqref{c:proptilt-3}, $\Hom((j_w)^*\tau_1, (j_w)^!\tau_2)$ appears as the last quotient. Therefore, every map $(j_w)^*\tau_1\to (j_w)^!\tau_2$ lifts to a map $\tau_1\to \tau_2$.

This particularly implies that for a cofree tilting sheaf $\tau$ supported on $Z_{\le w}$ and whose restriction to $Z_w$ is isomorphic to $\iota_{\dot{w}}\dcm$, the restriction to $Z_w$ induces a surjective map
\[
\End(\tau(w)_\chi)\to \Fun(\frf_\chi).
\]
Note that by locality of $\End(\tau(w)_\chi)$, any lifting of $1\in \Fun(\frf_\chi)$ is an invertible element in $\End(\tau(w)_\chi)$.

{\em Step 3.} Now let $\tau(w)_\chi,\tau'$ be two indecomposable cofree tilting sheaves supported on $Z_{\le w}$ whose restrictions to $Z_w$ are non-zero. In addition, suppose $\tau(w)_\chi|_{Z_w}\simeq \iota_{\dot{w}}\dcm$. Then there exists a split injective map $\tau(w)_\chi|_{Z_w}\hookrightarrow \tau'|_{Z_w}$ by Corollary \ref{l:cofreeprop}. By Step 2, we obtain maps
$$
\tau(w)_\chi\to \tau'\to \tau(w)_\chi
$$
whose composition is an automorphism of $\tau(w)_\chi$. It follows that $\tau(w)_\chi$ is a direct summand of $\tau'$. Therefore, $\tau(w)_\chi\simeq \tau'$, as desired.
\end{proof}

\begin{rem}
    We remark that the statement is known for the case when $E$ is a field. See for example \cite[Proposition 5.12]{BR}, \cite[Proposition 3.14]{IY}, and \cite[Proposition 5.4.5]{CD}. Note the last reference is explicitly phrased in the language of cofree monodromic (rather than free monodromic) objects.
\end{rem}

\subsection{Properties of cofree tiltings} Having constructed the cofree tilting sheaves, let us now collect some of their properties.

\sss{} We begin with describing the indecomposable tilting sheaf associated to a minimal element $w^{\beta}$. 

\begin{lem} \label{l:minclean}
The indecomposable tilting sheaves $\tau(w^{\beta})$ is clean for any coset $\beta$ in ${}_{\chi} \extw _{\chi'}$. In other words, we have $$\std(w^{\beta})_\chi \simeq \tau(w^{\beta})_\chi \simeq \costd(w^{\beta})_\chi.$$ 
\end{lem}

\begin{proof}
We will prove via an induction, over all cosets $\beta$ in $\tilW_\chi^\circ \bs \tilW$ simultaneously, on $\ell(w^{\beta})$, i.e., with respect to the usual length function on $\tilW$. 

If $\ell(w^{\beta}) = 0$, the claim is clear as $\tilI w^{\beta} \tilI$ is a closed stratum in $\wt{LG}$. 

If $\ell(w^{\beta}) > 0$, we may choose a simple reflection $s$ in $\tilW$ such that $w^{\beta} s < w^{\beta}$. By the minimality of $w^{\beta}$ in $$\tilW_\chi^\circ w^{\beta} = w^{\beta} \tilW_{w^{\beta,-1} \chi}^\circ,$$it follows that $s \notin \tilW_{w^{\beta,-1} \chi}^\circ$. In particular, by Lemma \ref{stdeqcostd}, we have the cleanness \begin{equation} \label{e:clean}\std(s)_{w^{\beta, -1}\chi} \simeq \tau(s)_{w^{\beta, -1}\chi} \simeq \costd(s)_{w^{\beta, -1}\chi}.\end{equation}
Moreover, by the minimality of $w^{\beta}$ in $\tilW_\chi^\circ w^{\beta}$, and the tautological minimality of $s$ in $\tilW_{w^{\beta,-1}\chi}^\circ s$, it follows that $w^{\beta} s$ is minimal in $\tilW_\chi^\circ w^{\beta}s$, cf. Lemma \ref{l:minmult}. By the cleanness \eqref{e:clean}, we have 
$$- \star \tau(s)_{w^{\beta, -1}\chi}: {}_\chi \cT_{w^{\beta,-1}\chi} \simeq {}_{\chi} \cT_{sw^{\beta,-1}\chi},$$
so in particular we deduce  
\begin{equation} \label{eq: mintiltconv}
    \tau(w^{\beta})_\chi \simeq \tau(w^{\beta}s)_\chi \star \tau(s)_{sw^{\beta, -1}\chi},
\end{equation}
whence the result follows by induction. 
\end{proof}

We next check that convolving with a minimal tilting sheaf preserves indecomposable tilting sheaves. 

\begin{lem}\label{l: convmintilt}
    Let $\beta$ be a coset in ${}_{\chi} \extw _{\chi'}$ and $x \in \extw$. Then we have $$\tau(x)_{x\chi} \star \tau(w^{\beta})_{\chi} \simeq \tau(xw^{\beta})_{x\chi}.$$ Similarly, we have $$\tau(w^{\beta})_{\chi} \star \tau(x)_{\chi'} \simeq \tau(w^{\beta}x)_{\chi}.$$
\end{lem}

\begin{proof}
    As in Lemma \ref{l:minclean}, we prove via an induction on $\ell(w^\beta)$. If $\ell(w^\beta)=0$, the statement follows from Lemma \ref{multi}. If $\ell(w^\beta)>0$, we may choose a simple reflection $s$ in $\tilW$ such that $w^{\beta} s < w^{\beta}$. Using \eqref{eq: mintiltconv} and the inductive hypothesis, we obtain 
    $$\tau(x)_{x\chi} \star \tau(w^{\beta})_{\chi} \simeq \tau(x)_{x\chi} \star \tau(w^{\beta}s)_{\chi} \star \tau(s)_{s\chi} \simeq \tau(xw^{\beta}s)_{x\chi} \star \tau(s)_{s\chi}.$$ The induction is completed by applying \eqref{eq: mintiltconv} once more.
\end{proof}

For the future use, we prove the following basic lemma concerning tilting sheaves and the partial order $\leq_{\beta}$.

\begin{lem} \label{l:supptilt} Let $\beta$ be a coset in ${}_{\chi} \extw _{\chi'}$. For any $x \in {}_{\chi} \extw _{\chi'}^{\beta}$, consider the indecomposable tilting sheaf $\tau(x)_\chi$. Then any $w \in \tilW$ for which $\std(w)_\chi$ or $\costd(w)_\chi$ occurs in a $\std$-flag or $\costd$-flag of $\tau(x)_\chi$ satisfies the inequality $$w \leq_{\beta} x.$$ 
\end{lem}

\begin{proof} Let $\Xi$ be the $\extw$-orbit of $\chi$. Let ${}_{\Xi} \tilW_{\Xi}$ be the union of ${}_{\chi'} \tilW_{\chi''}$ for $\chi', \chi'' \in \Xi$. We will proceed via induction, over all cosets $\beta$ in ${}_{\Xi} \tilW_{\Xi}$ simultaneously, on $\ell_{\beta}(x)$. 

We begin with the base case of $\ell_{\beta}(x) = 0$, i.e., $x = w^{\beta}$, which follows directly from Lemma \ref{l:minclean}. Suppose now that $\ell_{\beta}(x) > 0$. In this case, we may choose a simple reflection $r$ in $\tilW_{x^{-1} \chi}^\circ$ such that $xr <_{\beta} x.$ 

If $r$ is also a simple reflection of $\extw$, we know that $\tau(x)_\chi$ is a summand in 
$$\tau(xr)_\chi \star \tau(r)_{rx^{-1}\chi}.$$
By using the inductive hypothesis for $\tau(xr)_\chi$, and the fact that $\tau(r)_{rx^{-1}\chi}$ is supported on the strata corresponding to $1$ and $r$, it is straightforward to deduce the desired result for $\tau(xr)_\chi$. 

In general, there exists a minimal element $u \in \extw$ for some block, such that $uru^{-1}$ is a simple reflection in both $\extw$ and $\extw_{ux^{-1}\chi}^{\circ}$ by Proposition \ref{p:conjsimple}. Thanks to Lemma \ref{l: convmintilt}, we have the equivalence
$$\tau(u)_{u \chi} \star - \star \tau(u^{-1})_{x^{-1} \chi}: {}_{\chi} \cT_{x^{-1} \chi} \simeq {}_{u\chi} \cT_{ux^{-1} \chi}$$
and the resulting identities 
\begin{equation}\label{eq: conjminx}\tau(u)_{u \chi} \star \tau(x)_\chi \star \tau(u^{-1})_{x^{-1} \chi} \simeq \tau(uxu^{-1})_{u\chi}.
\end{equation}
\begin{equation} \label{eq: conjminxr}
    \tau(u)_{u \chi} \star \tau(xr)_\chi \star \tau(r)_{rx^{-1}\chi} \star \tau(u^{-1})_{x^{-1} \chi} \simeq \tau(uxru^{-1})_{u\chi} \star \tau(uru^{-1})_{urx^{-1}\chi}.
\end{equation}
Since $uru^{-1}$ is a simple reflection of $\extw$, the discussion above and the inductive hypothesis prove the statement for $\tau(uxu^{-1})_{u\chi}$.

To prove the desired result for $\tau(x)_{\chi}$, we note that conjugation by the minimal element $u$ preserves Bruhat order, as stated in Proposition \ref{p:conjmin}.
\end{proof}

\sss{}
For later purposes in studying the Soergel functor in our set-up, we need to establish some further properties of $\tau(s)_\chi$, for $s$ a simple reflection in $\extw$ which is in $\extw_\chi^\circ$.

Recall from Remark \ref{rem: monodromy action} that the left and the right actions of $D_{\chi\mon}(\wt\AA)$ on $\cM$ induce a homomorphism
\begin{equation}\label{eq: mono of tilting}
\Fun(\frf_\chi)\hat\otimes_{E_\chi} \Fun(\frf_{\chi})=\Fun(\frf_\chi\times_{E_\chi}\frf_\chi)\to \End(\tau(s)_\chi).
\end{equation}

For $\gamma\in \{1,s\}=W_s$, we let $\frf_\chi\simeq \Gamma^\gamma\subset \frf_\chi\times_{E_\chi}\frf_\chi$ denote the graph of the action map given by $\gamma$. For example, when $\gamma=1$ is the identity element, then $\Gamma^1$ is the diagonal. We note that $\Gamma^1\cap\Gamma^s\simeq \frf_\chi^s$ is the closed subscheme of $s$-fixed points. On the other hand, let $\alpha_s$ (resp. $\halpha_s$) be the affine root (resp. coroot) associated to $s$. 
We have 
$$
\frf_\chi\cap \ker \halpha_s^\vee\subset \frf_\chi\cap \ker (\alpha_s^\vee\circ \halpha_s^\vee)=\frf_\chi^s.
$$
Note that if the characteristic of $\fre$ is two, and if the affine simple root $\alpha_s$ is the twice of a character of $\wt\AA$, then the inclusion is strict. Otherwise,
$\frf_\chi\cap \ker \halpha_s^\vee= \frf_\chi^s$.

\begin{lem}\label{lem: endo of taus}
    We have the short exact sequence of $\Fun(\frf_\chi)\hat\otimes_{E_\chi} \Fun(\frf_{\chi})$-modules
    \[
    0\to \End(\tau(s)_\chi)\to \Fun(\Gamma^1)\oplus \Fun(\Gamma^s)\to \Fun(\frf_\chi\cap \ker\halpha_s^\vee)\to 0
    \]
    In particular, if the characteristic of $\fre$ is not two, or if the affine simple root $\alpha_s$ is not twice of a character of $\wt\AA$, then \eqref{eq: mono of tilting} is surjective, identifying $\End(\tau(s)_\chi)$ with $\Fun(\Gamma^1\cup \Gamma^s)$.
\end{lem}
\begin{proof}
Note that by the construction of $\tau(s)_\chi$, we have the following cofiber sequence of $\Fun(\frf_\chi)\hat\otimes\Fun(\frf_\chi)$-modules
\begin{multline*}
\End(\tau(s)_\chi)\to \Hom(\Delta(e)_\chi,\tau(s)_\chi)\oplus \Hom(\Delta(s)_\chi,\tau(s)_\chi)\to \\
\Hom(j_{e,!}(\CH(\omega_{\ker\halpha_s^\vee}\cap\frf_\chi)),\tau(s)_\chi)\to. 
\end{multline*}
We note that
\begin{multline*}
\Hom(j_{e,!}(\CH(\omega_{\ker\halpha_s^\vee}\cap\frf_\chi)),\tau(s)_\chi)=\Hom(\omega_{\ker\halpha_s^\vee\cap\frf_\chi}, \omega_{\frf_\chi})\\
=\Hom(\omega_{\ker\halpha_s^\vee\cap \frf_\chi}, \omega_{\ker\halpha_s^\vee\cap \frf_\chi})=\Fun(\ker\halpha_s^\vee\cap \frf_\chi).
\end{multline*}
Thus the cofiber sequence gives the desired short exact sequence.
\end{proof}

\subsection{Monoidal structure} 
Finally, using results in previous subsections, we can discuss the monoidal structure for tilting sheaves. 

\sss{}We first show that cofree tilting sheaves are closed under convolution. 
	
\begin{prop}\label{p:conv tilt} Given three character sheaves $\chi_i$, $i = 1, 2, 3$, and tilting objects
$$\tau_{12} \in {}_{\chi_1}\cT_{\chi_2} \quad \text{and} \quad \tau_{23} \in {}_{\chi_2}\cT_{\chi_3},$$
the convolution $\tau_{12} \star \tau_{23}$ is again tilting, i.e., 
$$\tau_{12} \star \tau_{23} \in {}_{\chi_1}\cT_{\chi_3}.$$
\end{prop}

\begin{proof} Without loss of generality, we may assume that $\tau_{23}$ is indecomposable, i.e., $\tau_{23} \simeq \tau(w)_{\chi_2}$ for some $w \in {}_{\chi_2}\tilW_{\chi_3}$. We will induct, simultaneously over all choices of $\chi_2$ and $\chi_3$, on the length of $w$. 

If $w$ is length zero, the result is clear. If $w = s$ is a simple reflection, it is the content of Lemma \ref{l: conv tilting s}. 

Finally, for the inductive step, as $w$ has positive length, we may write $w = xs$, for $x < w$, and $s$ a simple reflection. It follows from Proposition \ref{p:classtilts} that $\tau(w)_{\chi_2}$ is a summand in $\tau(x)_{\chi_2} \star \tau(s)_{x \chi_2}$. The result therefore follows by induction and Corollary \ref{c:proptilt}(3). \end{proof}    
     
The following is an immediate corollary of Proposition \ref{p:conv tilt}. 
    
\begin{cor} \label{c:xtx}If for a $\tilW$-orbit of monodromies $\Xi$  we consider the direct sum of additive categories
$$\xtx := \underset{\chi, \chi' \in \Xi} \oplus \hspace{1mm} \ctc$$
then $\xtx$ inherits the structure of a(n ordinary, non-unital) monoidal category for which the inclusion functor $\xtx \rightarrow \cM_\Xi$ is monoidal. Similarly, the analogous statement holds for ${}_\chi\cT_{\chi}^{\circ} \inj {}_\chi\cM_{\chi}^{\circ}$.
\end{cor}

\begin{rem} \label{r:units}Here is a minor technical remark, regarding the (non)-unitality of $\xtx$. Consider the ind-completion of $\xtx$, which we denote by $\on{Ind}(\xtx)$. Explicitly, this simply consists of arbitrary direct sums, not necessarily finite, of indecomposable tilting objects. Then $\on{Ind}(\xtx)$ is naturally monoidal, and contains a monoidal unit, namely 
$$\underset{\chi \in \Xi} \oplus \hspace{1mm} \dcm.$$Since we have opted to only consider finite direct sums of tilting sheaves in $\xtx$, this unit may not lie in $\xtx$ itself, if the $\tilW$-orbit of $\chi$ is infinite. Nonetheless, $\xtx$ is a semi-rigid monoidal category, in the sense that $\on{Ind}(\xtx)$ is a unital monoidal category, is compactly generated, and every compact object is dualizable. 
\end{rem}

Similarly, although we do not explicitly need it, we obtain the monoidal subcategory $$\cT := \underset{\Xi} \oplus \hspace{1mm} \xtx \rightarrow \cM.$$

\sss{}We now describe the monoidal generators for the tilting objects. For this, we recall the subgroup $\Omega_\chi \hookrightarrow \tilW_\chi$, cf. Section \ref{s:comboblock}.

\begin{lem} \label{l:tiltgens}The following hold. 

\begin{enumerate}
\item $\xtx$ is generated as a Karoubian monoidal category by the objects
$$\tau(s)_\chi \quad \text{and} \quad \tau(\omega)_\chi,$$
for $s$ a simple reflection in $\tilW$, $\omega \in \Omega$, and $\chi \in \Xi$. That is, every indecomposable object of $\xtx$ is a summand in a convolution of such objects.

\item ${}_\chi\cT_\chi$ is generated as a Karoubian monoidal category by the objects 
$$\tau(r)_\chi \quad \text{and} \quad \tau(\omega)_\chi,$$
for $r$ a simple reflection in $\tilW_\chi^\circ$ and $\omega \in \Omega_\chi. $ That is, every indecomposable object of ${}_\chi\cT_\chi$ is a summand in a convolution of such objects. 

\item ${}_\chi\cT_\chi^{\circ}$ is generated as a Karoubian monoidal category by the objects 
$$\tau(r)_\chi,$$
for $r$ a simple reflection in $\tilW_\chi^\circ$. That is, every indecomposable object of ${}_\chi\cT_\chi^{\circ}$ is a summand in a convolution of such objects. 
\end{enumerate}
    
\end{lem}

\begin{proof}Assertion (1) was proven in the course of the proof of Proposition \ref{p:conv tilt}.    

For assertion (2), consider any $x \in \tilW_\chi$ and a simple reflection $r$ in $\tilW_\chi^\circ$ such that $xr <_{\chi} x$ in the Bruhat order on $\tilW_\chi$, i.e., for which $x(\halpha_r)$ is a negative integral coroot. It is enough to show that in this case$$\tau(x)_\chi \overset{\oplus}\hookrightarrow \tau(xr)_\chi \star \tau(r)_\chi,$$i.e., that $\tau(x)_\chi$ is a summand in the convolution of $\tau(xr)_\chi$ and $\tau(r)_\chi$, in fact appearing with multiplicity one. 

When $r$ is also a simple reflection in $\tilW$, then the assertion is clear for support reasons. 

For a general reflection $r$, we choose a minimal element $u$ such that $uru^{-1}$ is a simple reflection in both $\extw$ and $\extw_{u\chi}^{\circ}$, as in the proof of Lemma \ref{l:supptilt}. The previous paragraph has shown that $$\tau(uxu^{-1})_{u\chi} \overset{\oplus}\hookrightarrow \tau(uxru^{-1})_{u\chi} \star \tau(uru^{-1})_{u\chi}.$$ appearing with multiplicity one. 

The result now follows from the identities \eqref{eq: conjminx} and \eqref{eq: conjminxr}.

Assertion $(3)$ can be proved similarly. 
\end{proof}

\subsection{Reconstructing the affine Hecke category via tiltings} 

We have isolated within the affine Hecke category the additive monoidal subcategory of cofree monodromic tilting sheaves, and established some of its basic properties, in particular a bijection between isomorphism classes of indecomposable objects and strata. The latter result suggests that one should be able to reconstruct the affine Hecke category itself from the tilting sheaves, and in this subsection we explain how this may be done, which is essentially a renormalization procedure. 

\sss{} To state the result, recall that to any (ordinary) additive category $\mathscr{A}$ one can attach its (dg-enhanced) bounded homotopy category $\on{K}^b\mathscr{A}$, which is a stable category. Namely, the objects are bounded chain complexes $M^\bullet$ of objects of $\mathscr{A}$, and the morphisms between two such complexes $M^\bullet, N^\bullet$ are the standard chain complex whose $i^{th}$ cohomology calculates maps of complexes from $M^\bullet \rightarrow N^\bullet[i]$, modulo chain homotopy. I.e., we have 
    \begin{equation*}\Hom^i(M^\bullet, N^\bullet) = \underset{j \in \mathbb{Z}}\Pi \hspace{.5mm} \Hom_{\mathscr{A}}(M^j, N^{j+i}),\end{equation*}
    with differential $$d(\phi) = d_M \circ \phi + (-1)^{\lvert \phi \rvert} \phi \circ d_N,$$where $\lvert \phi \rvert$ denotes the degree of $\phi$, and $d_M, d_N$ the differentials on $M$ and $N$, respectively.

    In particular, we have $\on{K}^b\xtx$, the bounded homotopy category of cofree monodromic tilting sheaves in $\cM_\Xi$, as well as its ind-completion $$\on{K}^b\xtx \hookrightarrow \on{Ind}(\on{K}^b\xtx).$$ Similarly, for $\chi \in \Xi$, we have the corresponding full subcategories $\on{K}^b\cT_\chi \hookrightarrow \on{Ind}(\on{K}^b\cT_\chi)$.

\sss{} Optimistically, one might hope that the Hecke category is simply equivalent to the above category, i.e., ask whether there is an equivalence 
$$\on{Ind}(\on{K^b}\cT_\Xi) \overset{?} \simeq \sM_\Xi.$$This cannot be literally true, as for any $\chi \in \Xi$, the monoidal unit $\tau(e)_\chi$ is not compact, whereas the corresponding object of $\Ind(\on{K}^b \cT_\chi)$ is by definition. As we will see, this is essentially the unique obstruction to above na\"ive hope, and may be accounted for exactly as in Corollary \ref{c:compsincofree}.

Namely, note that $\on{Ind}(\on{K}^b\cT_\Xi)$ inherits a monoidal structure from that on $\cT_\Xi$. In particular, if we denote its monoidal unit by $$\delta^{\Xi\mon} \in \cT_\Xi,$$via the right monodromy $\id \simeq - \star \delta_{\Xi\mon}$ our category is enriched over 
\begin{equation} \label{e:linearityunit}\End(\delta^{\Xi\mon}) \simeq \Fun(\frf_\Xi).\end{equation}
Let us utilize this linearity in the following manner. Consider within $\indcoh(\frf_\Xi)$ the idempotent complete stable subcategory generated by the dualizing sheaves on all the components of $\frf_\Xi$: 
$$\langle \omega_{\frf_\chi}, \chi \in \Xi \rangle \subset \indcoh(\frf_\Xi),$$
and similarly for a fixed component $\langle \omega_{\frf_\chi} \rangle \subset \indcoh(\frf_\chi).$ Note these are preserved under $!$-tensor product, i.e., inherit natural symmetric monoidal structures, as do their ind-completions, which we denote by 
$$\wt{\indcoh}(\frf_\Xi) := \Ind( \langle \omega_{\frf_\chi}, \chi \in \Xi \rangle), \quad \quad \wt{\indcoh}(\frf_\chi) := \Ind( \langle \omega_{\frf_\chi} \rangle).$$
\begin{rem} In fact, it is not difficult to shows that, for each $\chi \in \Xi$, $$\langle \omega_{\frf_\chi} \rangle \subset \indcoh(\frf_\chi)$$ agrees with the full subcategory of dualizable objects with respect to the $!$-tensor product, and moreover with the full subcategory consisting of objects whose $!$-restriction to any closed subscheme $Z \subset \frf_\chi$ lies in $\Coh(Z)$.
\end{rem}
By Lemma \ref{l:compsincofree}, we have adjunctions 
\begin{equation} \label{e:indcohadj} \iota_!: \indcoh(\frf_\Xi) \rightleftharpoons \wt{\indcoh}(\frf_\Xi): \iota^!, \quad \text{and} \quad \iota_!: \indcoh(\frf_\chi) \rightleftharpoons \wt{\indcoh}(\frf_\chi): \iota^!,\end{equation}
wherein $\iota_!$ is fully faithful and $\iota^!$ is monoidal. With this notation in hand, the above linearity \eqref{e:linearityunit} in particular endows our categories with actions 
$$\Ind(\on{K}^b \cT_\Xi) \circlearrowleft \wt{\indcoh}(\frf_\Xi) \quad \text{and} \quad \Ind(\on{K}^b \cT_\chi) \circlearrowleft \wt{\indcoh}(\frf_\chi).$$
We now show we obtain the usual Hecke categories via base changing along $\iota^!$.

\begin{thm} \label{t:renorm}  There is a canonical adjunction
\begin{equation} \label{e:adjxi}\iota_!: \cM_\Xi \rightleftharpoons \on{Ind}(\on{K}^b \xtx): \iota^! \end{equation}
wherein $\iota^!$ is monoidal and $\iota_!$ is fully faithful, with essential image the ind-completion of an explicit monoidal subcategory of $\on{K}^b \cT_\Xi$ describable purely in terms of $\cT_\Xi$. 

Moreover, $\iota^!$ factors through an equivalence
\begin{equation} \label{e:basechangexi}\on{Ind}(\on{K}^b \cT_\Xi) \underset{\wt{\indcoh}(\frf_\Xi)} \otimes \indcoh(\frf_\Xi) \simeq \cM_\Xi,\end{equation}
identifying the adjunction \eqref{e:adjxi} with the base change adjunction $\on{id}_{\Ind(\on{K}^b(\cT_\Xi))} \otimes \eqref{e:indcohadj}$.

Similarly, for any $\chi \in \Xi$ the adjunction \eqref{e:adjxi} restricts to an adjunction on the full subcategories
\begin{equation} \label{e:adjchi} \iota_!: \cM_\chi \rightleftharpoons \on{Ind}(\on{K}^b \cT_\chi): \iota^!,\end{equation}
wherein again $\iota^!$ is monoidal, $\iota_!$ is fully faithful, and $\iota^!$ factors through an equivalence
\begin{equation} \label{e:basechangechi}\Ind(\on{K}^b(\cT_\chi)) \underset{\wt{\indcoh}(\frf_\chi)} \otimes \indcoh(\frf_\chi) \simeq \sM_\chi.\end{equation}

\end{thm}

\begin{proof}[Proof of Theorem \ref{t:renorm}] For ease of reading, we break the proof into a sequence of steps and sublemmas.

{\it Step 1.} The monoidal inclusion $\cT_\Xi \rightarrow \cM_\Xi$ prolongs to a fully faithful monoidal functor 
    \begin{equation*}\iota^!: \on{K}^b \cT_\Xi \hookrightarrow \cM_\Xi,\end{equation*}
    for the monoidality, one may use the results of \cite{CD}, Appendix A. By construction, this restricts to a monoidal fully faithful embedding
    $$\iota^!: \on{K}^b \cT_\chi \hookrightarrow \cM_\chi.$$

{\it Step 2.} Recall the indecomposable cofree monodromic tilting sheaves $$\tau(y)_\chi, \quad \quad  y \in \tilW, \chi \in \Xi.$$ For a fixed $y$ and $\chi$, let $\beta \in {}_{\chi}\Omega_{\chi'}$ denote the block containing $y$. We denote by \begin{equation*}\ktc(\leqslant_{\beta} \! y) \hookrightarrow \on{K}^b \cT_\Xi\end{equation*} the full subcategory generated by the objects $\tau(w)_{\chi}$, for $w \in {}_{\chi}\extw^{\beta}_{\chi'}$  satisfying $w \leqslant_{\beta} y$. We define the further full subcategory 
$$i_*: \ktc(<_{\beta} \! y) \hookrightarrow \ktc(\leqslant_{\beta} \! y)$$similarly, where we replace $w \leqslant_{\beta} y$ with $w <_{\beta} y$. Let us denote the {\em a priori} partially defined left and right adjoints to $i_*$ by  $i^*$ and $i^!$, respectively.

\begin{lem} The functors $i^*$ and $i^!$ are defined on all of $\ktc(\leqslant_{\beta} \! y)$. Moreover, we have canonical distinguished triangles
    \begin{align*} i_* i^! \tau(y)_{\chi} \rightarrow \tau(y)_{\chi} \rightarrow \nabla(y)_\chi &\xrightarrow{+1} \\ \Delta(y)_\chi \rightarrow \tau(y)_{\chi} \rightarrow i_* i^* \tau(y)_{\chi} &\xrightarrow{+1}.\end{align*}
    In particular, the (co)standard cofree monodromic sheaves lie in $\on{K}^b \cT_\Xi \hookrightarrow \cM_\Xi$. 
        \label{l:stdcost}\end{lem}
    
    \begin{proof}  We proceed via induction on $\ell_{\beta}(y)$. The $\costd$-flag for $\tau(y)_\chi$ yields in particular an exact sequence
        \begin{equation*}0 \rightarrow K \rightarrow \tau(y)_\chi \rightarrow \nabla(y)_\chi \rightarrow 0,\end{equation*}
        where $K$ admits a finite filtration by $\nabla(z)_\chi$, $z  <_{\beta}  y$ by Lemma \ref{l:supptilt}. By inductive hypothesis, all of those lie in $\ktc(<_{\beta} \! y)$, whence so does $K$. As for any object $\xi$ of $\ktc(<_{\beta} \! y)$, one has $\Hom(\xi, \nabla(y)_\chi) = 0$, it follows that the above exact sequence satisfies the required properties of the triangle \begin{equation*}i_* i^! \tau(y)_{\chi} \rightarrow \tau(y)_{\chi} \rightarrow \nabla(y)_\chi \xrightarrow{+1}.\end{equation*}It moreover follows that $i^!$ is defined on all of $\ktc(\leqslant_{\beta} \! y)$, as $i^!$ is now defined on the generating objects $\tau(z)_\chi$, $z \leqslant_{\beta} y$. 
        
        The argument for $\Delta(y)_\chi$ and $i^*$ is entirely similar. 
    \end{proof}
    
    {\it Step 3.} For $y$ and $\chi$ as above, recall that for any lift $\wt{\chi} \in \mathscr{L}_\chi$ we have two associated compact objects of $\cM_\Xi$, namely $j(y)_{!,\wt{\chi}}$ and  $j(y)_{*, \wt{\chi}},$ cf. Section \ref{ss:compgens}. Recall in addition that the lift $\wt{\chi}$ corresponds to a closed subscheme 
    $$Z_{\wt{\chi}} \hookrightarrow \frf_\chi,$$
    cf. the proof of Lemma \ref{l:compsincofree}. In the following corollary, we implicitly equip $\on{Fun}(Z_{\wt{\chi}})$ with the structure of a $\on{Fun}(\frf_\chi)$-module via this embedding.

    \begin{cor} \label{cor: standard characterization} The full subcategory $\ktc(\leqslant_{\beta} \! y) \hookrightarrow \cM_\Xi$ contains the objects $j(y)_{!, \wt{\chi}}$ and $j(y)_{*, \wt{\chi}}$. 
    
    Moreover, $j(y)_{!, \wt{\chi}}$ may be characterized as the unique object of $\ktc(\leqslant_{\beta} \! y)$, up to unique isomorphism, equipped with identifications\begin{align} \label{e:stdhoms1}&\Hom(j(y)_{!, \wt{\chi}}, \xi) \simeq 0, &\text{for }\xi \in \ktc(<_{\beta} \! y), \text{ and} \\ \label{e:stdhoms2} &\Hom(j(y)_{!, \wt{\chi}}, \std(y)_\chi) \simeq \on{Fun}(Z_{\wt{\chi}}).&\end{align}
    
    Similarly, $j(y)_{*, \wt{\chi}}$ may be characterized as the unique object, up to unique isomorphism, equipped with identifications
        \begin{equation*}\Hom(\xi, j(y)_{*,\wt{\chi}}) = 0, \text{} \xi \in \ktc(<_{\beta} \! y), \quad \text{and} \quad \Hom(j(y)_{*, \wt{\chi}}, \std(y)_\chi) \simeq \on{Fun}(Z_{\wt{\chi}}).\end{equation*}    
    \end{cor}

    \begin{proof} We only explicitly discuss the case of $j(y)_{!, \wt{\chi}}$, with the $*$-extension being similar. 
    
    We saw in Lemma \ref{l:stdcost} that $\ktc(\leqslant_{\beta} \! y)$ contains $\Delta(y)_\chi$. By Lemma \ref{l:compsincofree}, $j(y)_{!, \wt{\chi}}$ lies in the small stable subcategory generated by $\Delta(y)_\chi$, hence in $\ktc(\leqslant_{\beta} \! y)$.

  It remains to verify the characterization of $j(y)_{!, \wt\chi}$ by the listed mapping properties. So, consider an object  $$\eta \in \ktc(\leqslant_{\beta} \! y)$$ equipped with isomorphisms as in \eqref{e:stdhoms1}-\eqref{e:stdhoms2}. 
  
  Using \eqref{e:stdhoms1}, in combination with the block decomposition, i.e. Proposition \ref{blockdec}, it follows that, as an object of $\cM_\Xi$, $\eta$ must be $!$-extended from the stratum $\wt{I}y\wt{I}$. That \eqref{e:stdhoms2} then determines the restriction of $\eta$ to the stratum itself, which we {\em a priori} know lies in the small stable subcategory generated by $\dcm$, follows from Lemma \ref{l:enddcm}. 
    \end{proof}
    
    {\it Step 4.} As a consequence, we may do the following. Denote by $\mathscr{C} \hookrightarrow \ktc$ the full subcategory generated by the objects 
    \begin{equation*}j(w)_{!, \wt{\chi}}, \quad \quad w \in \tilW, \chi \in \Xi, \wt{\chi} \in \mathscr{L}_\chi.\end{equation*}
    Note that $\mathscr{C}$ maps, via $\iota$, isomorphically onto the full subcategory of compact objects of $\cM_{\Xi}$. In particular, $\mathscr{C}$ is preserved by the convolution product on $\ktc$, and hence inherits the structure of a monoidal category. 
    
    If we form its ind-completion $\on{Ind}(\mathscr{C})$, this again inherits the structure of a monoidal category. Moreover, as an immediate consequence of the construction, we have that the composition of monoidal functors $\mathscr{C} \hookrightarrow \ktc \hookrightarrow \cM_{\Xi}$, induces, via the universal property of ind-completion, a monoidal functor
    \begin{equation}\label{eq: equivalence ind tilting}F: \on{Ind}(\mathscr{C}) \rightarrow \cM_{\Xi},\end{equation}which is moreover an equivalence. 
    By construction, the composition 
    \begin{equation*}\cM_\Xi \xrightarrow{F^{-1}} \on{Ind}(\mathscr{C}) \hookrightarrow \on{Ind}(\ktc) \end{equation*}
    provides a left adjoint $\iota_!$ to $\iota^!$, as desired. The assertion adjunction \eqref{e:adjchi} for a fixed $\chi$ and its properties follows from the construction.

{\em Step 5.} Having established the adjunctions \eqref{e:adjxi} and \eqref{e:adjchi}, it remains to explain why the Verdier quotient $\iota^!$ factors through the equivalences \eqref{e:basechangexi} and \eqref{e:basechangechi}. To see this, consider any category $\mathscr{N}$, equipped with a set of not necessarily compact generators $n_\alpha$ and an action 
$$ \mathscr{N} \circlearrowleft \wt{\indcoh}(\frf_\Xi).$$Then in the adjunction 
$$\id \otimes \iota_!: \mathscr{N} \underset{\wt{\indcoh}(\frf_\Xi)} \otimes \indcoh(\frf_\Xi) \rightleftharpoons \mathscr{N} \underset{\wt{\indcoh}(\frf_\Xi)} \otimes \wt{\indcoh}(\frf_\Xi): \id \otimes \iota^!,$$
$\id \otimes \iota_!$ is a fully faithful embedding, with essential image generated under colimits by objects of the form $n_\alpha \otimes \omega_{Z_{\wt{\xi}}}$, for connected closed subschemes $Z_{\wt{\xi}} \subset \frf_{\Xi}$. 

In particular, let us take $\mathscr{N} = \Ind(\on{K}^b\cT_\Xi)$, and the $n_\alpha$ to be the indecomposable tilting sheaves. We then note that, for a fixed $\tau(w)_\chi$, and $Z_{\wt{\xi}}$. we have that $\tau(w)_\chi \otimes \omega_{Z_{\wt{\xi}}}$ vanishes unless $\frf_{\wt{\xi}} \subset \frf_{w^{-1}\chi}$, and in the latter case it admits a finite filtration with successive quotients of the form $$j(x)_{!, w \wt{\xi}}, \quad \quad x \leqslant_\beta w,$$with $j(w)_{!, w \wt{\xi}}$ occuring with multiplicity one. It therefore follows that the essential image of $\id \otimes \iota_!$ is precisely the formerly constructed fully faithful embedding $$\sM_\Xi \subset \Ind(\on{K}^b\cT_\Xi),$$which shows the adjunctions \eqref{e:adjxi} and \eqref{e:basechangexi} indeed coincide. Moreover, the argument applies {\em mutatis mutandis} for fixed $\chi \in \Xi$ to match the adjunctions \eqref{e:adjchi} and \eqref{e:basechangechi}. \end{proof}

\section{Soergel functor}\label{s:Soergel}

In this section, we construct a functor from ${}_{\chi}\cM_{\chi}$ to a certain category of Soergel bimodules. This is achieved by consider the action of ${}_{\chi}\cM_{\chi}$ on an affine Whittaker category ${}_{\phi}\cM_{\chi}$, which will then be ``collapsed" to $D(\wt \AA/(\wt \AA, \chi\mon))$.

\subsection{Affine Whittaker category} \label{ss:AffineWhittaker}
We define a module category ${}_{\phi}\cM_{\chi}$ for ${}_{\chi}\cM_{\chi}$ by imposing an affine version of the Whittaker condition on sheaves on the affine flag variety.  This is a generalization of the Whittaker model constructed in \cite{BY} to the monodromic setting for loop groups. On the other hand, ${}_{\phi}\cM_{\chi}$ can be viewed as a category of automorphic sheaves for $\p^{1}$, and is a variant of the automorphic category considered in \cite{HNY} that gave rise to generalized Kloosterman sheaves.

We consider a twisted loop group $LG$ arising from a quasi-split group $G$ over $K=k\lr{t}$ as in \S\ref{ss:tw loop}, which admits a globalization $\cG$ over $\Gm$ as in \S\ref{ss:glob G}. 

\sss{Loop group at $\infty$}

Let $K_{\infty}=k\lr{t^{-1}}$ be the completed local ring of $\p^{1}$ at $\infty$. Then $\cG|_{\Spec K_{\infty}}$ is a reductive group over $K_{\infty}$ that is isomorphic to $G$ if we identify $K_{\infty}$ with $K$ via $t\mapsto  t^{-1}$. Let $L_{\infty}G$ be the loop group of $\cG|_{\Spec K_{\infty}}$. We have
\begin{equation*}
L_{\infty}G=\GG\lr{t^{-1/e}}^{\mu_{e}}
\end{equation*}
with the same pinned $\mu_{e}$-action on $\GG$ and $\z\in \mu_{e}$ sends $t^{-1/e}\mapsto\z^{-1}t^{-1/e}$. Let $\II_{\infty}\subset \GG\tl{t^{-1/e}}$ be the Iwahori subgroup of $\GG\lr{t^{-1/e}}$ corresponding to $\BB^{\opp}$. Let $I_{\infty}$ be the neutral component of $L_{\infty}G \cap \II_{\infty}$. The Lie algebra $\Lie I_{\infty}$ is the $t^{-1}$-adic completion of the span of $\fra$ and the negative affine root spaces $(L\frg)_{\a}$, $\a\in \Phi_{\aff}^{-}$.

Let $I_{\infty}^{+}$ be the pro-unipotent radical of $I_{\infty}$. Projection to the negative simple affine roots give a surjective homomorphism
\begin{equation*}
I_{\infty}^{+}\sur \prod_{\a_{i}\in S_{\aff}}(L\frg)_{-\a_{i}}=:V_{\infty}
\end{equation*}
whose kernel we denote by $I^{++}_{\infty}$. Let $\phi$ be the composition
\begin{equation}\label{aff gen char}
\phi: I_{\infty}^{+}\sur \prod_{\a_{i}\in S_{\aff}}\Ga\xr{add}\Ga
\end{equation}
where ``add'' means adding all coordinates.

Let 
\begin{equation*}
\cL_{\phi}=\phi^{*}\AS_{\psi}
\end{equation*}
be the pullback rank one character sheaf on $I_{\infty}^{+}$, where $\AS_{\psi}$ is as in \S \ref{SSS: char sheaf on Ga}.

Let $\Bun_{\cG}$ be the moduli stack of $\cG$-bundles over $\p^1$.
Let $\Bun_{\cG}(I^{+}_{0}, I^{++}_{\infty})$ be the moduli stack of $\cG$-torsors over $\p^{1}$ with $I^{+}$-level structure at $0$ and $I^{++}_{\infty}$-level structure at $\infty$. 
We have the uniformization map at $0$
\begin{equation}\label{unif}
LG\to LG/I^{+}\to \Bun_{\cG}(I^{+}_{0}, I^{++}_{\infty})\to \Bun_{\cG}.
\end{equation}

Recall the central extension $\tLG$ we construct in Section \ref{ss: cent ext}. It comes from the pullback $\cL_{\det, V}$ of the determinant line bundle via a homomorphism $\GG \to \GL(V)$. The line bundle $\cL_{\det, V}$ in fact descends to a line bundle on $\Bun_{\cG}$ as follows. Choose a vector bundle $\cV$ over $\p^1$ with $\cG$-action so that its fiber over $\Spec K$ is identified with $V$ with the $G$-action, then $\Bun_{\cG}$ carries a line bundle $\cL_{\det,\cV}$ whose fiber at a $\cG$-torsor $\cE$ is the determinant of the cohomology of the associated vector bundle $\cE_\cV:=\cE\times^\cG_{\p^1}\cV$. It is easy to see that $\cL_{\det, \cV}$ pulls back to $\cL_{\det,V}$ under the uniformization map \eqref{unif}.  

Let $\wt\Bun_{\cG}\to \Bun_{\cG}$ be the $\Gm$-torsor which is the complement of the zero section of $\cL_{\det, \cV}$. Let $\wt \Bun_{\cG}(I^{+}_{0}, I^{++}_{\infty})\to \Bun_{\cG}(I^{+}_{0}, I^{++}_{\infty})$ be the pullback $\Gm$-torsor via the map $\wt \Bun_{\cG}(I^{+}_{0}, I^{++}_{\infty})\to\Bun_{\cG}$. By \eqref{unif}, this $\Gm$-torsor pulls back to the central extension $\tLG$ of $LG$.

The torus  $\wt \AA \cong \wt I/I^{+}$ acts on $\wt\Bun_{\cG}(I^{+}_{0}, I^{++}_{\infty})$ by changing the level structure at $0$. The group $V_{\infty}=I^{+}_{\infty}/I^{++}_{\infty}$ acts on $\Bun_{\cG}(I^{+}_{0}, I^{++}_{\infty})$ by changing the level structure at $\infty$. Since the $V_{\infty}$ action does not change the underlying $\cG$-bundle on $\p^{1}$, and $\cL_{\det}$ is pulled back from $\Bun_{\cG}$, this action canonically lifts to an action on $\wt\Bun_{\cG}(I^{+}_{0}, I^{++}_{\infty})$.

The natural projection
\begin{equation*}
\pi: \wt\Bun_{\cG}(I^{+}_{0}, I^{++}_{\infty})\to \Bun_{\cG}(I_{0}, I^{+}_{\infty})
\end{equation*}
is a $\wt \AA \times V_{\infty}$-torsor.

Let $\cG(\p^1 \bsl \{ 0\})$ be the automorphism group of the trivial $\cG|_{\p^1 \bsl \{ 0\}}$-torsor $\cG |_{\p^1 \bsl \{ 0\}}$. When $G$ is split (so $G\cong (G_{0})_{K}$ for $G_{0}$ split over $k$), $\cG(\p^1 \bsl \{ 0\})\cong G_{0}[t^{-1}]$. The uniformization map \eqref{unif} realizes $\wt\Bun_{\cG}(I^{+}_{0}, I^{++}_{\infty})$ as the quotient
\begin{equation*}
(I_{\infty}^{++}\cap \cG(\p^1 \bsl \{ 0\}))\bs \tLG/I^{+}
\end{equation*}
Similarly,
\begin{equation}\label{double coset}
\wt \Bun_{\cG}(I_{0}, I^{+}_{\infty})\simeq \Gamma^{+}_{\infty}\bs \tLG /I
\end{equation}
where $\Gamma^{+}_{\infty}=I_{\infty}^{+}\cap \cG(\p^1 \bsl \{ 0\})$.

We may also uniformize at $\infty$ to get a map $L_\infty G\to \Bun_{\cG}$. Pulling back $\Gm$-torsor $\wt\Bun_G$ gives a central extension $\wt{L_\infty G}\to L_\infty G$. Then uniformization at $\infty$ gives an isomorphism:
\begin{equation}\label{unif at infty}
\wt \Bun_{\cG}(I^+_{0}, I^{++}_{\infty})\cong I^{++}_{\infty}\bs \wt{ L_{\infty}G}/(I^{+}_{0}\cap \cG(\p^1\bsl \{ \infty\})).
\end{equation}

\begin{defn}[Affine-Whittaker] Under the notations above, define the presentable stable category
\begin{equation*}
{}_{\phi}\cM_{\chi}=D((V_{\infty},\cL_{\phi}\mon)\bs\wt\Bun_{\cG}(I^{+}_{0}, I^{++}_{\infty}) / (\wt \AA, \chi\mon))
\end{equation*}
as the $(V_{\infty}\times\wt\AA, \cL_{\phi}\boxtimes\chi)$-monodromic constructible complexes on $\wt\Bun_{\cG}(I^{+}_{0}, I^{++}_{\infty})$. 
\end{defn}

The Hecke category ${}_{\chi}\cM_{\chi}$ acts on ${}_{\phi}\cM_{\chi}$ by modifications at $0$. To rigorously define the monoidal action, refer to the convolution pattern in \S \ref{sss:sheaf conv}. Namely, using notations there, we let $X=\B I$ and $Y=\B\tLG$ and $W$ be the moduli space of $\cG$-bundles on $\p^1\bsl \{0\}$ with $I_\infty^{++}$-level at $\infty$. Then $W\times_YX=\wt\Bun_{\cG}(I_0^+,I_{\infty}^{++})$ and therefore, $D(\wt\Bun_{\cG}(I_0^+,I_{\infty}^{++}))$ is acted by $D(I_0^+\bs \tLG/I_0^+)$, which induces a right convolution:
\begin{equation*}
(-)\star(-): {}_{\phi}\cM_{\chi}\otimes {}_{\chi}\cM_{\chi}\to {}_{\phi}\cM_{\chi}. 
\end{equation*}
More generally, we have \begin{equation*}
(-)\star(-): {}_{\phi}\cM_{\chi}\otimes {}_{\chi}\cM_{\chi'}\to {}_{\phi}\cM_{\chi'}. 
\end{equation*}

The connected components of $\wt\Bun_{\cG}(I^{+}_{0}, I^{++}_{\infty})$ (resp. $\Bun_{\cG}(I_{0}, I^{+}_{\infty})$) are in natural bijection with that of $\Bun_{\cG}$, hence with $\Om$. For $\om\in \Om$, let  $\wt\Bun^{\om}_{\cG}(I^{+}_{0}, I^{++}_{\infty})$ (resp. $\Bun^{\om}_{\cG}(I_{0}, I^{+}_{\infty})$) be the corresponding component.
Let ${}_{\phi}\cM^{\om}_{\chi}$ be the full subcategory of ${}_{\phi}\cM_{\chi}$ consisting of objects supported on $\wt\Bun^{\om}_{\cG}(I^{+}_{0}, I^{++}_{\infty})$. 

Recall the Birkhoff decomposition says that the double cosets in \eqref{double coset} (i.e.  points in $\Bun_{\cG}(I_{0}, I^{+}_{\infty})$) are in bijection with $\tilW$: for each $w\in \tilW$, any lifting $\dot w$ of it to $N_{LG}(\AA)$ gives a point in $\Bun_{\cG}(I_{0}, I^{+}_{\infty})$ via \eqref{unif}. The point in $\Bun_{\cG}(I_{0}, I^{+}_{\infty})$ represented by $w$ has automorphism group $I_{0}\cap \Ad(w)I_{\infty}^{+}$, which is a unipotent group of dimension $\ell(w)$. Let $\Bun_{\cG}(I_{0}, I^{+}_{\infty})_{w}$ be the locally closed substack whose unique point is $\dot w$; let $\wt\Bun_{\cG}(I^{+}_{0}, I^{++}_{\infty})_{w}=\pi^{-1}(\Bun_{\cG}(I_{0}, I^{+}_{\infty})_{w})$.

For each $\om\in\Om$, viewed as a length zero element in $\tilW$, $\Bun_{\cG}(I_{0}, I^{+}_{\infty})_{\om}$ is the unique point with trivial automorphism in the component $\Bun^{\om}_{\cG}(I_{0}, I^{+}_{\infty})$. Choosing a lifting $\dot \om\in N_{\tLG}(\tI)$, using the uniformization map, we get an open embedding
\begin{equation*}
j_{\dot \om}: V_{\infty} \times \wt \AA\incl \wt\Bun^{\om}_{\cG}(I^{+}_{0}, I^{++}_{\infty})
\end{equation*}
sending $(h,v)$ to the double coset of $v\dot \om h$ in $\Gamma^{++}_{\infty}\bs \tLG /I^{+}$. For $\om=1$ we always set $\dot \om=1$.  Setting the $V_{\infty}$-factor to be $1$, we get a locally closed embedding
\begin{equation*}
i_{\dot \om}: \wt \AA \incl \wt\Bun^{\om}_{\cG}(I^{+}_{0}, I^{++}_{\infty}), \quad \Gamma^{++}_{\infty }h\mapsto \dot \om h.
\end{equation*}

\begin{lemma}\label{l:psi clean}
For each $\om\in\Om$, the functor 
\begin{equation*}
i^{!}_{\dot \om}: {}_{\phi}\cM^{\om}_{\chi}\to D(\wt \AA/ (\wt \AA,\chi\mon))
\end{equation*}
is an equivalence.  The inverse functor is given by
\begin{equation*}
D(\wt \AA /(\wt \AA,\chi\mon))\ni \cF\mapsto j_{\dot \om!}(\cL_{\phi} \boxtimes \cF)\cong j_{\dot \om*}(\cL_{\phi} \boxtimes \cF). 
\end{equation*}
\end{lemma}

\begin{proof}
	For any length non-zero element $w \in \tilW$, there exists a simple affine root $\alpha \in S_{\aff}$ such that $w(\alpha)$ is negative. Therefore the one-dimensional parameter unipotent subgroup corresponding to $-\alpha$ acts trivially on the $w$-stratum. Since $\phi$ is non-trivial on that unipotent subgroup, the stalk of any sheaves on the  $w$-stratum has to vanish. Therefore, taking stalk at $\dot{\omega}$ is an equivalence from ${}_{\phi}\cM^{\om}_{\chi}$ to $(\wt \AA, \chi)$-monodromic objects on the torus $\wt \AA$.  
\end{proof}

In the sequel, we will write
\[
{}_\phi\Delta(\omega)_\chi=j_{\dot{\omega}*}(\cL_\phi\boxtimes\dcm).
\]

\sss{}\label{sss:wt M} 
Recall the group $M\subset L_{\mathrm{pol}}G$ introduced in \S\ref{sss:M}. We now use the same notation $M$ to denote
\begin{equation*}
    M:=N_{\Lpol G}(I_\infty,\AA)^{\red}
\end{equation*}
i.e., we change $I$ in the original definition of $M$ to $I_\infty$ \footnote{It can be shown that this a priori change from $I=I_0$ to $I_\infty$ does not change $M$.}. Using \eqref{unif at infty} we get a map
$M\to \Bun_{\cG}(I^+_0, I^{++}_\infty)$. Let $\wt M\to M$ be the central extension obtained by pulling back $\wt{LG}\to LG$. Then we have a map
\begin{equation}\label{M to Bun}
    \wt M\to \wt\Bun_{\cG}(I^+_0, I^{++}_\infty)
\end{equation}
equivariant under the right actions by $\wt\AA$.

\begin{cor}\label{c:describe Mphi} Restriction along the map \eqref{M to Bun} gives an equivalence
\[
{}_{\phi}\cM_{\chi}\cong D(\wt{M}/(\wt\AA,\chi\mon)).
\]
\end{cor}
\begin{proof} Choosing a lifting $\dot \om\in M$ for each $\om\in\Om$, the pullback functor along \eqref{M to Bun} becomes the direct sum of $i^*_{\dot \om}$ in Lemma \ref{l:psi clean} for  $\om\in \Om$, hence an equivalence.
\end{proof}

\subsection{Identifying components of affine Whittaker category}
Our goal is to construct an action of 
${}_{\chi}\cM_{\chi}$ on $D(\wt \AA/(\wt \AA, \chi\mon))$, which will lead to the construction of the Soergel functor. Above we have constructed an action of ${}_{\chi}\cM_{\chi}$ on ${}_{\phi}\cM_{\chi}$, which by Lemma \ref{l:psi clean} is a direct sum of copies of $D(\wt \AA/(\wt \AA, \chi\mon))$ indexed by $\Om$. In this subsection we will identify these copies of $D(\wt \AA/(\wt \AA, \chi\mon))$ in a ${}_{\chi}\cM_{\chi}$-equivariant way.

We now use the subgroup $\Sig\subset \wt M_\phi(k)$ from Construction \ref{cons: Sigma} to construct an action of ${}_{\chi}\cM_{\chi}$ on $D(\wt \AA/(\wt \AA, \chi\mon))$. 

\begin{thm} \label{thm: monodromic-action} 
\begin{enumerate}
    \item For $\chi\in \Ch(\wt\AA)$, there is a canonical action of ${}_{\chi}\cM_{\chi}$ on $D(\wt \AA/(\wt \AA, \chi\mon))$.
    
    \item Let $\Xi\subset \Ch(\wt\AA)$ be a $\tilW$-orbit. There is a canonical action of $\cM_{\Xi}$ on $\op_{\chi\in \Xi}D(\wt \AA/(\wt \AA, \chi\mon))$. 
\end{enumerate}
\end{thm}

\begin{proof}
    For ease of reading, we break the proof into a sequence of steps and sublemmas. 

    {\it Step 1.} We first construct a monoidal functor 
\begin{equation*}
\a: D(\Sig_{1}\bs\Sig/\Sig_{1})\to \Modu_E.
\end{equation*}

Recall from \S \ref{sss:sheaf conv} that for $\cF_{1}, \cF_{2}\in D(\Sig_{1}\bs \Sig/\Sig_{1})$, $\cF_{1}\star\cF_{2}$ is defined to be $m_{*}(p_{1}^{!}\cF_{1}\ot^! p_{2}^{!}\cF_{2})$ using the following diagram (where $m$ is induced by the multiplication on $\Sig$, $p_{i}$ is the projection to the $i$th factor)
\begin{equation*}
\xymatrix{ &\Sig_{1}\bs \Sig\times^{\Sig_{1}}\Sig/\Sig_{1}\ar[rr]^-{m} \ar[dl]_{p_{1}}\ar[dr]^{p_{2}} && \Sig_{1}\bs\Sig/\Sig_{1}\\
\Sig_{1}\bs \Sig/\Sig_{1} && \Sig_{1}\bs \Sig/\Sig_{1}
}
\end{equation*}
Consider the forgetful functors
\begin{equation*}
D(\Sig_{1}\bs \Sig/\Sig_{1})\xr{\om_{1}} D(\Sig/\Sig_{1})=D(\Om)\xr{\om_{2}} \Modu_E
\end{equation*}
where $\om_{1}$ forgets the left $\Sig_{1}$-structure. Using that $\Sig_{1}$ is normal in $\Sig$, it acts trivially on $\Om=\Sig/\Sig_{1}$, hence $(p_{1},p_{2})$ give a bijection
\begin{equation*}
(p_{1},p_{2}): \Sig\times^{\Sig_{1}}\Sig/\Sig_{1}\isom \Om\times \Om.
\end{equation*}
Using this one checks that $\om_{1}$ is monoidal. The functor $\om_{2}$ is push-forward to a point, which is clearly monoidal.

{\it Step 2. } Recall the definition of $\phi$ in \eqref{aff gen char}. As $\Sig$ act trivial on $I^{+}/\ker\phi\stackrel{\phi}{\simeq}\GG_a$, we may form the category
\[
{}_\phi\cM'_\chi:=D(\Sig_1\bs ((\GG_a,\AS_\psi \mon)\bs  \wt\Bun_{\cG}(I_0^+,\ker\phi)/(\wt\AA,\chi\mon))),
\]
which by the convolution pattern is acted by $D(\Sig_1\bs\Sig/\Sig_1)$ from the left, commuting the action of ${}_\chi\cM_{\chi}$ from the right. From Step 1, we have a canonical monoidal functor $ \alpha : D(\Sig_1\bs\Sig/\Sig_1)\to \Modu_E$. Therefore, we can form
\begin{equation} \label{eq: chi m chi action}
    \Modu_E\otimes_{D(\Sig_1\bs\Sig/\Sig_1)}{}_\phi\cM'_\chi,
\end{equation}
on which ${}_\chi\cM_{\chi}$-acts.

Note that we can write
\[\wt\Bun_{\cG}(I_0^+,\ker\phi)=\Sig\times^{\Sig_1}\wt\Bun^1_{\cG}(I_0^+,\ker\phi),\]
so 
$$
    \Sig_1\bs \wt\Bun_{\cG}(I_0^+,\ker\phi)= \Sig_1\bs\Sig\times^{\Sig_1}\wt\Bun^1_{\cG}(I_0^+,\ker\phi).
$$
Therefore, as $D(\Sig_1\bs \Sig/\Sig_1)\otimes {}_\chi\cM(e)_{\chi}= D(\Sig_1\bs \Sig/\Sig_1)\otimes D_{\chi\mon}(\wt\AA)$-modules, we may rewrite
\[
{}_\phi\cM'_\chi\simeq D(\Sig_1\bs \Sig/\Sig_1)\otimes_{D(\BB \Sig_1)} D(\Sig_1\bs ((\GG_a,\AS_\psi \mon)\bs  \wt\Bun^{1}_{\cG}(I_0^+,\ker\phi)/(\wt\AA,\chi\mon))).
\]
Here $\BB\Sig_1=\Sig_1\bs\Sig_1/\Sig_1\subset \Sig_1\bs\Sig/\Sig_1$ is the relative diagonal (so pushforward is monoidal). We also write
\[
{}_\phi\cM^{'1}_\chi := D(\Sig_1\bs ((\GG_a,\AS_\psi \mon)\bs  \wt\Bun^{1}_{\cG}(I_0^+,\ker\phi)/(\wt\AA,\chi\mon)))\simeq D(\Sig_1\bs \wt\AA/(\wt\AA,\chi\mon)).
\]

Combining with \eqref{eq: chi m chi action}, we see that there is a canonical action of ${}_\chi\cM_{\chi}$ on 
\[
\Modu_E\otimes_{D(\BB\Sig_1)} D(\Sig_1\bs \wt\AA/(\wt\AA,\chi\mon)).
\]
In addition, there is a canonical equivalence
\[
\Modu_E\otimes_{D(\BB\Sig_1)}D(\Sig_1\bs \wt\AA/(\wt\AA,\chi\mon))\simeq D(\wt\AA/(\wt\AA,\chi\mon)),
\]
as $D_{\chi\mon}(\wt\AA)$-modules. To see this, we first argue as in \cite[Proposition 4.100]{Zh} (using \cite[Corollary 8.72]{Zh}) that there is a fully faithful embedding
\begin{equation} \label{eq: sigma 1 equivariant}
\Modu_E\otimes_{D(\BB\Sig_1)}D(\Sig_1\bs \wt\AA)=D(\Sig_1/\Sig_1)\otimes_{D(\Sig_1\bs \Sig_1/\Sig_1)} D(\Sig_1\bs \wt\AA)\to D(\wt\AA).
\end{equation}

This is clearly $D(\wt\AA)$-linear, with respect to the right convolution action. From Section \ref{sss: chi-monodromic category}, we know that the inclusion $D_{\chi\mon}(\wt\AA)\subset D(\wt\AA)$ is $D(\wt\AA)$-linear with a $D(\wt\AA)$-linear right adjoint. Therefore, by applying $(-)\otimes_{D(\wt\AA)}D_{\chi\mon}(\wt\AA)$ to the fully faithful embedding \eqref{eq: sigma 1 equivariant}, we obtain a fully faithful embedding 
\begin{equation} \label{eq: sigma 1 full faithful}   \Modu_E\otimes_{D(\BB\Sig_1)}D(\Sig_1\bs \wt\AA/(\wt\AA,\chi\mon))\to D(\wt\AA/(\wt\AA,\chi\mon)).
\end{equation}
To see that it is essential surjective, it is enough to show that its right adjoint is conservative. Notice that the composition of the functor $D(\Sig_1\bs \wt\AA/(\wt\AA,\chi\mon))\to \Modu_E\otimes_{D(\BB\Sig_1)}D(\Sig_1\bs \wt\AA/(\wt\AA,\chi\mon))$ with \eqref{eq: sigma 1 full faithful} is nothing but the $!$-pullback functor, so its right adjoint is given by the $!$-pushforward (equivalently $*$-pushforward by Proposition \ref{lem: functoriality depth zero geom Langlands for tori}). Thus it is enough to show that $*$-pushforward $D_{\mono}(\wt\AA)\to D_{\mono}(\Sig_1\bs \wt\AA)$ is conservative, which is clear.

Therefore, we have constructed a canonical action of ${}_{\chi}\cM_{\chi}$ on $D(\wt \AA/(\wt \AA, \chi\mon))$.

{\it Step 3.} For the statement concerning $\Xi$, we define

\[
{}_\phi\cM'_{\Xi}:= \bigoplus_{\chi \in \Xi} D(\Sig_1\bs ((\GG_a,\AS_\psi \mon)\bs  \wt\Bun_{\cG}(I_0^+,\ker\phi)/(\wt\AA,\chi\mon))).
\]
Note that ${}_\phi\cM'_{\Xi}$ carries a $D(\Sigma_1 \bsl \Sigma / \Sigma_1)$-action on the left and a $\cM_{\Xi}$-action on the right. The argument in Step $2$ implies that there is a canonical action of $\cM_{\Xi}$ on $\op_{\chi\in \Xi}D(\wt \AA/(\wt \AA, \chi\mon))$. 
\end{proof}

\subsection{Soergel functor} \label{ss:soergfunctor}

In the previous subsection, we constructed an action of $ \cM_\Xi$ on $$\underset{\chi \in \Xi} \oplus \hspace{.5mm} D(\wt{\mathbb{A}} / (\wt{\mathbb{A}}, \chi\mon)).$$For later use, we note some properties  and reformulations of this action.

\begin{enumerate}

    \item We first note that this module category, which arose as a modification of a nondegenerate affine Whittaker category, by construction admits a $\cM_\Xi$-equivariant insertion functor from the unmodified category
\begin{equation} \alpha_{\Xi}: \label{e:insertionofwhit} \underset{\chi \in \Xi} \oplus \hspace{.5mm} {}_\phi\cM_\chi \rightarrow \underset{\chi \in \Xi} \oplus \hspace{.5mm} D(\wt{\AA} / (\wt{\AA}, \chi\mon)).\end{equation} For each individual $\chi$, we have a ${}_{\chi}\cM_{\chi}$-equivariant functor $$ \alpha_{\chi}: {}_\phi\cM_\chi \rightarrow D(\wt{\AA} / (\wt{\AA}, \chi\mon)). $$

From the construction of $\alpha_{\chi}$, we know that for any $\om \in \Om$, the composition \begin{equation} \label{eq: Soergel average is identity}
    D(\wt \AA/ (\wt \AA,\chi\mon)) \ra {}_\phi\cM^{\om}_\chi \ra {}_\phi\cM_\chi \xra{\alpha_{\chi}} D(\wt \AA/ (\wt \AA,\chi\mon))
\end{equation} is isomorphic to the identity functor, where the first functor is given in Lemma \ref{l:psi clean}.

\item Second, by summing over all $\Xi$, we therefore obtain an action of $\cM$ on 
$D_{\on{mon}}(\wt{\mathbb{A}})$. By Corollary  \ref{cor: Obtaining Soergel}, this is equivalent to the data of a monoidal functor 
$$\mathbb{V}: \cM  \rightarrow \indcoh(\mathrm{LS}_{\wt{\mathbb{A}}^\vee}^{t,\Box}\times \mathrm{LS}_{\wt{\mathbb{A}}^\vee}^{t,\Box}).$$

\item Recall in addition that to each monodromy $\chi$ we attached a formal subscheme $$\frf_\chi \hookrightarrow \on{LS}^{t, \square}_{\wt{\mathbb{A}^\vee}},$$cf. Section \ref{s:tiltbasics}. By construction, each composite
$$\cmc \hookrightarrow \cM \rightarrow \indcoh(\mathrm{LS}_{\wt{\mathbb{A}}^\vee}^{t,\Box}\times \mathrm{LS}_{\wt{\mathbb{A}}^\vee}^{t,\Box})$$
factors through a functor 
$$\cmc \xrightarrow{ \cvc} \indcoh(\frf_\chi \times \frf_{\chi'}) \hookrightarrow \indcoh(\mathrm{LS}_{\wt{\mathbb{A}}^\vee}^{t,\Box}\times \mathrm{LS}_{\wt{\mathbb{A}}^\vee}^{t,\Box})$$
where the latter fully faithful embedding is given by $*$-pushforward. Equivalently, after summing over all $\chi, \chi' \in \Xi$, we obtain a monoidal functor
$$\mathbb{V}_\Xi: \cM_\Xi \rightarrow \underset{\Xi \times \Xi} \oplus \hspace{.5mm} \indcoh(\frf_\chi \times \frf_{\chi'}).$$

By restricting the functor ${}_{\chi} \mathbb{V} _{\chi}$ to the subcategory ${}_{\chi} \cM _{\chi} ^{\circ}$, we obtain a monoidal functor $${}_{\chi}\mathbb{V} _{\chi}^{\circ} : {}_{\chi} \cM _{\chi} ^{\circ} \rightarrow \indcoh(\frf_\chi \times \frf_\chi).$$

\end{enumerate}

\begin{rem} As a special case, if $\chi$ is $\fre$-linear, as opposed to merely linear over a finite extension $\fre'$ of $\fre$, under the translation isomorphisms $\frf_{\chi} \simeq \frf \simeq \frf_{\chi'}$, we equivalently obtain functors 
\begin{equation} \label{e:vblock} \cvc: \cmc \rightarrow \indcoh(\frf^2),\end{equation}which carry a datum of monoidality for $\chi = \chi'$. Similarly, summing over the orbit,  we obtain a monoidal functor
$$\mathbb{V}_\Xi: \cM_\Xi \rightarrow \underset{\Xi \times \Xi} \oplus \hspace{.5mm} \indcoh(\frf^2).$$
\end{rem}

\section{Equivalence with Soergel bimodules}

In this section, we use the Soergel functor constructed in the previous section to give the promised combinatorial description of monodromic Hecke categories, and explain how to use it to obtain endoscopic equivalences between different Hecke categories. 

This section consists of two main parts. First, we will argue that the Soergel functor behaves favorably on tilting sheaves, identifying them in Theorem \ref{t:soergequivtitSEMIDIRECT} with an appropriately defined version of Soergel bimodules in the present setting. Then, in Theorem \ref{t:mainthmsec7}, we will state the desired consequence for endoscopy. 

\subsection{Tilting sheaves and Soergel bimodules I: standard sheaves} 

\sss{} Our first task in this section is to describe some explicit objects of interest within $\I(\on{LS}^{t, \square}_{\wt{\AA}^\vee})$; these will generate the relevant category of Soergel bimodules. As usual, we will therefore need two basic ingredients, namely standard objects and Bott--Samelson objects. 

In this subsection, we deal with the standard objects, and deduce some first consequences for the $\mathbb{V}$-functor; its sequel will handle the Bott--Samelson objects. 

\sss{} Let us now give the definition of the standard objects. To do so, note that the action of $\tilW$ on $\wt{\mathbb{A}}$ induces a dual action on $\wt{\mathbb{A}}^\vee$, and whence on $\on{LS}_{\wt{\mathbb{A}}^\vee}^{t, \square}$. It is straightforward to see the latter action preserves the closed sub ind-scheme $$\frf_\Xi := \underset{\chi \in \Xi}{\sqcup} \hspace{.5mm} \frf_\chi \hookrightarrow \on{LS}_{\wt{\mathbb{A}}^\vee}^{t, \square}.$$

In particular, to a fixed $w \in \tilW$ and $\chi \in \Xi$, we may consider the graph 
\begin{equation} \label{gammaw}\Gamma^w: \frf_{w^{-1}\chi} \rightarrow \frf_{\chi} \times \frf_{w^{-1}\chi}, \quad \quad x \mapsto (w(x), x).\end{equation}
\begin{defn}We define the {\em standard sheaf} associated to $w$ and $\chi$ to be $$\omega(w)_\chi := \Gamma^{w}_*(\omega_{\frf_{w^{-1}\chi}}) \in \indcoh(\frf_{\chi} \times \frf_{w^{-1}\chi}).$$  
\end{defn} 

\sss{}  The relevance of standard sheaves to affine Hecke categories is given by the following proposition. To state it, by a mild abuse of notation let us write $\on{Fun}(\Gamma^w)$ for $\on{Fun}(\frf_{w^{-1}\chi})$, thought of as a module for $\Fun(\frf_\chi \times \frf_{w^{-1}\chi})$ via the closed embedding \eqref{gammaw}.

\begin{prop} \label{p:std2std} For $w \in \cwc$, the $\mathbb{V}$-functor \eqref{e:vblock} satisfies
$$\cvc(\std(w)_\chi) \simeq \cvc(\costd(w)_\chi) \simeq \omega(w)_\chi,$$
i.e., sends (co)standard sheaves to standard sheaves. Moreover, it induces isomorphisms
$$\on{End}(\std(w)_\chi) \simeq \on{End}(\costd(w)_\chi) \simeq \tau^{\leqslant 0} \End(\omega(w)_\chi) \simeq \on{Fun}(\Gamma^w).$$
\end{prop} 

\begin{rem}\label{r:canonicalstd} We emphasize that, while $\std(w)_\chi$ and $\costd(w)_\chi$ were previously only specified up to non-unique isomorphism, i.e., were defined as non-contractible groupoids of objects,  Proposition \ref{p:std2std} now gives a canonical rigidification. Namely, an object $\Delta(w)_\chi$ equipped with an isomorphism $$\iota: \cvc(\std(w)_\chi) \simeq \omega(w)_\chi$$is now defined uniquely to unique isomorphism.     
\end{rem}

\begin{rem}
We also remark that, summing over all residual characters $\chi$ and $\chi'$, an equivalent formulation of the proposition is that $\mathbb{V}(\Delta(w))$ is the dualizing sheaf on the graph of $w$: $$\Gamma^w: \on{LS}_{\wt{\mathbb{A}}^\vee}^{t, \square}\rightarrow \on{LS}_{\wt{\mathbb{A}}^\vee}^{t, \square} \times \on{LS}_{\wt{\mathbb{A}}^\vee}^{t, \square}.$$
\end{rem}

\begin{proof}[Proof of Proposition \ref{p:std2std}.] Recall the action map 
$$- \star -: {}_\phi\cM_\chi \otimes \cmc \rightarrow {}_\phi\cM_{\chi'}.$$
Recall in addition, that if we write $\omega \in \Omega \simeq \pi_0(\wt{LG})$ for the component supporting $\Delta(w)_\chi$, the above action map restricts to 
$$- \star -: {}_\phi\cM_\chi^e \otimes \cmc^\omega \rightarrow {}_\phi\cM_{\chi'}^\omega.$$
Note that, if we consider the functors  
$$\alpha_\chi: {}_\phi \sM_\chi \rightarrow D(\wt{\AA} / (\wt{\AA}, \chi\mon)) \quad \text{and} \quad \alpha_{\chi'}: {}_\phi \sM_{\chi'} \rightarrow D(\wt{\AA} / (\wt{\AA}, \chi'\mon)),$$cf. Equation \eqref{e:insertionofwhit}, they are $\indcoh(\frf_\chi)$-linear and $\indcoh(\frf_{\chi'})$-linear, respectively. Consider the commutative diagram 
$$\xymatrix{ {}_\phi \cM^e_\chi \ar[d] \ar[r]^{- \star \Delta(w)_\chi} & {}_\phi \cM^\omega_{\chi'} \ar[d] \\ {}_\phi \cM_\chi \ar[d]^{\alpha_\chi}  & {}_\phi \cM_{\chi'} \ar[d]^{\alpha_{\chi'}}  \\ D(\wt{\AA} / (\wt{\AA}, \chi\mon)) \ar[d]^\sim & D(\wt{\AA} / (\wt{\AA}, \chi'\mon)) \ar[d]^\sim\\  \indcoh(\frf_\chi) \ar[r]^{\cvc(\Delta(w)_\chi)} & \indcoh(\frf_{\chi'}).}$$
As in the previous discussion, the left and right vertical compositions are  $\indcoh(\frf_\chi)$-linear and $\indcoh(\frf_{\chi'})$-linear equivalences, respectively. Therefore, to prove the proposition it is enough to show that the upper horizontal arrow corresponds, up to isomorphism, to the integral kernel $\omega(w)_\chi$.  

Using Lemma \ref{multi} and the monoidality of $\mathbb{V}$, it is straightforward to reduce to the case of $w$ being length zero or a simple reflection. If $w$ is length zero, then the functor $- \star \std(w)_{\chi}$ sends ${}_\phi\Delta(e)_\chi \in {}_\phi \cM^e_\chi $ to ${}_\phi\Delta(w)_\chi \in {}_\phi \cM^{\om}_\chi$. Moreover, the induced map $$\Fun( \frf_{\chi})\simeq \End({}_\phi\Delta(e)_\chi) \ra \End({}_\phi\Delta(e)_\chi \star \std(w)_{\chi}) \simeq \Fun( \frf_{\chi'})$$ is given by the $w$-twist. The claim now follows from Corollary \ref{c:compsincofree} and \eqref{eq: Soergel average is identity}. Finally, if $w = s$ is a simple reflection, this follows from Lemma \ref{c:std=graph} below.
\end{proof}

\begin{lem}\label{c:std=graph}
The composition 
$$\indcoh(\frf_\chi) \simeq  {}_\phi \sM_\chi^e \xrightarrow{ - \star \Delta(s)_\chi} {}_\phi \sM_{s\chi}^e \simeq \indcoh(\frf_{s\chi})$$
identifies, up to isomorphism, with the equivalence given by $!$-tensoring with the integral kernel
$$\Gamma^s_*(\omega_{\frf_\chi}) \in \indcoh(\frf_{s\chi} \times \frf_{\chi}),$$
where $\Gamma^s: \frf_\chi \rightarrow \frf_{s \chi} \times \frf_{\chi}$ denotes the graph of $\on{Ad}_s: \frf_\chi \simeq \frf_{s\chi}$.     
\end{lem}
\begin{proof}
Let $\alpha_s$ be the simple affine root associated to $s$, and $U_{-\alpha_s}$ the root subgroup for $-\alpha_s$.
  Note that we have the following commutative diagram 
    \[
    \xymatrix{
{}_{\phi}\cM_{\chi}^{e}    \ar^-{\star \nabla(s)}[r]\ar^\simeq[d]& {}_{\phi}\cM_{s\chi}^{e}\ar^\simeq[d]\\
   D((U_{-\alpha_s},\AS_\psi\mon)\bs\tLG_{\le s}/(I,\chi\mon))\ar^-{\star \nabla(s)}[r]& D((U_{-\alpha_s},\AS_\psi\mon)\bs\tLG_{\le s}/(I,s\chi\mon)) }
    \]
    with the vertical map given by $!$-pullback along the map $\tLG_{\le s}/I^+\to \tLG/I^+\to \wt\Bun^{\om}_{\cG}(I^{+}_{0}, I^{++}_{\infty})$. 
    
    Using this commutative diagram, we note that \eqref{eq: shrek stalk of average of AS} implies that for $\cF\in D(\wt\AA/(\wt\AA,\chi\mon))$
    $$
    j_{e!}(\cL_\phi\boxtimes\cF)\star \nabla(s)\simeq j_{e!}(\cL_\phi\boxtimes s(\cF))\otimes_{E_\chi} \cG(\halpha_s^{\vee}(\chi^{-1}),\psi^{-1}),
    $$
where $j_{e!}(\cL_\phi\boxtimes\cF)$ is as in
    Lemma \ref{l:psi clean}. A choice of an isomorphism $\cG(\halpha_s^{\vee}(\chi^{-1}),\psi^{-1})\simeq E_\chi$ gives the desired statement of the lemma.
\end{proof}

\sss{} Let us note a first useful consequence of Proposition \ref{p:std2std}. To state it, recall the monoidal categories  of tilting sheaves $\xtx$ and ${}_\chi\cT_{\chi}^{\circ}$, cf. Corollary \ref{c:xtx}. Recall the notation 
$$\frf_\Xi := \underset{\chi \in \Xi} \sqcup \hspace{.5mm} \frf_\chi.$$

\begin{cor}\label{c: tilt neutral ff}
The composition of monoidal functors $${}_\chi\cT_{\chi} \rightarrow {}_\chi\cM_{\chi} \xra{{}_{\chi}\mathbb{V} _{\chi}}   \indcoh(\frf_{\chi} \times \frf_\chi) $$lands in the abelian category of ordinary coherent sheaves on $\frf_\chi\times\frf_\chi$, and the obtained functor restricts to a fully faithful embedding
$$
{}_{\chi}\mathbb{V} _{\chi}^{\circ}: {}_\chi\cT_{\chi}^{\circ} \hookrightarrow \indcoh(\frf_{\chi}^2)^\heartsuit.
$$ 
\end{cor}

\begin{proof}
    That tilting objects land in the heart follows from their $\std$-filtrations and Proposition \ref{p:std2std}. 

To see the asserted fully faithfulness, by the monoidality of the functor, and the semi-rigidity of ${}_\chi\cT_{\chi}^{\circ}$, cf. Remark \ref{r:units},  it suffices to check for any object $\tau \in {}_\chi\cT_{\chi}^{\circ}$ that the map 
$$\Hom_{{}_\chi\cT_{\chi}^{\circ}}( \dcm, \tau) \rightarrow \Hom_{\indcoh(\frf_{\chi}^2)^\heartsuit}( {}_{\chi}\mathbb{V} _{\chi}^{\circ}(\dcm), {}_{\chi}\mathbb{V} _{\chi}^{\circ}(\tau))$$
is an equivalence. As $\tau$ is filtered by cofree monodromic costandard sheaves, it is enough to see, for any $w \in {}_{\chi}\tilW_{\chi}^{\circ}$, that the map 
\begin{equation} \label{kicks}\Hom_{{}_\chi\cM_{\chi}^{\circ}}( \dcm , \costd(w)_{\chi}) \xrightarrow{} \Hom_{\indcoh(\frf_{\chi}^2)^\heartsuit}( {}_{\chi}\mathbb{V} _{\chi}^{\circ}(\dcm), {}_{\chi}\mathbb{V} _{\chi}^{\circ}(\costd(w)_{\chi}))\end{equation}
is an equivalence. For $w = e$, \eqref{kicks} is an equivalence follows from Proposition \ref{p:std2std}. For $w \neq e$, it suffices to again argue that both sides vanish. Indeed, this is clear on the left-hand side of \eqref{kicks}. 

On the right-hand side of \eqref{kicks}, we first note that $w$ acts on $\on{Fun}(\frf_\chi)$ by an automorphism distinct from the identity. Indeed, $W_{\aff}$ acts faithfully on $\wt{\mathbb{A}}$, thanks to Proposition \ref{p: affine weyl faithful}, whence on $\wt{\mathbb{A}}^\vee$, and so, by the connectedness of $\wt{\mathbb{A}}^\vee$, the claim follows by passing to algebra of functions on the formal completion $\frz_\chi$, and recalling that $\Fun(\frf_\chi) \simeq \Fun(\frz_\chi)$. 

Let us now show that any element of $\Hom_{\indcoh(\frf_\chi^2)^\heartsuit}( \omega(e), \omega(w) )$ is necessarily zero. Indeed, consider the second projection $$\pi: \frf_\chi^2 \rightarrow \frf_\chi,$$and the associated injection 
\begin{align*}\Hom_{\indcoh(\frf_{\chi}^2)^\heartsuit}( \omega(e)_\chi, \omega(w)_\chi )  \hookrightarrow  &\Hom_{\indcoh(\frf_\chi)^\heartsuit}( \pi_*(\omega(e)_\chi), \pi_*(\omega(w)_\chi)) \\ \simeq &\Hom_{\indcoh(\frf_\chi)^\heartsuit}( \omega_{\frf_\chi}, \omega_{\frf_\chi}) \\ \simeq &\on{Fun}(\frf_\chi). \end{align*}
By considering the first projection as well, it follows the image  in $\on{Fun}(\frf_\chi)$ lies in the subspace of elements $\phi$ annihilated by functions of the form $f - w \cdot f, f \in \Fun(\frf_\chi)$. However, recalling from above that the latter subspace of functions is nonzero, and that $\on{Fun}(\frf_\chi)$ is a domain,   it follows the image is zero, as desired.   
\end{proof}

Another corollary we need is the following.

\begin{cor}\label{cor: relating to classical SBim}
    The composition of monoidal functors $${}_\chi\cT_{\chi} \rightarrow {}_\chi\cM_{\chi} \xra{{}_{\chi}\mathbb{V} _{\chi}}   \indcoh(\frf_{\chi} \times \frf_\chi) $$lands in $\langle\omega_{\frf_\chi\times\frf_\chi}\rangle$, the idempotent complete subcategory of $\indcoh(\frf_{\chi} \times \frf_\chi)$ generated by $\omega_{\frf_\chi\times\frf_\chi}$. Under the equivalence
    $$
    \Hom(\omega_{\frf_\chi\times\frf_\chi},-): \langle\omega_{\frf_\chi\times\frf_\chi}\rangle\simeq \Perf_{\Fun(\frf_\chi\times\frf_\chi)},
    $$
    where $\Perf_{\Fun(\frf_\chi\times\frf_\chi)}$ denotes the category of perfect $\Fun(\frf_\chi\times\frf_\chi)$-modules, the object
    ${}_{\chi}\mathbb{V} _{\chi}(\tau)$ corresponds to the $E$-module $\Hom({}_\phi\Delta(e)_{\chi}, {}_\phi\Delta(e)_{\chi}\star \tau)$, equipped with $\Fun(\frf_\chi\times\frf_\chi)$-module structure coming from the composed maps
    $$
    \Fun(\frf_\chi\times\frf_\chi)\to \End(\tau)\to \Hom({}_\phi\Delta(e)_{\chi}, {}_\phi\Delta(e)_{\chi}\star \tau),
    $$
    where the first map is the monodromy action as in Remark \ref{rem: monodromy action}.
\end{cor}
\begin{proof}
    Since ${}_{\chi}\mathbb{V} _{\chi}(\Delta(w)_\chi)\simeq \omega_{\Gamma^s}$ belonging to $\langle\omega_{\frf_\chi\times\frf_\chi}\rangle$, the first statement follows. 

    For the second statement, first notice that ${}_{\chi}\mathbb{V} _{\chi}$ is $D_{\chi\mon}(\wt\AA)\simeq \indcoh(\frf_\chi)$-bilinear. So it is enough to show that 
    \begin{multline*}
    \Hom(\omega_{\frf_\chi}, (\pr_2)_*((\Delta_{12})^!(\omega_{\frf_\chi}\boxtimes {}_{\chi}\mathbb{V} _{\chi}(\tau))))\\= \Hom(\omega_{\frf_\chi}, (\pr_2)_*({}_{\chi}\mathbb{V} _{\chi}(\tau)))\simeq \Hom(\omega_{\frf_\chi\times\frf_\chi}, {}_{\chi}\mathbb{V} _{\chi}(\tau)).
    \end{multline*}
    Again, this holds if $\tau$ is replaced by (co)standard sheaves, and therefore holds for $\tau$.
\end{proof}

Similarly, we have the following.

\begin{cor} \label{c:tiltff}The composition of monoidal functors $$\xtx \rightarrow \cM_{\Xi} \rightarrow \underset{(\chi', \chi) \in \Xi^2} \bigoplus \hspace{1mm} \indcoh(\frf_{\chi'} \times \frf_\chi) \simeq \indcoh(\frf_\Xi^2)$$lands in the heart, and the obtained functor 
$$\xvx: \xtx \rightarrow \indcoh(\frf_\Xi^2)^\heartsuit$$is faithful. Moreover, if $G$ is split, the functor $\xvx$ is a fully faithful embedding. 
\end{cor}

\begin{proof} These may be proved by the same argument as Corollary \ref{c: tilt neutral ff}, where for the second assertion for untwisted loop groups one uses the faithfulness of the action of $\extw$ on $\wt \AA$, cf. Corollary \ref{c: untwisted faithful}.
\end{proof}

\subsection{Tilting sheaves and Soergel bimodules II: Bott--Samelson sheaves}

\sss{} Corollary \ref{c: tilt neutral ff} provides us a fully faithful embedding
$${}_{\chi}\mathbb{V} _{\chi}^{\circ}: {}_\chi\cT_{\chi}^{\circ} \hookrightarrow \indcoh(\frf_{\chi}^2)^\heartsuit.$$
To characterize the essential image, we need one more ingredient, namely the Bott--Samelson sheaves. We begin with their definition. 

\sss{} For a reflection $r$ in ${}_\chi\tilW_\chi^{\circ}$, consider the subgroup $$W_r := \{1, r\} \subset {}_\chi\tilW_\chi^{\circ},$$and the ind-scheme obtained by taking the invariant theory quotient  
$\frf_{\chi} /\!\!/ W_r.$ Explicitly, we may write $\frf_{\chi}$ as a colimit of (affine) subschemes $Z_\alpha$ preserved by the action of $W_r$, and we then have $$\frf_{\chi} /\!\!/ W_r \simeq \varinjlim Z_\alpha /\!\!/ W_r,$$
where the latter are the usual invariant theory quotients of affine schemes. 

We have a tautological projection $\frf_{\chi} \rightarrow \frf_{\chi} /\!\!/ W_r$, and in particular may consider the fiber product 
$$i: \frf_{\chi} \underset{\frf_{\chi} /\!\!/ W_r} \times \frf_{\chi} \hookrightarrow \frf_{\chi}^2.$$
\begin{defn} 
We define the {\em Bott--Samelson sheaf} associated to $r$ to be $$\beta(r)_\chi := i_*(\omega_{\frf_{\chi} \underset{\frf_{\chi} /\!\!/ W_r} \times \frf_{\chi} }).$$
\end{defn}

\sss{} The relevance of Bott--Samelson sheaves to affine Hecke categories is as follows. 

\begin{prop} \label{p:bs} 
Assume that the characteristic of $\fre$ is not two.
Let $r$ denote a simple reflection of ${}_\chi \tilW_{\chi}^\circ$. Then there is an isomorphism
$${}_\chi\V_\chi(\tau(r)_\chi) \simeq \beta(r)_\chi.$$
\end{prop}

\begin{proof} 
From Proposition \ref{p:conjsimple}, we know that there exists a minimal element $u$ such that $uru^{-1}$ is a simple reflection in both $\extw_{u\chi}^{\circ}$ and $\extw$. 
It follows from Lemma \ref{l:minclean} that the (co)standard objects $\costd(u^{-1})_\chi$ and $\std(u)_{u\chi}$ are clean and tilting. 
From Lemma \ref{l: convmintilt}, we have an isomorphism 
$$\std(u)_{u\chi} \star \tau(r)_\chi \star \costd(u^{-1})_\chi \simeq \tau(uru^{-1})_{u\chi}.$$
Using this and Proposition \ref{p:std2std}, it is now straightforward to see that the assertion of the proposition for $uru^{-1}$ and the twist $u\chi$ implies the assertion for $r$ and the twist $\chi$. 

It therefore remains to prove the proposition for $r=s$ which is moreover a simple reflection in $\tilW$. Using Corollary \ref{cor: relating to classical SBim}, the desired statement now follows from Lemma \ref{lem: endo of taus} and the fact that $\Hom({}_\phi\Delta(e)_\chi, {}_\phi\Delta(e)_\chi\star\tau(s)_\chi)$ is a free $\End(\tau(s)_\chi)$-module of rank one, which in turn follows from Lemma \ref{lem: adjoint of Wh average} below.
\end{proof}

\begin{lemma}\label{lem: adjoint of Wh average}
    The functor 
    $$
    {}_\phi\Delta(e)_\chi \star -: {}_\chi\cM(\leq s)_{\chi}\to {}_\phi\cM^e_{\chi}
    $$
    admits both left and right adjoints, both of which sends ${}_\phi\Delta(e)_\chi$ to $\tau(s)_\chi$.
\end{lemma}
\begin{proof}
    As in the proof of Lemma \ref{c:std=graph}, we may instead consider the adjoint functors of the following
    $$
    {}_\phi\Delta(e)_\chi\star -: D((I,\chi\mon)\bs\tLG_{\le s}/(I,\chi\mon)) \to D((U_{-\alpha_s},\AS_\psi\mon)\bs\tLG_{\le s}/(I,\chi\mon)),
    $$
    where ${}_\phi\Delta(e)_\chi$ now is the $*$-extension of $\AS_\psi\boxtimes\dcm\in D(U_{-\alpha_s}\times\wt\AA)$ along the open embedding $U_{-\alpha_s}\times\wt\AA\subset \tLG_{\leq s}/I^+$. Since $\dcm$ is the unit of $D_{\chi\mon}(\wt\AA)$, the functor in question can be rewritten as
    $$
    \cF\mapsto a_*(\AS_\psi[1]\boxtimes \cF)\simeq \mult_!(\AS_\psi[1]\boxtimes \cF),
    $$
    Here $a: U_{-\alpha_s}\times \tLG_{\le s}/I^+$ is the action map. It follows from the right  adjoint of the functor in question is 
    $$
    D((U_{-\alpha_s},\AS_\psi\mon)\bs\tLG_{\le s}/(I,\chi\mon))\to D((I,\chi\mon)\bs\tLG_{\le s}/(I,\chi\mon)),\quad \cG\mapsto b_!(\cG[1])
    $$
    where $b: U_\alpha\times \tLG_{\le s}/I^+$ is the action map. Similarly the left adjoint is given $\cG\mapsto b_*(\cG[1])$. Since $\cG$ is $(U_{-\alpha_s}\times \wt\AA,\AS_\psi\boxtimes \chi^{-1})$-monodromic, we have
    \[b_!(\cG[1])\simeq {}_\chi\Delta(e)_{\psi}\star \cG,\quad b_*(\cG[1])\simeq {}_\chi\nabla(e)_{\psi}\star \cG,\]
    where ${}_\chi\Delta(e)_{\psi}\in D((I,\chi\mon)\bs \tLG_{\le s}/(U_{-\alpha_s},\AS_\psi\mon))$ (resp. ${}_\chi\nabla(e)_{\psi}$) is the $!$-extension (resp. $*$-extension) of $\dcm\boxtimes \AS_{\psi}[1]$ along $\wt\AA\times U_{-\alpha_s}\hookrightarrow I^+\bs \tLG_{\le s}$. Now the canonical isomorphism ${}_\chi\nabla(e)_{\psi}\simeq {}_\chi\Delta(e)_{\psi}$ gives the desired isomorphism between left and the right adjoint functors. We simply denote the obtained functor by $\Av$. Note that written as convolution with ${}_\chi\Delta(e)_{\psi}$, the functor is clearly ${}_{\chi}\cM(\le s)_{\chi}$-linear. Therefore it remains to show that $\Av({}_\phi\Delta(e)_{\chi})\simeq \tau(s)_\chi$.

    Note that for $w\in\{1,s\}$, we have
    \[    \Hom(\Delta(w)_{\chi},\Av({}_\phi\Delta(e)_{\chi}))\simeq \Hom({}_\phi\Delta(e)_{\chi},{}_\phi\Delta(e)_{\chi})\simeq \Hom({}_\phi\Delta(e)_{\chi}, \nabla(w)_{\chi}).
    \]
    Therefore, $\Av({}_\phi\Delta(e)_{\chi})$ is cofree tilting, supported on $Z_{\le s}$, and therefore is isomorphic to $\tau(s)_\chi$.
\end{proof}

\begin{rem}\label{rem:true BS sheaf}
    If the characteristic of $\fre$ is two and $\ker(\alpha_s^\vee: \GG_m\to \wt\AA^\vee)$ is $\mu_2$, i.e. $\alpha_s$ is twice of a character of $\wt\AA$, then the inclusion $\frf_\chi\cap \ker \halpha_s^\vee\subset \frf_\chi^s=\frf_\chi\cap \ker (\alpha_s^\vee\circ \halpha_s^\vee)$ is strict. The proof of Lemma \ref{lem: endo of taus} suggests that in general, one shall define the Bott--Samelson sheaf as the dualizing sheaf of the ind-scheme $\Gamma^1\sqcup_{\ker \halpha_s^\vee} \Gamma^s$, which is an ind-scheme finite over (but not necessarily closed in) $\frf_\chi\times\frf_\chi$. Then Proposition \ref{p:bs} will hold in all characteristic will this modified definition.
\end{rem}

\sss{} To leverage Proposition \ref{p:bs}, we recall that we have described monoidal generators of $\cT_{\Xi}$, ${}_\chi\cT_\chi$ and ${}_\chi\cT_\chi^{\circ}$ in Lemma \ref{l:tiltgens}. We may now give the definition of Soergel bimodules in the present setting. 

To state it, for a simple reflection $s$ in $\tilW$ and $\chi \in \Xi$, recall that if $s \in {}_\chi\tilW_\chi^\circ$, we have previously defined the corresponding Bott--Samelson sheaf 
$$\beta(s)_\chi \in \indcoh(\frf_\chi^2) \hookrightarrow \indcoh(\frf_{\Xi}^2).$$
If $s \notin {}_\chi\tilW_\chi^\circ$, we will extend the definition by taking the Bott--Samelson sheaf to simply be the standard, i.e., 
$$\beta(s)_\chi := \omega(s)_{\chi} \in \indcoh(\frf_{s\chi} \times \frf_\chi) \hookrightarrow \indcoh(\frf_\Xi^2), \quad \quad s \notin {}_\chi\tilW_\chi^\circ.$$

\sss{} Let us now give the definition of the relevant categories of Soergel bimodules for us; we will defer generalities to a subsequent remark. To state it, we recall that $S_\chi$ denotes the reflections of $\tilW_\chi^\circ.$

\begin{defn} Consider the quadruples of formal schemes equipped with actions of extended Coxeter systems 
$$(\frf_\Xi, \tilW, W_{\aff}, S), \quad  (\frf_\chi, \tilW_\chi, \tilW_\chi^\circ, S_\chi), \quad \text{and} \quad (\frf_\chi, \tilW_\chi^\circ, \tilW_\chi^\circ, S_\chi).$$
We define the associated categories of {\em Soergel bimodules} as follows.

\begin{enumerate}

\item The category of Soergel bimodules associated to $(\frf_\Xi, \tilW, W_{\aff}, S)$
$$\on{SBim}_\Xi \subset \indcoh(\frf^2_\Xi)^\heartsuit$$
is the Karoubian monoidal subcategory generated by the objects 
$$\beta(s)_\chi \quad \text{and} \quad \omega(w)_\chi,$$ 
for $s$ a simple reflection in $\tilW$, $w \in \Omega$, and $\chi \in \Xi$. 

\item The category of Soergel bimodules associated to $(\frf_\chi, \tilW_\chi, \tilW_\chi^\circ, S_\chi)$
$$\on{SBim}_\chi \subset \indcoh(\frf^2_\chi)^\heartsuit$$
is the Karoubian monoidal subcategory generated by the objects 
$$\beta(r)_{\chi} \quad \text{and} \quad \omega(w)_{\chi},$$ 
for $r \in S_\chi$, and $w \in \Omega_\chi$. 

\item The category of Soergel bimodules associated to $ (\frf_\chi, \tilW_\chi^\circ, \tilW_\chi^\circ, S_\chi)$
$$\on{SBim}_\chi^{\circ} \subset \indcoh(\frf^2_\chi)^\heartsuit$$
is the Karoubian monoidal subcategory generated by the objects 
$$\beta(r)_{\chi},$$ 
for $r \in S_\chi$. 
\end{enumerate}
\end{defn}

\begin{rem} Let us in passing place the previous definition in a more general context. Consider an extended Coxeter group $$(W, W^\circ,S),$$ i.e., $(W^\circ, S)$ is a Coxeter group, and $W \simeq \Omega \ltimes W^\circ$, where $\Omega$ is a group acting on $(W^\circ, S)$ by automorphisms of the Coxeter system. 

Consider a torus $A$ equipped with a faithful action of $W$ by automorphisms, such that the $W$-action on the character lattice decomposes as 
$$\X^*(A) \underset{\Z} \otimes \Q \simeq \on{refl}^* \oplus (\X^*(A) \underset{\Z} \otimes \Q)^{W^\circ},$$
where $\on{refl}^*$ denotes the dual reflection representation of $W$, or a quotient thereof on which $W^\circ$ acts faithfully. Fix a $W$-orbit $\mathfrak{o}$ in $\Ch(A; \fre)$, and write 
$$\frf_{\mathfrak{o}} := \underset{\chi \in \mathfrak{o}} \sqcup \hspace{.5mm} \frf_\chi \hookrightarrow \on{LS}_{A^\vee}^{t, \square}$$
for the corresponding formal $E$-scheme, which inherits a natural action of $W$. Given the quadruple 
$$(\frf_{\mathfrak{o}}, W, W^\circ, S),$$
one can define an associated monoidal subcategory $\on{SBim}_{\mathfrak{o}} \subset \indcoh(\frf_{\mathfrak{o}}^2)$ exactly as above. 
\end{rem}

\sss{} With these definitions and results in hand, we are ready to deduce the first main theorem of this section. Recall the monoidal functors $${}_{\chi}\mathbb{V}_\chi^{\circ}: {}_{\chi}\cM_\chi^{\circ} \ra \indcoh(\frf^2_\chi),$$ $$\xvx: \cM_\Xi \rightarrow \indcoh(\frf^2_\Xi) \quad \text{and} \quad {}_{\chi}\mathbb{V}_\chi: {}_{\chi}\cM_\chi \rightarrow \indcoh(\frf^2_\chi).$$ 

\begin{thm}\label{t:soergequivtit} Assume that the characteristic of $\fre$ is not two. The functor ${}_{\chi}\mathbb{V}_\chi^{\circ}$ restricts to monoidal equivalence:  $${}_{\chi}\mathbb{V}_\chi^{\circ}:  {}_{\chi}\cT_\chi^{\circ} \ra \on{SBim}_\chi^{\circ}.$$ 
Moreover, if $G$ is a split reductive group, the functors $\xvx$ and $\mathbb{V}_\chi$ restrict to monoidal equivalences
$$\xvx: \xtx \simeq \on{SBim}_\Xi \quad \text{and} \quad \mathbb{V}_\chi: {}_{\chi}\cT_\chi \simeq \on{SBim}_\chi.$$    
\end{thm}

\begin{proof} We begin with the assertion for ${}_{\chi}\mathbb{V}_\chi^{\circ}$. In view of Corollary \ref{c: tilt neutral ff}, it remains to argue that the essential image of ${}_{\chi}\mathbb{V}_\chi^{\circ}$ is $\on{SBim}_\chi^{\circ}$. However, using the tilting generators from Lemma \ref{l:tiltgens}, this now follows from our calculation of their images in Propositions \ref{p:std2std} and \ref{p:bs}. 

The arguments for $\mathbb{V}_\Xi$ and $\mathbb{V}_\chi$ are similar, where we we replace Corollary \ref{c: tilt neutral ff} with Corollary \ref{c:tiltff}. 
\end{proof}

\sss{} In view of the previous theorem, our final task in this subsection is to give a Soergel theoretic description of all of ${}_\chi\cT_\chi$ which applies to not necessarily split groups, i.e., to circumvent the non-faithfulness of the action of the extended affine Weyl group on the extended torus.

We do so as follows. Let us denote the ordinary monoidal category of finitely generated projective $E$-modules by $\Proj_E$.\footnote{While in the present paper we only explicitly discuss the case of $E$ at most global dimension one, we phrase the discussion below to work equally in the Betti setting with $E$ a complete local Noetherian ring of finite global dimension.} Given an abstract group $\Gamma$, consider the monoidal category $\Proj_{E,\Gamma}$ of finitely supported sheaves on $\Gamma$ of finitely generated projective $E$-modules, with the convolution monoidal structure. 

Explicitly, the objects of $\Proj_{E,\Gamma}$ are finite formal direct sums of objects of the form $\delta_\gamma \otimes P,$ for $\gamma \in \Gamma$ and  $P \in \Proj_E.$ Homomorphisms are determined by the formula
\begin{equation} \label{e:homssemidirect}\Hom(\delta_{\gamma} \otimes P, \delta_{\gamma'} \otimes Q) = \begin{cases} \Hom(P,Q), & \text{if }\gamma = \gamma', \text{ and}\\ 0, & \text{otherwise}. \end{cases},\end{equation}
with the evident composition rule. Finally, the monoidal structure is determined by the formula 
$$(\delta_\gamma \otimes P) \star (\delta_{\gamma'} \otimes Q) := \delta_{\gamma \gamma'} \otimes (P \underset{E} \otimes Q),$$
along with the evident upgrade of the above formula on objects to a functor $\Proj_{E,\Gamma} \times \Proj_{E, \Gamma} \rightarrow \Proj_{E, \Gamma}$, and with the tautological associativity equivalences and unit object. 

 Let $\cM$ be an $E$-linear ordinary monoidal category which is moreover Karoubian, and write $\mathbb{1}$ for its monoidal unit. We now recall the universal property of $\Proj_{E, \Gamma}$, whose proof is immediate.

\begin{lem}\label{l:mapprojg}The data of an $E$-linear monoidal functor $\Proj_{E, \Gamma} \rightarrow \cM$ is equivalent to:

\begin{enumerate} \item  For each $\gamma \in \Gamma$, an object $m_\gamma$ of $\cM$, with $m_e = \mathbb{1},$ and

\item for each pair $(\gamma, \gamma') \in \Gamma \times \Gamma$, an isomorphism
$$\alpha_{\gamma, \gamma'}: m_\gamma \star m_{\gamma'} \simeq m_{\gamma \gamma'},$$
which for $\gamma = e$ or $\gamma' = e$ are the unit isomorphisms, such that 

\item for each triple $(\gamma, \gamma', \gamma'') \in \Gamma \times \Gamma \times \Gamma$, we have a commutative diagram 
$$\xymatrix{(m_\gamma \star m_{\gamma'}) \star m_{\gamma''} \ar[d]_\sim \ar[rr]^{\alpha_{\gamma, \gamma'} \star \id} && m_{\gamma \gamma'} \star m_{\gamma''} \ar[rr]^{\alpha_{\gamma \gamma', \gamma''}} && m_{\gamma \gamma' \gamma''} \ar[d]^{\id} \\ m_\gamma \star (m_{\gamma'} \star m_{\gamma''}) \ar[rr]^{\id \star \alpha_{\gamma', \gamma''}} && m_\gamma \star m_{\gamma' \gamma''} \ar[rr]^{\alpha_{\gamma, \gamma' \gamma''}}&& m_{\gamma \gamma' \gamma''},}$$
where the left vertical arrow is the associativity isomorphism of $\sM$. 

\end{enumerate} 
\end{lem}

\sss{} Consider the Karoubian monoidal subcategory of ${}_\chi\cT_\chi$ generated by the minimal tilting sheaves $${}_\chi\cT_\chi(\Omega_\chi) \subset {}_\chi\cT_\chi.$$
Explicitly, this is the full subcategory with objects of the form 
$$\underset{\beta \in \Omega_\chi}\oplus \hspace{.5mm} \tau(w^\beta)_\chi \underset{E} \otimes P_\beta,$$ 
where $\tau(w^\beta)_\chi$ denotes any object in the isomorphism class of indecomposable tilting sheaves with support $w^\beta$, and $P_\beta$ are objects of $\Proj_E$, all but finitely many of which are zero. Noting that one has the equality of isomorphism classes
\begin{equation}\label{e:multgroth}\tau(w^\beta)_\chi \star \tau(w^\gamma)_\chi \simeq \tau(w^{\beta \gamma})_\chi,\end{equation}cf. Lemmas \ref{l:minmult} and \ref{l: convmintilt}, it follows that ${}_\chi\cT_\chi(\Omega_\chi)$ is closed under convolution.  

Note that to lift the identities \eqref{e:multgroth} to a genuine monoidal functor $\Proj_{E, \Gamma} \rightarrow {}_\chi\cT_\chi$, we encounter an obstructing 3-cocycle arising from condition (3) of Lemma \ref{l:mapprojg}.\footnote{Said differently, one could instead encounter a similar convolution category, albeit on a multiplicative gerbe over $\Omega_\chi$.} We now trivialize it using the Soergel functor, as follows.

\begin{defn} For $\beta \in \Omega_\chi$, a {\em Whittaker normalized} tilting sheaf is a pair $(\tau, \square),$ where $\tau$ is an indecomposable tilting sheaf with support $w^\beta$, and $\square$ is an isomorphism 
$$\square: \mathbb{V}_\chi(\tau) \simeq \omega(w^\beta)_\chi.$$
\end{defn}

Note that for fixed $\beta$, Whittaker normalized tilting sheaves form a category in a natural way, where a homomorphism $(\tau, \square) \rightarrow (\tau', \square')$ is a morphism of tilting sheaves $\alpha: \tau \rightarrow \tau'$ which intertwines their Whittaker normalizations, i.e., for which the following diagram commutes
$$\xymatrix{ \mathbb{V}_\chi(\tau) \ar[d]_{\square} \ar[r]^{\mathbb{V}_\chi(\alpha)} & \mathbb{V}_\chi(\tau') \ar[d]^{\square'} \\ \omega(w^\beta)_\chi \ar[r]^{\id} & \omega(w^\beta)_\chi.}$$
Let us denote this category by ${}_\chi\cT_\chi(\beta, \square).$

\begin{lem}\label{l:contracgrpd} For each $\beta \in \Omega_\chi$, ${}_\chi\cT_\chi(\beta, \square)$ is a contractible groupoid.
\end{lem}

\begin{proof}We must show that there is a unique isomorphism class of objects in ${}_\chi\cT_\chi(\beta, \square)$ and the endomorphisms of any object consist only of the identity. However, these assertions are immediate from Proposition \ref{p:std2std}. 
\end{proof}

\sss{} To proceed, we next note that for any pair $(\beta, \beta') \in \Omega_\chi \times \Omega_\chi$, we have a naturally defined functor 
$$- \star -: {}_\chi\cT_\chi(\beta, \square) \times {}_\chi\cT_\chi(\beta', \square) \rightarrow {}_\chi\cT_\chi(\beta \beta', \square),$$
which on objects sends a pair $((\tau, \square), (\tau', \square'))$ to $(\tau \star \tau', \square \star \square')$, where the latter isomorphism is the composition 
$$\mathbb{V}_\chi(\tau \star \tau') \simeq \mathbb{V}_\chi(\tau) \star \mathbb{V}_\chi(\tau') \xrightarrow{ \square \star \square'} \omega(w^\beta)_\chi \star \omega(w^{\beta'})_\chi \simeq \omega(w^{\beta \beta'})_\chi,$$
where the final isomorphism is the tautological identification of the convolution of graphs with the graph of the composition. 

Moreover, by construction, for any triple $(\beta, \beta', \beta'') \in \Omega_\chi \times \Omega_\chi \times \Omega_\chi$, we have a canonical identification of functors
\begin{equation} \label{e:assoc}(- \star -) \star - \simeq - \star (- \star-): {}_\chi\cT_{\chi}(\beta, \square) \times  {}_\chi\cT_{\chi}(\beta', \square) \times  {}_\chi\cT_{\chi}(\beta'', \square) \rightarrow  {}_\chi\cT_{\chi}(\beta \beta' \beta'', \square).\end{equation}
where we in particular use that the tautological identifications of convolution of graphs with graphs of compositions satisfy the cocycle condition of Lemma \ref{l:mapprojg}.  Therefore, they give rise to a monoidal functor $$\square:  \Proj_{E, \Omega_\chi} \rightarrow \on{SBim}_\chi. \quad \quad \beta \mapsto \omega(w^\beta)_\chi.$$

\sss{} With these preparations in hand, we are now ready to trivialize the arising multiplicative gerbe over $\Omega_\chi$. 

\begin{prop} \label{p:rigidaction}There exists a monoidal functor 
$$\blacksquare: \Proj_{E, \Omega_\chi} \rightarrow {}_\chi\cT_\chi$$
together with a natural equivalence of monoidal functors $\iota: \mathbb{V}_\chi \circ \blacksquare \simeq \square$ such that for any $\beta \in \Omega_\chi$, $\square(\beta)$ is an indecomposable tilting sheaf with support $w^\beta$. 

Moreover, the pair $(\blacksquare, \iota)$ satisfying the above property is unique up to unique equivalence. 
\end{prop}

\begin{proof} We begin by constructing a functor $\blacksquare$ using the criterion of Lemma \ref{l:mapprojg}. For each $\beta \in \Omega_\chi$, let us choose an object  $(\tau_\beta, \square_\beta)$ of ${}_\chi\cT_\chi(\beta, \square)$. For any pair $(\beta, \beta') \in \Omega_\chi \times \Omega_\chi$, note by Lemma \ref{l:contracgrpd} we have a unique isomorphism in ${}_\chi\cT_\chi(\beta\beta', \square)$
$$\alpha_{\beta, \beta'}: (\tau_\beta, \square_\beta) \star (\tau_{\beta'}, \square_{\beta'}) \simeq (\tau_{\beta \beta'}, \square_{\beta \beta'}).$$
Moreover, for any triple $(\beta, \beta', \beta'') \in \Omega_\chi \times \Omega_\chi \times \Omega_\chi$, the diagram 
$$\xymatrix{((\tau_\beta, \square_\beta) \star (\tau_{\beta'}, \square_{\beta'})) \star (\tau_{\beta''}, \square_{\beta''}) \ar[rr]^{\eqref{e:assoc}} \ar[d]^{\alpha_{\beta, \beta'} \star \id} &&
(\tau_\beta, \square_\beta) \star ((\tau_{\beta'}, \square_{\beta'}) \star (\tau_{\beta''}, \square_{\beta''})) \ar[d]^{\id \star \alpha_{\beta', \beta''}} \\ (\tau_{\beta \beta'}, \square_{\beta \beta'}) \star (\tau_{\beta''}, \square_{\beta''}) \ar[d]^{\alpha_{\beta \beta', \beta''}} && (\tau_\beta, \square_\beta) \star (\tau_{\beta' \beta''}, \square_{\beta' \beta''}) \ar[d]^{\alpha_{\beta, \beta' \beta''}} \\ (\tau_{\beta \beta' \beta''}, \square_{\beta \beta' \beta''}) \ar[rr]^{\id}   && (\tau_{\beta \beta' \beta''}, \square_{\beta \beta' \beta''})}$$
automatically commutes by Lemma \ref{l:contracgrpd}. In particular, remembering only the objects $\tau_\beta, \beta \in \Omega_\chi$, and the underlying morphisms between their convolutions in isomorphisms $\alpha_{\beta, \beta'}$, we obtain the desired functor $\blacksquare$ along with a natural equivalence $\iota$. The uniqueness up to unique equivalence again follows from Lemma \ref{l:contracgrpd}.
\end{proof}

\sss{} To proceed, we next recall the following. Suppose $\cM^\circ$ is a Karoubian monoidal 1-category, equipped with a strict action of $\Gamma$. That is, for each $\gamma$, one has a monoidal automorphism 
$$c_\gamma: \cM^\circ \simeq \cM^\circ,$$
and for each pair $(\gamma, \gamma') \in \Gamma \times \Gamma$ a natural isomorphism of monoidal automorphisms
$$\iota_{\gamma, \gamma'}: c_{\gamma} \circ c_{\gamma'} \simeq c_{\gamma \gamma'},$$
satisfying the usual cocycle identity for triples $(\gamma, \gamma', \gamma'') \in \Gamma \times \Gamma \times \Gamma$. In this case, one can form the {\em smash product} monoidal category 
$$\Proj_{E, \Gamma} \ltimes \cM^\circ.$$
Namely, for the underlying category, objects consist of formal direct sums $$\underset{\gamma \in \Gamma} \oplus \hspace{.5mm} \delta_\gamma \otimes n_\gamma,$$
for $n_\gamma$ objects of $\cM^\circ$, all but finitely many of which are zero,  and homomorphisms are defined as in Equation \eqref{e:homssemidirect}. For the monoidal structure, the underlying binary product of the monoidal structure is determined by the formula 
$$(\delta_\gamma \otimes n_\gamma) \star (\delta_{\gamma'} \otimes n_{\gamma'}) := \delta_{\gamma \gamma'} \otimes (c_{(\gamma')^{-1}}(n_\gamma) \star n_{\gamma'}),$$
 with the evident functoriality on morphisms. Finally, the associativity isomorphisms are obtained from those on $\Proj_{E, \Gamma}$ and $\cM^\circ$ and the transformations $i_{\gamma, \gamma'}$. More explicitly, note that 
$$((\delta_\gamma \otimes n_\gamma) \star (\delta_{\gamma'} \otimes n_{\gamma'})) \star (\delta_{\gamma''} \otimes n_{\gamma''}) = \delta_{\gamma \gamma' \gamma''} \otimes  c_{(\gamma'')^{-1}}( c_{(\gamma')^{-1}}(n_\gamma) \star n_{\gamma'}) \star n_{\gamma''},$$
$$(\delta_\gamma \otimes n_\gamma) \star ((\delta_{\gamma'} \otimes n_{\gamma'}) \star (\delta_{\gamma''} \otimes  n_{\gamma''})) =   \delta_{\gamma \gamma' \gamma''} \otimes c_{(\gamma' \gamma'')^{-1}}(n_\gamma) \star (c_{(\gamma'')^{-1}}(n_{\gamma'}) \star n_{\gamma''}).$$
The required identification given by the composition 
\begin{align*} c_{(\gamma'')^{-1}}( c_{(\gamma')^{-1}}(n_\gamma) \star n_{\gamma'}) \star n_{\gamma''} &\simeq (c_{(\gamma'')^{-1}} \circ c_{(\gamma')^{-1}}(n_\gamma) \star c_{(\gamma'')^{-1}}(n_{\gamma'})) \star n_{\gamma''} \\ & \simeq ( c_{(\gamma' \circ \gamma'')^{-1}}(n_\gamma) \star  c_{(\gamma'')^{-1}}(n_{\gamma'})) \star n_{\gamma''} \\ & \simeq c_{(\gamma' \circ \gamma'')^{-1}}(n_\gamma) \star  (c_{(\gamma'')^{-1}}(n_{\gamma'})) \star n_{\gamma''}),\end{align*}
where the identifications are given by (i) the monoidality data for $c_{(\gamma'')^{-1}}$, (ii) the natural transformation $\iota_{(\gamma'')^{-1}, (\gamma')^{-1}}$, and (iii) the associativity natural isomorphism for $\cM$, respectively. Finally, the monoidal unit is given by $\delta_e \otimes \mathbb{1}$, with the evident trivializations 
$$(\delta_e \otimes \mathbb{1}) \star - \simeq - \simeq  - \star (\delta_e \otimes \mathbb{1}).$$

\sss{} In particular, suppose one is given a monoidal category $\cM$ equipped with a monoidal functor $\Proj_{E, \Gamma} \rightarrow \cM$, and for each $\gamma \in \Gamma$ let us write $m_\gamma$ for the corresponding object of $\cM$ as in Lemma \ref{l:mapprojg}. Note that, in particular, each $m_\gamma$ is an invertible object of $\cM$, with inverse $m_{\gamma^{-1}}$ and evaluation pairing $$\alpha_{\gamma, \gamma^{-1}}: m_\gamma \star m_{\gamma^{-1}} \simeq m_{e} = \mathbb{1}.$$

Suppose one is in addition given a full monoidal subcategory $\cM^\circ \subset \cM$ which is normalized by $\Gamma$, i.e., for each $\gamma \in \Gamma$ we have that $\cM^\circ$ is preserved by the autoequivalence 
$$ m_{\gamma} \star - \star m_{\gamma^{-1}}: \cM \simeq \cM.$$
In this case, as above one can form the semi-direct product $\Proj_{E, \Gamma} \ltimes \cM^\circ$, and we have the following basic  observation. 

\begin{lem} \label{l:semidirect2all}There is a canonical monoidal functor
$$\mu: \Proj_{E, \Gamma} \ltimes \cM^\circ \rightarrow \cM$$
which on objects sends $\delta_\gamma \otimes m^\circ_\gamma \mapsto m_\gamma \star m^\circ_\gamma$.
\end{lem} 

\begin{proof}At the level of underlying $E$-linear additive, Karoubian categories, note that $$\Proj_{E, \Gamma} \ltimes \sM^\circ \simeq \underset{\gamma \in \Gamma} \oplus \hspace{.5mm} m_\gamma \star \sM^\circ,$$and the functor in question is the direct sum of the inclusion functors $m_\gamma \star \sM^\circ \subset \sM.$ Moreover, if for ease of notation we suppress associativity natural isomorphisms in $\cM$, the required natural isomorphism intertwining the binary products is given by the composition 
\begin{align*}
\mu(m_\gamma \otimes n_\gamma) \star  \mu(m_{\gamma'} \otimes n_{\gamma'}) & \simeq m_\gamma \star n_\gamma \star m_{\gamma'} \star n_{\gamma'} \\ & \simeq m_\gamma \star (m_{\gamma'} \star m_{(\gamma')^{-1}}) \star n_\gamma \star m_{\gamma'} \star n_{\gamma'} \\ & \simeq (m_\gamma \star m_{\gamma'}) \star (m_{(\gamma')^{-1}} \star n_\gamma \star m_{\gamma'})  \star n_{\gamma'} \\ & \simeq m_{\gamma \gamma'} \star c_{(\gamma')^{-1}}(n_\gamma) \star n_{\gamma'} \\ & \simeq \mu( (m_\gamma \otimes n_\gamma) \star (m_{\gamma'} \otimes n_{\gamma'})). \end{align*}The unital data is defined similarly, and we leave the verification of the compatibility with associativity constraints and triangle identities as a straightforward exercise for the reader. 
\end{proof}

\sss{}  We are now ready to give the promised Soergel theoretic description of the full category. Note that for any $\beta \in \Omega_\chi$, the conjugation map 
$$\omega(w^\beta)_\chi \star - \star \omega(w^{\beta^{-1}})_\chi: \on{SBim}_\chi \simeq \on{SBim}_\chi$$restricts to a monoidal automorphism of the neutral block
$$\omega(w^\beta)_\chi \star - \star \omega(w^{\beta^{-1}})_{\chi}: \on{SBim}_\chi^\circ \simeq \on{SBim}_\chi^\circ,$$essentially by diagram automorphisms. In particular, we may consider the smash product category 
$$\Proj_{E, \Omega_\chi} \ltimes \on{SBim}_\chi^\circ.$$

\begin{thm} \label{t:soergequivtitSEMIDIRECT}Assume that the characteristic of $\fre$ is not two. Given a Soergel functor $\mathbb{V}_\chi$, there is a canonical up to unique equivalence associated monoidal equivalence 
$${}_\chi\cT_\chi \simeq \Proj_{E, \Omega_\chi} \ltimes \on{SBim}_\chi^\circ.$$
\end{thm}

\begin{proof} Recall the pair $(\blacksquare, \iota)$ of Proposition \ref{p:rigidaction}. Using $\blacksquare$, we may form the smash product monoidal category 
$$\Proj_{E, \Omega_\chi} \ltimes {}_\chi \cT_\chi^\circ,$$
and using $\iota$ we obtain a monoidal functor
\begin{equation} \label{e:idVchi}\on{id} \ltimes \mathbb{V}_\chi: \Proj_{E, \Omega_\chi} \ltimes {}_\chi \cT_\chi^\circ \xrightarrow{} \Proj_{E, \Omega_\chi} \ltimes \on{SBim}_\chi^\circ,\end{equation}
which is an equivalence by Theorem \ref{t:soergequivtit}. 

On the other hand, by Lemma \ref{l:semidirect2all}, we have a monoidal functor 
\begin{equation} \label{e:idVchi2} \Proj_{E, \Omega_\chi} \ltimes {}_\chi \cT_\chi^\circ \xrightarrow{} {}_\chi \cT_\chi,
\end{equation}
which is moreover an equivalence by Lemma \ref{l: convmintilt} and Proposition \ref{blockdec}. The composition of Equations \eqref{e:idVchi2} and \eqref{e:idVchi} then yields the desired equivalence. \end{proof}

\subsection{Endoscopy for affine Hecke categories}

\sss{} By combining the previous results of this section, we are ready to state and prove the desired, final endoscopy result for affine Hecke categories. Consider two affine Hecke categories associated to possibly different groups and characters $$D( (\wt{I}_1, \chi_1\mon) \bs \wt{LG}_1 / (\wt{I}_1, \chi_1\mon)) \quad \text{and} \quad D( (\wt{I}_2, \chi_2\mon) \bs \wt{LG}_2 / (\wt{I}_2, \chi_2\mon)).$$

Suppose that one has an isomorphism of extended Coxeter systems \begin{equation*}\alpha: \widetilde{W}_{\widetilde{LG}_1, \chi_1} \simeq \widetilde{W}_{\widetilde{LG}_2,\chi_2}\end{equation*}
and a compatible identification of their monodromy representations, i.e. an $\alpha$-equivariant identification of the formal schemes, in evident notation, over $\on{Spec } E$ \begin{equation*}\beta: \frf_{\widetilde{LG_1}, \chi_1} \simeq \frf_{\wt{LG}_2, \chi_2}.\end{equation*}

Suppose in addition we fix Soergel functors $\mathbb{V}_{\chi_i}$, $i = 1, 2$, for both categories. 

\begin{thm}\label{t:mainthmsec7} Assume that the characteristic of $\fre$ is not two. The quadruple $(\alpha, \beta, \mathbb{V}_{\chi_1}, \mathbb{V}_{\chi_2})$ canonically yields a $t$-exact equivalence of monoidal categories \begin{equation*} \epsilon: D( (\wt{I}_1, \chi_1\mon) \bs \wt{LG}_1 / (\wt{I}_1, \chi_1\mon)) \simeq D( (\wt{I}_2, \chi_2\mon) \bs \wt{LG}_2 / (\wt{I}_2, \chi_2\mon)).\end{equation*}  

Moreover, this equivalence  exchanges (co)standard objects with (co)standard objects and tilting objects with tilting objects. That is, for any $w \in \tilW_{\wt{LG}_1, \chi_1}$, there are canonical isomorphisms of objects
$$\epsilon(\std(w)_{\chi_1}) \simeq \std(\alpha(w))_{\chi_2} \quad \text{and} \quad \epsilon(\costd(w)_{\chi_1}) \simeq \costd(\alpha(w))_{\chi_2},$$
where the canonical (co)standard objects on both sides are defined as in Remark \ref{r:canonicalstd}, and there is a canonical isomorphism of groupoids of objects
$$\epsilon( \tau(w)_{\chi_1}) \simeq \tau(\alpha(w))_{\chi_2}.$$

\end{thm}

\begin{proof} In what follows, for ease of notation we suppress the groups $\wt{LG}_i$ from the notation. The pair $(\alpha, \beta)$ gives rise to a canonical monoidal equivalence $$\Proj_{E, \Omega_{\chi_1}} \ltimes \on{SBim}_{\chi_1}^\circ \simeq \Proj_{E, \Omega_{\chi_2}} \ltimes \on{SBim}_{\chi_2}^\circ.$$Therefore from Theorem \ref{t:soergequivtitSEMIDIRECT} we obtain a composite monoidal equivalence $\epsilon$
$$\cT_{\chi_1} \overset{\mathbb{V}_{\chi_1}}\simeq \Proj_{E, \Omega_{\chi_1}} \ltimes \on{SBim}_{\chi_1}^\circ \simeq \Proj_{E, \Omega_{\chi_2}} \ltimes \on{SBim}_{\chi_2}^\circ \overset{\mathbb{V}_{\chi_2}^{-1}} \simeq \cT_{\chi_2},$$Note that, by construction, $\epsilon$ exchanges the groupoids of objects $\tau(w)_{\chi_1}$ and $\tau(\alpha(w))_{\chi_2}$. We deduce in particular an equivalence 
$$\epsilon: \on{Ind}(\on{K}^b \cT_{\chi_1}) \simeq \on{Ind}(\on{K}^b \cT_{\chi_2})$$
which, via their induced actions on $\frf_{\wt{LG}_1} \simeq \frf_{\wt{LG}_2}$, exchanges the canonical (co)standard objects, cf. Corollary \ref{cor: standard characterization}. 

Moreover, Theorem \ref{t:renorm}, particularly the base change isomorphism \eqref{e:basechangechi}, shows that $\epsilon$ exchanges the Verdier quotients 
$$\xymatrix{\on{Ind}(\on{K}^b \cT_{\chi_1}) \ar[r]^\epsilon \ar[d]_{\iota^!} & \on{Ind}(\on{K}^b \cT_{\chi_2}) \ar[d]^{\iota^!} \\ D((\widetilde{I}_1, \chi_1\mon) \bsl \widetilde{LG}_1/(\widetilde{I}_1, \chi_1\mon)) \ar[r]^\epsilon & D((\widetilde{I}_2, \chi_2\mon) \bs \widetilde{LG}_2/(\widetilde{I}_2, \chi_2\mon)).}$$

It remains therefore only to verify the $t$-exactness of the obtained equivalence of affine Hecke categories. However, note that $\beta$ induces a canonical isomorphism $\beta: \mathscr{L}_{\chi_1} \simeq \mathscr{L}_{\chi_2}$, as both identify with with the set of closed subschemes of 
$$\frf_{\wt{LG}_1} \simeq \frf_{\wt{LG}_2}.$$
In particular, $\beta$ restricts to equivalences between the closed subschemes
$$\beta: \frf_{\wt{\chi}_1} \simeq \frf_{\beta(\wt{\chi}_1)}, \quad \quad \wt{\chi}_1 \in \mathscr{L}_{\chi_1}.$$
It follows from Corollary \ref{cor: standard characterization} that the equivalence $\cT_{\chi_1} \simeq \cT_{\chi_2}$ exchanges the corresponding compact objects, i.e.,. we have a canonical isomorphism
$$\epsilon( j(w)_{!, \wt{\chi}_1}) \simeq j(\alpha(w))_{!, \beta(\wt{\chi}_1)}.$$
As the $t$-structure may be characterized by maps out of these objects, cf. Corollary \ref{c:compgensfixedmon}, we are done. 
\end{proof}

\section{From monodromic to strict and parahoric equivalences}\label{s:mon vs strict}

So far, we have considered monodromic affine Hecke categories
\begin{equation*} 
D((\wt{I}, \chi\mon) \bs \tLG / (\wt{I}, \chi\mon)),
\end{equation*}
and showed how to obtain endoscopic equivalences among such categories in Theorem \ref{t:mainthmsec7}. In applications however, one instead finds strictly equivariant affine Hecke categories, or more generally, affine Hecke 2-categories equivariant against varying parahorics. 

For this reason, in this section we develop some formalism for passing from monodromic to strict affine Hecke categories, and unwind some consequences for endoscopy. 

\subsection{Equivariant categories}
\sss{} 

We retain the notations from \S \ref{sss: int coef}. Let $\chi$ be a residual character sheaf on $H$, and let $\wt\chi\in\mathscr{L}_\chi$ be a lifting of $\chi$ to an $E$-character sheaf of $H$. We recall that the category of $E$-linear functors from $\Modu_E$ to an $E$-linear category $\cC$ can be identified with $\cC$ itself. Therefore, $\wt\chi$ determines an $E$-linear functor $\io_{\wt\chi}:\Modu_E\to D(H)$, which admits a factorization
\begin{equation}\label{eq: inclusion of character sheaves}
\Modu_E\xrightarrow{\iota_{\wt \chi}^{\mono}} D_{\chi\mon}(H)\subset D_{\mono}(H)\subset D(H).
\end{equation}
As shown in Proposition \ref{p: monoidal unit}, we know that the last two inclusion functors admit $E$-linear continuous right adjoints. In fact, the argument there shows that $\iota_{\wt \chi}^{\mono}$ admits continuous $E$-linear monoidal right adjoint. Hence, the right adjoint $(\io_{\wt\chi})^R: D(H) \ra \Modu_E$ is continuous and monoidal. We denote the last $D(H)$-module category as $(\Modu_E)_{\wt \chi}$. With the equipped $D(H)$-module structures, the functors in \eqref{eq: inclusion of character sheaves} are $D(H)$-linear.

\sss{}  We let 
    \[
    D((H,\tilde\chi)\bs X):=\Hom_{D(H)}((\Modu_E)_{\wt\chi},D(X)),
    \]
    be the category of $(H,\wt \chi)$-equivariant sheaves on $X$.

    We have 
    \[
    \Hom_{D(H)}((\Modu_E)_{\wt\chi},D(H))\cong (\Modu_E)_{\wt\chi}
    \]
    as right $D(H)$-modules and for general $X$, we can identify $D((H,\tilde\chi)\bs X)$ with
    \begin{equation}\label{eq: monodromic to equivariant}
    (\Modu_E)_{\wt\chi}\otimes_{D(H)}D(X)\cong (\Modu_E)_{\wt\chi}\otimes_{D_{\chi\mon}(H)}D((H,\chi\mon)\bs X).
    \end{equation}
    Note that $D((H,\wt\chi)\bs X)$ is generally not a subcategory of $D(X)$.

\begin{rem} \label{r:avcons} As $(\Modu_E)_{\wt\chi}\to D_{\chi\mon}(H)$ is $D(H)$-linear and the latter is generated by the images of all such functors, for $\wt\chi\in\mathscr{L}_\chi$ we see 
that $D((H,\chi\mon)\bs X)\subset D(X)$ is generated by the essential images of the functors $D((H, \wt{\chi})\bs X) \rightarrow D(X)$. Let $$\on{Oblv}: D((H, \wt\chi) \bs X) \rightleftharpoons D((H, \chi\mon) \bs X): \on{Av}^{\wt\chi}$$ be the adjunction. The above discussion implies that $$\underset{\wt{\chi}}\Pi\hspace{.5mm} \on{Av}^{\wt{\chi}}: D((H, \chi\mon) \bs X) \rightarrow \underset{\wt{\chi}} \Pi \hspace{.5mm} D((H, \wt{\chi}) \bs X)$$
is conservative, where $\wt{\chi}$ runs over all $E$-lifts of $\chi$. 

\end{rem}

\begin{exmp}
    If $\chi$ is the trivial character sheaf, in which case we also denote it by $u$, then $D((H,u)\bs X)\cong D(H\bs X)$. In particular, we see that $D((H,u\mon)\bs X)\subset D(X)$ is generated by $*$-pullbacks of sheaves on $H\bs X$.
\end{exmp}

\subsection{Twisted categories of sheaves}

\sss{} The goal of this subsection is to recall some standard facts and definitions concerning sheaves on a group with a fixed equivariance along a central subgroup.

$$D((A, {\wt{\chi}}) \bs X) \simeq (\Modu_E)_{\wt{\chi}} \underset{D(A)} \otimes D(X),$$
cf. Equation \eqref{eq: monodromic to equivariant}.  

A particular case of the above which will be useful to us is the following. Suppose we have a central extension of algebraic groups 
$$1 \rightarrow A \rightarrow \wt{H} \rightarrow H \rightarrow 1.$$
By the commutativity of $A$, $D(A)$ carries a canonical symmetric monoidal structure. Moreover, by the centrality of $A$ in $\wt{H}$, the tautological monoidal functor $D(A) \rightarrow D(\wt{H})$ lifts to an $E_2$-monoidal functor to the Drinfeld center
$$D(A) \rightarrow Z(D(\wt{H})) := \Hom_{D(\wt{H})\on{-mod-}\hspace{-.5mm}D(\wt{H})}(D(\wt{H})).$$
Similarly, if we fix an $E$-linear character sheaf $\wt{\tau}$ on $A$, the tautological functor $D(A) \rightarrow (\Modu_E)_{\wt{\tau}}$ carries a datum of symmetric monoidality, so that in particular the tensor product 
$$D_{\wt{\tau}}(H) := D((A, {\wt{\tau}}) \bs \wt{H}) \simeq (\Modu_E)_{\wt{\tau}} \underset{D(A)} \otimes D(\wt{H})$$again is naturally a monoidal category. Explicitly, this is the usual convolution product of ${\wt{\tau}}$-equivariant sheaves on $\wt{H}$.

Let us write ${\wt{\tau}}^\vee$ for the dual character sheaf on $A$ to ${\wt{\tau}}$. Note that inversion on $\wt{H}$ induces an equivalence of monoidal categories 
$$D_{\wt{\tau}}(H) \simeq D_{{\wt{\tau}}^\vee}(H)^{\on{rev}},$$
where the superscript `rev' refers to opposite. i.e., reversed, monoidal structure. In particular, we use this to identify left modules for $D_{\wt{\tau}}(H)$ and right modules for $D_{{\wt{\tau}}^\vee}(H)$.

\sss{} The binary product
$$m_*: D_{\wt{\tau}}(H) \otimes D_{\wt{\tau}}(H) \rightarrow D_{\wt{\tau}}(H)$$again admits a left adjoint $m^*$, so that we can speak of $E$-linear character sheaves ${\wt{\chi}}$ in $D_{\wt{\tau}}(H)$. The (discrete) groupoid of such character sheaves is canonically equivalent to the (discrete) groupoid of character sheaves on $\wt{H}$ whose restriction to $A$ is ${\wt{\tau}}$. In what follows, by a mild abuse of notation we use the single letter ${\wt{\chi}}$ to denote the corresponding object of either groupoid. It follows that passing to the dual character sheaf yields a contravariant equivalence between character sheaves ${\wt{\chi}}$ in $D_{\wt{\tau}}(H)$ and character sheaves ${\wt{\chi}}^\vee$ in $D_{{\wt{\tau}}^\vee}(H)$.

\sss{} In particular, for any category $\cC$ equipped with a left action of $D_{\wt{\tau}}(H)$, and a character sheaf ${\wt{\chi}}$ in $D_{\wt{\tau}}(H)$, we may again form the strictly ${\wt{\chi}}$-equivariant category 
$$\cC_{(H, {\wt{\chi}})} := (\Modu_E)_{{\wt{\chi}}^\vee} \underset{D_{\wt{\tau}}(H)} \otimes \cC.$$
Similarly, given a category $\cD$ equipped with a right action of $D_{\wt{\tau}}(H)$, we set 
$$\cD_{(H, {\wt{\chi}})} := \cD \underset{D_{\wt{\tau}}(H)} \otimes (\Modu_E)_{{\wt{\chi}}}.$$

Note that $*$-tensoring with the dual character sheaf ${\wt{\chi}}^\vee$ gives a monoidal equivalence between $D_{\wt{\tau}}(H)$ and $D(H)$, exchanging ${\wt{\chi}}$ and the trivial character sheaf. In particular, it follows from the untwisted case that we have a canonical identification 
$$\cC_{(H, {\wt{\chi}})} \simeq \Hom_{D_{\wt{\tau}}(H)\mod}( (\Modu_E)_{\wt{\chi}}, \cC),$$
i.e., an identification of twisted invariants and coinvariants. 

\begin{exmp}
A basic example of relevance for us is the following. Namely, if $X$ is a variety or ind-scheme of ind-finite type equipped with an action of $\wt{H}$, then the category of equivariant sheaves $$D_{\wt{\tau}}(X) := D((A, {\wt{\tau}})\bs X)$$ carries a canonical action of $D_{\wt{\tau}}(H)$. Moreover, in this case we have a canonical equivalence
$$D_{\wt{\tau}}((H, {\wt{\chi}}) \bs X) := D_{\wt{\tau}}( X)_{(H, {\wt{\chi}})} \simeq D((\wt{H}, {\wt{\chi}}) \bs X).$$    
\end{exmp}

A second basic example for us is the following. 
\begin{exmp} Given ${\wt{\chi}}$ a character sheaf in $D_{\wt{\tau}}(H)$, we may consider the associated Hecke category 
$$\Hom_{D_{\wt{\tau}}(H)\mod}( (\Modu_E)_{\wt{\chi}}, (\Modu_E)_{\wt{\chi}}) \simeq D_{\wt{\tau}}((H, {\wt{\chi}}) \bs H / (H, {\wt{\chi}})).$$
Then $*$-tensoring by ${\wt{\chi}}^\vee$ yields a monoidal equivalence 
$$D_{\wt{\tau}}((H, {\wt{\chi}}) \bs H / (H, {\wt{\chi}})) \simeq D(H \bs H / H) \simeq D(\mathbb{B}H).$$
By contrast, note that when working with untwisted sheaves on $\wt{H}$ itself, we instead have 
$$\Hom_{D(\wt{H})\mod}((\Modu_E)_{\wt{\chi}}, (\Modu_E)_{\wt{\chi}}) \simeq D(\mathbb{B}\wt{H}).$$
\end{exmp}

\subsection{Monodromy operators}

\subsubsection{} Let $H$, ${\wt{\tau}}$, and $\wt{\chi}$ be as in the previous subsection. As in Section \ref{S: monodromic and equivariant categories}, we may consider the full subcategory generated by $\wt{\chi}$ under colimits within $D_{\wt{\tau}}(H)$, which we denote by 
$$\iota: D_{{\wt{\tau}}, \wt{\chi}\mon}(H) \rightleftharpoons D_{\wt{\tau}}(H): \iota^R.$$
For any category $\cC$ with a left action of $D_{\wt{\tau}}(H)$, we have the associated monodromic (co)invariants 
$$\cC_{(H, \wt{\chi}\mon)} := D_{{\wt{\tau}}^\vee, \wt{\chi}^\vee\mon}(H) \underset{D_{\wt{\tau}}(H)} \otimes \cC \simeq \Hom_{D_{\wt{\tau}}(H)\mod}(D_{{\wt{\tau}}, \wt{\chi}\mon}(H), \cC),$$
where the second equivalence may be deduced from its analogue in the untwisted case via tensoring with $\wt{\chi}^\vee$.

Similarly, for a right $D_{\wt{\tau}}(H)$-module $\cD$, we have the associated monodromic invariants 
$$\cD_{(H, \wt{\chi}\mon)} := \cD_{(H, \wt{\chi}^\vee\mon)}.$$

When $X$ is a variety equipped with a left action of $\wt{H}$, for ease of notation we will write 
$$D_{\wt{\tau}}((H, \wt{\chi}\mon) \bs X) := D_{\wt{\tau}}(X)_{(H, \wt{\chi}\mon)},$$
and similarly for $Y$ a variety equipped with a right action of $\wt{H}$
$$D_{\wt{\tau}}(Y / (H, \wt{\chi}\mon)) := D_{{\wt{\tau}}}(Y)_{(H, \wt{\chi}\mon)}.$$

\sss{} By definition, we have an adjunction 
$$\on{Oblv}: \cC_{(H, \wt{\chi})} \rightleftharpoons \cC_{(H, \wt{\chi})\mon}: \on{Av}_*,$$
Our goal in this subsection will be to explain two different ways to reconstruct either category from the other, i.e., to pass between strict and monodromic invariants. 

To this end, let us introduce some convenient notation for the associated monodromic Hecke category,
\begin{align*}
\sH := D_{\wt{\tau}}( (H, \wt{\chi}\mon) \bs H / (H, \wt{\chi}\mon)) &\simeq \Hom_{D_{\wt{\tau}}(H)\mod}( D_{{\wt{\tau}}, \wt{\chi}\mon}(H), D_{{\wt{\tau}}, \wt{\chi}\mon}(H)). \intertext{Explicitly, via the fully faithful forgetful functor from monodromic invariants to the ambient category $\cC$, in this case $\cC \simeq D_{\wt{\tau}}(H)$, $\sH$ is simply $D_{{\wt{\tau}}, \wt{\chi}\mon}(H)$, and the underlying non-unital monoidal category structure is simply convolution. Similarly, we may consider the strictly equivariant Hecke category}  \stMst := D_{\wt{\tau}}((H, \wt{\chi}) \bs H / (H, \wt{\chi})) &\simeq \Hom_{D_{\wt{\tau}}(H)\mod}((\Modu_E)_{\wt{\chi}}, (\Modu_E)_{\wt{\chi}}), \intertext{and finally their tautological bimodules} \Mst := D_{\wt{\tau}}((H, \wt{\chi}\mon) \bs H / (H, \wt \chi)) &\simeq \Hom_{D_{\wt{\tau}}(H)\mod}(D_{{\wt{\tau}}, \wt{\chi}\mon}(H), (\Modu_E)_{\wt{\chi}}), \\ \stM := D_{\wt{\tau}}((H,  \wt{\chi}) \bs H / (H, \wt{\chi}\mon)) &\simeq \Hom_{D_{\wt{\tau}}(H)\mod}(   (\Modu_E)_{\wt{\chi}}, D_{{\wt{\tau}}, \wt{\chi}\mon}(H)).
\end{align*}

The first presentation of how to pass back and forth between strict and monodromic invariants is the fact that the above bimodules provide inverse equivalences between the categories of modules for the strict and monodromic Hecke categories. That is, we have the following.

\begin{prop} \label{p:monmorita} Tensoring with the tautological bimodules yields inverse Morita equivalences
\begin{equation*} \stM \underset{\H}\otimes - :\H\mod \simeq \stMst\mod: \Mst \underset{\stMst} \otimes -.
\end{equation*}
\end{prop}

\begin{proof} It is enough to show that the natural convolution, i.e., composition, functors
\begin{equation} \label{e:twoconvs} \quad \Mst \underset{\stMst} \otimes \stM  \rightarrow \H  \quad \text{and} \quad  \stM  \underset{\H} \otimes \Mst \rightarrow \stMst \end{equation}
are equivalences.  

Moreover, by tensoring with the inverse character sheaf, we may assume that ${\wt{\tau}}$ and $\wt \chi$ are both trivial.

Let us begin with the left-hand functor in \eqref{e:twoconvs}. Note that it may be rewritten more explicitly as the pullback 
\begin{equation} \label{e:convmon1}D(\on{pt}) \underset{D(\mathbb{B}H)} \otimes D(\on{pt}) \rightarrow D(\on{pt} \underset{\mathbb{B}H} \times \on{pt}) \simeq D(H),\end{equation}
The fully faithfulness of this functor follows from \cite[(4.68)]{Zh}.

To characterize the essential image of \eqref{e:convmon1}, recall the tensor product is
generated under colimits by the insertion 
$$\Modu_E \simeq \Modu_E \otimes \Modu_E \simeq D(\pt) \otimes D(\pt) \rightarrow D(\pt) \underset{D(\mathbb{B}H)} \otimes D(\pt),$$so it follows the essential image of \eqref{e:convmon1} is the full subcategory generated by the constant sheaf, i.e. $\sH$.

It remains to address the right-hand functor in \eqref{e:twoconvs}. To see this, let us compute the appearing tensor product via the bar resolution, i.e., as the colimit of the diagram
$$\begin{tikzcd}
   \cdots \arrow[r, shift left=3] \arrow[r, shift left = 1]  \arrow[r, shift right=3] \arrow[r, shift right = 1] &  \stM \otimes \sH \otimes \sH \otimes \Mst  \arrow[r]
\arrow[r, shift left=2]
\arrow[r, shift right=2] &  \stM \otimes \sH \otimes \Mst 
 \arrow[r, shift left]
\arrow[r, shift right]
&  \stM \otimes \Mst.
\end{tikzcd}
$$
As all the appearing functors are right adjointable, we may equivalently compute this as the limit of the diagram 
$$\begin{tikzcd}
 \stM \otimes \Mst \arrow[r, shift left]
\arrow[r, shift right] &  \stM \otimes \sH \otimes \Mst  \arrow[r]
\arrow[r, shift left=2.2]
\arrow[r, shift right=2.2] & \stM \otimes \sH \otimes \sH \otimes \Mst   \arrow[r, shift left=3] \arrow[r, shift left = 1]  \arrow[r, shift right=3] \arrow[r, shift right = 1]  & \cdots.
\end{tikzcd}
$$
Using Lemma \ref{lem: functoriality depth zero geom Langlands for tori}, or precisely its extension general $H$ \cite[Proposition 4.15]{Zh}, one can identify this limit diagram with the usual Bernstein--Lunts presentation of $\stMst$, i.e., the Cech diagram arising from the smooth cover $\pt \rightarrow \mathbb{B}H$, as desired. \end{proof}

Let us note the following double centralizer property, which is a formal consequence of the previous proposition. 
\begin{cor} \label{c:doublecent}Suppose the reductive quotient of $H$ is a torus. Let $\mathscr{A}$ be an algebra object in $\H$-bimodules, and 
$${}^{\on{s}}\mathscr{A}^{\on{s}} := \stM \underset{\H} \otimes \mathscr{A} \underset{\H} \otimes \Mst$$the corresponding algebra object in $\stMst$-bimodules. Then if we consider the tautological bimodule 
$$\mathscr{A} \circlearrowright \mathscr{A}^{\on{s}} := \mathscr{A} \underset{\H} \otimes \Mst \circlearrowleft {}^{\on{s}}\mathscr{A}^{\on{s}},$$we have monoidal equivalences
$${}^{\on{s}}\mathscr{A}^{\on{s}, \on{rev}} \simeq  \End_{\mathscr{A}\mod}(\mathscr{A}^{\on{s}}) \quad \quad \mathscr{A} \simeq \End_{\on{mod}\text{\textendash}\,{}^{\on{s}}\mathscr{A}^{\on{s}}}(\mathscr{A}^{\on{s}}).$$
\end{cor}

\subsubsection{} It will be very useful for us to recall an alternative way to describe the mutually inverse functors in Proposition \ref{p:monmorita}. Namely, we will reformulate the operation of passing from monodromic invariants to strict  invariants as turning off certain monodromy operators, and similarly the operation of passing from strict invariants to monodromic invariants as turning off the action of the equivariant cohomology of a point. Let us now spell this out in more detail.

\subsubsection{} First, we explain how to pass from monodromic to strict invariants via turning off monodromy operators. 

So, let $\cC$ be a category acted on $\sH$, and write $\delta$ for the monoidal unit of the latter category. Recall that its endomorphisms \begin{equation*}\End(\delta):= \End_{\H}(\delta)\end{equation*} map to the center of $\cC$, i.e., one has a map of $E_2$-algebras
$\End(\delta) \rightarrow \on{HH}^*(\cC),$
so that we may also regard $\cC$ as a module category over $\End(\delta)\mod$. As a particular case, for the (right) module $\cC \simeq \stM$, we obtain an augmentation $$\End(\delta) \rightarrow \on{HH}^*(\stM) = \on{HH}^*(\Modu_E) \simeq E.$$

Note that the averaging functor $\on{Av}^{{\wt{\chi}}}$ can be written as the insertion:
$$\cC \simeq \stM \otimes \cC \rightarrow \stM \underset{\H} \otimes \cC,$$
and as such tautologically factors through an enhanced averaging functor 
\begin{equation} \label{e:lune}\on{Av}^{{\wt{\chi}}, enh}: \stM \underset{\End(\delta)\mod} \otimes \cC \rightarrow \stM \underset{\sH} \otimes \cC.\end{equation}
The desired reformulation of passing from monodromic to strict invariants, via base changing the monodromy operators along their augmentation, then reads as follows. 
\begin{prop}\label{p:avenh} The functor \eqref{e:lune} is an equivalence. 
\end{prop}

\begin{proof} It is enough to address the universal case of $\cC \simeq \sH$. Consider the adjunctions of $\End(\delta)$-modules 
$$\on{Oblv}: \stM \rightleftharpoons \H: \on{Av}^{\wt{\chi}} \quad \text{and} \quad \iota_!: \H \rightleftharpoons \End(\delta)\mod: \iota^!.$$Here, to see the second adjunction, consider the endomorphism algebra
$$\on{C}^*(H,E) := \End_{\sH}(\wt{\chi}),$$
and the resulting tautological identification $$\sH \simeq \on{C}^*(H,E)\mod.$$ Recall that, under this, $\delta$ corresponds to the augmentation module $E$. We then have the following assertion. 

\begin{lem} Inside $\on{C}^*(H,E)\mod$, the compact objects lie in the small stable subcategory generated by the augmentation module. 
\end{lem}

\begin{rem} If $E$ is a field, the claim of the lemma, and its analogue for any variety $X$, follows from general results on modules over coconnective algebras $A$ with $\on{H}^0(A)$ a semisimple algebra over a field. Below, for simplicity we only address for $E$ integral the immediately relevant case of algebraic groups. 
\end{rem}
     
    \begin{proof} By tensoring with $\wt{\chi}^\vee$, it is enough to address the case of $\wt{\chi}$ being the trivial $E$-linear character sheaf. Let us first note two immediate reductions. First, it is enough to verify that $\on{C}^*(H,E)$ itself lies in the small stable category generated by $E$. Second, by replacing $H$ by its quotient by its unipotent radical, we may assume $H$ is a connected reductive group. 

      We next prove the assertion for $H = \mathbb{G}_m$. In this case, as in Proposition \ref{prop: tame geometric Langlands for tori}, we have $\on{C}^*(\mathbb{G}_m,E)\lmod$ canonically identifies with the full subcategory of $E[x, x^{-1}]\lmod$ generated by $E[x,x^{-1}]/(x-1)$. This equivalence by definition exchanges $\on{C}^*(\mathbb{G}_m, E)$ with $E[x,x^{-1}]/(x-1)$ itself, and the augmentation module with $$\omega := E(\!(x-1)\!)dx/E\tl{x-1}dx.$$Therefore, in this case, the claim follows from the exact sequence
$$0 \rightarrow E[x,x^{-1}]/(x-1) \rightarrow \omega \xrightarrow{x-1} \omega \rightarrow 0.$$
      
The analogous claim for an algebraic variety $X$ isomorphic to a product $\mathbb{A}^n \times \mathbb{G}_m^r$ follows straightforwardly by the K\"unneth formula. Similarly, the claim for a smooth algebraic variety $Y$ admitting a stratification $Y = \sqcup_\alpha X_\alpha$, where each $X_\alpha$ is as above, follows from using the associated Cousin filtration 
$$\on{gr} \on{C}^*(Y,E) \simeq \underset{\alpha} \oplus \hspace{.5mm} \on{C}^*(X_\alpha, E)[2 \dim X_\alpha - 2 \dim Y],$$
and the fact that by definition the action of $\on{C}^*(Y, E)$ on $\on{C}^*(X_\alpha, E)$ is the tautological one arising from restriction. In particular, by taking $Y$ to be $H$ equipped with the Bruhat stratification, we are done. \end{proof}

    Ind-completing this inclusion of compact objects of $\H$ into the small subcategory generated by $\delta$ yields $\iota_!$, and  the right adjoint $\iota^!$ inherits a datum of $\End(\delta)$-linearity by the rigidity of $\End(\delta)\mod$, which moreover, by $\iota^!(\delta) \simeq \delta$, must be the tautological functor $\End(\delta)\mod \rightarrow \H$.     

In particular, if we consider the composite adjunctions 
$$\stM \rightleftharpoons \sH \rightleftharpoons \End(\delta)\lmod$$
it follows in particular that the kernel of $\iota^!$ acts by zero on $\stM$. Therefore, we have 
$$\stM \underset{\End(\delta)\mod} \otimes \H \simeq \stM \underset{\End(\delta)\mod} \otimes \End(\delta)\mod \simeq \stM.$$
as desired. 
\end{proof}

\subsubsection{} Let us next explain how to pass from strict invariants to monodromic invariants via turning off the equivariant cohomology of a point. 

So, let $\cD$ be a category acted on by $\stMst$, and write $\sds$ for the monoidal unit of the latter category. We again have a map of of $E_2$-algebras 
\begin{equation*}\End(\sds) \rightarrow \on{HH}^*(\cD).\end{equation*}
As before, the forgetful functor, i.e., the insertion
$$\on{Oblv}: \cD \simeq \Mst \otimes \cD \rightarrow \Mst \underset{\stMst} \otimes \cD,$$
factors through an enhanced forgetful functor
\begin{equation} \label{e:oblvenh}\on{Oblv}^{enh}: \Mst \underset{\End(\usds)\mod} \otimes \cD \rightarrow \Mst \underset{\stMst} \otimes \cD,\end{equation}
and we have the following. 

\begin{prop} \label{p:oblvenh}The functor \eqref{e:oblvenh} is an  equivalence.    
\end{prop}

\begin{proof} One argues as in the proof of Proposition \ref{p:avenh}, using instead the adjunctions
$$\on{Oblv}^L: \Mst \rightleftharpoons \stMst: \Oblv \quad \text{and} \quad i_!: \stMst \rightleftharpoons \End(\sds)\mod: i^!, $$where $\iota_!$ is again induced from the observation that $\sds$, while not compact, contains the compact objects of $\stMst$ in the small stable subcategory generated by it. Namely, recall the adjunction 
$$\pi_!: \Mst \simeq D(\pt) \rightleftharpoons D(\mathbb{B}H) \simeq \stMst:  \pi^!,$$
gives rise, via the conservativity of $\pi^!$ and Barr--Beck--Lurie, to an identification 
$$\stMst \simeq \pi_! \circ \pi^! (E)\lmod \simeq \on{C}_*(H, E)\lmod.$$
The claim therefore follows from noting that $\sds$ corresponds to the augmentation module $E$ for $\on{C}_*(H, E)$, and recalling that, thanks to the connectivity of $\on{C}_*(H,E)$, the category $\on{C}_*(H,E)\lmod$ carries a unique $t$-structure for which the forgetful functor to $\Modu_E$ is $t$-exact. 
\end{proof}

\begin{rem} In the case when $E$ is a field of characteristic zero, and no twist ${\wt{\chi}}$, Propositions \ref{p:avenh} and \ref{p:oblvenh} also appear e.g. in \cite[Theorem 5.2.6]{CD}, and presumably other places as well. 
\end{rem}

\subsection{Monodromic and strict Hecke categories} \label{ss: monodromic strict hecke}

\subsubsection{} Let us now apply the results of the previous subsection to Hecke categories, as in Section \ref{s:heckecat}. In particular, we resume the use of notation from there. 

Let $K$ be a subgroup of $LG$ containing $I$, and write $\wt{K} \subset \wt{LG}$ for the corresponding central extension. Consider the monoidal category $$D(I^+ \bs \wt{K} / I^+) \subset D(I^+ \bs \wt{LG} / I^+),$$and note that as before we have an $E_2$-monoidal functor 
$$D(\mathbb{G}_m^{\on{cen}}) \rightarrow Z(D(I^+ \bs \wt{K} / I^+)).$$
If we fix an $\fre$-linear character sheaf $\chi$, and denote by $\tau$ its restriction to $\mathbb{G}_m^{\on{cen}}$, the above functor induces on colocalizations
$$D_{\tau\on{-mon}}(\mathbb{G}_m^{\on{cen}}) \rightarrow Z(D((\wt{I}, \chi\mon)\bs \wt{K} / (\wt{I}, \chi\mon))).$$
In particular, if we fix an $E$-linear character sheaf $\wt{\chi} \in \mathscr{L}_\chi$, and denote by $\wt{\tau}$ its restriction to $\mathbb{G}_m^{\on{cen}}$, we similarly obtain the further colocalization 
$$D_{\wt{\tau}\on{-mon}}(\mathbb{G}_m^{\on{cen}}) \rightarrow Z(D((\wt{I}, \wt{\chi}\mon)\bs \wt{K} / (\wt{I}, \wt{\chi}\mon))).$$
Finally, we may obtain the monoidal category of strictly $\wt{\tau}$-equivariant sheaves along the central torus
\begin{equation} \label{e:stricequivbasechange}
D_{\wt{\tau}}((I, \wt{\chi}\mon) \bs K / (I, \wt{\chi}\mon)) := (\Modu_E)_{\wt{\tau}} \underset{D_{\wt{\tau}\mon}(\mathbb{G}_m^{\on{cen}})} \otimes D((\wt{I}, \wt{\chi}\mon)\bs \wt{K} / (\wt{I}, \wt{\chi}\mon)).\end{equation}While the above notation will be employed on occasion below, for ease of reading we additionally introduce the shorthand
$$ \label{e:basechangemon}\sM_{\wt{\tau}, \wt{\chi}\mon}(K) := D_{\wt{\tau}}((I, \wt{\chi}\mon) \bs K / (I, \wt{\chi}\mon)),
$$
and when $K = LG$, the further shorthand $\sM_{\wt{\tau}, \wt{\chi}\mon} := \sM_{\wt{\tau}, \wt{\chi}\mon}(LG)$.

These monoidal categories, for any $K$, are tautologically bimodules over their monoidal full subcategories supported over the neutral stratum, i.e., 
$$(\Modu_E)_{\wt{\tau}} \underset{D_{\wt{\tau}\on{-mon}}(\mathbb{G}_m^{\on{cen}})} \otimes D((\wt{I}, \wt{\chi}\mon)\bs \wt{I} / (\wt{I}, \wt{\chi}\mon)) \simeq D_{\wt{\tau}, \wt{\chi}\mon}(\AA).$$
In particular, we may form their left modules 
$$ \sM_{\wt{\tau}, \wt{\chi}\mon}^{\on{s}}(K) := \sM_{\wt{\tau}, \wt{\chi}\mon}(K) \underset{D_{\wt{\tau}, \wt{\chi}\mon}(\AA)} \otimes (\Modu_E)_{\wt{\chi}}$$
of sheaves which are strictly $\wt{\chi}$-equivariant on the right. 

We may then complete the passage to strictly equivariant sheaves on both sides, i.e., the category
$${}^{\on{s}}\sM_{\wt{\tau}, \wt{\chi}}^{\on{s}}(K) := (\Modu_E)_{\wt{\chi}} \underset{D_{\wt{\tau}, \wt{\chi}\mon}(\AA)} \otimes\sM_{\wt{\tau}, \wt{\chi}\mon}(K) \underset{D_{\wt{\tau}, \wt{\chi}\mon}(\AA)} \otimes (\Modu_E)_{\wt{\chi}}$$
as follows. 

\begin{lemma} \label{l:doubcent}One has a canonical equivalence of categories 
$$\End_{\sM_{\wt{\tau}, \wt{\chi}\mon}(K)\mod}(\sM_{\wt{\tau}, \wt{\chi}\mon}^{\on{s}}(K)) \simeq    {}^{\on{s}}\sM_{\wt{\tau}, \wt{\chi}}^{\on{s}}(K)$$and the double centralizer identity 
$$\End_{{}^{\on{s}}\sM_{\wt{\tau}, \wt{\chi}}^{\on{s}}(K)\mod}(\sM_{\wt{\tau}, \wt{\chi}\mon}^{\on{s}}(K)) \simeq \sM_{\wt{\tau}, \wt{\chi}\mon}(K).$$

\end{lemma}

Note that the lemma in particular equips ${}^{\on{s}}\sM_{\wt{\tau}, \wt{\chi}}^{\on{s}}(K)$
with a monoidal structure. Explicitly, the binary product is given by the usual convolution product of strictly equivariant (twisted) sheaves, cf. \cite{LY}.

\begin{proof} This is a particular case of Corollary \ref{c:doublecent}. \end{proof}

\sss{} We are ready to deduce, after introducing notation, the desired conclusion for endoscopy. 

Consider a pair of monodromic Hecke categories associated to possibly different groups 
$$\sM_{\chi_i} := D( (\wt{I}_i, \chi_i\mon) \bs \wt{LG}_i / (\wt{I}_i, \chi_i\mon)), \quad \quad i = 1, 2.$$
Here, as before, $\chi_i$ denote $\fre$-linear character sheaves, and let us denote by $\tau_i$ their restrictions to the central tori $$\mathbb{G}_m^{\on{cen}} \hookrightarrow \wt{LG}_i.$$To discuss specializing to strict monodromy, let us fix $E$-linear character sheaves $\wt{\chi}_i \in \mathscr{L}_{\chi_i}$, and denote by $\wt{\tau}_i$ their restrictions to the central tori.

\begin{prop} \label{p:pinky}\label{p:mon2strict}Suppose that one has a monoidal equivalence of monodromic affine Hecke categories  $$\iota: \sM_{\chi_1} \simeq \sM_{\chi_2}$$
satisfying the two conditions listed below. 
    
    \begin{enumerate}

        \item The equivalence $\iota$ exchanges, up to isomorphism, the objects $j(e)_{!, \wt{\chi}_i}$, $i = 1, 2$.

        \item There is a (unique) isomorphism $\gamma$ of algebras making the following diagram commute 
                    $$\xymatrix{ \End_{D_{\tau_1\mon}(\mathbb{G}_m)}( \delta^{\tau_1\mon}) \ar[r] \ar[d]_\gamma & \End_{\cM_{\chi_1}}(\delta^{\chi_1\mon}) \ar[d]^{\iota}  \\ \End_{D_{\tau_2\mon}(\mathbb{G}_m)}( \delta^{\tau_2\mon}) \ar[r] & \End_{\cM_{\chi_2}}(\delta^{\chi_2\mon}).}$$

    \end{enumerate} 
    Then $\iota$ induces a monoidal equivalence of strictly equivariant affine Hecke categories
	\begin{equation*}\ssMss_{\wt{\tau}_1, \wt{\chi}_1} \simeq \ssMss_{\wt{\tau}_2, \wt{\chi}_2}.\end{equation*} 
\end{prop}

\begin{rem} \label{r:mon2str}The two listed conditions, which as we will see below are often easy to check in practice, are clearly necessary for the conclusion of the proposition to hold; the content of the proposition is that they are in fact sufficient.    

We should also remark that, if $E$ is a field, the augmentation module for the endomorphisms of $\End(\delta^{\chi_i\mon})$ is unique, so condition (1) is always satisfied. 

Finally, though we do not need it, let us mention that the above result holds, {\em mutatis mutandis}, after replacing the centrally extended groups $\wt{I} \subset \wt{LG}$ with any $\wt{H} \subset \wt{L}$, where $L$ is a placid group ind-scheme and $H \subset L$ is a connected compact open group subscheme whose prounipotent radical has finite codimension. 
\end{rem}

\begin{proof} Note that, thanks to Corollary \ref{c:compsincofree}, assumption (1) is equivalent to the assertion that $\iota$ exchanges the augmentation modules $$\wt{\chi}_i: \End_{\cM_{\chi_i}}(\delta^{\chi_i\mon}) \circlearrowright \Hom_{\cM_{\chi_i}}(j(e)_{!, \wt{\chi}_i}, \delta^{\chi_i\mon}) \simeq E.$$

By assumptions (1) and (2) and Proposition \ref{p:avenh}, the equivalence $\iota$ base changes to a monoidal equivalence
$$\iota': \sM_{\wt{\tau}_1, \wt{\chi}_1\mon} \simeq \sM_{\wt{\tau}_2, \wt{\chi}_2\mon},$$
cf. Equation \eqref{e:stricequivbasechange}. 

By Proposition \ref{p:avenh} and assumption (1), the monoidal equivalence $\iota'$ exchanges the left modules $\sM_{\wt{\tau}_i, \wt{\chi}_i\mon}^{\on{s}}$. Therefore, by considering their endomorphisms, and applying Lemma \ref{l:doubcent}, we are done. \end{proof}

\subsection{Parabolic Hecke categories}

\subsubsection{} Consider a standard parahoric, i.e., a connected group subscheme $K$ intermediate between $I$ and $LG$. In particular, passing to the central extensions, we have inclusions
$$\wt{I} \subset \wt{K} \subset \wt{LG}.$$Let us write $U$ for the pro-unipotent radical of $K$,  $L \simeq K/U$ for its Levi quotient, and $B_L := I/U$ for its standard Borel. Recalling that the central extension $\wt{LG}$ splits uniquely on $U$, let us write $\wt{L}$ for the corresponding central extension of $L$, and similarly for $\wt{B}_L$. 

Note that pullback yields an equivalence between character sheaves on $\wt{I}/I^+$ and $\wt{B}_L$. Let us suppose that the $\fre$-linear character sheaf $\chi$ on $\wt{B}_L$ extends (uniquely) to a character sheaf on $\wt{L}$, which we again denote by $\chi$.

For an $E$-lift $\wt{\chi}$ of $\chi$ on $K$,  consider as before the monoidal categories
\begin{align*}\cM_{I \bs K /I} &:=  D_{\wt{\tau}}((I, \wt{\chi}\mon) \bs K / (I, \wt{\chi}\mon)) \simeq D_{\wt{\tau}}( (B_L, \wt{\chi}\mon) \bs L / (B_L, \wt{\chi}\mon)).  \end{align*}
\begin{align*}\cM_{K \bs K/ K} &:= D_{\wt{\tau}}((L, \wt{\chi}\mon) \bs L /(L, \wt{\chi}\mon)).\end{align*}

\begin{lem}The tautological inclusion  \label{l:hitwithdelta}
	\begin{equation} \label{e:inky}\cM_{K \bs K / K} \hookrightarrow \cM_{I \bs K / I},\end{equation} is (non-unital) monoidal, with essential image a $\cM_{I \bs K / I}$-bimodule. Moreover the inclusion admits a right adjoint which is a map of $\cM_{I \bs K / I}$ bimodules. 
\end{lem}

\begin{proof}Consider their natural bimodule 
	\begin{equation*}\cM_{K \bs K / I} := D_{\wt{\tau}}((L, \wt{\chi}\mon) \bs L / (B_L, \wt{\chi}\mon)).\end{equation*}
	On the one hand, the tautological inclusion
	\begin{equation*}\cM_{K \bs K / K} \rightarrow \cM_{K \bs K / I}\end{equation*}is an equivalence of $\cM_{K \bs K / K}$-modules. On the other hand, the inclusion
	\begin{equation*}\cM_{K \bs K / I} \hookrightarrow \cM_{I\bs K / I}\end{equation*}is a fully faithful embedding of right $\cM_{I \bs K / I}$-modules, which admits an equivariant right adjoint given by monodromic averaging. The claims of the lemma now follow. \end{proof}

In particular, we have the following. 

\begin{cor} \label{c:monoidempotent}For a category $\cC$ acted on by $\cM_{I \bs K / I}$, one has an adjunction 
	\begin{equation*}\cC_{(K, \wt{\chi}) \mon} := \cC \underset{\cM_{I \bs K / I}} \otimes \cM_{I \bs K / K} \rightleftharpoons \cC,\end{equation*}
	where the left adjoint is a fully faithful embedding, and the composition
	\begin{equation*}\cC \rightarrow \cC_{(K, \wt{\chi}\mon)} \rightarrow \cC\end{equation*}
	is given by convolution with the monoidal unit $\delta_K$ of $\cM_{K \bs K / K}.$
\end{cor}

Finally, as notation, consider the left $\sM_{I \bs K / I}$-module  
$$\sM_{ I \bs K / K}^{\on{s}} := D_{\wt{\tau}}( (B_L, \wt{\chi}\mon) \bs L / (L, \wt{\chi})) \simeq \Modu_E.$$

\subsubsection{} Let us now deduce some consequences for endoscopy for Hecke 2-categories. We first introduce some notation and basic definitions.  \label{sss:biggerco}

Consider a set $\mathscr{I}$ of standard parahorics
\begin{equation*}I \subset K^\alpha \subset LG, \quad \quad \text{for } \alpha \in \mathscr{I}.\end{equation*}
equipped with character sheaves $\wt{\chi}$ on their centrally extended Levi quotients $\wt{L}^\alpha$ as above. For ease of notation below, let us assume, by possibly adjoining one element to $\mathscr{I}$, that $K^{\alpha_\circ} = I$ for some $\alpha_\circ \in \mathscr{I}$.
\subsubsection{} For any pair $\alpha, \beta$ in $\mathscr{I}$ consider the associated category
\begin{align*} \cM_{K^\alpha \bs LG / K^\beta} &:=  \cM_{K^\alpha \bs K^\alpha / I} \underset{\cM_{I \bs K^\alpha / I}} \otimes \cM_{I \bs LG / I} \underset{\cM_{I \bs K^\beta / I}} \otimes \cM_{I \bs K^\beta / K^\beta}. \end{align*}
More explicitly, this may be identified with $D_{\wt{\tau}}((K^\alpha, \wt{\chi} \mon) \bs {LG} / (K^\beta, \wt{\chi}\mon)),$ i.e., strictly $\wt{\tau}$-equivariant sheaves on $\wt{LG}$ which are moreover monodromically (twisted) equivariant with respect to $K^\alpha$ and $K^\beta$ on the left and right, respectively.\footnote{Strictly speaking, in the formalism used in the present paper, we recall that we define the appearing monodromic categories, for group schemes of infinite type, via considering their Levi quotients, and their natural action on $U^\alpha \bs \wt{LG}/ U^\beta$.} 

Similarly, we have the strictly equivariant category, 
 \begin{align*} \ssMss_{K^\alpha \bs LG / K^\beta} &:= {}^{\on{s}}\cM_{K^\alpha \bs K^\alpha / I} \underset{\cM_{I \bs K^\alpha / I}} \otimes \cM_{I \bs LG / I} \underset{\cM_{I \bs K^\beta / I}} \otimes \cM_{I \bs K^\beta / K^\beta}^{\on{s}},\end{align*}
which can more explicitly be written as $D_{\wt{\tau}}((K^\alpha, \wt{\chi}) \bs LG / (K^\beta, \wt{\chi})).$ 

These are the underlying categories of Hecke 2-categories associated to $\mathscr{I}$. To see this, we  consider in addition the one-sided equivariant categories $$\cM_{K^\alpha \bs LG / K^\beta}^{\on{s}} := \cM_{K^\alpha \bs K^\alpha / I} \underset{\cM_{I \bs K^\alpha / I}} \otimes \cM_{I \bs LG / I} \underset{\cM_{I \bs K^\beta / I}} \otimes \cM_{I \bs K^\beta / K^\beta}^{\on{s}}.$$

\begin{lem} \label{l:algbr}Consider the natural action of $\cM_{\wt{\chi}} \simeq \sM_{I \bs LG / I}$ on \begin{equation*}\underset{\alpha} \oplus \hspace{.5mm} \cM_{I \bs LG / K^\alpha}.\end{equation*} 
	Then one has a canonical equivalence
	\begin{equation*}\End_{\cM_{\wt{\chi}}}(\underset{\alpha} \oplus  \hspace{.5mm} \cM_{I \bs LG / K^\alpha}) \simeq \underset{\alpha, \beta} \oplus  \hspace{.5mm} \cM_{K^\alpha \bs LG / K^\beta}. \end{equation*}
	Similarly, one has a canonical equivalence
	\begin{equation*}\End_{\cM_{\wt{\chi}}}(\underset{\alpha} \oplus  \hspace{.5mm} \cM_{I \bs LG / K^\alpha}^{\on{s}}) \simeq \underset{\alpha, \beta} \oplus  \hspace{.5mm} \ssMss_{K^\alpha \bs LG / K^\beta}. \end{equation*}
\end{lem}

\begin{proof} We unwind that 
    \begin{align*} & \hspace{5mm} \Hom_{\cM_{\wt{\chi}}}(\underset{\alpha} \oplus  \hspace{.5mm} \cM_{I \bs LG / K^\alpha}, \underset{\alpha} \oplus  \hspace{.5mm} \cM_{I \bs LG / K^\alpha}) \\ & \simeq \underset{\alpha, \beta} \oplus \Hom_{\cM_{\wt{\chi}}}( \cM_{I \bs LG / K^\alpha}, \cM_{I \bs LG / K^\beta}) \\ &   \simeq \underset{\alpha, \beta} \oplus \Hom_{\cM_{\wt{\chi}}}( \cM_{I \bs LG / I} \underset{ \cM_{I \bs K^\alpha / I}} \otimes \cM_{I \bs K^\alpha / K^\alpha}  , \cM_{I \bs LG / K^\beta}) \\ & \simeq \underset{\alpha, \beta} \oplus \Hom_{\cM_{I \bs K^\alpha / I}}( \cM_{I \bs K^\alpha / K^\alpha}, \cM_{I \bs LG / K^\beta}) \\ & \simeq \underset{\alpha, \beta} \oplus  \hspace{.5mm}\cM_{K^\alpha \bs LG / K^\beta}, \end{align*}
	as desired. The case of strict equivariance may be argued similarly.\end{proof}

\subsubsection{} Let us now explicitly see how to apply the previous results to endoscopy for affine Hecke categories. 

So, suppose that we have an equivalence of monodromic affine Hecke categories
\begin{equation*} \epsilon: D( (\wt{I}_1, \chi_1\mon) \bs \wt{LG}_1 / (\wt{I}_1, \chi_1\mon)) \simeq D( (\wt{I}_2, \chi_2\mon) \bs \wt{LG}_2 / (\wt{I}_2, \chi_2\mon)).\end{equation*}  
as in Theorem \ref{t:mainthmsec7}, and fix compatible $E$-lifts $$ \wt{\chi}_1 \in \mathscr{L}_{\chi_1} \simeq \mathscr{L}_{\chi_2} \ni \wt{\chi}_2,$$and denote their restrictions to the central $\mathbb{G}_m$'s by $\wt{\tau}_1$ and $\wt{\tau}_2$.

Let $\mathscr{I}$ index the set of simple reflections in the integral Weyl group $\tilW_{\wt{\chi}_1} \simeq \tilW_{\wt{\chi}_2}$ which correspond to simple reflections in both ambient affine Weyl groups \begin{equation*} \tilW_{\wt{LG}_1} \hookleftarrow \tilW_{\wt{\chi}_1} \simeq \tilW_{\wt{\chi}_2} \hookrightarrow \tilW_{\wt{LG}_2}.\end{equation*}For $\alpha \in \mathscr{I}$ let us denote the corresponding standard parahoric subgroups by  $$K^{\alpha}_i  \hookrightarrow LG_i, \quad \quad 1 \leqslant i \leqslant 2.$$
\begin{thm} \label{thm: parahoric equivalence}Suppose that the equivalence $\epsilon$ satisfies the assumptions of Proposition \ref{p:mon2strict} above. Then $\epsilon$ induces an equivalence of monodromic Hecke 2-categories
\begin{equation*}
    \underset{\alpha, \beta} \oplus \hspace{.5mm} D_{\wt{\tau}_1}((K^\alpha_1, \wt{\chi}_1\mon) \bs \wt{LG}_1 / (K^\beta_1, \wt{\chi}_1\mon)) \simeq \underset{\alpha, \beta} \oplus \hspace{.5mm} D_{\wt{\tau}_2}((K^\alpha_2, \wt{\chi}_2\mon) \bs \wt{LG}_2 / (K^\beta_2, \wt{\chi}_2\mon)),
\end{equation*}
and of strictly equivariant Hecke 2-categories
\begin{equation*}
 \underset{\alpha, \beta} \oplus \hspace{.5mm} D_{\wt{\tau}_1}((K^\alpha_1, \wt{\chi}_1) \bs \wt{LG}_1 / (K^\beta_1, \wt{\chi}_1)) \simeq \underset{\alpha, \beta} \oplus \hspace{.5mm}  D_{\wt{\tau}_2}((K^\alpha_2, \wt{\chi}_2) \bs \wt{LG}_2 / (K^\beta_2, \wt{\chi}_2)).
\end{equation*}
\end{thm}

\begin{proof}We begin by addressing the equivalence of bi-monodromic Hecke categories. Let us again denote by $\epsilon$ the equivalence 
$$D_{\wt{\tau}_1}((I, \wt{\chi}_1\mon) \bs \wt{LG}_1 / (I, \wt{\chi}_1\mon)) \simeq  D_{\wt{\tau}_2}((I, \wt{\chi}_2\mon) \bs \wt{LG}_2 / (I, \wt{\chi}_2\mon)) $$
induced by Proposition \ref{p:pinky}, and note that it remains a $t$-exact equivalence, which exchanges the corresponding (co)standard objects as in Theorem \ref{t:mainthmsec7}. 

By Lemma \ref{l:algbr}, it is enough to argue that, for each fixed $\alpha$, $\epsilon$ exchanges up to non-canonical equivalence the left modules 
\begin{equation} \label{e:leftmods}D_{\wt{\tau}_i}((I, \wt{\chi}_i\mon) \bs \wt{LG}_i / (I, \wt{\chi}_i\mon)) \circlearrowright D_{\wt{\tau}_i}((I, \wt{\chi}_i\mon) \bs \wt{LG}_i / (K^\alpha, \wt{\chi}_i\mon)).\end{equation}
We will in fact supply a canonical equivalence, as follows. By Lemma \ref{l:hitwithdelta}, the left modules  
$$D_{\wt{\tau}_i}((I, \wt{\chi}_i\mon) \bs \wt{LG}_i / (K^\alpha_i, \wt{\chi}_i\mon))$$
are canonically submodules of $D_{\wt{\tau}_i}((I, \wt{\chi}_i\mon) \bs \wt{LG}_i / (I, \wt{\chi}_i\mon))$, obtained by right convolution with the idempotents $\delta_i^\alpha$ given by the monoidal units of $$D_{\wt{\tau}_i}((K^{\alpha}_i, \wt{\chi}_i\mon) \bs K^{\alpha}_i / (K^{\alpha}_i, \wt{\chi}_i\mon)).$$

It is therefore enough to check that the equivalence $\epsilon$ exchanges, up to isomorphism, the objects $\delta_i^\alpha$. To see this, it is enough to obtain an expression for $\delta_i^\alpha$ in terms of objects we know are exchanged under $\epsilon$, which we do as follows.

Let us write $w_{\circ, i}^\alpha$ for the longest element of the parabolic subgroup $W_i^\alpha \subset \wt{W}_{\wt{LG}_i}$ corresponding to $K_i^\alpha$. Consider the corresponding standard and costandard objects 
$j(w_{\circ, i}^\alpha)_{!, \wt{\chi}_i}$ and $j(w_{\circ, i}^\alpha)_{\wt{\chi}_i. *}$. As these are exchanged under $\epsilon$, by the $t$-exactness of $\epsilon$ so is the image factorization of the canonical map  
$$j(w_{\circ, i}^\alpha)_{\wt{\chi}_i, !} \twoheadrightarrow I_i \hookrightarrow j(w_{\circ, i}^\alpha)_{\wt{\chi}_i, *}.$$ 
We then note that $I_i$ is simply the character sheaf $\wt{\chi}_i$ on $K_i^\alpha$, up to a shift by $\ell(w_{\circ, i}^\alpha)$. In particular, if we write $(I_i)$ for the full subcategory generated under colimits and shifts by $I_i$, we have an adjunction 
$$\iota_!: (I_i) \rightleftharpoons D_{\wt{\tau}_i}((I, \wt{\chi}_i\mon) \bs \wt{LG}_i / (I, \wt{\chi}_i\mon)): \iota^!,$$and $\delta_i^\alpha$ is obtained by applying $\iota^!$ to the monoidal unit. This implies that $\epsilon$ exchanges $\delta_1^\alpha$ and $\delta_2^\alpha$, in fact up to canonical isomorphism, as desired. 

Finally, to address the equivalence of strictly equivariant Hecke 2-categories, 
note that the equivalence $\epsilon$ exchanges the augmentation modules 
$$\End(\delta_i^\alpha) \circlearrowright \Hom(I_i, \delta_i^\alpha).$$
In particular, it follows from the canonical equivalence \eqref{e:leftmods} and Proposition \ref{p:avenh} that $\epsilon$ exchanges up to canonical equivalence the left modules 
\begin{equation*} D_{\wt{\tau}_i}((I, \wt{\chi}_i\mon) \bs \wt{LG}_i / (I, \wt{\chi}_i\mon)) \circlearrowright D_{\wt{\tau}_i}((I, \wt{\chi}_i\mon) \bs \wt{LG}_i / (K^\alpha, \wt{\chi}_i)).\end{equation*}
Therefore, we are done by Lemma \ref{l:algbr}.
\end{proof}

\section{Applications to the quantum Langlands correspondence} \label{s: Quantum Langlands}

Our goal in this final section is to apply the general structural results  developed in this paper to some concrete problems in representation theory. While there are many such possible applications, some more of which we will return to in a  sequel, here we concern ourselves with a pair of conjectures in the local quantum geometric Langlands correspondence. 

The first of these conjectures identifies the affine Hecke 2-categories associated to Langlands dual split groups at dual Kac--Moody levels.
We emphasize this assertion makes sense for all levels.  In particular, we may consider the case of a purely central monodromy at a generic level, where it essentially reduces to the celebrated identification of the finite Hecke 2-categories of Langlands dual groups due to Soergel. 

Unlike their affine counterparts, the finite Hecke 2-categories do not appear to have natural descriptions as derived categories of coherent sheaves. For this reason, the second of these conjectures focuses on rational levels, where it gives a spectral description of the affine Hecke 2-category in terms of coherent sheaves on the Steinberg stacks of the {\em metaplectic dual group}. As a particular case of the above, we obtain the {\em metaplectic derived Satake equivalence}.

\subsection{Setup and notation} 

\subsubsection{}  Let $k$ be an algebraically closed field of characteristic zero. 

\subsubsection{}  In this section, we are only concerned with untwisted loop groups. So, as in Section \ref{ss:loop group}, let $\GG$ be a reductive group over $k$ equipped with a pinning and in particular a Borel and Cartan subgroup \begin{equation*} \GG \supset \BB \supset \AA.\end{equation*}Let us denote the corresponding Lie algebras by $\gg \supset \bb \supset \aa$. Fix in addition a nondegenerate {\em level}, i.e., a nondegenerate Ad($\GG$)-invariant bilinear form 
\begin{equation*}\kappa: \gg \otimes \gg \rightarrow k.\end{equation*}

\subsubsection{}  Denote the Langlands dual of $\GG$, along with the Borel and Cartan induced by its dual pinning, by 
\begin{equation*}\chGG \supset \chBB \supset \chAA, \end{equation*}
with respective Lie algebras $\chgg \supset \chbb \supset \chaa.$ Let us denote the dual level by 
\begin{equation*}\ckappa: \chgg \otimes \chgg \rightarrow k.\end{equation*}
That is, $\ckappa$ is the unique Ad($\chGG)$-invariant form with the property that, after restriction to the respective Cartan subalgebras, 
\begin{equation*}\kappa: \aa \otimes \aa \rightarrow k \quad \text{and} \quad \ckappa: \chaa \otimes \chaa \rightarrow k.\end{equation*}
are dual bilinear forms.

\begin{rem} We recall that in the quantum geometric Langlands literature, one usually asks that $\kappa$ and $\ckappa$ be dual forms only after shifting both by the critical levels for $\gg$ and $\chgg$. As we are only concerned with affine Hecke 2-categories below, and the twistings associated to critical levels arise from line bundles, we may and do suppress this shift in what follows.     
\end{rem}

\sss{} We denote the Weyl group of $\GG$ and $\chGG$ by $\WW$, and the coroot and cocharacter lattices of $\AA$ and $\chAA$ respectively by 
	\begin{equation*}\check{Q} \subset \cLambda \subset \chaa \and Q \subset  \Lambda \subset \aa.\end{equation*} 

\subsubsection{} 
Let us denote the loop and arc groups of $\GG$ by $LG$ and $L^{+}G$, and write $I = \II \subset L^{+}G$ for the Iwahori subgroup induced by our choice of $\BB$. Let us denote the corresponding groups associated to $\chGG$ by 
\begin{equation*}L\chG \supset L^{+}\chG \supset \chI.\end{equation*}

\subsubsection{}  
Recall that to the form $\kappa$ we may canonically associate a category of twisted D-modules on $LG$. Explicitly, if we write $\on{Tw}(LG)$ for the $k$-vector space of isomorphism classes of twistings on $LG$, the twisting assigned to $\kappa$ is uniquely specified by the conditions that it (i) arises from a $k$-linear map 
$$(\Sym^2 \gg^*)^\GG \rightarrow \on{Tw}(LG),$$
and (ii) if $\kappa$ is the trace form on a representation $V$, the associated twisting is the one arising from the corresponding determinant line bundle on $LG$.

The monodromy associated to $\kappa$ is canonically trivialized on $L^{+}G$ and we restrict this to obtain a trivialization on $I$. Therefore, as in Section \ref{ss: monodromic strict hecke}, we may form the monodromic affine Hecke category 
\begin{equation*}\Dmod_\kappa(I \bs LG / I).\end{equation*}
Let us denote similarly the monodromic affine Hecke category associated to $-\ckappa$ on $L\chG$ by 
\begin{equation*}\Dmod_{-\ckappa}(\chI \bs L\chG / \chI).\end{equation*}

\subsection{Quantum Langlands duality for affine Hecke 2-categories}

\subsubsection{}  Our first goal in this section is to prove the following statement, which was conjectured by Gaitsgory.  

\begin{conj} [Conjecture 3.10 of \cite{Ga}] \label{c:GaitsgIwahori} There is an equivalence of monoidal categories
	\begin{equation*}\Dmod_\kappa(I \bs LG / I) \simeq \Dmod_{-\ckappa}(\chI \bs L\chG / \chI). \end{equation*}
\end{conj}

In fact, the same considerations in local quantum Langlands which lead to the statement of Theorem \ref{t:GaitsgIwahori}, namely the equivalence for tori (in depth zero) and an expected compatibility for parabolic induction, predict a more general series of equivalences which incorporate a further twist along the Iwahori subgroup.

 To set up their statement, let us identify $\aa$ and $\chaa$ using $$\kappa: \aa \rightarrow \chaa, \quad \quad \clambda \mapsto \kappa(\clambda, -) \in \chaa,$$with inverse $\ckappa: \chaa \rightarrow \aa.$ Then, for an element \begin{equation*}\theta \in \chaa \simeq \aa \ni  \ckappa(\theta) =: \cht,\end{equation*} we have associated rigidified character sheaves on $\AA$ and $\chAA$, and hence may form the monoidal categories of twisted D-modules
\begin{equation} \label{e:qahc}\Dmod_\kappa((I, \theta) \bs LG / (I, \theta)) \quad \text{and} \quad \Dmod_{-\ckappa}((\chI, \cht) \bs L\chG / (\chI, \cht)).\end{equation}
Local quantum Langlands then predicts the following.

\begin{thm} \label{t:GaitsgIwahori} \label{thm: dual with finite twist}There is an equivalence of monoidal categories 
    \begin{equation*}\Dmod_\kappa((I, \theta) \bs LG / (I, \theta))  \simeq \Dmod_{-\ckappa}((\chI, \cht) \bs L\chG / (\chI, \cht)). \end{equation*} 
\end{thm}

Note that, if we specialize to $\theta = 0$, this confirms Conjecture \ref{c:GaitsgIwahori}.

\begin{proof}[Proof of Theorem \ref{t:GaitsgIwahori}] For ease of reading, we break the proof into the several steps.

{\em Step 1.} To begin an analysis of the relevant combinatorics, we first set up some notation. Consider the affine Lie algebra 
$$0 \rightarrow k \cdot \mathbf{1} \rightarrow \widehat{\gg}_\kappa \rightarrow \gg(\!(t)\!) \rightarrow 0,$$
which we recall is canonically split as a vector space with commutator pairing 
$$[X \otimes p, Y \otimes q] = [X,Y] \otimes (p \cdot q) - \kappa(X,Y) \cdot  \underset{t = 0}{\on{Res}}(p \cdot dq),$$
for $X,Y \in \gg,$ and $p,q \in k(\!(t)\!).$ In particular, the centrally extended Cartan refers to the pulled back trivial extension 
$$0 \rightarrow k \cdot \mathbf{1} \rightarrow \widehat{\aa}_\kappa \rightarrow \aa \rightarrow 0, \quad \quad \widehat{\aa}_\kappa \simeq (k \cdot \mathbf{1}) \oplus \aa.$$
and its dual therefore takes the form$$0 \rightarrow \chaa \rightarrow (\widehat{\aa}_\kappa)^* \rightarrow (k \cdot \mathbf{1})^* \rightarrow 0, \quad \quad (\widehat{\aa}_\kappa)^* \simeq (k \cdot \mathbf{1}^*) \oplus \chaa,$$
where $\mathbf{1}^* \in (\widehat{\aa}_\kappa)^*$ denotes the unique element which is perpendicular to $\aa$ and satisfies $\langle \mathbf{1}^*, \mathbf{1} \rangle = 1$. 

Recall the extended affine Weyl group of $\tilW_{LG} = \check{\Lambda} \rtimes \WW$ naturally acts on $(\widehat{\aa}_\kappa)^*$, and that this action preserves the fibers of the projection $(\widehat{\aa}_\kappa)^* \rightarrow (k \cdot \mathbf{1})^*$, which are torsors over $\chaa$. This action is explicitly given as follows. For any element $\lambda \in \chaa$, consider the automorphism 
$$\tau^\lambda := \begin{pmatrix} \id & \\ \lambda & \id \end{pmatrix}: (k \cdot \mathbf{1}^*) \oplus \chaa \simeq  (k \cdot \mathbf{1}^*) \oplus \chaa,$$
and for any $w \in \WW$, consider the automorphism 
$$\rho^w := \begin{pmatrix} \id & \\ & w \end{pmatrix}: (k \cdot \mathbf{1}^*) \oplus \chaa \simeq  (k \cdot \mathbf{1}^*) \oplus \chaa.$$
Then in this notation, $w \in \WW$ acts via $\rho^w$, and $\clambda \in \check{\Lambda}$ acts via $\tau^{-\kappa(\clambda)}$. In particular, if we consider the action on the affine hyperplane $\mathbf{1}^* + \chaa$, we have the formula 
$$(t^{\clambda} w)( \mathbf{1}^* + \lambda) = \mathbf{1}^* + w(\lambda) - \kappa(\clambda).$$

Similarly, we may consider the affine Lie algebra 
$$0 \rightarrow k \cdot \check{\mathbf{1}} \rightarrow \widehat{\chgg}_{-\ckappa} \rightarrow \chgg(\!(t)\!) \rightarrow 0,$$
along with its centrally extended Cartan subalgebra $$0 \rightarrow k \cdot \check{\mathbf{1}} \rightarrow \widehat{ \chaa}_{-\ckappa} \rightarrow \chaa \rightarrow 0, \quad \quad \widehat{\chaa}_{-\ckappa} \simeq (k \cdot \check{\mathbf{1}}) \oplus \chaa.$$Again for the dual centrally extended Cartan $(\widehat{\chaa}_{-\ckappa})^* \simeq (k \cdot \check{\mathbf{1}}^*) \oplus \aa$ for $\clambda \in \aa$ and $w \in W$ we may consider the operators
$$\tau^{\clambda} := \begin{pmatrix} \id \\ \clambda & \id \end{pmatrix}: (k \cdot \check{\mathbf{1}}^*) \oplus \aa \simeq (k \cdot \check{\mathbf{1}}^*) \oplus \aa$$
$$\check{\rho}^w := \begin{pmatrix} \id \\ & w \end{pmatrix}: (k \cdot \check{\mathbf{1}}^*) \oplus \aa \simeq (k \cdot \check{\mathbf{1}}^*) \oplus \aa,$$
Then the action of $\tilW_{L\check{G}} = \Lambda \rtimes \WW$ is given by $w \in \WW$ acts by $\check{\rho}^w$ and $\lambda \in \Lambda$ acts via $\tau^{\ckappa(\lambda)}$. In particular, when restricting to the affine hyperplane $\check{\mathbf{1}}^* + \aa$, we have the formula 
$$(t^\lambda w)(\check{\mathbf{1}}^* + \clambda) = \check{\mathbf{1}}^* + w(\clambda) - (-\ckappa)(\lambda).$$

{\em Step 2.} We may now begin the analysis of the relevant combinatorics.  Let us denote the integral Weyl groups of the affine Hecke categories \eqref{e:qahc} in the extended affine Weyl groups of $\GG$ and $\chGG$   by\begin{equation*}\extw_{LG,\kappa, \theta} \subset \WW \ltimes \cLambda \quad  \text{and} \quad  \extw_{L\chG,-\ckappa, \cht} \subset \WW \ltimes \Lambda,\end{equation*}respectively.  Let us determine these integral Weyl groups explicitly.  From e.g. the description of the $\extw$ action on $\wt \AA$ in \eqref{eq: commutator pairing}, we know that an element $$t^{\clambda} w \in \cLambda \rtimes \WW$$ belongs to $\extw_{LG,\kappa, \theta}$ if and only if it fixes the element $\mathbf{1}^* + \theta$ up to a translation by the character lattice $\Lambda$, i.e., 
\begin{equation} \label{eq:fulllattice} w(\theta) - \theta - \kappa(\clambda) = -\lambda, \quad \text{for some } \lambda \in \Lambda.\end{equation}Similarly, an element $$t^{\lambda}w \in \Lambda \rtimes \WW$$belongs to $W_{\chGG, -\ckappa, \cht}$ if and only if it satisfies
\begin{equation}\label{eq:fulllatticedual}w(\cht) - \cht - (- \ckappa)(\lambda) = \clambda, \quad \text{for some } \clambda \in \cLambda.\end{equation}
For later use in Step 4, let us note the following. Consider a pair $(\clambda, \lambda) \in \cLambda \times \Lambda$ satisfying Equation \eqref{eq:fulllattice}. Applying $\ckappa$ to both sides, we obtain this is equivalent to
$$w(\cht) - \cht - \clambda = \ckappa(-\lambda).$$Rearranging terms, this again is equivalent to 
$$w(\cht) - \cht - (-\ckappa)(\lambda) = \clambda.$$
That is, the pair $(\clambda, \lambda)$ satisfies Equation \eqref{eq:fulllattice} if and only if it satisfies Equation \eqref{eq:fulllatticedual}. Equivalently, for the pair $(\clambda, \lambda)$, the element $t^{\clambda}w$ lies in the integral Weyl group of $\theta$ if and only if $t^{\lambda}w$ lies in the integral Weyl group of $\cht$. 

{\em Step 3.} Let us denote the neutral blocks of the integral Weyl groups by $$\extw_{LG,\kappa, \theta}^{\circ} \subset \tilW_{LG, \kappa, \theta} \quad \text{and} \quad \extw_{L\chG,-\ckappa, \cht}^{\circ} \subset \extw_{L \chG, -\ckappa, \cht}.$$To describe these explicitly, we must determine the integral affine coroots for $\theta$ and $\cht$.

Let us write $\cPhi_{\GG}$ for the coroots of $\GG$, and $\cPhi_{LG}$ for the real affine coroots of $LG$. We recall that there is a canonical bijection of sets \begin{equation*}\cPhi_{LG} \simeq \cPhi_{\GG} \times \Z,\end{equation*}which assigns to a pair $(\halpha, n) \in \cPhi_{\GG} \times \Z$ the associated coroot $\halpha_n$ in the centrally extended Cartan given by the formula
\begin{equation*}\halpha_n := \halpha + n \cdot \frac{\kappa(\halpha, \halpha)}{2} \cdot \mathbf{1}. \end{equation*}
In particular, we have that $\halpha_n$ is an integral coroot of $\theta$ if and only if we have 
\begin{equation} \label{e:rodentdream} \langle \mathbf{1}^* + \theta, \halpha_n \rangle =  \langle \theta,  \halpha \rangle + n \cdot  \frac{\kappa(\halpha, \halpha)}{2} \in \mathbb{Z}, \end{equation} which is equivalent to having the associated reflection $s_{\halpha_{n}} =  t^{-n \halpha}s_{\halpha}$ satisfy 
\begin{equation} \label{eq:soluble} 
 (\mathbf{1}^* + \theta) - s_{\halpha_n}(\mathbf{1}^* + \theta)  = \theta - s_{\halpha}(\theta) + \kappa(n \halpha) \in \mathbb{Z} \alpha.\end{equation}
Similarly, an affine coroot of $L\chG$ \begin{equation*}\a_m := \a + m \cdot \frac{-\ckappa(\alpha, \alpha)}{2} \cdot \check{\mathbf{1}} \end{equation*} is an integral coroot of $\cht$ if and only if \begin{equation}\label{eq:solubledual}  \cht - s_{\alpha}(\cht)  + (-\ckappa)(m\alpha) \in \mathbb{Z}\halpha.\end{equation}

{\em Step 4.} Consider the inverse linear isomorphisms 
$$\iota:= \begin{pmatrix} 1 & \\ \cht & -\ckappa \end{pmatrix}: (k \cdot \mathbf{1}^*) \oplus \chaa \simeq (k \cdot \check{\mathbf{1}}^*) \oplus \aa: \begin{pmatrix} 1 & \\ \theta & -\kappa \end{pmatrix}=: \iota^{-1},$$
where the upper left entries of $`1$' denotes the isomorphisms $k \cdot \mathbf{1}^* \simeq k \cdot \check{\mathbf{1}}^*$ exchanging $\mathbf{1}^*$ and $\check{\mathbf{1}}^*$. It is straightforward to check these exchange translations with translations, i.e., for any $\lambda \in \chaa$, one has the equality 
$$ \iota \circ \tau^{\lambda} = \tau^{-\ckappa(\lambda)} \circ \iota.$$
Similarly, this exchanges the affine linear action of $\WW$ on $\mathbf{1}^* +\chaa$ centered at $\mathbf{1}^*$ to the affine linear action of $\WW$ on $\check{\mathbf{1}}^* + \aa$ centered at $\check{\mathbf{1}}^* + (-\cht)$, i.e., 
$$\iota \circ \rho^w =  \tau^{w(-\cht) - (-\cht)} \circ \check{\rho}^w \circ \iota.$$
In particular, given an element $t^{\clambda}w \in W_{LG, \kappa, \theta}$, again writing as in Equation \eqref{eq:fulllattice}
$$w(\theta) - \theta - \kappa(\clambda) = -\lambda, \quad \quad \lambda \in \Lambda,$$
we unwind that 
\begin{align*} \iota \circ (t^{\clambda}w) &= \iota \circ t^{-\kappa(\clambda)} \circ \rho^w \\ &= \tau^{w(-\cht) - (-\cht) + \clambda} \circ \check{\rho}^w \circ \iota \\ &= \tau^{\ckappa(\lambda)} \circ \check{\rho}^w \circ \iota \\ &= t^{\lambda}w \circ \iota.\end{align*}
Therefore, by the analysis at the end of Step 2, it follows that conjugation by $\iota$ induces an isomorphism of groups 
\begin{equation*}\extw_{LG,\kappa, \theta} \simeq \extw_{L\chG,-\ckappa, \cht}.\end{equation*}

Moreover, for a reflection $t^{n\halpha} s_{\halpha}$ in $\tilW_{LG, \kappa, \theta}^\circ$, writing as in Equation \eqref{eq:soluble}$$s_{\halpha}(\theta) - \theta - \kappa(n\halpha) = - m \alpha, \quad \quad m \in \ZZ,$$
we in particular have $\iota$ exchanges $t^{n\halpha}s_{\halpha}$ with $t^{m\alpha} s_{\alpha}$. Therefore by Equation \eqref{eq:soluble} and \eqref{eq:solubledual}, it follows that $\iota$ exchanges reflections in $\extw_{LG,\kappa, \theta}^\circ$ with reflections in $\extw_{L\chG,-\ckappa, \cht}^{\circ}$, and therefore $\iota$ restricts to an isomorphism between the neutral blocks.

{\em Step 5.} Consider the obtained isomorphism $$\iota: \tilW_{LG, \kappa, \theta} \simeq \tilW_{L\chG, -\ckappa, \cht}.$$In this step, and its two sequels, we will argue that, up to postcomposition with conjugation by an element of the neutral block, this is an isomorphism of extended Coxeter systems, i.e., exchanges the simple reflections and length zero elements on either side. \footnote{Let us include in passing two orienting comments for the reader. First, we remind that a given abstract group can have distinct Coxeter presentations up to conjugacy, e.g. the Weyl group of $G_2$ is abstractly isomorphic to the Weyl group of $SL_2 \times SL_3$, but this cannot exchange the associated reflection representations. With this in mind, we also note that if $\kappa$ and $\theta$ are defined over $\QQ$, or more generally the real numbers $\RR$, the assertion of this step is evident, as the identification $\iota$ of the dual affine Cartans by construction restricts to an isomorphism of the real reflection representations; below we adapt this observation to general $\kappa$ and $\theta$.} The element in the neutral block will be non-unique in general, so we will explicitly characterize the ambiguity as well. 

To begin, note that by construction $\iota$ exchanges the two canonical projections to the finite Weyl group 
$$\tilW_{LG, \kappa, \theta} \hookrightarrow \tilW_{LG} \twoheadrightarrow \WW \quad \text{and} \quad \tilW_{L\chG, -\ckappa, \cht} \hookrightarrow \tilW_{L\chG} \twoheadrightarrow \WW.$$In particular, if we denote the images of the two maps by $\WW_{G, \theta}$ and $\WW_{\chG, \cht}$, we have these subgroups of $\WW$ coincide.

We next record the observation that, by construction, the underlying linear isomorphism $\iota$ of dual affine Cartans tautologically fits into a short exact sequence of $\tilW_{LG, \kappa, \theta} \simeq \tilW_{L\chG, -\ckappa, \cht}$-modules
\begin{equation} \label{e:SESaffinecartan}\xymatrix{0 \ar[r] & \chaa \ar[d]_{-\ckappa} \ar[r] & (\widehat{\aa}_{\kappa})^* \ar[r] \ar[d]^\iota & k \cdot \mathbf{1}^* \ar[d]^1 \ar[r] & 0 \\ 0 \ar[r] & \aa \ar[r] & (\widehat{\chaa}_{-\ckappa})^* \ar[r] & k \cdot \check{\mathbf{1}}^* \ar[r] & 0.}\end{equation}
Notice further that within $\chaa \simeq \aa$, the action of $\tilW_{LG, \kappa, \theta} \simeq \tilW_{L\chG, -\ckappa, \cht}$ factors through $\WW_{G, \theta} \simeq \WW_{\chG, \cht}$. The goal of the next step is to study this action. 

{\em Step 6.} If we similarly denote by $\WW_{G, \theta}^\circ \simeq \WW_{\chG, \cht}^\circ$ the images of the neutral blocks $$\tilW_{LG, \kappa, \theta}^\circ \simeq \tilW_{L\chG, -\ckappa, \cht}^\circ,$$these are generated by the reflections corresponding to the finite root systems given by the images of, in evident notation, the compositions
\begin{equation} \label{e:projroot}\cPhi_{LG, \kappa, \theta} \hookrightarrow \cPhi_{LG} \twoheadrightarrow \cPhi_G \quad \quad \text{and} \quad \quad \cPhi_{L\chG, -\ckappa, \cht} \hookrightarrow \cPhi_{L\chG} \twoheadrightarrow \cPhi_{\chG}.\end{equation}Let us denote these images by 
$$\cPhi_{G, \theta} \subset \cPhi_G, \quad \quad \text{and} \quad \quad \cPhi_{\chG, \cht} \subset \cPhi_{\chG},$$respectively. To see that they are indeed root subsystems of $\cPhi_{G}$ and $\cPhi_{\chG}$, one must check that for each element of them, e.g. $\halpha \in \cPhi_{G, \theta}$, the corresponding reflection $$s_{\halpha}: \cPhi_G \simeq \cPhi_G.$$restricts to an involution of $\cPhi_{G, \theta}$. However, this follows from the fact that some lift ${\halpha_n}$ lies in $\cPhi_{LG, \kappa, \theta}$, for $n \in \ZZ$, and the fact that $s_{\halpha_n}$ preserves $\cPhi_{LG, \kappa, \theta}$.

In particular, the root systems $\cPhi_{G, \theta}$ and $\cPhi_{\chG, \cht}$, without fixing a preferred set of positive roots, equip $\WW^\circ_{G, \theta}$ and $\WW^\circ_{\chG, \cht}$ with a system of Coxeter generators well-defined up to conjugacy. Moreover, as by construction the two are interchanged under the canonical bijection $\cPhi_G \simeq \cPhi_{\chG}$, it follows that the identification $$\iota: \WW^\circ_{G, \theta} \simeq \WW^\circ_{\chG, \cht}$$exchanges these conjugacy classes of Coxeter presentations, and in particular the associated reflection representations. 

 Having discussed the images of the neutral blocks, let us now turn to the action of the entire integral affine Weyl groups. Note that the natural actions of $\tilW_{LG, \kappa, \theta}$ and $\tilW_{L\chG, -\ckappa, \cht}$ on $\cPhi_{LG, \kappa, \theta}$ and $\cPhi_{L\chG, -\ckappa, \cht}$, respectively, descend to actions 
$$\WW_{G, \theta} \circlearrowright \cPhi_{G, \theta} \quad \quad \text{and} \quad \quad \WW_{\chG, \cht} \circlearrowright \cPhi_{\chG, \cht}.$$
In particular, if we denote by $\on{Irr}(\cPhi_{G, \theta})$ the set of constituents of the unique decomposition of $\cPhi_{G, \theta}$ as a disjoint union of irreducible and mutually orthogonal root subsystems, we have an induced action of $\WW_{G, \theta}$ on $\on{Irr}(\cPhi_{G, \theta})$. In particular, if $\mathbb{O}_i, i \in \mathscr{I},$ denotes the set of orbits of the action of $\WW_{G, \theta}$ on $\on{Irr}(\cPhi_{G, \theta})$, we obtain a decomposition 
\begin{equation} \label{e:breakrootsystem} \cPhi_{G, \theta} = \underset{i \in \mathscr{I}}\sqcup \hspace{.5mm} \cPhi_{G, \theta}^i,\end{equation}
where $\cPhi_{G, \theta}^i$ is by definition the union of the root subsystems in $\mathbb{O}_i$. Note that, if we define similarly the decomposition 
$$\cPhi_{\chG, \cht} = \bigsqcup_{i \in \check{\mathscr{I}}} \cPhi^i_{\chG, \cht},$$
it is clear these decompositions are exchanged under the canonical bijection $\cPhi_{G, \theta} \simeq \cPhi_{\chG, \cht}$, and in particular we have a canonical identification $\mathscr{I} \simeq \check{\mathscr{I}}$; we denote both simply by $\mathscr{I}$ going forwards. 

For $i \in \mathscr{I}$, denote by $\chaa_i \subset \chaa$ the span of the elements of $\cPhi_{G, \theta}^i$. Summing over $i$ the inclusions, we obtain a canonical direct sum decomposition as vector spaces
$$\chaa = (\underset{i \in \mathscr{I}} \oplus \hspace{.5mm} \chaa_i ) \oplus \chaa^{\WW^\circ_{G, \theta}}.$$
Moreover, it is straightforward to observe that this is in fact a decomposition as $\WW_{G, \theta}$-representations, and that each $\chaa_i$ is an irreducible $\WW_{G, \theta}$-module, with underlying $\WW_{G, \theta}^\circ$-module the direct sum of the reflection representations of the components belonging to $\mathbb{O}_i$. In addition, note that the identification $-\ckappa: \chaa \simeq \aa$ exchanges this decomposition with the analogously defined 
$$\aa =  (\underset{i \in \mathscr{I}} \oplus \hspace{.5mm} \aa_i ) \oplus \aa^{\WW^\circ_{G, \theta}}.$$

In particular, for any $i \in \mathscr{I}$, we may push out \eqref{e:SESaffinecartan} along the projections $$(\chaa \simeq \aa) \rightarrow (\chaa_i \simeq \aa_i)$$to obtain a short exact sequence of $\tilW_{LG, \kappa, \theta} \simeq \tilW_{L\chG, -\ckappa, \theta}$-modules 
\begin{equation} \label{e:2SESaffinecartan}\xymatrix{0 \ar[r] & \chaa_i \ar[d]_{-\ckappa} \ar[r] & \wt{\chaa}_i \ar[r] \ar[d]^\iota & k \cdot \mathbf{1}^* \ar[d]^1 \ar[r] & 0 \\ 0 \ar[r] & \aa_i \ar[r] & \wt{\aa}_i \ar[r] & k \cdot \check{\mathbf{1}}^* \ar[r] & 0.}\end{equation}
In the next step, we would like to explain how to recover the extended Coxeter presentations from these representations, and in particular to deduce they agree up to conjugacy from the isomorphism \eqref{e:2SESaffinecartan}. To implement this, we must really analyze certain canonical $\QQ$-forms of the representations instead, and argue that the isomorphism \eqref{e:2SESaffinecartan} ensures the $\QQ$-forms are isomorphic as well. 

{\em Step 7.} To proceed, we recall that the action of $\tilW_{LG}$ on the $k$-vector space $(\widehat{\aa}_\kappa)^*$ is canonically extended from an action on the character lattice of $\wt{\AA}$, and in particular we have in evident notation a $\QQ$-form of the representation 
$$(\widehat{\aa}_\kappa)^* \simeq   k \underset{\QQ} \otimes (\widehat{\aa}_{\kappa})^*_\QQ,$$
which again sits in a short exact sequence $$0 \rightarrow \chaa_\Q \rightarrow (\widehat{\aa}_\kappa)^*_\QQ \rightarrow \QQ \cdot \mathbf{1}^* \rightarrow 0.$$
In particular, we may restrict this representation to $\tilW_{LG, \kappa, \theta}$. Via the tautological inclusion $\cPhi_{G, \theta} \subset \chaa_{\QQ} \subset \chaa$, the previous analysis similarly affords a decomposition of the rational form of $\chaa$
$$\chaa_\QQ = (\underset{i \in \mathscr{I}} \oplus \hspace{.5mm} \chaa_{i. \QQ} ) \oplus \chaa_{\QQ}^{\WW^\circ_{G, \theta}},$$
and in particular a rational form of the pushout $\tilW_{LG, \kappa, \theta}$-module
$$0 \rightarrow \chaa_{i, \QQ} \rightarrow \widehat{\chaa}_{i, \QQ} \rightarrow \QQ \cdot \mathbf{1}^* \rightarrow 0.$$By construction, the extended Coxeter structure on $\tilW_{LG, \kappa, \theta}$, up to conjugacy by an element of $\tilW_{LG, \kappa, \theta}^\circ$, may be recovered from the collection of representations 
$$\widehat{\chaa}_{i, \QQ}, \quad \quad i \in \mathscr{I}.$$
Namely, as usual one can consider their extension of scalars to $\RR$, and the decomposition of $\mathbf{1}^* + \chaa_{i, \RR}$ minus the reflection hyperplanes of $\cPhi_{LG, \kappa, \theta}$ into connected components. With this, a choice of component in each $$\mathbf{1}^* + \chaa_{i, \RR},  \quad \quad i \in \mathscr{I},$$affords an extended Coxeter presentation, and $\tilW^\circ_{LG, \kappa, \theta}$ acts simply transitively on the set of $\mathscr{I}$-tuples of alcoves.  For the latter assertion, one notes that the decomposition \eqref{e:breakrootsystem} yields a product decomposition of Coxeter groups 
$$\tilW_{LG, \kappa, \theta}^\circ = \underset{i \in \mathscr{I}} \Pi \hspace{.5mm} \tilW_{LG, \kappa, \theta}^{\circ, i},$$
where $\tilW_{LG, \kappa, \theta}^{\circ, i}$ denotes the subgroup of $\tilW^\circ_{LG, \kappa, \theta}$ generated by the integral reflections whose corresponding finite coroot lies in $\cPhi_{G, \theta}^i$. Moreover, it is straightforward to see that this collection of Coxeter presentations all agree, up to conjugacy, with the one obtained via the embedding $\tilW_{LG, \kappa, \theta} \hookrightarrow \tilW_{LG}$. Moreover, by construction the conjugation action of $\tilW_{LG, \kappa, \theta}$ on its neutral block preserves each $\tilW^{\circ, i}_{LG, \kappa, \theta}$ separately, for $i \in \mathscr{I}$, so it follows the induced extended Coxeter presentations agree as well. 

By combining the observations of the previous paragraph with their analogues for $\tilW_{L\chG, -\ckappa, \cht}$, to prove the assertion of this step it is therefore enough to argue that the isomorphism $$ \iota: \tilW_{LG, \kappa, \theta} \simeq  \tilW_{L\chG, -\ckappa, \cht}$$exchanges, for each $i \in \mathscr{I}$,  the $\Q$-representations 
\begin{equation} \label{e:tail2cities}\widehat{\chaa}_{i, \QQ} \quad \text{and} \quad \widehat{\aa}_{i, \QQ}\end{equation}
up to isomorphism. I.e., we would like to obtain an isomorphism as in \eqref{e:2SESaffinecartan} for not merely the $k$-forms, but for the $\QQ$-forms as well.

To see this, we argue as follows. Note that a choice of $\WW$-equivariant isomorphism $\chaa_\QQ \simeq \aa_\QQ$ restricts to a $\WW_{G, \theta} \simeq \WW_{\chG, \theta}$-equivariant isomorphism 
$$\chaa_{i, \QQ} \simeq \aa_{i, \QQ}.$$
Recall in addition that, for two finite dimensional representations $V_\Q$ and $W_\Q$ of an abstract group $\Gamma$, we have a natural isomorphism 
\begin{equation} k \underset{\QQ} \otimes \Hom_\Gamma(V_\Q, W_\Q) \xrightarrow{\sim} \Hom_\Gamma(k \underset{\QQ} \otimes V_\QQ, k \underset{\QQ} \otimes W_\QQ) \label{e:homsext}.\end{equation}
In particular, both representations \eqref{e:tail2cities} are extensions of the trivial representation $\QQ$ by $\aa_{i. \QQ}$, and the latter two representations are both irreducible and have endomorphisms $\QQ$ by \eqref{e:homsext}. Again by \eqref{e:homsext}, it follows that each extension is split over $\QQ$ if and only if it is split over $k$, and by \eqref{e:2SESaffinecartan} the latter problems are equivalent. 

In particular, either both of the extensions \eqref{e:tail2cities} are split or both are non-split. If both are split, they are tautologically isomorphic. If they are both non-split, it is straightforward to see their endomorphisms are both $\QQ$, as a nontrivial extension of distinct irreducibles, both of whose endomorphisms are $\QQ$. By the latter observation, \eqref{e:2SESaffinecartan}, and \eqref{e:tail2cities}, it follows that $$\Hom_{\tilW_{LG, \kappa, \theta}}( \widehat{\chaa}_{i, \QQ}, \widehat{\aa}_{i, \QQ}) \quad \quad \text{and} \quad \quad \Hom_{\tilW_{LG, \kappa, \theta}}( \widehat{\aa}_{i, \QQ}, \widehat{\chaa}_{i, \QQ})$$
are both lines over $\Q$, and any nonzero elements are isomorphisms, as desired.

{\em Step 8.} The previous argument shows that $$\iota:  \tilW_{LG, \kappa, \theta} \simeq \tilW_{L\chG, -\ckappa, \cht}$$exchanges their extended Coxeter presentations, up to conjugacy by an element $y$ of the neutral block $\tilW^\circ_{LG, \kappa, \theta}$. The set of relevant $y$ form a torsor over the group $\mathscr{G}$ of elements $z$ of $\tilW_{LG, \kappa, \theta}$ such that 
$$z \circ - \circ z^{-1}: \tilW_{LG, \kappa, \theta} \simeq \tilW_{LG, \kappa, \theta}$$is an automorphism of the extended Coxeter system, i.e., preserves lengths. In this step, let us determine $\mathscr{G}$ explicitly. 

Let us denote the set of $i \in \mathscr{I}$ such that the corresponding Coxeter group $\tilW_{LG, \kappa, \theta}^{i. \circ}$ is finite by $$\mathscr{I}_f \subset \mathscr{I}.$$
Then, if we denote by $w_{\circ}^i, i \in \mathscr{I}_f$ their longest elements, we will show these pairwise commuting involutions belong to $\mathscr{G}$ and they generate it, i.e., 
\begin{equation} \label{e:whoisG} \mathscr{G} \simeq \underset{i \in \mathscr{I}_f} \Pi \hspace{.5mm} \ZZ/2\ZZ.\end{equation}

Indeed, this follows from noting that (i) any element $z$ of $\mathscr{G}$ must in particular preserve lengths within the neutral block $$z \circ - \circ z^{-1}: \tilW_{LG, \kappa, \theta}^\circ \simeq \tilW_{LG, \kappa, \theta}^\circ.$$If we write $\Upsilon$ for the set of connected components of the Coxeter diagram of $\tilW_{LG, \kappa, \theta}^\circ$, note the group canonically factors as the product of its parabolic subgroups corresponding to the $\upsilon \in \Upsilon$, i.e., 
$$\tilW_{LG, \kappa, \theta} \simeq \underset{\upsilon}\Pi \hspace{.5mm} W_\upsilon$$
It is straightforward to see that $z$ must be a product of longest elements in a union $\Upsilon_z$ of components for which the corresponding parabolic subgroups are finite. Moreover, (ii) as conjugation by $z$ must preserve length zero elements, it follows that $\Upsilon_z$ is stable under the action of the length zero elements $\Omega_{\kappa, \theta}$ on $\Upsilon$. Finally, (iii) to make contact with $\mathscr{I}_f$, recall the set $\on{Irr}(\cPhi_{\GG, \theta})$, cf. Step 6 of the argument. We will need the following observation. 

\begin{lem} For any $j \in \on{Irr}(\cPhi_{\GG, \theta})$, write $\cPhi^j_{\GG, \theta} \subset \cPhi_{\GG, \theta}$ for the corresponding irreducible coroot subsystem. Then its preimage under the projection 
$$\cPhi_{LG, \kappa, \theta} \twoheadrightarrow \cPhi_{\GG, \theta},$$
which we denote by $\cPhi^j_{LG, \kappa, \theta} \subset \cPhi_{LG, \kappa, \theta}$, is again an irreducible coroot system.  
\end{lem}

\begin{proof} For any $\halpha \in \cPhi^j_{\GG, \theta}$, if we denote its preimage by 
$$\cPhi^{\pm \halpha}_{LG, \kappa, \theta} \subset \cPhi^j_{LG, \kappa, \theta},$$a straightforward calculation using Equation \eqref{e:rodentdream} shows that $\cPhi^{\pm \halpha}_{LG, \kappa, \theta}$ either consists of two elements, generating a copy of the coroots of $SL_2$, or infinitely many elements, generating a copy of the real coroots of $LSL_2$, and in particular is irreducible. The claim of the lemma follows straightforwardly from this and the fact that $\cPhi^j_{\GG, \theta}$ itself is irreducible.      
\end{proof}

By the lemma, it follows that we have a canonical bijection $\on{Irr}(\cPhi_{\GG, \theta}) \simeq \Upsilon$. It is straightforward to see that this bijection is $\Omega_{\kappa, \theta}$-equivariant, for its natural action on $\Upsilon$ and its action on $\on{Irr}(\cPhi_{\GG, \theta})$ via the surjective composition 
$$\Omega_{\kappa, \theta} \hookrightarrow \tilW_{LG, \kappa, \theta} \twoheadrightarrow \WW_{\GG, \theta} \twoheadrightarrow \WW_{\GG, \theta}/ \WW_{\GG, \theta}^\circ,$$
and the natural action of $\WW_{\GG, \theta}/ \WW_{\GG, \theta}^\circ$ on $\on{Irr}(\cPhi_{\GG, \theta})$. In particular, $\Upsilon_z$ corresponds via this bijection to a subset of $\mathscr{I}_f$, which implies the desired identity \eqref{e:whoisG}.

{\em Step 9.} Finally, having dispensed with the combinatorial analysis, let us obtain the desired equivalence of monoidal categories. Denote by $$\iota': (\widehat{\aa}_\kappa)^* \simeq (\widehat{\chaa}_{-\ckappa})^*$$the postcomposition of $\iota$ with an element $y$ of the neutral block of $\tilW_{L\chG, -\ckappa, \cht}^\circ$ obtained from the previous four steps, so that $\iota'$ in particular induces an isomorphism of extended Coxeter systems
$$\iota': \tilW_{LG, \kappa, \theta} \simeq \tilW_{L\chG, -\ckappa, \cht}.$$

 Denote the formal completions of $(\widehat{\aa}_\kappa)^*$ and $(\widehat{\chaa}_{-\ckappa})^*$ at zero by $(\widehat{\aa}_\kappa)^*_{0}$ and $(\widehat{\chaa}_{-\ckappa})^*_0$, respectively. Note that, as we are in characteristic zero, and working in the de Rham setting, the logarithm of monodromy yields canonical isomorphisms of formal schemes 
$$(\widehat{\aa}_\kappa)^*_{0} \simeq \frf_{(\kappa, \theta)}, \quad \text{and} \quad (\widehat{\chaa}_{-\ckappa})^*_0 \simeq \frf_{(-\ckappa, \cht)}$$
compatibly with the action of $\tilW_{LG, \kappa, \theta}$ and $\tilW_{L\chG, -\ckappa, \cht}$, respectively. In particular, the isomorphism $\iota'$ yields, thanks to Theorem \ref{t:mainthmsec7}, a monoidal equivalence 
$$D_{\kappa\mon}((I, \theta\mon) \bs LG / (I, \theta\mon)) \simeq D_{-\ckappa\mon}((\chI, \cht\mon) \bs L\chG / (\chI, \cht\mon)).$$
To prove the theorem, it remains to check that the hypotheses of Proposition \ref{p:mon2strict} are satisfied. As in Remark \ref{r:mon2str}, condition (1) of the proposition is vacuously satisfied, as we are working over a field. For condition (2), we note that by the construction of $\iota'$, it exchanges the projections 
$$(k \cdot \mathbf{1}^*) \oplus \chaa \twoheadrightarrow k \cdot \mathbf{1}^* \overset{1} \simeq k \cdot \check{\mathbf{1}}^* \twoheadleftarrow (k \cdot \check{\mathbf{1}}^*) \oplus \aa.$$
By passing to formal completions at the origins, and the associated pullbacks on functions, we see the remaining condition (2) of Proposition \ref{p:mon2strict} is also satisfied, hence we are done. 
\end{proof}

\sss{} Having identified quantum Langlands dual affine Hecke categories, let us next identify the corresponding affine Hecke 2-categories. To set up the relevant canonical bijection between the appropriate standard parahorics for $LG$ and $L\chG$, which is sensitive to the level, let us assume that 
\begin{equation} \label{e:assumelevelrat} \kappa(\halpha, \halpha) \in \QQ^\times, \quad \quad \halpha \in \cPhi_\GG. \end{equation}

Thanks to this assumption, we obtain a canonical decomposition of the coroot system into the disjoint union of two coroot systems, depending on the sign of the level on each simple factor, namely $$\cPhi_\GG \simeq \cPhi_\GG^{\kappa_+} \sqcup \cPhi_\GG^{\kappa_-},$$
$$\cPhi_\GG^{\kappa_{\pm}} := \{ \halpha \in \cPhi_\GG: \kappa(\halpha, \halpha) \in \pm \QQ^{>0} \}.$$
In particular, we may consider the involution of the coroot system $$\mathbf{Chev}_{\kappa}: \cPhi_{\GG} \simeq \cPhi_{\GG}$$
given by the Chevalley involution on $\cPhi_\GG^{\kappa_+}$, i.e., $\halpha \mapsto - w_\circ(\halpha)$, where $w_\circ$ denotes the longest element of $\WW$, and the identity on $\cPhi_\GG^{\kappa_-}.$ 

\sss{} In particular, if we write $\Delta_{\GG}$ and $\Delta_{\chGG}$ for the sets of simple coroots of $\GG$ and $\chGG$, respectively, we may consider the composite bijection 
\begin{equation} i_\kappa: \Delta_{\GG} \overset{\mathbf{Chev}_\kappa} \simeq \Delta_\GG \simeq \Delta_{\chGG},\end{equation}
where the second identification is the standard bijection between the simple roots of $\GG$ and $\chGG$.

 Therefore, for any standard parabolic subgroup $\PP$ in $\GG$, i.e.,$$\GG \supset \PP \supset \BB,$$
we may use $i_\kappa$ to obtain a standard parabolic subgroup of $\chGG$ 
$$\chGG \supset \chPP_\kappa \supset \chBB.$$Similarly, let us denote the corresponding standard parahoric subgroups in the loop groups by 
$$L^+ \GG \supset I_{\PP} \supset I \quad \text{and} \quad L^+\chGG \supset I_{\chPP_\kappa} \supset \chI,$$
so in particular $L^+ \GG = I_{\GG}$ and $I_{\BB} = I$, and similarly on the dual side.  

\sss{} The desired equivalence of affine Hecke 2-categories then reads as follows.

\begin{cor} \label{c:quantLang2HEcke} Let $\kappa$ satisfy \eqref{e:assumelevelrat}. Then for any standard parabolics $\PP^1$ and $\PP^2$, there are equivalences
	\begin{equation*}\Dmod_\kappa( I_{\PP^1} \bs LG / I_{\PP^2}) \simeq \Dmod_{-\ckappa}(I_{\chPP^{1}_\kappa} \bs L\chG / I_{\chPP^{2}_\kappa}),\end{equation*}
compatibly with convolution, i.e., which fit into an equivalence of affine Hecke 2-categories 
$$\underset{\PP^1, \PP^2} \oplus \hspace{.5mm} \Dmod_\kappa( I_{\PP^1} \bs LG / I_{\PP^2}) \simeq \underset{\chPP^1, \chPP^2} \oplus \hspace{.5mm} \Dmod_{-\ckappa}(I_{\chPP^{1}_\kappa} \bs L\chG / I_{\chPP_\kappa^{2}}).$$
\end{cor}

\begin{proof} This follows from applying Theorem \ref{thm: parahoric equivalence} to the equivalence constructed in the proof of Theorem \ref{t:GaitsgIwahori} in the case of $\theta = 0$. Namely, we note that in the isomorphism between the extended Coxeter data in the proof of Theorem \ref{t:GaitsgIwahori}, in Step 8 one finds that the group $\mathscr{G}$ is trivial with our assumption on $\kappa$. Moreover, it is straightforward to see that the unique element $y$ as in Step 9 is the longest element of the parabolic subgroup of $\WW$ corresponding to $\cPhi_\GG^{\kappa_+}$, which leads to the desired matching of parahorics. \end{proof}

\begin{rem} \label{r:quantum2cat} The proof of Corollary \ref{c:quantLang2HEcke} also implies that one has an equivalence of affine Hecke 2-categories for any level, but in the proof of Theorem \ref{t:GaitsgIwahori} the group $\mathscr{G}$ becomes a product of $\mathbb{Z}/2$'s for each simple factor of $\gg$ on which $\kappa$ is irrational, corresponding to the Chevalley involution on the corresponding appearing finite Hecke category. 
\end{rem}

As a particular case of the preceding corollary and remark, we obtain the following quantum derived Satake isomorphism. 

\begin{thm} \label{thm: metaplectic derived Satake}There is an equivalence of monoidal categories
	\begin{equation*}\Dmod_\kappa(L^{+}G \bs LG / L^{+}G) \simeq \Dmod_{-\ckappa}(L^{+}\chG \bs L\chG / L^{+}\chG).\end{equation*}
	
\end{thm}

\subsection{Spectral description: the metaplectic endoscopic group}\label{ss:meta dual}
\subsubsection{} In the previous subsection, we identified quantum Langlands dual pairs of affine Hecke 2-categories, for all Kac--Moody levels. As discussed in the introduction to this section, for certain levels one expects a further identification of both with a 2-category of coherent sheaves associated to the Steinberg stack of another reductive group, the {\em metaplectic dual group}. In the remainder of this section, we will obtain the desired identification.
 
Unlike the prior subsection, which concerned D-modules, we will return to our general sheaf theoretic setting, with our coefficients $E$ a field of characteristic zero. 

\sss{} For $\GG$ as before a reductive group over $k$. Recall that the affine Cartan of $LG$ fits into a short exact sequence
$$1 \rightarrow \GG^{\cen}_m \rightarrow  \wt{\AA} \rightarrow \AA \rightarrow 1,$$
which is moreover split via the restriction of the canonical splitting of the central extension over $L^+G$. In particular, this splitting gives us a decomposition of the discrete groupoids of $E$-linear character sheaves
\begin{equation} \label{e:splitchars}\Ch(\wt{\AA}; E) \simeq \Ch(\GG^{\cen}_m; E) \times \Ch(\AA; E).\end{equation}
In particular, given a character sheaf $\chi$ on $\GG_m^{\cen}$, we may pull it back to $\wt{\AA}$ via this splitting, and we denote the resulting character sheaf by $\chi_c$. In what follows, we will assume that $\chi_{c}$ is of finite order.

We now review how to attach to the data of $\chi_c$ another split reductive group, the {\em metaplectic dual group} $H^\vee$. In fact, for our purposes it will be convenient to first consider instead the Langlands dual of $H^\vee$, which we call the {\em metaplectic endoscopic group $H$}.

\sss{} By our assumption that $\chi_c$ is of finite order, it is straightforward to check that the intersection 
$$\X_*(\AA_H) := \X_*(\AA)  \cap \tilW_{\wt{LG}, \chi_c}$$
is a lattice of finite index in $\X_*(\AA).$

\begin{defn} The {\em metaplectic endoscopic torus} $\AA_H$ is the torus over $k$ with cocharacter lattice $\X_*(\AA_H)$.  
\end{defn}

Note that, by definition, we have a homomorphism $\AA_H \rightarrow \AA$, and let us denote the corresponding map on character lattices by $$\X^*(\AA) \hookrightarrow \X^*(\AA_H),$$
which again is an inclusion of finite index. 

\sss{} The coroots of the metaplectic endoscopic group $H$ will be proportional to those of $\GG$. For a coroot $\halpha_\GG$ of $\GG$, we denote by $\halpha_{\GG, n}$ ($n\in\ZZ$) the corresponding real affine coroot of $LG$. Let us set 
$$\halpha_H := N_{\halpha} \cdot \halpha_\GG,$$
where $N_{\halpha} \in \ZZ^{> 0}$ is the minimal positive integer such that $\halpha_{\GG, N_{\halpha}}$ is an integral real affine coroot of $\chi_c$ in the sense of \S\ref{s:comboblock}. 

\begin{rem}
When $G^\der$ is almost simple, we make the definition of $\halpha_H$ more concrete. Let $\e=1$ if $G$ is of type $A,D$ or $E$, $\e=2$ if $G$ is of type $B,C$ or $F$, and $\e=3$ if $G=G_2$. Let $N$ be the  order of $\chi_c^{\ot  \j{\check{\a}_\GG, \check{\a}_\GG}/2}$ for any short coroot $\check{\a}_\GG$. Here $\j{\cdot,\cdot}$ is the bilinear pairing on $\xcoch(\AA)$ attached to the central extension $\wt{LG}$ defined in Section \ref{sss:commutator}. Then 
\begin{itemize}
    \item If $N$ is coprime to $\e$, then $\halpha_H= N \cdot \halpha_\GG$ for all coroots $\halpha_\GG$ of $\GG$;
    \item If $\e>1$ and $\e|N$, then
    \begin{equation}
        \halpha_H= \begin{cases}N\cdot  \halpha_\GG, & \mbox{ if $\halpha_\GG$ is short},\\
        (N/\e)\cdot \halpha_\GG, & \mbox{ if $\halpha_\GG$ is long}.
        \end{cases}
    \end{equation}
\end{itemize}
\end{rem}

A standard calculation in rank one shows that, if we denote by $$\alpha_H \in \X^*(\AA_H) \underset{\ZZ} \otimes \QQ$$the unique element in the $\QQ$-span of the corresponding root $\alpha_\GG$ of $\GG$ which satisfies 
$$\langle \halpha_H, \alpha_H \rangle = 2,$$
then $\alpha_H$ in fact lies in $\X^*(\AA_H) \hookrightarrow \X^*(\AA_H) \underset{\ZZ} \otimes \QQ$. 

In particular, as we vary over all coroots of $\GG$, we obtain subsets $$\cPhi_H \subset \X_*(\AA_H) \quad \quad \text{and} \quad \quad \Phi_H \subset \X^*(\AA_H),$$which moreover come equipped with tautological bijections 
$$\cPhi_{\GG} \simeq \cPhi_H, \quad \halpha_\GG \mapsto \halpha_H, \quad \text{and} \quad \Phi_{\GG} \simeq \Phi_H, \quad \alpha_\GG \mapsto \alpha_H.$$

\sss{} The definition of the desired group $H$ is then the following. 

\begin{defn} A {\em metaplectic endoscopic group} $H$ for $(\GG, \chi_c)$ is a connected reductive group equipped with an embedding $$\AA_H \hookrightarrow H$$as a Cartan subgroup, and with corresponding root datum given by the quadruple
$$(\X_*(\AA_H), \cPhi_H, \X^*(\AA_H), \Phi_H)$$as defined above. 
\end{defn}

It is a standard fact that a metaplectic endoscopic group $H$ exists, i.e., that the above quadruple is a root datum. The above definition then characterizes $H$ uniquely up to the adjoint action of $\AA_{H}$, i.e., the groupoid of such $H$ is canonically equivalent to $\pt / \AA_{H}$.

\begin{exmp}
    Suppose $\GG$ is the symplectic group $\on{Sp}_{2n}$, and we consider the basic central extension $L\on{Sp}_{2n}$, i.e., associated to the vector representation. 
    
    If $\chi_c$ is the unique character sheaf on $\GG_m$ of order two, then the metaplectic endoscopic group of $(\on{Sp}_{2n}, \chi_c)$ is given by $H \simeq \on{SO}_{2n+1}$.  
\end{exmp}

\subsubsection{} Having introduced the metaplectic endoscopic group, let us now turn to the promised relationship between the affine Hecke 2-categories of $(\GG, \chi_c)$ and $H$.

By construction, if we consider the set of positive roots of $\GG$, the corresponding subset of $\Phi_H$ is a set of positive roots for $H$. Let us denote by $$\AA_{H} \subset \BB_{H} \subset H$$the corresponding Borel subgroup. Let us denote the loop group of $H$ by $LH$, and the Iwahori subgroup corresponding to $\BB_H$ and the arc subgroup respectively by 
$$I_H \subset L^+ H \subset LH.$$

\sss{} Given a character sheaf on the finite torus $\chi_f \in \Ch(\AA; E)$, via the splitting \eqref{e:splitchars} we may form the monodromic affine Hecke category 
$$D_{\chi_c}((I, \chi_f) \bs LG / (I, \chi_f)).$$
On the other hand, via the map $\AA_H \rightarrow \AA$, we may pull back $\chi_f$ along this map and form the affine Hecke category without central twist
$$D((I_H, \chi_f) \bs LH / (I_H, \chi_f)).$$

\sss{} It is unfortunately not true that these two Hecke categories can always be identified in a reasonable way, e.g. in a manner which exchanges (co)standard objects.

\begin{exmp} Consider the pair $(\GG, \chi_c)$, where $\GG = \PSp_6$, and $\chi_c$ is the order two character sheaf on the basic central extension.  In this case, one has $H = \on{Sp}_6$. We saw in Section \ref{sss: example} that one may choose a twist $\chi_f$ on the finite Cartan such that the resulting integral affine Weyl group $\tilW_{LG, \chi_c, \chi_f}$ is not isomorphic, as an abstract group, to the semi-direct product of a finite group and a lattice. 

By contrast, for the metaplectic endoscopic group, one knows that $\tilW_{LH, \chi_f}$ is the semi-direct product of its intersection with the finite Weyl group $\WW$ and the lattice $\X_*(\AA_H)$. In particular, there cannot exist any isomorphism 
$$D_{\chi_c}((I, \chi_f) \bs LG / (I, \chi_f)) \overset{?} \simeq D((I_H, \chi_f) \bs LH / (I_H, \chi_f))$$which exchanges (co)standard objects, as the integral affine Weyl groups are not isomorphic even as abstract groups.

 One may construct similar examples for $\GG = \PSp_{4n+2}$ and  dually, by Theorem \ref{thm: dual with finite twist}, for $\GG = \Spin_{4n+3}$, for any $n \geqslant 1$.
\end{exmp}

\sss{} With the above example in mind, we will instead relate full monoidal subcategories on both sides, defined as follows. Recall the tautological short exact sequence 
$$1 \rightarrow \X_*(\AA) \rightarrow \tilW_{LG} \rightarrow \WW \rightarrow 1,$$
where $\WW$ denotes the finite Weyl group. In particular, we may consider the subgroup of the integral affine Weyl group generated by its Coxeter part and its intersection with the cocharacter lattice, which we denote by 
$$\tilW_{LG, \chi_c, \chi_f}^{\bullet} := \langle \tilW_{LG, \chi_c, \chi_f}^\circ, \tilW_{LG, \chi_c, \chi_f} \cap \X_*(\AA) \rangle \subset \tilW_{LG, \chi_c, \chi_f}. $$
In particular, from the tautological inclusion $\tilW^\circ_{LG, \chi_c, \chi_f} \subset \tilW^{\bullet}_{LG, \chi_c, \chi_f}$, it follows that the latter again naturally an extended Coxeter group, and explicitly corresponds to the union of blocks given by the image of the composition 
$$\tilW_{LG, \chi_c, \chi_f} \cap \X_*(\AA) \hookrightarrow \tilW_{LG, \chi_c, \chi_f} \twoheadrightarrow \Omega_{\chi_c, \chi_f}.$$
We therefore have a corresponding full monoidal subcategory 
$$D_{\chi_c}((I, \chi_f) \bs LG / (I, \chi_f))^\bullet \subset D_{\chi_c}((I, \chi_f) \bs LG / (I, \chi_f))$$
which is the union of the corresponding blocks. 

By replacing $(LG, \chi_c, \chi_f)$ in the preceding discussion with $(LH, \chi_f)$, we similarly obtain 
$$D((I_H, \chi_f) \bs LH / (I_H, \chi_f))^\bullet \subset D((I_H, \chi_f) \bs LH / (I_H, \chi_f)).$$

\sss{} We now relate the two above monoidal categories as follows.

\begin{thm} \label{thm: twisted metaplectic}There is an equivalence of monoidal categories 
\begin{equation*}D_{\chi_c}((I, \chi_f) \bs LG / (I, \chi_f))^\bullet  \simeq D((I_H, \chi_f) \bs LH / (I_H, \chi_f))^\bullet.\end{equation*}
\end{thm}

\begin{proof} 
As we explain more carefully below, we will largely retain the notation of Theorem \ref{thm: dual with finite twist} and its proof. For ease of reading, we break the proof into a sequence of steps and sublemmas.

{\it Step 1.} For ease of notation, let us fix a topological generator of $\pi_1^t(K)$, and thereby identify for any split torus $A$ over $E$ the formal group scheme of rigidified $A$-local systems on $K$ with a formal subscheme of $A$
$$\on{LS}_{A}^{t, \square} \hookrightarrow A.$$ 
In particular, we will view our character sheaf on $LG$ as an $E$-point $$(\chi_f, \chi_c) \in (\AA^\vee \times \GG_m^{\cen})(E).$$

{\it Step 2.} In this step we will determine the integral coroots for $$D_{\chi_c}((I, \chi_f) \bs LG / (I, \chi_f))^\circ.$$
Recall the inner product $\langle -, - \rangle_V$ on $\X_*(\AA)$ from Section \ref{sss:commutator}. Fix a coroot $\halpha_\GG$ of $\GG$, and set 
$$Q(\halpha_\GG) := \frac{\langle \halpha_\GG, \halpha_\GG \rangle}{2} \in \ZZ;$$we note for later the tautological equality $Q(\halpha_\GG) = Q(w(\halpha_\GG))$ for $w \in \WW$. With this notation in hand, we have that for $i \in \ZZ$ the associated real affine coroot of $LG$ is given by $$\halpha_{\GG, i}  = (\halpha,  i \cdot Q(\halpha_\GG)) \in \X^*(\AA^\vee) \times \ZZ \simeq \X^*(\AA^\vee \times \GG_m^{\cen}).$$
In particular, we have that 
$$\halpha_{\GG, i}(\chi_f, \chi_c) = \halpha(\chi_f) \cdot \chi_c^{i \cdot Q(\halpha_\GG)},$$
so that it is an integral root if and only if 
\begin{equation} \label{e:7cats} \halpha_\GG(\chi_f) = \chi_c^{-i \cdot Q(\halpha_\GG)}.\end{equation}
Note that if $\chi_c^{Q(\halpha_\GG)}$ is a root of unity of order $d_{\halpha}$, this equation is satisfied if and only if $\halpha_\GG(\chi_f)$ is also a $d_{\halpha}$th root of unity, and the condition on $i$ only depends on the residue class $$i \in  \ZZ / d_{\halpha} \ZZ.$$ 

Let us denote the image of the integral affine coroots under the tautological projection to the finite coroots by $$\cPhi_{\GG, \chi_f, \chi_c} \subset \cPhi_\GG.$$It is straightforward to see this is a root subsystem, i.e., is closed under its corresponding reflections.

{\it Step 3.} Let us now turn to the computation of the integral coroots for $$D((I_{H}, \chi_f) \bs LH / (I_H, \chi_f))^\circ.$$
To do so, for a coroot of $\GG$, recall the defining identity 
$$\halpha_{H} := N_{\halpha} \cdot \halpha_\GG.$$
By the analysis of the previous step, applied to the purely central twist $(e, \chi_c)$, we have $N_{\halpha} = d_{\halpha}$. In particular, it follows that a real affine coroot $\halpha_{H, i}$ of $LH$, for $i \in \ZZ$, is an integral coroot if and only if 
$$\halpha_H(\chi_f) = \halpha_\GG(\chi_f)^{d_{\halpha}} = 1,$$
i.e., if and only if $\halpha_\GG(\chi_f)$ is a $d_{\halpha}$th root of unity. 

In particular, if we denote the image of the integral affine coroots under the tautological projection to the finite coroots by 
$$\cPhi_{H, \chi_f} \subset \cPhi_H,$$
it follows that this is again a root subsystem, and the defining bijection $\cPhi_\GG \simeq \cPhi_H$ restricts to a bijection 
$$\iota: \cPhi_{\GG, \chi_f, \chi_c} \simeq \cPhi_{H, \chi_f}.$$

{\it Step 4.} We would next like to record a convenient set of generators for the Coxeter parts of both integral affine Weyl groups. To do so, note the intersection $$\cPhi_{\GG, \chi_f, \chi_c} \cap \cPhi^+_\GG$$is a system of positive roots for $\cPhi_{\GG, \chi_f, \chi_c}$, and let us write $\Delta_{\GG, \chi_f}$ for the corresponding set of simple roots. Defining $\Delta_{H, \chi_f}$ similarly, note that $\iota$ restricts to a bijection 
$$\iota: \Delta_{\GG, \chi_f} \simeq \Delta_{H, \chi_f}.$$

To proceed, for an element $\halpha_\GG \in \cPhi_{\GG, \chi_f, \chi_c}$, consider the subset of integral affine coroots whose finite parts are $\halpha_\GG$; let us denote this by $$\cPhi_{LG, \chi_f, \chi_c}(\halpha_\GG) \subset \cPhi_{LG, \chi_f, \chi_c}.$$Let us denote the corresponding subset of reflections in the integral affine Weyl group by
$$\on{S}(\halpha_\GG) := \{ s_{\halpha_{\GG, i}}: \halpha_{\GG, i} \in \cPhi_{LG, \chi_f, \chi_c}(\halpha_\GG) \} \subset \tilW_{LG, \chi_f, \chi_c}^\circ.$$
Given an element $\halpha_H \in \cPhi_{H, \chi_f}$, define 
$$\on{S}(\halpha_H) \subset \tilW^\circ_{LH, \chi_f}$$similarly. With this, we may state the following. 

\begin{lem} \label{l:reflgens}The group  $\tilW^\circ_{LG, \chi_f, \chi_c}$ is generated by the reflections $$\on{S}(\halpha_\GG), \quad \quad \text{for }\halpha_\GG \in \Delta_{\GG, \chi_f}.$$Similarly, $\tilW^\circ_{LH, \chi_f}$ is generated by the reflections $$\on{S}(\halpha_H), \quad \quad \text{for }  \halpha_H \in \Delta_{H, \chi_f}.$$
\end{lem}

\begin{proof} By definition, $\tilW^\circ_{LG, \chi_f, \chi_c}$ is generated by the reflections $$\on{S}(\check{\beta}_\GG), \quad \quad \text{for } \check{\beta}_\GG \in \cPhi_{\GG, \chi_f, \chi_c},$$so it is enough to see all of these lie in the subgroup generated by the $S(\halpha_\GG),$ for $\halpha_\GG \in \Delta_{\GG, \chi_f}$. To proceed, for any $\check{\beta}_\GG \in \cPhi_{\GG, \chi_f, \chi_c}$, we may choose a sequence $\halpha_{\GG}^1, \halpha_{\GG}^2, \ldots, \halpha_{\GG}^N$ in $\Delta_{\GG, \chi_f}$, for some $N \geqslant 1$, such that 
$$\check{\beta}_\GG = s_{\halpha^1_\GG} \circ s_{\halpha^2_\GG} \circ \cdots \circ s_{\halpha^{N-1}_\GG} (\halpha_\GG^N).$$
If we choose arbitrary lifts $$s^i \in S(\halpha_{\GG}^i),  \quad \quad 1 \leqslant i \leqslant N - 1,$$it is straightforward to see that conjugation by $s^1 \circ s^2 \cdots \circ s^{N-1}$ restricts to an isomorphism $S(\halpha^N_\GG) \simeq S(\check{\beta}_\GG)$, which yields the claim for $LG$. The proof for the analogous claim for $LH$ is entirely similar. 
\end{proof}

{\em Step 5.} For each $\halpha_\GG \in \Delta_{\GG, \chi_f}$, let $i_{\halpha} \in \ZZ_{\geqslant 0}$ denote the minimal nonnegative integer for which $\halpha_{\GG, i_{\halpha}}$ lies in $\cPhi_{LG, \chi_f, \chi_c}$. If we consider the rationalization of the cocharacter lattice 
$$\X_{*, \QQ} := \X_*(\AA) \underset{\Z} \otimes \QQ \xleftarrow{\sim} \X_*(\AA_H) \underset{\Z} \otimes \QQ,$$
there is a unique element $\cmu$ in the $\QQ$-span of $\Delta_{\GG, \chi_f}$ satisfying the identities 
$$\langle \alpha_\GG, \cmu \rangle = -i_{\halpha},$$
for each $\halpha_\GG \in \Delta_{\GG, \chi_f}$. In particular, we have that $$\cmu - s_{\halpha}(\cmu) = -i_{\halpha}\cdot \halpha_\GG$$ for any $\halpha_\GG \in \Delta_{\GG, \chi_f}$.

Let us denote by $\on{Aff}(\X_{*, \QQ})$ the group of affine linear automorphisms of $\X_{*, \QQ}$. For any $v \in \X_{*, \QQ}$, let us denote by $\tau^v \in \on{Aff}(\X_{*, \QQ})$ the corresponding translation $w \mapsto v + w$. Consider the tautological embeddings 
$$\tilW_{LG} = \WW \ltimes \X_*(\AA) \hookrightarrow \on{Aff}(\X_{*, \Q}) \hookleftarrow  \WW \ltimes \X_*(\AA_H) = \tilW_{LH}.$$
where $\WW$ acts linearly on $\X_{*, \Q}$ in the usual way, and $$\X_*(\AA_H) \subset \X_*(\AA) \subset \X_{*, \Q}$$act by the corresponding translations. 

Note that, as the $i_{\halpha}$ are not necessarily zero, the tautological embeddings $$\tilW_{LG, \chi_f, \chi_c}^\bullet \hookrightarrow \tilW_{LG} \hookrightarrow \on{Aff}(\X_{*, \Q}) \hookleftarrow \tilW_{LH} \hookleftarrow \tilW_{LH, \chi_f}^\bullet.$$
do not have the same image, i.e., do not identify $\tilW^\bullet_{LG, \chi_f, \chi_c}$ and $\tilW^\bullet_{LH, \chi_f}$. To account for this, let us consider the affine linear automorphism 
$$\tau^{\cmu}: \X_{*, \Q} \simeq \X_{*, \Q},$$
which induces the conjugation automorphism $$\Aff(\X_{*, \Q}) \simeq \Aff(\X_{*, \Q}), \quad \quad g \mapsto \tau^{\cmu} \circ g \circ \tau^{-\cmu}.$$ 

We claim that this restricts to an isomorphism between $\tilW^\bullet_{LG, \chi_f, \chi_c}$ and $\tilW^\bullet_{LH, \chi_f}$, i.e., that 
\begin{equation} \label{e:conjeq} \tau^{\cmu} \circ \tilW^\bullet_{LG, \chi_f, \chi_c} \circ \tau^{-\cmu} = \tilW^\bullet_{LH, \chi_f}.\end{equation}
Indeed, it is clear it acts as the identity on their common lattice part, namely $\X_*(\AA_H)$, so it remains to show it exchanges their Coxeter parts. For $\halpha_\GG \in \Delta_{\GG,\chi_f}$, we compute directly for $i_{\halpha}$ and $d_{\halpha}$ as above, and any $j \in \Z$, that 
\begin{align*}
\tau^{\cmu} \circ s_{{\halpha}_{\GG, i_{\halpha} + j \cdot d_{\halpha}}}  \circ \tau^{-\cmu} &= \tau^{\cmu} \circ (\tau^{(i_{\halpha} + j \cdot d_{\halpha}) \cdot \halpha_\GG} \circ s_{\halpha}) \circ  \tau^{-\cmu} \\ & = \tau^{\cmu - s_{\halpha}(\cmu) + (i_{\halpha} + j \cdot d_{\halpha}) \cdot \halpha_\GG} \circ s_{\halpha} \\ & = \tau^{-i_{\halpha} \cdot \halpha_\GG + (i_{\halpha} + j \cdot d_{\halpha}) \cdot \halpha_\GG} \circ s_{\halpha} \\ & = \tau^{j \cdot d_{\halpha} \cdot \halpha_\GG} \circ s_{\halpha}  = \tau^{j \cdot \halpha_H} \circ s_{\halpha} = s_{\halpha_{H, j}}. 
\end{align*}
It therefore follows from Lemma \ref{l:reflgens} that \eqref{e:conjeq} holds, as claimed. 

Note that, under the standard homomorphism 
$$\on{Aff}(\X_{*, \Q}) \rightarrow GL( \X_{*, \Q} \oplus \Q),$$
and the identification $\X_{*, \Q} \simeq \X^*_\Q$ afforded by $\langle -, - \rangle_V$, this recovers the reflection representation of $\tilW_{LG}$, i.e., its  action on the rational dual affine Cartan. 

For $\tilW_{LH}$, it is straightforward to see that after possibly replacing $\langle -, - \rangle_V$ by a positive integer multiple, it 
again arises from the trace form on a representation. In particular, this also recovers the reflection representation of $\tilW_{LH}$, i.e., its action on the rational dual affine Cartan of a Kac--Moody central extension of $LH$. 

In particular, there exists a unique element $y$ of $\tilW_{LH. \chi_f}^\circ$ such that the composition 
$$j := y \circ \tau^{\cmu}: \X_{*, \Q} \simeq \X_{*, \Q}$$
exchanges the fundamental alcoves of $\tilW^\circ_{LG, \chi_f, \chi_c}$ and $\tilW^\circ_{LH, \chi_f}$, and in particular 
induces an isomorphism of extended Coxeter systems 
$$j: \tilW^\bullet_{LG, \chi_f, \chi_c} \simeq \tilW^\bullet_{LH, \chi_f}.$$

{\em Step 6.} To use $j$ to identify our Hecke categories, we now explain how to relate the rational dual affine Cartans to the formal schemes $\frf_{LG, \chi_f, \chi_c}$ and $\frf_{LH, \chi_f}$ appearing in the target of the Soergel functor.  Let $A$ be an algebraic group over $E$, with Lie algebra $\fra$. As $E$ has characteristic zero, the formal exponential map  identifies the formal completions of $\fra$ and $A$ at zero and the identity element, respectively, i.e.,  
$$\fra^\wedge_0 \simeq A^\wedge_{e},$$
compatibly with the natural actions of the group of automorphisms of $A$. It is straightforward to see this again holds for any formal group subscheme of $A$. 

In particular, it follows that for a torus $T$ over $E$ with Lie algebra $\frt$, we have a canonical isomorphism  
$$(\on{LS}_{T}^{t, \square})^\wedge_{e} \simeq (\on{H}^1(\GG_m, E) \underset{E} \otimes \frt^*)^\wedge_0.$$
In particular, if we tensor our previous isomorphism 
$$j: \X^*(\A)_\Q \oplus \Q \simeq \X^*(\AA_H)_\Q \oplus \Q$$
over $\Q$ with $\on{H}^1(\GG_m, E)$, and complete at the origin, we obtain our desired $\tilW_{LG, \chi_f, \chi_c}^\bullet \simeq \tilW_{LH, \chi_f}^\bullet$-equivariant identification 
$$j: \frf_{LG, \chi_f, \chi_c} \simeq \frf_{LH, \chi_f}.$$

With the isomorphism $j$ in hand, it remains to apply Theorems \ref{t:mainthmsec7} and Proposition \ref{p:mon2strict} to obtain the equivalence  of the affine Hecke categories. More carefully, while they were stated for the entire Hecke categories, the proofs apply equally well to identify full monoidal subcategories consisting of unions of blocks via their Soergel bimodule descriptions. From here, the verification of the hypotheses of Proposition \ref{p:mon2strict} are exactly as in Step 6 of the proof of Theorem \ref{t:GaitsgIwahori}, so we are done. \end{proof}

\sss{} We next note two favorable situations in which one {\em does} have an equivalence on all blocks.  First, under an assumption on the center of $H$, one is guaranteed an equivalence on all blocks for all $\chi_f$, as follows. 

\begin{cor} \label{c:connectedcenter}Suppose the center of $H$ is connected. Then one has an equivalence of monoidal categories 
\begin{equation*} D_{\chi_c}((I, \chi_f) \bs LG / (I, \chi_f)) \simeq D((I_{H}, \chi_f) \bs LH / (I_{H}, \chi_f)). \end{equation*}
\end{cor}

\begin{proof} In view of Theorem \ref{thm: twisted metaplectic}, it suffices to see that the inclusions 
$$\tilW_{LG, \chi_f, \chi_c} \hookleftarrow \tilW_{LG, \chi_f, \chi_c}^\bullet \simeq \tilW^\bullet_{LH, \chi_f} \hookrightarrow \tilW_{LH, \chi_f}$$
are both in fact bijections. Explicitly, the bijection for $LH$ is equivalent to the assertion that the images of natural maps 
$$\tilW_{LH, \chi_f}^\bullet \hookrightarrow \tilW_{LH} \twoheadrightarrow \WW \quad \text{and} \quad \tilW_{LH, \chi_f} \hookrightarrow \tilW_{LH} \twoheadrightarrow \WW $$
coincide, and similarly for $LG$. 

 Let us begin with the case of $LH$. Recall that an element  $t^{\clambda} w \in \tilW_{LH}$ lies in $\tilW_{LH, \chi_f}$ if and only if 
$$w(\chi_f) = \chi_f.$$By the assumption that $H$ has connected center, the stabilizer in $\WW$ of $$\chi_f \in \Ch(\AA_H; E)$$is generated by the reflections it contains, which implies the desired claim for $LH$.

Similarly, for the claim for $LG$, note that for any element $t^{\clambda} w \in \tilW_{LG}$ its action on $(\chi_f, \chi_c)$, cf. the notation of Step 1 of the proof of Theorem \ref{thm: twisted metaplectic}, is explicitly given by 
$$ t^{\clambda}w (\chi_f, \chi_c) = ( \clambda^*(\chi_c) \cdot w(\chi_f), \chi_c),$$
cf. Lemma \ref{l:conj action}. In particular, for $t^{\clambda}w$ to lie in $\tilW_{LG, \chi_f, \chi_c}$ is the assertion that 
$$\clambda^*(\chi_c) \cdot w(\chi_f) = \chi_f  \in \AA^\vee(E).$$
Applying the projection $\pi: \AA^\vee(E) \rightarrow \AA^\vee_H(E)$, and recalling our abuse of notation that we write $\pi(\chi_f) \in \AA^\vee_H(E)$ simply by $\chi_f$ again, we obtain that 
\begin{equation} \label{e:zhivago} \pi \circ \clambda^*(\chi_c) \cdot w(\chi_f) = \chi_f \in \AA^\vee_H(E).\end{equation}

We next observe that $\pi \circ \clambda^*(\chi_c)$ is the unit element $1$ of $\AA^\vee_H(E)$. Indeed, it suffices to see that after composition with any character $\cmu$ of $\AA^\vee_H$, we have the equality 
$$\cmu \circ \pi \circ \clambda^*(\chi_c) \overset{?} = 1 \in E^\times.$$
But explicitly, we may rewrite this as 
$$\cmu \circ \pi \circ \clambda^*(\chi_c) = \chi_c^{ \langle \cmu, \clambda \rangle_V }$$
using notation from Section \ref{sss:commutator}. From here, we note that the definition of $\X^*(\AA^\vee_H) \simeq \X_*(\AA_H)$ is precisely that $\cmu^*(\chi_c) = 1 \in \AA^\vee(E)$, i.e., that after composition with any character $\clambda$ of $\AA^\vee(E)$, we have that 
$$\clambda \circ \cmu^*( \chi_c) = \chi_c^{\langle \clambda, \cmu \rangle_V} = 1,$$
as desired.

Therefore, we may rewrite \eqref{e:zhivago} simply as 
$$w(\chi_f) = \chi_f \in \AA^\vee_H(E),$$
i.e., that $w$ lies in $\tilW_{LH, \chi_f}$. In particular, by our preceding observation concerning the case of $LH$, and the fact that the projections of the Coxeter parts of $\tilW_{LG, \chi_f, \chi_c}$ and $\tilW_{LH, \chi_c}$ coincide, cf. Step 4 of the proof of Theorem \ref{thm: twisted metaplectic}, we are done.  \end{proof}

\begin{rem} \label{r:allblocks} In fact,  the above argument shows more generally that the inclusion
$$D_{\chi_c}((I, \chi_f) \bs LG /(I, \chi_f))^\bullet \subset D_{\chi_c}((I, \chi_f) \bs LG / (I, \chi_f))$$is an equivalence if and only if the inclusion 
$$D((I_H, \chi_f) \bs LH / (I_H, \chi_f))^\bullet \subset D((I_H, \chi_f) \bs LH / (I_H, \chi_f))$$
is an equivalence, and both in turn are equivalent to the assertion that 
$\WW_{H, \chi_f} \subset \WW$ is generated by reflections. 
\end{rem}

\sss{} The second situation where we need not be careful about block issues is when there is no finite twisting $\chi_f$, and in fact we will obtain the desired identification of the entire affine Hecke 2-categories. 

To set this up, note that, via the bijection $\cPhi_\GG \simeq \cPhi_H$, to each standard parabolic in $\GG$ $$\BB \subset \PP \subset \GG$$one has an associated standard parabolic in $H$ 
$$\BB_H \subset \PP_H \subset H,$$
and let us write $I_{\PP_H} \subset L^+ H$ for the corresponding standard parahoric subgroup. We then have the following.

\begin{thm}\label{thm: derived satake} \label{thm: endoscopicHecke2category}For any pair of standard parabolics $\PP^1, \PP^2$ in $\GG$, there is a $t$-exact equivalence
$$D_{\chi_c}(I_{\PP^1} \bs LG / I_{\PP^2}) \simeq D(I_{\PP^1_H} \bs LH / I_{\PP^2_H}),$$fitting into an equivalence of affine Hecke 2-categories
$$\underset{\PP^1, \PP^2} \oplus \hspace{.5mm} D_{\chi_c}(I_{\PP^1} \bs LG / I_{\PP^2}) \simeq \underset{\PP^1_H, \PP^2_H} \oplus \hspace{.5mm} D(I_{\PP^1_H} \bs LH / I_{\PP^2_H}).$$
\end{thm}

\begin{proof} We first note that, as in Remark \ref{r:allblocks},  the inclusions 
$$\tilW_{LG, \chi_c} \hookleftarrow \tilW_{LG, \chi_c}^\bullet \simeq \tilW^\bullet_{LH} \hookrightarrow \tilW_{LH},$$are both bijections, as for $\chi_f$ trivial, we tautologically have that $\WW_{H, \chi_f} = \WW$, and in particular is generated by reflections. 

Therefore, by Theorem \ref{thm: twisted metaplectic}, we may obtain an equivalence 
$$D_{\chi_c}( I \bs LG / I) \simeq D( I_H \bs LH / I_H).$$
To match the parahorics, recall the element $\cmu$ of Step 5 of the proof of Theorem \ref{thm: twisted metaplectic}, and note in the present situation $\cmu = 0$. It is clear that the identity map $\X_{*, \R} \simeq \X_{*, \R}$ exhibits the fundamental alcove for $\tilW_{LG}^\circ$ as a subset of the fundamental alcove for $\tilW_{LH}^\circ$, so that the element $y$ from Step 5 of the proof of Theorem \ref{thm: twisted metaplectic} is also the trivial element.  

In particular, the isomorphism of extended Coxeter systems $$j: \tilW_{LG, \chi_c} \simeq \tilW_{LH}$$
restricts to the identity on the finite Weyl group $\WW$. We may then finish by applying Theorem \ref{thm: parahoric equivalence}. 
\end{proof}

As a particular case of the preceding theorem, we obtain the following.

\begin{thm} \label{thm: metaplectic derived Satake}There is an equivalence of monoidal categories	\begin{equation*} \epsilon: D_{\chi_c}(L^{+}G \bs LG / L^{+}G) \simeq D(L^{+}H \bs LH / L^{+}H).\end{equation*}	
\end{thm}

\subsection{Spectral description: the metaplectic dual group}
\label{ss:meta dual2}
\sss{} We are now ready to deduce the desired spectral description of the metaplectic affine Hecke 2-category.

Recall that to $(\GG, \chi_c)$ we attached the metaplectic endoscopic group $H$ in the previous subsection, along with its distinguished Borel and Cartan $$H \supset \BB_H \supset \AA_H.$$

\begin{defn} A {\em metaplectic dual group} of $(\GG, \chi_c)$ is a Langlands dual group over $E$ $$H^\vee \supset \BB_{H^\vee} \supset \AA_{H^\vee}$$to $H \supset \BB_H \supset \AA_H.$
\end{defn}

Note that with the present definition, which suffices for our purposes, metaplectic dual groups to $(\GG, \chi_c)$ form a groupoid canonically equivalent to $\pt / \AA_{H^\vee}(E)$.

\sss{} For each standard parabolic $\PP$ of $\GG$, consider the associated standard parabolic of $\chHH$, which we denote by $\chPP_H$. 

Denote the Lie algebras of their unipotent radicals by $\fru_{\chPP}$, and the Lie algebra of $\chHH$ by $\chhh$. Consider the disjoint union of the (equivariant) partial Springer resolutions 
$$\on{Spr}_{\chHH} := \underset{\chPP_H} \sqcup \hspace{.5mm} (\fru_{\chPP_H} / \chPP_H) \rightarrow \chhh / \chHH,$$
and the associated Steinberg derived stack 
$$\on{St}_{\chHH} := \on{Spr}_{\chHH} \overset{R}{\underset{\chhh/\chHH}\times}  \on{Spr}_{\chHH}.$$
If we denote its category of ind-coherent sheaves with nilpotent singular support by 
$$\indcoh_{\on{nilp}}(\on{St}_{\chHH}),$$
this is naturally a monoidal category under convolution. 

\sss{} We next recall the following theorem concerning the automorphic counterpart of the above category.  

\begin{thm}[\cite{ABG}, \cite{BF}, \cite{B}, \cite{BL}, \cite{CD}] \label{t:specside}  There is an equivalence of monoidal categories 
$$\underset{\PP_H^1, \PP_H^2} \oplus \hspace{.5mm} D(I_{\PP_H^1} \bs LH / I_{\PP^2_H}) \simeq  \indcoh_{\on{nilp}}(\on{St}_{\chHH}),$$which restricts, for any pair of parabolics $\PP_H^1$ and $\PP_H^2$ of $H$, to an equivalence 
$$ D(I_{\PP_H^1} \bs LH / I_{\PP^2_H}) \simeq \indcoh_{\on{nilp}}((\fru_{\chPP_H^1} / \chPP_H^1)  \overset{R}{\underset{\chhh/\chHH}\times}  (\fru_{\chPP_H^2} / \chPP_H^2)).$$
\end{thm}

 \begin{rem} Given the profusion of citations, let us briefly review the history of the above statement.\footnote{For simplicity, we disregard in this discussion the distinction between strict and monodromic invariants, and the issues of renormalization relating to singular support; let us recall however the  formulations as above were given explicitly for $\PP^1_H = \PP^2_H = H$, by Arinkin--Gaitsgory \cite{AG}. and the general case may be found in \cite{CD}.} For $$\PP^1_H = \BB_H \quad \text{and} \quad \PP^2_H = H,$$this equivalence is a theorem of Arkhipov--Bezrukavnikov--Ginzburg \cite{ABG}. The monoidal equivalence for $$\PP^1_H = \PP^2_H = H$$is the {\em derived Satake equivalence} of Bezrukavnikov--Finkelberg \cite{BF}, and similarly the monoidal equivalence for $$\PP^1_H = \PP^2_H = \BB_H$$is a theorem of Bezrukavnikov \cite{B}, who also explicitly conjecturally formulated the equivalence of 2-categories. For $\PP^1_H = \BB_H$, and $\PP^2_H$ arbitrary, this conjecture was proven by Bezrukavnikov--Losev \cite{BL}, and for general $\PP^1_H$ and $\PP^2_H$ by  Chen--D. \cite{CD}. \end{rem}

\sss{} Let us now obtain the desired spectral description of the metaplectic affine Hecke 2-category.

\begin{thm} There is an equivalence of monoidal categories 
$$ \underset{\PP^1, \PP^2} \oplus \hspace{.5mm} D_{\chi_c}(I_{\PP^1} \bs LG / I_{\PP^2}) \simeq  \indcoh_{\on{nilp}}(\on{St}_{\chHH}),$$which restricts, for any pair of parabolics $\PP^1$ and $\PP^2$ of $\GG$, to an equivalence 
$$ D_{\chi_c}(I_{\PP^1} \bs LG / I_{\PP^2}) \simeq \indcoh_{\on{nilp}}((\fru_{\chPP_H^1} / \chPP_H^1)  \overset{R}{\underset{\chhh/\chHH}\times}  (\fru_{\chPP_H^2} / \chPP_H^2)).$$
\end{thm}

\begin{proof} This follows by concatenating Theorem \ref{thm: endoscopicHecke2category} and Theorem \ref{t:specside}.  
\end{proof}

As a particular case of the previous theorem, we have the following statement for spherical Hecke categories, i.e., {\em metaplectic derived Satake}.

\begin{cor} \label{c:metaplectic derived satake} There is a monoidal equivalence 
$$\Phi: D_{\chi_c}(L^+G \bs LG / L^+G ) \simeq \indcoh_{\on{nilp}}( \on{pt} / \chHH \overset{R}{\underset{\chhh / \chHH}{\times}} \on{pt} / \chHH ).$$
\end{cor}

\begin{rem} To the best of our knowledge, a proof has not yet appeared in the literature of the assertion that the equivalence of Bezrukavnikov--Losev \cite{BL}, when specialized to $D(I \bs LH / L^+H)$, recovers the equivalence of Arkhipov--Bezrukavnikov--Ginzburg \cite{ABG}, and relatedly that the equivalence of Chen--D. \cite{CD}, when specialized to $D(L^+ H \bs LH / L^+H)$, recovers the equivalence of Bezrukavnikov--Finkelberg \cite{BF}. In what follows, we will use the equivalence of \cite{BF}, which we denote by $\beta$. I.e., in Corollary \ref{c:metaplectic derived satake} we will mean the composition $\Phi := \beta \circ \epsilon$, where $\epsilon$ was the endoscopic equivalence of Theorem \ref{thm: metaplectic derived Satake}. 
\end{rem}

\subsection{Theta sheaves}

\sss{} The goal of this final subsection is to record a property of metaplectic derived Satake which is useful in applications, cf. \cite{coloumb}. After introducing some relevant notation, we formulate the question we would like to address in Section \ref{sss:problem}, and record its answer in Theorem \ref{t:theta}.

\subsubsection{} Recall that given a cocomplete category $\mathcal{C}$ equipped with a $t$-structure whose truncation functors commute with filtered colimits, one can consider the corresponding full subcategory of {\em almost compact} objects
$$\mathcal{C}^{a.c.} \subset \mathcal{C},$$
i.e., objects $c$ for which $\Hom(c, -)$ commutes with uniformly $t$-bounded from below filtered colimits. Note that this full subcategory only depends on the $t$-structure up to bounded equivalence.

One then can obtain its renormalization as the ind-completion of its almost compact objects 
$$\mathcal{C}^{ren} := \Ind(\mathcal{C}^{a.c.}).$$
Moreover, if $\mathcal{C}$ is monoidal, the monoidal unit is almost compact, and the almost compact objects are closed under multiplication, then $\mathcal{C}^{ren}$ inherits a monoidal structure.

\sss{} In particular, as Theorem \ref{thm: metaplectic derived Satake} is $t$-exact, and the equivalence \cite{BF} is $t$-bounded, we have monoidal equivalences 
\begin{equation} D_{\chi_c}(L^+G \bs LG / L^+G)^{ren} \overset{\epsilon}\simeq D(L^+H \bs LH / L^+H)^{ren} \overset{\beta}\simeq \on{QCoh}( \fh^{\vee, *}[2] / H^\vee),\footnote{Here and below, given an object $V$ of $\Modu_E$, not necessarily in non-positive cohomological degrees, by a mild abuse of notation we write $\on{QCoh}(V)$ for $\Sym(V^*)\mod$.} \label{e:dersat} \end{equation}where the equivalence $$
\indcoh_{\on{nilp}}( \pt / H^\vee \overset{R}{\underset{\fh^\vee / H^\vee}{\times}} \pt / H^\vee)^{ren} \simeq \on{QCoh}( \fh^{\vee, *}[2] / H^\vee)$$
is given by Koszul duality, cf. \cite{AG}; in fact it is the equivalence on compact objects 
$$D(L^+H \bs LH / L^+H)^{ren, c} \simeq (\indcoh_{\on{nilp}}( \pt / H^\vee \overset{R}{\underset{\fh^\vee / H^\vee}{\times}} \pt / H^\vee)^{ren, c} \simeq \on{QCoh}(\fh^{\vee, *}[2] / H^\vee)^{c}$$
which explicitly appears in \cite{BF}.

\sss{} In what follows, for ease of reading we introduce the following notation. We will denote the relevant spherical Hecke categories by 
$$\Sphc := D_{\chi_c}(L^+ G \bs LG / L^+G)^{ren} \quad \text{and} \quad \Sphh := D(L^+H \bs LH / L^+H)^{ren},$$
so that the equivalences of \eqref{e:dersat} read as 
$$\Sphc \overset{\epsilon} \simeq \Sphh \overset{\beta} \simeq \on{QCoh}(\fh^{\vee, *}[2]/H^\vee).$$
We will continue to denote the composite equivalence by $\Phi := \beta \circ \epsilon$.

\subsubsection{} A basic compatibility of the equivalence of Bezrukavnikov--Finkelberg is as follows. 

On the automorphic side, we have the functor of (equivariant) global sections 
$$\Gamma: \Sphh \rightarrow \Modu_E;$$
explicitly, $\Sphh$ is the ind-completion of the usual category of bounded constructible complexes on $L^+H \bs LH / L^+H$, and the above functor is the ind-extension of the usual functor on bounded constructible complexes of $*$-pushforward along 
$$L^+H \bs LH / L^+H \rightarrow \on{pt}.$$ 

On the spectral side, we have the functor of Kostant--Whittaker reduction, $$\varkappa: \on{QCoh}(\fh^{\vee, *}[2] / H^\vee) \rightarrow \Modu_E,$$ which is essentially a cohomologically sheared version of restriction to the Kostant slice, cf. \cite{BF} for the details. 

Then a basic feature of the equivalence of Bezrukavnikov--Finkelberg is a datum of commutativity for the following diagram. 
\begin{equation} \label{e:BFink}  \xymatrix{\Sphh \ar[rr]^\beta \ar[d]_\Gamma && \on{QCoh}(\fh^{\vee, *}[2]/H^\vee) \ar[d]^{\varkappa} \\ \Modu_E \ar@{=}[rr] && \Modu_E. }\end{equation}

\sss{} Our goal is to determine what $\varkappa$ corresponds to under metaplectic derived Satake $\Phi$, and the answer is given in Theorem \ref{t:theta} below. In purely endoscopic terms, we equivalently would like to describe the composition
$$\Sphc \overset{\epsilon}\simeq \Sphh \xrightarrow{\Gamma} \Modu_E,$$
which we denote below by $\Gamma_{\chi_c}$.
 \label{sss:problem}

\sss{} We first reformulate the problem of describing $\Gamma_{\chi_c}$ as identifying a certain object of $\Sphc$. Given a dualizable category $\mathcal{C}$, with dual $\mathcal{C}^\vee \simeq \Hom_{\lincat_E}(\mathcal{C}, \Modu_E)$, we denote their perfect pairing by 
$$\langle - , - \rangle: \mathcal{C} \otimes \mathcal{C}^\vee \rightarrow \Modu_E.$$
Next, note that, thanks to our renormalization, $\Sphc \simeq \Sphh \simeq \on{QCoh}(\fh^{\vee, *}[2]/H^\vee)$ is a rigid monoidal category. Moreover, for any rigid monoidal category $\mathcal{M}$, if we write $1$ for its monoidal unit, and denote its underlying binary product by 
$$- \star - : \mathcal{M} \otimes \mathcal{M} \rightarrow \mathcal{M},$$
it is known that the pairing 
$$\mathcal{M} \otimes \mathcal{M} \rightarrow \Modu_E, \quad \quad m_1 \boxtimes m_2 \mapsto \Hom_{\mathcal{M}}(1, m_1 \star m_2)$$
is perfect, i.e., yields an identification of $\mathcal{M}$ with ${M}^\vee$. 

In particular, it follows that $\Gamma_{\chi_c}$ is given by pairing with a unique up to equivalence object $\mathcal{K}_{\chi_c}$ characterized by the equivalence 
$$\Gamma_{\chi_c}(-) \simeq \Hom_{\Sphc}( 1, - \star \mathcal{K}_{\chi_c}).$$
Moreover, as $\Gamma$ and $\varkappa$ are similarly given by pairing with objects
$$\Gamma(-) \simeq \Hom_{\Sphh}(1, - \star \mathcal{K}) \quad \text{and} \quad \varkappa(-) \simeq \Hom_{\on{QCoh}(\fh^{\vee, *}[2]/H^\vee)}(1, - \star \mathfrak{K})$$
it follows that one has 
$$\epsilon(\mathcal{K}_{\chi_c}) \simeq \mathcal{K}, \quad \beta(\mathcal{K}) \simeq \mathfrak{K}, \quad \text{and} \quad \Phi(\mathcal{K}_{\chi_c}) \simeq \mathfrak{K}.$$

\begin{rem} Although we will not explicitly use it below, let us mention that, by definition $\mathfrak{K}$ is up to shearing the pushforward of the structure sheaf of the Kostant slice $\mathscr{S}$ along the tautological map $\pi: \mathscr{S} \rightarrow \fh^{\vee, *}/H^\vee$.     
\end{rem}

\sss{} As a first step in describing $\mathcal{K}_{\chi_c}$, let us give an explicit description of $\mathcal{K}$. Let us write the affine Grassmannian of $H$ as an ascending union of $L^+H$-stable closed subschemes of finite type 
$$LH/L^+H \simeq \varinjlim_\alpha  Z_\alpha.$$
Then for each $Z_\alpha$, we have a dualizing sheaf $\omega_\alpha \in D(L^+H \bs Z_\alpha)^{a.c.}$, and these naturally form an inductive system with respect to the inclusions among the $Z_\alpha$, so that we may form the dualizing sheaf 
$$\omega := \varinjlim \omega_\alpha \in \Sphh.$$We now identify this with our kernel $\mathcal{K}$. 

\begin{lem} \label{l:bluesclues}We have an equivalence $\mathcal{K} \simeq \omega,$ 
and therefore an equivalence $$\mathcal{K}_{\chi_c} \simeq \epsilon^{-1}(\omega).$$
\end{lem}
\begin{proof} The second assertion follows immediately from the first. For the first assertion, ind-proper base change gives, for any $\mathcal{F}, \mathcal{G} \in \Sphh$ a functorial identification 
$$\Hom_{\Sphh}(1, \mathcal{F} \star \mathcal{G}) \simeq \Gamma( L^+H \bs LH / L^+H, \mathcal{F} \overset{!} \otimes \on{inv}_* \mathcal{G}),$$
where $\on{inv}: LH \rightarrow LH$ denotes the inversion map. Applying this with $\mathcal{G} \simeq \omega$ and noting that $\on{inv}_* \omega \simeq \omega$ yields the claim. 
\end{proof}

In particular, the compatibility of Bezrukavnikov--Finkelberg may be stated as an equivalence $\beta(\omega) \simeq \mathfrak{K}$. 

\sss{} We will next write the expression $\epsilon^{-1}(\omega)$ more explicitly by using the Radon transform. To do so, consider the moduli stacks of $G$- and $H$-bundles on $\PP^1$, which we denote by $\Bun_G$ and $\Bun_H$, respectively, and the corresponding categories of sheaves 
$$D_{\chi_c}(\Bun_G) \quad \text{and} \quad D(\Bun_H).$$
By definition, these are the inverse limit, under $*$-restriction, of the categories of sheaves on each quasi-compact open subset of $\Bun_G$ and $\Bun_H$, respectively. Modification of bundles at $0 \in \PP^1$ gives rise to actions 
$$D_{\chi_c}(L^+G \bs LG / L^+G) \circlearrowright D_{\chi_c}(\Bun_G) \quad \text{and} \quad D(L^+H \bs LH / L^+H) \circlearrowright D(\Bun_H).$$

To recall the construction of the Radon transform, note that the trivial bundle defines open substacks
$$j: \BB{G} \hookrightarrow \Bun_G \quad \text{and} \quad j: \BB H \hookrightarrow \Bun_H$$
and hence $!$-extension functors 
$$j_!: D(\BB G) \simeq D_{\chi_c}(\BB G) \rightarrow D_{\chi_c}(\Bun_G) \quad \text{and} \quad j_!: D(\BB H) \rightarrow D(\Bun_H).$$
The categories $D(\BB G)$ and $D( \BB H)$ contain a unique up to isomorphism irreducible object, which we denote by $E$. Moreover, it is well known that convolution with its $!$-extension yields an equivalence
$$ - \star j_!(E): D_{\chi_c}(L^+G \bs LG / L^+G) \simeq D_{\chi_c}(\Bun_G),$$
which restricts to an equivalence 
$$- \star j_!(E): D_{\chi_c}(L^+G \bs LG / L^+G)^{a.c.} \simeq D_{\chi_c}(\Bun_G)^{a.c.},$$
and hence renormalizes to an equivalence 
$$\on{RT}_{G, \chi_c}: \Sphc \simeq D_{\chi_c}(\Bun_G)^{ren}.$$
Similarly, for $H$, we obtain a renormalized equivalence 
$$\on{RT}_H: \Sphh \simeq D(\Bun_H)^{ren}.$$
We then have the following observation. 

\begin{lem} \label{l:poslevelt} The composite equivalence 
$$\on{RT}_H \circ \epsilon \circ \on{RT}_{G, \chi_c}^{-1}: D_{\chi_c}(\Bun_G)^{ren} \simeq D(\Bun_H)^{ren}$$is $t$-exact with respect to the perverse $t$-structures on both sides.     
\end{lem}

\begin{proof} Each dominant coweight $\clambda$ of $H$ defines strata 
$$j_{\clambda}: L^+H \bs L^+H \cdot t^{\clambda} \cdot L^+H/ L^+H \hookrightarrow L^+H \bs LH / L^+H \quad \text{and} \quad \mathbf{j}_{\clambda}: \Bun_H^{\clambda} \hookrightarrow \Bun_H.$$Moreover, if we denote the (in general non-perverse) $*$- and $!$-extensions of the unique irreducible objects of each category
$$D(L^+ H \bs L^+H \cdot t^{\clambda} \cdot L^+ H / L^+ H)^{ren} \quad \text{and} \quad D(\Bun_H^{\clambda})^{ren}$$
respectively by 
$$j_{\clambda, *} \in \Sphh \quad \text{and} \quad \mathbf{j}_{\clambda, !} \in D(\Bun_H)^{ren},$$
a standard property of the Radon transform is the identification 
\begin{equation} \label{e:radonswap}\on{RT}_H( j_{\clambda, *}) \simeq \mathbf{j}_{\clambda, !}.\end{equation}
In particular, the pullback of the perverse $t$-structure on $D(\Bun_H)^{ren}$ to $\Sphh$ can be characterized by the condition that an object $\xi$ is coconnective if and only if 
\begin{equation} \label{e:pervradon}\Hom( j_{\clambda, *}, \xi) \in \Modu_E^{\geqslant 0}, \quad \text{for all } \clambda.\end{equation}
As the analogous assertions hold, {\em mutatis mutandis}, for $G$, and the equivalence $\epsilon$ by construction exchanges the objects $j_{\clambda, *}$ for $H$ with the similarly named objects for $G$, the claim of the lemma follows. \end{proof}

In particular, by considering the constant perverse sheaf $E_{\Bun_H}$ on $\Bun_H$, we are led to the following. Consider a block of $D_{\chi_c}(\Bun_G)^{ren}$,
$$i_\gamma: D_{\chi_c}(\Bun_G)_\gamma^{ren} \hookrightarrow D_{\chi_c}(\Bun_G)^{ren}.$$
Let us call a stratum $j_{\clambda}: \Bun_G^{\clambda} \rightarrow \Bun_G$ {\em relevant} to our block if the restriction functor $(j_{\clambda})^! \circ i_\gamma$ is nonzero. Then we note that there is a unique relevant stratum of maximal dimension, and we call the corresponding simple object of $D_{\chi_c}(\Bun_G)_\gamma^{ren}$ the IC sheaf of {\em maximal support} in the block. 

\begin{defn} The {\em theta sheaf} $\Theta_{\chi_c} \in D_{\chi_c}(\Bun_G)^{ren}$ is the direct sum of the IC sheaves with maximal support in each block. 
\end{defn}

\begin{cor} The equivalence of Lemma \ref{l:poslevelt} exchanges the theta sheaf and the constant sheaf, i.e.,
$$\on{RT}_H \circ \epsilon \circ \on{RT}_{G, \chi_c}^{-1} ( \Theta_{\chi_c}) \simeq E_{\Bun_H}.$$ \label{c:theta}
\end{cor}

\begin{proof}This follows immediately from Lemma \ref{l:poslevelt}. 
\end{proof}

Finally, let us obtain the promised more explicit description of our kernel  $\mathcal{K}_{\chi_c}$.

\begin{prop} We have an equivalence     \label{p:whoisK}
$$\on{RT}^{-1}_{G, \chi_c}(\Theta_{\chi_c}) \simeq \mathcal{K}_{\chi_c}.$$
\end{prop}

\begin{proof} By considering the tautological commutative diagram of equivalences
$$\xymatrix{ D_{\chi_c}(\Bun_G)^{ren} \ar[d]_{\on{RT}_{G, \chi_c}^{-1}} \ar[rr]^{\on{RT}_H \circ \epsilon \circ \on{RT}_{G, \chi_c}^{-1}} & & D(\Bun_H)^{ren} \ar[d]^{\on{RT}_H^{-1}} \\ \Sphc \ar[rr]^{\epsilon} & &  \Sphh},$$
and combining Lemma \ref{l:bluesclues} and Corollary \ref{c:theta}, it is enough to verify the analogous equivalence for $H$, i.e., 
$$ \on{RT}^{-1}_H( E_{\Bun_H}) \overset{?}\simeq \omega.$$
However, this is a standard observation, whose proof we sketch for the reader. 

{\em Step 1.} We recall that the functor
$$\Gamma: \Sphh \rightarrow \Modu_E$$
naturally factors through bimodules for $\on{C}^*(\BB L^+ H) \simeq \on{C}^*(\BB H)$, and moreover this functor carries by ind-proper base change a datum of monoidality  
$$\Sphh \overset{\otimes}{\rightarrow}\on{C}^*(\BB L^+ H)\on{-bimod}  \xrightarrow{\on{Oblv}} \Modu_E;$$
we denote this monoidal functor by $\Gamma^{\on{enh}}$. (We remind the reader in passing that, by contrast,  $\on{Oblv} \circ \Gamma^{enh}$ cannot be equipped with a monoidal structure.)

{\em Step 2.} Let us denote by $\widehat{\Bun}_H$ the moduli stack of $H$-bundles on $\PP^1$ equipped with full level structure at $t = 0$, which is naturally an $L^+H$ torsor over $\Bun_H$ via the forgetful map 
$$\widehat{\Bun}_H \rightarrow \Bun_H.$$ 
By applying ind-proper base change, and considering the $!$-pullback of the dualizing sheaf along 
$$L^+H \bs \widehat{\Bun}_H \rightarrow L^+H \bs \on{pt},$$
it follows that $E_{\Bun_H}$ carries a datum of derived Hecke equivariance for the action of $\Sphh$; we use in particular that $\Bun_H$ is a smooth Artin stack, so its dualizing and constant sheaves agree up to shift. We therefore obtain that the convolution map 
$$\Sphh \rightarrow D(\Bun_H)^{ren}, \quad \quad \mathcal{F} \mapsto \mathcal{F} \star E_{\Bun_H}$$
canonically factors as  
$$\Sph_H \xrightarrow{\Gamma^{\on{enh}}} \on{C}^*(\BB L^+ H)\on{-bimod} \rightarrow D(\Bun_H).$$

{\em Step 3.} As the analogue of the previous assertion for $E_{\Bun_H}$ holds by definition for $\omega$ and hence for $\on{RT}_H(\omega)$, it is therefore enough to show an equivalence of
$$E_{\Bun_H} \simeq \on{RT}_H(\omega)$$
as derived Hecke equivariant objects. However, by using the Radon transform we know
\begin{align*}&\Hom_{\Sphh\mod}( \on{C}^*(\BB L^+ H)\on{-bimod}, D(\Bun_H)^{ren})\\ \simeq &\Hom_{\Sphh\mod}(\on{C}^*(\BB L^+H)\on{-bimod}, \Sphh) \\ \simeq &D(\BB L^+ H)^{ren} \simeq D(\BB H)^{ren}, \end{align*}
where the final equivalences explicitly send a derived Hecke equivariant object in $\Sphh$ to its $!$-restriction to the trivial double coset. By the definition of the Radon transform, it follows that the composite equivalence sends a derived Hecke equivariant object on $\Bun_H$ to its $!$-restriction to the trivial stratum $$\BB H = \Bun_H^0 \hookrightarrow \Bun_H.$$ 
Therefore, by comparing the equivalences 
$$\Hom_{\Sphh}(1, \omega) \simeq \on{C}^*(\BB H) \quad \text{and} \quad \Hom_{D(\Bun_H)^{ren}}( \on{RT}(1), E_{\Bun_H}) \simeq \on{C}^*(\BB H),$$
the claim follows. 
\end{proof}

Before proceeding, we will need one more basic observation characterizing the Radon transform of the theta sheaf. For a block of the Satake category 
$$i_\gamma: \Sph_{G, \chi_c, \gamma} \hookrightarrow \Sphc,$$
as it maps under the Radon transform to a block of $D_{\chi_c}(\Bun_G)$,
$$\on{RT}_{G, \chi_c}: \Sph_{G, \chi_c, \gamma} \simeq D_{\chi_c}(\Bun_G)_\gamma^{ren},$$the perverse $t$-structure on the latter restricts to a $t$-structure on $\Sph_{G, \chi_c, \gamma}$. Let us call a stratum 
$$j_{\clambda}: L^+G \cdot t^{\clambda} \cdot L^+G \hookrightarrow LG$$
{\em relevant} if the $*$-extension of the constant IC sheaf $j_{\clambda, *}$ belongs to $\Sph_{G, \chi_c, \gamma}$. Finally, let us denote by $\Theta_{\chi_c, \gamma}$ the IC sheaf of maximal support in $D_{\chi_c}(\Bun_G)_\gamma^{ren}$.  

\begin{lemma} \label{l:charrttheta}There exists a unique relevant stratum $L^+G \cdot t^{\clambda} \cdot L^+G$ minimal under the closure partial order on double cosets. Moreover, $\on{RT}^{-1}(\Theta_{\chi_c, \gamma})$ is the unique simple object of $\Sph_{G, \chi_c, \gamma}$ with respect to the pulled back $t$-structure satisfying the nonvanishing 
$$\Hom_{\Sphc}(j_{\clambda, *}, \on{RT}_{G, \chi_c}^{-1}(\Theta_{\chi_c, \gamma})) \neq 0.$$
\end{lemma}

\begin{proof} The uniqueness of a minimal relevant stratum is equivalent under the Radon transform to the uniqueness of a maximal relevant stratum on $\Bun_G$, and the second assertion concerning nonvanishing follows from \eqref{e:radonswap}. 
\end{proof}

\sss{} With these preparations, we are ready to deduce our final identification of Kostant--Whittaker reduction. 

To state it, consider the dual twisting $\chi_c^\vee$ for $G$, the tautological perfect pairing, 
$$\langle -, - \rangle: \Sphc \otimes \Sph_{G, \chi_c^\vee} \xrightarrow{ - \overset{!} \otimes - } \Sph_G \xrightarrow{\Gamma} \Modu_E$$
and in particular the functor given by pairing with the dual theta sheaf  
\begin{equation} \label{e:dualtheta}\langle -, \on{RT}^{-1}_{G, \chi_c^\vee}(\Theta_{\chi_c^\vee}) \rangle: \Sphc \rightarrow \Modu_E.\end{equation}

\begin{thm} \label{t:theta} Metaplectic derived Satake exchanges Kostant--Whittaker reduction with pairing with the dual theta sheaf, i.e., the composite functor 
$$\Sphc \overset{\Phi}\simeq \on{QCoh}(\fh^{\vee, *}[2]/H^\vee) \xrightarrow{\varkappa} \Modu_E$$
is equivalent to \eqref{e:dualtheta}, and we have an equivalence of kernel objects 
\begin{equation} \label{e:dweezil} \Phi( \on{RT}^{-1}_{G, \chi_c}(\Theta_{\chi_c})) \simeq \mathfrak{K}.\end{equation}
\end{thm}

\begin{proof} We recall that $\varkappa \circ \Phi$ was identified with the functor 
\begin{align*} &\Hom_{\Sphc}(1, - \star \mathcal{K}_{\chi_c}). \intertext{
By ind-proper base change, if we write $\on{inv}: LG \simeq LG$ for the inversion map $g \mapsto g^{-1}$, we may rewrite this as}
\simeq & \langle -, \on{inv}_* \mathcal{K}_{\chi_c} \rangle. \intertext{Using Proposition \ref{p:whoisK}, we rewrite this as}\simeq &\langle -, \on{inv}_* \on{RT}_{G, \chi_c}^{-1}(\Theta_{\chi_c}) \rangle, \intertext{which in particular yields \eqref{e:dweezil}. To proceed, we note that inversion 
$$\on{inv}: \Sphc \simeq \Sph_{G, \chi_c^\vee}$$
is $t$-exact with respect to the $t$-structures pulled back from the perverse $t$-structures on $D_{\chi_c}(\Bun_G)$ and $D_{\chi_c^\vee}(\Bun_G)$ via Radon transform, as follows from the characterization of Equation \eqref{e:pervradon}. By combining this with Lemma \ref{l:charrttheta}, it follows straightforwardly that inversion exchanges the Radon transforms of the IC sheaves with maximal support, hence we may rewrite this as} \simeq &\langle -, \on{RT}_{G, \chi_c^\vee}^{-1}(\Theta_{\chi_c^\vee}) \rangle,
\end{align*}
 as desired. 
\end{proof}

\begin{rem} For $\chi_c$ the quadratic twist on the basic central extension of the symplectic group, the theta sheaf, its Radon transform, and their ramified and higher genus analogues have been studied by V. Lafforgue--Lysenko intensively, see \cite{lysenko3, lysenko, lysenko2} for a partial indication. The presentation of the corresponding local category as modules over the Weyl vertex operator algebra appears in their unpublished work, and was later rediscovered in the construction of Coulomb branches of non-cotangent type  \cite{coloumb}.  

The theta sheaves considered here again have local counterparts, namely the modules for a simple current extension of a boundary admissible vertex operator algebra, and the conformal blocks of the latter give the extension to higher genus. 

Moreover, these local and higher genus counterparts have analogues over function fields which should recover, after trace of Frobenius, the theta representations and theta series of coverings of $p$-adic groups studied by Kubota, Patterson, Deligne, Kazhdan, Savin, Bump--Friedburg--Ginzburg, Gao, Leslie, and many others, see \cite{kubota, patterson1, patterson2, deligne, kazhdanpatterson, Savin, BFG, BFG2, BFG3, FG, FG2, FG3, gao, leslie} and references therein for a partial indication.\footnote{We thank Arakawa for emphasizing this latter connection.}  We will return  to  these points in more detail elsewhere. 
\end{rem}

\bibliographystyle{amsplain}
\bibliography{reference}

\providecommand{\bysame}{\leavevmode\hbox to3em{\hrulefill}\thinspace}
\providecommand{\MR}{\relax\ifhmode\unskip\space\fi MR }
\providecommand{\MRhref}[2]{%
  \href{http://www.ams.org/mathscinet-getitem?mr=#1}{#2}
}
\providecommand{\href}[2]{#2}
\begin{thebibliography}{10}

\bibitem{AFMO}
J.~D. Adler, J.~Fintzen, M.~Mishra, and K.~Ohara, \emph{Reduction to depth zero for tame p-adic groups via {Hecke} algebra isomorphisms}, arXiv.2408.07805 (2024).

\bibitem{AG}
D.~Arinkin and D.~Gaitsgory, \emph{Singular support of coherent sheaves and the geometric {Langlands} conjecture}, Selecta Mathematica \textbf{21} (2015), no.~1, 1--199.

\bibitem{AGKRRV.restricted.local.systems}
D.~Arinkin, D.~Gaitsgory, D.~Kazhdan, S.~Raskin, N.~Rozenblyum, and Y.~Varshavsky, \emph{The stack of local systems with restricted variation and geometric {Langlands} theory with nilpotent singular support}, arXiv:2010.01906 (2020).

\bibitem{ABG}
S.~Arkhipov, R.~Bezrukavnikov, and V.~Ginzburg, \emph{Quantum groups, the loop {G}rassmannian, and the {S}pringer resolution}, J. Amer. Math. Soc. \textbf{17} (2004), no.~3, 595--678. \MR{2053952}

\bibitem{BN}
D.~Ben-Zvi and D.~Nadler, \emph{The character theory of a complex group}, arXiv:0904.1247 (2009).

\bibitem{BZN}
\bysame, \emph{Betti geometric {L}anglands}, Algebraic geometry: {S}alt {L}ake {C}ity 2015, Proc. Sympos. Pure Math., vol. 97.2, Amer. Math. Soc., Providence, RI, 2018, pp.~3--41. \MR{3821166}

\bibitem{B}
R.~Bezrukavnikov, \emph{On two geometric realizations of an affine {Hecke} algebra}, Publications Mathématiques de l'Institut des Hautes Sciences \textbf{123} (2016), 1--67.

\bibitem{BF}
R.~Bezrukavnikov and M.~Finkelberg, \emph{Equivariant {Satake} category and {Kostant--Whittaker} reduction}, Moscow Mathematical Journal \textbf{8} (2008), no.~1, 39--72.

\bibitem{BFO}
R.~Bezrukavnikov, M.~Finkelberg, and V.~Ostrik, \emph{Character {D}-modules via {Drinfeld} center of {Harish}-{Chandra} bimodules}, Inventiones Mathematicae \textbf{188} (2012), 589--620.

\bibitem{BL}
R.~Bezrukavnikov and I.~Losev, \emph{Dimensions of modular irreducible representations of semisimple {L}ie algebras}, J. Amer. Math. Soc. \textbf{36} (2023), no.~4, 1235--1304. \MR{4618958}

\bibitem{BR}
R.~Bezrukavnikov and S.~Riche, \emph{A topological approach to {Soergel} theory}, {Representation Theory and Algebraic Geometry}: A Conference Celebrating the Birthdays of {Sasha Beilinson} and {Victor Ginzburg}, Springer International Publishing, Cham, 2021, pp.~267--343.

\bibitem{BY}
R.~Bezrukavnikov and Z.~Yun, \emph{On {Koszul} duality for {Kac}-{Moody} groups}, Representation Theory of the American Mathematical Society \textbf{17} (2013), no.~1, 1--98.

\bibitem{Bo}
N.~Bourbaki, \emph{Groupes et algèbres de {Lie}. {Chapitres} 4, 5 et 6}, Hermann, Paris, 1968.

\bibitem{coloumb}
A.~Braverman, G.~Dhillon, M.~Finkelberg, S.~Raskin, and R.~Travkin, \emph{Coulomb branches of noncotangent type (with appendices by {Gurbir Dhillon} and {Theo Johnson-Freyd})}, arXiv:2201.09475 (2022).

\bibitem{BFG}
D.~Bump, S.~Friedberg, and D.~Ginzburg, \emph{On the cubic {S}himura lift for {${\rm PGL}_3$}}, Israel J. Math. \textbf{126} (2001), 289--307. \MR{1882041}

\bibitem{BFG3}
\bysame, \emph{Small representations for odd orthogonal groups}, Int. Math. Res. Not. (2003), no.~25, 1363--1393. \MR{1968295}

\bibitem{BFG2}
\bysame, \emph{Lifting automorphic representations on the double covers of orthogonal groups}, Duke Math. J. \textbf{131} (2006), no.~2, 363--396. \MR{2219245}

\bibitem{CD}
H.~Chen and G.~Dhillon, \emph{A {Langlands} dual realization of coherent sheaves on the nilpotent cone}, arXiv:2310.10539 (2023).

\bibitem{deligne}
P.~Deligne, \emph{Sommes de {G}auss cubiques et rev\^etements de {${\rm SL}(2)$}\ [d'apr\`es {S}. {J}. {P}atterson]}, S\'eminaire {B}ourbaki (1978/79), Lecture Notes in Math., vol. 770, Springer, Berlin, 1980, pp.~Exp. No. 539, pp. 244--277. \MR{572428}

\bibitem{EE}
J.~N. Eberhardt and A.~Eteve, \emph{Universal {Koszul} {Duality} for {Kac}-{Moody} {Groups}}, arXiv:2408.14716 (2024).

\bibitem{FL}
M.~Finkelberg and S.~Lysenko, \emph{Twisted geometric {Satake} equivalence}, Journal of the Institute of Mathematics of Jussieu \textbf{9} (2010), no.~4, 719--739.

\bibitem{FG}
S.~Friedberg and D.~Ginzburg, \emph{Theta functions on covers of symplectic groups}, Bull. Iranian Math. Soc. \textbf{43} (2017), no.~4, 89--116. \MR{3711824}

\bibitem{FG2}
\bysame, \emph{Descent and theta functions for metaplectic groups}, J. Eur. Math. Soc. (JEMS) \textbf{20} (2018), no.~8, 1913--1957. \MR{3854895}

\bibitem{FG3}
\bysame, \emph{Classical theta lifts for higher metaplectic covering groups}, Geom. Funct. Anal. \textbf{30} (2020), no.~6, 1531--1582. \MR{4182832}

\bibitem{Ga}
D.~Gaitsgory, \emph{Quantum {Langlands} {Correspondence}}, arXiv:1601.05279 (2007).

\bibitem{gao}
F.~Gao, \emph{Distinguished theta representations for certain covering groups}, Pacific J. Math. \textbf{290} (2017), no.~2, 333--379. \MR{3681111}

\bibitem{Go}
V.~Gouttard, \emph{Perverse {Monodromic} sheaves}, Ph.D. thesis, Université Clermont Auvergne, 2021.

\bibitem{HNY}
J.~Heinloth, B.C. Ng\^{o}, and Z.~Yun, \emph{Kloosterman sheaves for reductive groups}, Annals of Mathematics (2013), 241--310.

\bibitem{IY}
A.~Ionov and Z.~Yun, \emph{Tilting sheaves for real groups and {Koszul} duality}, arXiv:2301.05409 (2023).

\bibitem{IM}
N.~Iwahori and H.~Matsumoto, \emph{On some {Bruhat} decomposition and the structure of the {Hecke} rings of $p$-adic {Chevalley} groups}, Publications Mathématiques de l’Institut des Hautes Scientifiques \textbf{25} (1965), 5--48.

\bibitem{kazhdan1980schubert}
D.~Kazhdan and G.~Lusztig, \emph{Schubert varieties and {P}oincar{\'e} duality}, Geometry of the Laplace operator (1980), 185--203.

\bibitem{KL}
\bysame, \emph{Proof of the {Deligne}-{Langlands} conjecture for {Hecke} algebras}, Inventiones Mathematicae \textbf{87} (1987), 153--215.

\bibitem{kazhdanpatterson}
D.~A. Kazhdan and S.~J. Patterson, \emph{Metaplectic forms}, Inst. Hautes \'Etudes Sci. Publ. Math. (1984), no.~59, 35--142. \MR{743816}

\bibitem{kubota}
T.~Kubota, \emph{On automorphic functions and the reciprocity law in a number field}, Lectures in Mathematics, Department of Mathematics, Kyoto University, vol. No. 2, Kinokuniya Book Store Co., Ltd., Tokyo, 1969. \MR{255490}

\bibitem{lysenko2}
V.~Lafforgue and S.~Lysenko, \emph{Geometric {W}eil representation: local field case}, Compos. Math. \textbf{145} (2009), no.~1, 56--88. \MR{2480495}

\bibitem{leslie}
S.~Leslie, \emph{A generalized theta lifting, {CAP} representations, and {A}rthur parameters}, Trans. Amer. Math. Soc. \textbf{372} (2019), no.~7, 5069--5121. \MR{4009400}

\bibitem{Li}
Y.~W. Li, \emph{Endoscopy for affine {Hecke} categories}, Selecta Mathematica, New Series \textbf{30} (2024), 93.

\bibitem{Lupadic}
G.~Lusztig, \emph{Classification of unipotent representations of simple $p$-adic groups}, International Mathematics Research Notices (1995), no.~11, 517--589.

\bibitem{LuTrunc}
\bysame, \emph{Truncated convolution of character sheaves}, Bulletin of the Institute of Mathematics, Academia Sinica \textbf{10} (2015), no.~1, 1--72.

\bibitem{LuUnip}
\bysame, \emph{Unipotent representations as a categorical centre}, Representation Theory of the American Mathematical Society \textbf{19} (2015), no.~9, 211--235.

\bibitem{LY}
G.~Lusztig and Z.~Yun, \emph{Endoscopy for {Hecke} categories, character sheaves and representations}, Forum of Mathematics, Pi \textbf{8} (2020), e12.

\bibitem{lysenko3}
S.~Lysenko, \emph{Moduli of metaplectic bundles on curves and theta-sheaves}, Ann. Sci. \'Ecole Norm. Sup. (4) \textbf{39} (2006), no.~3, 415--466. \MR{2265675}

\bibitem{lysenko}
\bysame, \emph{Geometric theta-lifting for the dual pair {$\Bbb S\Bbb O_{2m},\Bbb S{\rm p}_{2n}$}}, Ann. Sci. \'Ec. Norm. Sup\'er. (4) \textbf{44} (2011), no.~3, 427--493. \MR{2839456}

\bibitem{M}
L.~Morris, \emph{Tamely ramified intertwining algebras}, Inventiones Mathematicae \textbf{114} (1993), no.~1, 1--54.

\bibitem{PR}
G.~Pappas and M.~Rapoport, \emph{Twisted loop groups and their affine flag varieties}, Advances in Mathematics \textbf{219} (2008), no.~1, 118--198.

\bibitem{patterson1}
S.~J. Patterson, \emph{A cubic analogue of the theta series}, J. Reine Angew. Math. \textbf{296} (1977), 125--161. \MR{563068}

\bibitem{patterson2}
\bysame, \emph{A cubic analogue of the theta series. {II}}, J. Reine Angew. Math. \textbf{296} (1977), 217--220. \MR{563069}

\bibitem{Ro}
A.~Roche, \emph{Types and {Hecke} algebras for principal series representations of split reductive p-adic groups}, Annales Scientifiques de l’École Normale Supérieure \textbf{31} (1998), no.~3.

\bibitem{Sa}
G.~Savin, \emph{Local shimura correspondence}, Mathematische Annalen \textbf{280} (1988), 185--190.

\bibitem{Savin}
\bysame, \emph{An analogue of the {W}eil representation for {$G_2$}}, J. Reine Angew. Math. \textbf{434} (1993), 115--126. \MR{1195692}

\bibitem{soergel1990kategorie}
W.~Soergel, \emph{{K}ategorie $\mathscr{O}$, perverse {G}arben und {M}oduln {\"u}ber den {K}oinvarianten zur {W}eylgruppe}, Journal of the American Mathematical Society \textbf{3} (1990), no.~2, 421--445.

\bibitem{We}
M.~Weissman, \emph{L-groups and parameters for covering groups}, Astérisque \textbf{398} (2018), 33--186.

\bibitem{Xia}
J.~Xia, \emph{Equivalence of hecke categories with deeper level structures}, Harvard PhD thesis (2024).

\bibitem{Yang}
D.~Yang, \emph{A filtration on the moduli stack of local systems on the punctured disc}, in preparation.

\bibitem{zhu2020coherent}
X.~Zhu, \emph{Coherent sheaves on the stack of {Langlands} parameters}, arXiv:2008.02998 (2020).

\bibitem{Zh}
\bysame, \emph{Tame categorical local {Langlands} correspondence}, arXiv:2504.07482 (2025).

\end{thebibliography}

\end{document}